%% file: master_BOSV.tex
\documentclass[reqno]{amsart}

\usepackage{amsfonts}
\usepackage{amscd}
\usepackage{amsbsy}
\usepackage{amsxtra}
\usepackage{amssymb}
\usepackage{epsfig}
\usepackage{epic,eepic}
\usepackage{graphicx}
\usepackage[all]{xy}
\usepackage{xypic}
\usepackage{psfrag}
\usepackage{caption}
\usepackage{amsmath}
\usepackage{amsthm}
\usepackage{setspace}
\usepackage{hyperref}
\usepackage{url}
\usepackage{here}
\usepackage{todonotes}
\newlength{\halfbls}
\usepackage{tikz}
\usepackage{tikz}
\usetikzlibrary{calc}
\usetikzlibrary{decorations.pathreplacing,decorations.markings}
\usetikzlibrary{arrows}

\DeclareRobustCommand{\SkipTocEntry}[9]{}

\include{macros_BOSV}

\makeatletter
\renewcommand{\@bibtitlestyle}{%
  \@xp\part\@xp*\@xp{\refname}%
}
\makeatother

\title[Siegel-Veech constants]{
Quasimodularity and large genus limits \\ 
of Siegel-Veech constants}

\begin{document}
\author{Dawei Chen, Martin M\"oller and Don Zagier}
\maketitle

\tableofcontents

\noindent
\SaveTocDepth{1} 

\newpage
\begin{abstract}
Quasimodular forms were first studied systematically in the context of
counting torus coverings. Here we show that a weighted version 
of these coverings with Siegel-Veech weights also provides
quasimodular forms. We apply this to prove conjectures of Eskin and 
{Zorich on} the large genus limits of Masur-Veech volumes 
and of Siegel-Veech constants.
\par
In Part~I we connect the geometric definition of Siegel-Veech constants
both with a combinatorial counting problem and with intersection numbers on
Hurwitz spaces. We also {introduce certain} modified Siegel-Veech weights 
{whose generating functions will later be shown to be quasimodular}.
\par
Parts II and III are devoted to the study of the (quasi) modular properties 
of the generating functions arising {from weighted} counting of
{torus coverings.  These two parts} contain little geometry and can be read independently of the rest of the paper.  The
starting point is the theorem of Bloch and Okounkov {saying} that certain
weighted averages, called $q$-brackets, of shifted symmetric functions 
on { partitions are} quasimodular forms.
In Part II we give an expression for the growth polynomials 
(a certain polynomial invariant of quasimodular forms) of these $q$-brackets 
in terms of Gaussian integrals and use this to obtain a closed formula 
for the generating series of cumulants that is the basis for studying large 
genus asymptotics. In Part III we show that the even hook-length moments of partitions 
are shifted symmetric polynomials and prove a surprising formula for the
$q$-bracket of the product of such a hook-length moment with an arbitrary
shifted symmetric polynomial as a linear combination of derivatives of
Eisenstein series.  This {formula gives} a quasimodularity statement {also}
for the $(-2)$-nd hook-length moments by an appropriate extrapolation, {and
this in turn implies} the quasimodularity of the Siegel-Veech weighted counting functions.
\par
Finally, in Part IV these results are used to give explicit generating functions
for the volumes and Siegel-Veech constants in the case of the principal stratum of abelian differentials. 
The {generating functions have an amusing} form in terms of the inversion of 
a power series {(with multiples of Bernoulli numbers as coefficients)} that 
gives the asymptotic expansion of a Hurwitz zeta function.  To apply these
exact formulas to the Eskin-Zorich conjectures on large genus asymptotics 
(both for the volume and the Siegel-Veech constant) we provide 
in a separate appendix 
a general framework for computing the asymptotics of rapidly divergent power series. 
\end{abstract}

\part*{Introduction} \label{sec:intro}

\input{intro}

\input{Part1_BOSV}
\input{Part2_BOSV}
\input{Part3_BOSV}

\input{Part4_BOSV}


\bibliography{my}
\bibliographystyle{plain}

\newpage
\input{appendix}
\end{document}

%% file: macros_BOSV.tex
\newtheorem{Defi}{Definition}[section]  
    \newtheorem{Prop}[Defi]{Proposition}
\newtheorem{Lemma}[Defi]{Lemma}    \newtheorem{Cor}[Defi]{Corollary}
\newtheorem{Thm}[Defi]{Theorem}

\newtheorem{Conj}[Defi]{Conjecture}

\newcommand{\CC}{\mathbb{C}}

\newcommand{\HH}{\mathbb{H}}

\newcommand{\NN}{\mathbb{N}}
 
\newcommand{\PP}{\mathbb{P}} 
\newcommand{\QQ}{\mathbb{Q}} 
\newcommand{\RR}{\mathbb{R}}

\newcommand{\ZZ}{\mathbb{Z}}
\def\Z{\ZZ} 

\newcommand{\frakI}{\mathfrak{I}}

\newcommand{\fZ}{\mathfrak{Z}}

\def\fd{\mathfrak d}

\newcommand{\cFF}{{\mathcal F}}  
 \newcommand{\cXX}{{\mathcal X}}  
\newcommand{\cCC}{{\mathcal C}} 
\newcommand{\cBB}{{\mathcal B}} 
 \newcommand{\cOO}{{\mathcal O}}
  
 \newcommand{\sgn}{{\rm sgn}}

\newcommand{\bfa}{{\bf a}}

\newcommand{\dd}{{\bf d}}
\newcommand{\bfm}{{\bf m}}
\newcommand{\bfn}{{\bf n}}

\newcommand{\bfu}{{\bf u}}

\renewcommand\P{{\bf P}}

\newcommand{\bfW}{{\boldsymbol{W}}}

\newcommand{\ual}{{\boldsymbol{\alpha}}}
\newcommand{\ube}{{\boldsymbol{\beta}}}
\newcommand{\uga}{{\boldsymbol{\gamma}}}

\newcommand{\usi}{{\boldsymbol{\sigma}}}
\def\DS{\bf D\bf S} \def\DE{\bf D\bf E}

\newcommand{\Aut}{{\rm Aut}}

\newcommand{\na}{{\rm na}} 
\newcommand{\cyl}{{\rm cyl}} 
\newcommand{\sac}{{\rm sc}} 
 
\newcommand{\area}{{\rm area}} 
\newcommand{\Cov}{{\rm Hur}} 

  \newcommand{\Tr}{{\rm Tr}}

\newcommand{\other}{\operatorname{other}}

\newcommand{\wgt}{{\rm wt}} 
\newcommand{\irr}{\operatorname{irr}}
\DeclareMathOperator{\vol}{vol}
\DeclareMathOperator{\Res}{Res}

\def\={\;=\;}  \def\+{\,+\,} \def\m{\,-\,}       \def\h{\tfrac12}  
\def\i{^{-1}}      \def\ssm{\smallsetminus}  
  \def\c{\circ} \def\es{\emptyset} \def\sse{\subseteq}  
  
 \def\mapt#1{\overset{#1}\mapsto}

\def\pp{{\boldsymbol \partial}}  

\DeclareMathAlphabet{\eucal}{U}{eus}{m}{n}
\DeclareMathAlphabet{\newcal}{U}{dutchcal}{m}{n}
\newcommand{\PPP}{{\mathcal P}}

\def\RRR{{\newcal R}}

\def\Z{\ZZ} \def\Q{\QQ}   \def\fd{\mathfrak d} \def\pp{\partial}
  
\def\ev{{\rm ev}}  \def\Ev{{\rm Ev}}  \def\evX{\Ev}  \def\evh{\ev}
   \def\sbrX#1{\langle#1\rangle_X}
   \def\sbrh#1{\langle#1\rangle_h}
\def\Gv{\mathfrak G} \def\Gva{\mathfrak G_{asy}} 
 \def\Br{\mathfrak{B}} \def\br{\mathfrak{b}}
\def\wta{\widetilde a} \def\wtb{\widetilde b} \def\wtc{\widetilde c}
\def\la{\langle}  \def\ra{\rangle}

\def\Gva{\mathfrak G_{\rm asy}}

 \def\t#1{\tilde{#1}}  \def\wt#1{\widetilde{#1}}     
\newcommand{\ol}{\overline}

\def\wT{\wt T}  \def\wM{\wt M} 

\def\wF{\widehat F}  
\def\wG{\widehat G}  

\def\SL#1{{\rm SL}(2,#1)}       
\newcommand{\SO}{{\rm SO}}

\DeclareMathOperator{\lda}{\langle\langle}
\DeclareMathOperator{\rda}{\rangle\rangle}
\def\bq#1{\bigl\langle#1\bigr\rangle_q}   \def\sbq#1{\langle#1\rangle_q}
\def\bqs#1{\bigl\langle#1\bigr\rangle_q^\star}   \def\sbqs#1{\langle#1\rangle_q^\star}  
\def\bL#1{\bigl\langle#1\bigr\rangle_L}   
   
\def\rs{\rho^\star}

\def\a{\alpha} \def\b{\beta}  \def\th{\theta}    \newcommand{\ve}{{\varepsilon}}
\def\g{\gamma}   \def\l{\lambda} 
  
   \def\t{\tau} \def\z{\zeta}  \def\p{\partial}  \def\ph{\varphi}
\def\G{\Gamma}  \def\L{\Lambda} \def\D{\Delta}     \def\Om{\Omega}

\def\be{\begin{equation}}   \def\ee{\end{equation}}     \def\bes{\begin{equation*}}    \def\ees{\end{equation*}}
\def\ba{\be\begin{aligned}} \def\ea{\end{aligned}\ee}   \def\bas{\bes\begin{aligned}}  \def\eas{\end{aligned}\ees}

\newcommand{\proj}{{\mathbb P}}
\newcommand{\moduli}[1][g]{{\mathcal M}_{#1}}
\newcommand{\omoduli}[1][g]{{\Omega \mathcal M}_{#1}}
\newcommand{\oamoduli}[1][g]{{\Omega_1 \mathcal M}_{#1}}
\newcommand{\barmoduli}[1][g]{{\overline{\mathcal M}}_{#1}}

\newcommand{\pobarmoduli}[1][g]{{\proj\Omega\overline{\mathcal M}}_{#1}}

\newcommand{\Teichmuller}{Teich\-m\"uller}
\newcommand{\BH}{\overline{H}}
\newcommand{\Hmu}{\Pi}
\newcommand{\he}{{\rm h}{\rm t}} 
\newcommand{\hgt}{{\rm ht}} 
\newcommand{\pd}{d}

\newcommand{\Part}{{\P}} 

\newcounter{savedtocdepth}
\newcommand*{\SaveTocDepth}[1]{%
  \addtocontents{toc}{%
    \protect\setcounter{savedtocdepth}{\protect\value{tocdepth}}%
    \protect\setcounter{tocdepth}{#1}%
  }%
}

\newcommand{\one}{{\bf 1}_N}

\newcommand{\tens}[1]{%
  \mathbin{\mathop{\otimes}\limits_{#1}}%
}


%% file: intro.tex
This paper grew out of an attempt to understand the algebraic and combinatorial 
nature of Siegel-Veech constants on flat surfaces (Part I) and culminates
in a proof of the Eskin-Zorich conjecture (\cite{ezvol}) on large genus asymptotics 
of Masur-Veech volumes and Siegel-Veech constants for
the case of the principal stratum of abelian differentials (Part IV).
Along the way we discovered properties of Bloch-Okounkov correlators
and growth polynomials of quasimodular forms, of interest independently
of the geometric background. 
Consequently, we start with the motivation through  Siegel-Veech constants
but a reader with focus on Bloch-Okounkov correlators and quasimodular
forms may skip to the presentation of Part II and Part III below, where
no background on flat surfaces is required.
\par
\smallskip
{\bf Part I: Siegel-Veech constants on Hurwitz spaces.}
The number of closed geodesics of bounded length on a flat surface, 
i.e.\ a Riemann surface with a flat metric (induced from an abelian differential), 
has quadratic growth. The moduli space of flat surfaces is stratified by 
the number and multiplicities of zeros of the differential and the leading term
of the quadratic asymptotic is the same for all generic flat surfaces in a given
stratum. This leading term is called the Siegel-Veech constant (\cite{veech82}, \cite{eskinmasur}). 
In fact, there are several variants of Siegel-Veech constants (e.g.\ \cite{vorobets} and \cite{baugou})
obtained by counting the trajectories with different weights. Among them is the {\em area Siegel-Veech
constant} (see Section~\ref{sec:SVconf} for the definition), whose importance is 
due to the connection with intersection numbers on the moduli space of curves
and with Lyapunov exponents. We will focus on the area Siegel-Veech
constant throughout the paper.
\par
The strata of the moduli space of flat surfaces have  an integral affine structure
and thus a natural volume form, due to Masur and Veech. 
The area Siegel-Veech constants for strata have been computed recursively using Masur-Veech volumes
of strata in low genera by Eskin-Masur-Zorich (\cite{emz}). This procedure is combinatorially
quite involved and sheds little light on the algebro-geometric significance of Siegel-Veech constants.
\par
Now consider the {\em Hurwitz space} $H_d(\Hmu)$  of degree~$d$ torus coverings with ramification 
profile~$\Hmu$ (see Section~\ref{sec:Hurwitz} for the background and notation). These spaces
are dense in every stratum and the same definition of area Siegel-Veech constants through
quadratic asymptotics applies here as well. The advantage of Hurwitz spaces is that there we
can provide  a transparent combinatorial and intersection-theoretic explanation
of Siegel-Veech constants. We define the $p$-th {\em part-length moment} 
of a partition~$\alpha$ to be
\be \label{intro:SVweight} 
S_p(\alpha) \= \sum_{i=1}^r \alpha_i^p, \qquad (\alpha = (\alpha_1,\alpha_2,\ldots,\alpha_r),\;\, \alpha_i \geq 1)\,, \ee
where $p$ is any complex number.  {{A torus covering in $H_d(\Hmu)$ can be described by the associated Hurwitz tuple $(\alpha,\beta,\gamma_i)$ where $\alpha$, $\beta$ and $\gamma_i$ are elements in $S_d$ arising from monodromy of the covering (see \eqref{eq:HT} for the definition).}}
Let the {\em $p$-weighted Siegel-Veech constant} $c_{p}^0(d, \Hmu)$ 
of a Hurwitz space be the sum of $S_{p}(\alpha)$ over all Hurwitz tuples $(\alpha,\beta,\gamma_i)$ 
for $H_d(\Hmu)$ and $N^0_d(\Hmu)$ the number of these tuples. Using the case $p=-1$, in Theorem~\ref{thm:SV} we give the 
following combinatorial formula for Siegel-Veech constants on Hurwitz spaces: 
\par
\begin{itemize}
\item[$\bullet$] The area Siegel-Veech constant $c_{\rm area}$ for a Hurwitz space $H_d(\Hmu)$ 
is equal to 
$$ c_{\rm area}(d, \Hmu)\=\frac{3}{\pi^2} \, \frac{c_{-1}^0(d, \Hmu)}{N^0_d(\Hmu)}\,.$$
\end{itemize}
The proof is a standard application of the Siegel-Veech transform. Defining and using the 
 $p$-weighted Siegel-Veech constants for all~$p \in \ZZ$ will be crucial in Part~III, although 
we are not aware of a flat geometric interpretation of the counting functions for $p \neq -1$.
\par
The sum of Lyapunov exponents for the Teichm\"uller geodesic flow is another quantity
of dynamic origin, defined for strata and Hurwitz spaces (in fact for any $\SL\RR$-invariant
submanifold of strata), whose algebraic nature still awaits to be understood completely.
Here we do not rely on the dynamic definition of Lyapunov exponents via growth rates of cohomology classes 
(see \cite{zorich06} for more detail). The point of departure is rather the reinterpretation 
of Kontsevich-Zorich (\cite{kontsevich}) for the sum of Lyapunov exponents
as a ratio of two intersection numbers with a foliation 
class~$\beta$ (recalled in Section~\ref{sec:SVtoLyap}). Intersection with~$\beta$ is well-defined as a transverse measure class, but since an
interpretation of~$\beta$ as a rational cohomology class is still missing, there is currently no
direct algebraic proof of the rationality of the sum of Lyapunov exponents for strata. However, 
Eskin-Kontsevich-Zorich (\cite{ekz}) managed to prove this indirectly with a beautiful generalization
of Noether's formula (recalled in~\eqref{eq:ekzmain} below), showing that the sum of Lyapunov exponents 
differs from the area Siegel-Veech constant by an easily computable rational number, an evaluation 
of the $\kappa$-class. 
\par
In the case of Hurwitz spaces we show that all of the above quantities have transparent algebro-geometric
interpretations. In Section~\ref{sec:beta} we show that $\beta$ is indeed proportional to 
a cohomology class and relate~$\beta$ to the tautological classes $\psi_i$ (see 
Theorem~\ref{thm:betainH2}): 
\par
\begin{itemize}
\item[$\bullet$] On the moduli space $\barmoduli[{1,n}]$ the classes $\beta$ and
$\psi_2\cdots \psi_n$ are proportional.
\end{itemize}
\par
Moreover, the Siegel-Veech constant $c_{-1}^0(d, \Hmu)$ with weight $p=-1$ appears in the pushforward of
the nodal locus in the universal curve over the Hurwitz space to $\barmoduli[{1,n}]$, namely
as coefficient of the boundary divisor~$\delta_{{\rm irr}}$ (Theorem~\ref{thm:intnumber}).
Combining these observations, we give a proof of the main result of \cite{ekz} 
for Hurwitz spaces using only intersection theory calculations (Theorem~\ref{thm:hurwitzLckappa}).
\par
In order to understand the combinatorial nature of the $p$-weighted Siegel-Veech constants 
we form the generating series 
$$c_{p}(\Hmu) = \sum_{d \geq 1} c_{p}^0(d,\Hmu) \,q^d.$$
As we will explain in the motivation for Part~II in more detail, the generating function for 
counting covers without weights is a {\em quasimodular form} for $\SL\ZZ$, i.e.\
a polynomial in the Eisenstein series $E_2$, $E_4$, and $E_6$. Our first main structure
result (Theorem~\ref{thm:prstr_minus1}) states that quasimodularity still holds with odd 
Siegel-Veech weight $p \geq -1$, 
despite the unusual counting involving inverses (if $p=-1$) of part-lengths: 
\par
\begin{itemize}
\item[$\bullet$] The generating series of Siegel-Veech constants $c_{-1}(\Hmu)$ with weight
\hbox{$p=-1$} for Hurwitz spaces
in the stratum $\omoduli[g](m_1,\ldots,m_n)$ as well as its \hbox{$p$-weighted} variants for odd $p>0$
are quasimodular forms of mixed weight $\leq p+1 + \sum_{i=1}^n (m_i + 2)$.
\end{itemize}
\par
The proof of this result  requires all the material of 
Part~III and will be completed only in Section~\ref{sec:applSV}. In covering theory
it is a standard argument that coverings without unramified components can be counted
by counting all coverings and then dividing by the partition function. The generating
series for counting connected coverings is then obtained by their linear combinations. 
In Proposition~\ref{prop:cpconversion} we show that a similar procedure works
in the presence of a Siegel-Veech weight, though with a different formula since Siegel-Veech
weights are additive (rather than multiplicative) on disjoint unions of partitions. 
Consequently, we need to understand $q$-brackets to prove  Theorem~\ref{thm:prstr_minus1}.
\par
\medskip
{\bf Part II: Bloch-Okounkov correlators and their growth polynomials.}
The point of departure for Part II is a beautiful theorem of Bloch and Okounkov (\cite{blochokounkov})
saying that the $q$-bracket 
\bes 
  \sbq f \= \frac{\sum_{\l\in\P} f(\l)\,q^{|\l|}}{\sum_{\l\in\P} q^{|\l|}}\;\,\in\;\QQ[[q]]\,,
\ees
of any ``shifted symmetric polynomial~$f$" on the set of all partitions is a quasimodular form. 
This theorem continued the ideas of Dijkgraaf (\cite{Dijkgraaf}, with a rigorous proof given in~\cite{KanZag}),
that the generating series for the number of connected covers of a torus with simple 
branching is a quasimodular form. Eskin and Okounkov (\cite{eo}) used the Bloch-Okounkov
theorem to show that quasimodularity holds for any type of branching profile.
We recall in Sections~\ref{sec:Partitions} and~\ref{sec:QMofqb} the background on the
ring $\RRR$ of shifted symmetric polynomials, on quasimodular forms, and the Bloch-Okounkov theorem.
We refer to the generating functions $F(z_1,\ldots,z_n)$ of $q$-brackets
for a fixed number~$n$ of monomials as {\em Bloch-Okounkov correlators}.
\par
Understanding the quasimodular forms arising this way is difficult, even though
e.g.\ the top term as a polynomial in $E_2$ had been computed in~\cite{KanZag}. 
However, there is a ring homomorphism~$\evX$ associating to each quasimodular form a 
``growth polynomial" (given on generators by $E_4 \mapsto X^2$, $E_6 \mapsto X^3$, while
$E_2 \mapsto X + 12$) that governs the growth of its Fourier coefficients {{and describes the asymptotic behavior 
of the quasimodular form near the cusp}} 
(Proposition~\ref{prop:evAsy}):
\begin{itemize}
\item[$\bullet$] Let $F$ be a quasimodular form of weight~$k$ with $\evX[F] = AX^h + \cdots$ 
and the leading coefficient $A \neq 0$. Then the sum of the first~$N$ Fourier coefficients of $F$ has the asymptotic 
behaviour
$$ \sum_{n=1}^N a_n(F) \= (-4\pi^2)^h A \frac{N^{h+k}}{(h+k)!} \+ {\rm O}(N^{h+k-1} \log(N))\,. $$
\end{itemize}
The growth polynomial is essentially equivalent to an expansion used by Eskin and 
Okounkov in \cite{eo}, but we give a different presentation and several further properties. 
\par 
The main new ideas of this part start in Section~\ref{sec:growthBO}. Previously in \cite{blochokounkov} and \cite{eo}, the focus had been on the generating 
functions $F(z_1,\ldots,z_n)$ and their $\evX$-images. Instead, we introduce the 
partition function 
\bes \Phi(\bfu)_q
\= \Bigl\langle \exp\Bigl(\sum_{\ell \geq 1} p_\ell\, u_\ell\Bigr)\,\Bigr\rangle_q
  \=\sum_{\bfn\geq 0}\,\langle\underbrace{p_1,\ldots,p_1}_{n_1},\underbrace{p_2,\ldots,p_2}_{n_2},\ldots\rangle_q \,\frac{\bfu^\bfn}{\bfn!} 
   \nonumber  
\ees
where {$p_\ell$} are power sum generators of the algebra of shifted symmetric polynomials.
After passing to the growth polynomial (and only then!) the structure of this partition
function becomes transparent (Theorem~\ref{thm:hbracketbfU}):
\par
\begin{itemize}
\item[$\bullet$] The $\evX$-image $\Phi(\bfu)_X  = \evX[\Phi(\bfu)_q]$ 
of the partition function can be expressed as the formal
Gaussian integral 
\be \label{intro:Gauss}
\Phi(\bfu)_X  \= \frac1{\sqrt{2\pi}}\int_{-\infty}^\infty e^{-y^2/2 \+ \cBB(\bfu,iy,X)}\,dy\;. \ee
where we use the coefficients of $\sum_{k\geq 0} \beta_k z^k = \tfrac{z/2}{\sinh(z/2)}$ to define
\bes \cBB(\bfu,y,X) \= \sum_{\genfrac{}{}{0pt}{2}{ \bfa > 0 }{r \geq 0}}
(a_1+2a_2+3a_3+\cdots)!\,\, \beta_{2-r+w(\bfa)} \sqrt{X}^{2-r + w(\bfa)}\, \frac{\bfu^\bfa}{\bfa!}\frac{y^r}{r!}\,,
\ees
{{with $w({\bf a}) = a_2 + 2a_3 + 3a_4 + \cdots$. }}
\end{itemize}
Note that the right hand side of~\eqref{intro:Gauss} is purely algebraic and does not really involve integration.
Our proof of this theorem uses the formula for $\evX[F(z_1,\ldots,z_n)]$ of Eskin-Okounkov, 
for which we also give an independent proof in Theorem~\ref{thm:EO}.
\par
In Section~\ref{sec:gensercumu} we apply these results to the computation of {\em connected brackets}
\bes
\langle f_1|\ldots|f_n\rangle_q \= \sum_{\alpha \in \PPP(n)} (-1)^{|\a|-1} 
(|\a|-1)!\, \prod_{A\in\a} \Bigl\la \prod_{a\in A} f_a \Bigr\ra_q
\ees
of shifted symmetric polynomials $f_i$. (Here $\PPP(n)$ is the set of partitions of the set~$\{1,\ldots,n\}$.) 
Geometrically these arise when counting 
{\em connected} coverings (compare the definition to~\eqref{eq:NNfromNpr}). The connected
brackets involving only the algebra generators~$p_\ell$ all appear in the generating function 
\bas
\Psi(\bfu)_q \,:=\, \sum_{n=0}^\infty \frac1{n!} \sum_{\ell_1,\ldots,\ell_n \geq 1} 
   \la p_{\ell_1}|\cdots|p_{\ell_n} \ra_q\, u_{\ell_1}\cdots u_{\ell_n}   \= \log \Phi(\bfu) _q\,,
\eas
the logarithm of the partition function introduced above. For all asymptotic questions
the important quantities are the leading terms of the growth polynomials of these connected brackets, 
called {\em cumulants}. We denote the passage from brackets to cumulants by decorating the
function with a subscript~$L$. An efficient evaluation of these cumulants is not obvious, 
since the degree of the growth polynomial drops by one for every insertion of a slash
into a bracket. This was observed in \cite[Theorem~6.2]{eo}, and we provide an independent proof in 
Proposition~\ref{prop:degdrop}. Computationally, the key to the evaluation of cumulants
is the following consequence of the representation of $\Phi(\bfu)_X$ as Gaussian integral
(Theorem~\ref{thm:psi}): 
\begin{itemize}
\item[$\bullet$]  The generating series of cumulants is given by
\be \label{intro:PsiL} \Psi(\bfu)_L = \Bigl. \cBB(\bfu,y_0)   \;+\; \frac{y_0^2}{2} \,,
\ee
where $y_0=y_0(\bfu)$ is the unique power series with $\tfrac{\partial}{\partial y}\cBB(\bfu,y_0)+y_0=0\,$.
\end{itemize}
This is shown  using the principle of least action. We obtain as a
corollary also the generating function of cumulants for a fixed number of variables
(formula~\eqref{eq:cumnfixed}, equivalent to \cite[Theorem~6.7]{eo}), but it is the formula
for $\Psi(\bfu)_L$ that turns out to be useful for all applications to asymptotic questions.
\par
When we specialize to cumulants with only $2$'s with application to the counting problems for 
simple branching in mind, the computation of $\psi(u) = \Psi(\bfu)_L|_{\bfu = (0,u,0,\cdots)}$ allows 
further simplification. While~\eqref{intro:PsiL} is a two-variable expression  even after
this specialization, we {show} in Theorem~\ref{thm:GFvn} that the cumulants
\be \label{eq:introvn}
v_n 
 \= \frac1{n!}\,\la \underbrace{p_2|\cdots|p_2}_n \ra_L \qquad(n>0), \quad v_{-2}=v_0=-\frac1{24},\, v_{-1} =0 
\ee
can be obtained by manipulating only one-variable series, as follows:
\begin{itemize}
\item[$\bullet$] 
 Define a Laurent series
 \bes
\Br_{1/2}(X) \= X^{1/2} \+ \frac{X^{-3/2}}{96} \m \frac{7X^{-7/2}}{6144} \+ \frac{31X^{-11/2}}{65536} \m \cdots \ees
in $X^{-1/2}$ as the unique solution in $X^{-1/2}\QQ[[1/X]]$ of the functional equation
$$  \Br_{1/2}(X+\h) \m \Br_{1/2}(X-\h) \= \frac{X^{-1/2}}2\;. $$
{Then the} rational cumulants $v_n$  are given by the inversion formula
   \be\label{intro:vnGenFn} Y=\Br_{1/2}(X) \;\quad\Longleftrightarrow \;\quad
     X \= \sum_{n=-2}^\infty \frac{2n+1}{2^{2n+1}}\,v_n\,Y^{-2n-2}\;. \ee
\end{itemize}
We mention that the power series $\Br_{1/2}(X)$ gives the asymptotic expansion 
of the Hurwitz zeta function $-\h\,\zeta(\h,X+\h)$ as $X\to\infty$, and that 
its Taylor coefficients are simple multiples of the numbers $\beta_n$ used 
above.
\par
Similar one-variable inversion formulas are given in Theorems~\ref{thm:GFpsik} and~\ref{thm:GFkappa} 
for other linear combinations of cumulants involving $2$'s that are relevant for the computation
of Siegel-Veech asymptotics. The precise form of the linear combinations is motivated
by the operators $T_p$ that appear in Part III.
\par
\medskip
{\bf Part III: The hook-length moments $T_p$.} 
In the Bloch-Okounkov theorem, one obtains quasimodularity for the $q$-brackets of
shifted symmetric polynomials. For our applications we need a quite different looking class 
of functions on partitions, the {\it hook-length moments}
\bes
T_p(\lambda) \= \sum_{\sigma \in Y_\lambda}  h(\sigma)^{p-1} \,,
\ees
where $h(\sigma)$ denotes the hook-length of the cell~$\sigma$ of the Young diagram of $\lambda$.
Surprisingly, half of these functions {\em do} lie in the ring of shifted symmetric polynomials, 
as we show in Theorem~\ref{prop:qofTp}:  
\begin{itemize} 
\item[$\bullet$] For $p \geq 1$ odd, the function $\wT_p(\lambda) = T_p(\lambda) + \tfrac12\zeta(-p)$
is a homogeneous shifted symmetric polynomial of weight~$p+1$, given by
\bes
\wT_p(\lambda) \= \frac{(p-1)!}2\; \sum_{k=0}^{p+1}(-1)^k Q_k(\l)\,Q_{p+1-k}(\l) \,
\ees
where $Q_0 = 1$ and $Q_{\ell+1} = \ell! {\,p_\ell\,}$. 
\end{itemize}
\par
The hook-length moments appear first in our study of Siegel-Veech constants in the amusing 
formula (Corollary~\ref{cor:SumSpTp})
\be \label{intro:SchurToHook}
\frac{1}{\pd!}\sum_{\mu \in \Part(\pd)} S_p(\mu)\, z_\mu\, \chi^\lambda(\mu)^2 \= T_p(\l)\,,
\ee
obtained by inserting the part-length moment~\eqref{intro:SVweight} into the left hand side 
that would simplify to just~$1$ by Schur orthogonality without the weight.
\par
A leitmotiv for Part~III is the observation that many more
functions on partitions than just shifted symmetric polynomial have quasimodular
or nearly quasimodular $q$-brackets. In \cite{zagBO} identities like~\eqref{intro:SchurToHook} were used 
to create many non-trivial examples. Here we are more specifically interested in $T_{-1}$, 
which is certainly not {a} shifted symmetric polynomial. Yet, we will show that 
for $f$ any shifted symmetric polynomial, $\bq{T_{-1}\,f}$ is nearly quasimodular: it is in the 2-dimensional module over quasimodular forms generated 
additively by $1$ and $\log(q^{-1/24}\eta)$. 
\par
The way we show this property of $T_{-1}$ is very indirect, but reveals many beautiful
properties of the operators~$T_p$ and $\wT_p$. We show the quasimodularity of expressions of the form
\be  \label{intro:Tpexpr}
\bq{T_p\,f} - \bq{T_p}\bq{f} \quad \text{for $f$ a shifted symmetric polynomial, \ $p \geq -1$ odd}\,
\ee
by extrapolating a formula for $p>0$ to $p=-1$. The key property of the operators~$\wT_p$, 
discovered experimentally and discussed in detail in Section~\ref{sec:tpop1}, is that
\be \label{intro:effofTp}
\bq{\wT_p\,f} \= \sum_{i,\,j\ge0} \bq{\rho_{i,j}(f)} \,G^{(j)}_{p+i+1}\qquad\text{for all odd $p\ge1$}\,,
\ee
where $G^{(j)}_{k}$ is the $j$-th derivative of the Eisenstein series $G_k$ and where
$\rho_{i,j}: \RRR \to \RRR$ is a differential operator of the form 
\bes 
\rho_{i,j} \= \sum_{k=0}^\infty  Q_k\, \rho_{i,j}^{(k)} \Bigl(\frac{\partial}{\partial p_1}, 
\frac{\partial}{\p p_2},\ldots  \Bigr)\,
\ees
for some polynomials $\rho_{i,j}^{(k)}$ that are given explicitly in Theorem~\ref{thm:rhoijDiffOp}. 
\par
We will give several different descriptions of the operators $\rho_{i,j}$
in Section~\ref{sec:tpop1}. For the proofs, it turns out to be convenient to 
reinterpret the key property~\eqref{intro:effofTp} in terms of the Bloch-Okounkov correlators 
$F(z_1,\ldots,z_n)$. We show in particular that~\eqref{intro:effofTp} is equivalent
to the following statement: 
\par
\begin{itemize} \item[$\bullet$]
A correlator with two arguments $u$ and $-u$ that add up to zero can be expressed 
in terms of certain nearly-elliptic functions of one variable $Z_\ell(u)$ given 
explicitly in~\eqref{Zag01} and correlators not involving~$u$ by the formula
\be  \label{intro:uminusu}
 F(u,-u,\fZ_N) \= \sum_{I\sse J\sse N\atop\ve\in\{\pm1\}} (-1)^{|J\ssm I|} \,Z_{|J|}(z_I+\ve u)\, F(\fZ_{N\ssm J},z_J)\,,
\ee
where $N=\{1,\dots,n\}$, $\fZ_J  = (z_j, j\in J)$, and $z_J = \sum_{j \in J} z_j$ for~$J\sse N$. 
\end{itemize}
\par
The basic strategy to prove such identities is to show that both 
sides have the same elliptic transformation laws, the same poles, 
and that they agree at one point. This idea has already been used
in \cite{blochokounkov} and the formulas for the elliptic transformation laws
of~$F$ are given there. They involve summing the contributions
of correlators for all subsets of arguments. But~\eqref{intro:uminusu} 
as it stands is completely inadapted to recursive arguments. To
overcome this, we give in Theorem~\ref{thm:mainconjuv} a formula to express a 
Bloch-Okounkov correlator involving two distinguished variables $u$ and~$v$
as a linear combination of products of a correlator involving only $u+v$ 
and a nearly elliptic function $Z_\ell$ involving only one of the 
variables $u$ and~$v$. This formula specializes to~\eqref{intro:uminusu}
for $v=-u$ and allows for a straightforward (though somewhat tedious) proof
following the basic strategy outlined above.
\par
The formula~\eqref{intro:effofTp} enables us to extrapolate in Section~\ref{sec:applSV} 
the effect of $T_p$ to $p=-1$, to prove the quasimodularity of~\eqref{intro:Tpexpr} 
also for $p=-1$, and thus to complete the proof of Theorem~\ref{thm:prstr_minus1} on the quasimodularity
of Siegel-Veech weighted counting functions announced at the end of Part~I.
\par
\medskip
{\bf Part IV: Volumes and Siegel-Veech constants for large genus.} In this part
we come back to the geometric applications. So far, in Part~I, we have been
studying one Hurwitz space at a time, but we have packaged the resulting functions
into generating series. It was the motivation of the work of Eskin-Okounkov (\cite{eo})
that the Masur-Veech volume of a stratum can be expressed as the limit of volumes
of the Hurwitz spaces contained in that stratum, and hence in terms of
cumulants (see\ formula~\eqref{eq:EOvol}). A similar statement also holds for Siegel-Veech 
constants. It appeared for arithmetic Teichm\"uller curves in the appendix of \cite{chenrigid}, and we 
give a self-contained statement and proof in Proposition~\ref{prop:HurStrLim}.
We also mention in Section~\ref{sec:HurtoStrata} an interpretation of the non-varying 
phenomenom for the sum of Lyapunov exponents (\cite{cmNV}) in the light of the quasimodularity 
theorem for Siegel-Veech generating series. 
\par
\smallskip
The main goal of Part IV is to study the large genus limits of both Masur-Veech
volumes and Siegel-Veech constants. Large genus geometry of the moduli space 
has already attracted a lot of attention in the parallel world of Weil-Peterson
volumes (\cite{kmz}, \cite{mizo}) and also in algebraic geometry in the form of the 
slope conjecture (\cite{hamoslope}, \cite{FaPo}, see \cite{chenrigid} for 
some connections), and it is natural to ask similar questions in Teichm\"uller geometry. 
\par
Based on numerical data Eskin and Zorich conjectured (around the time \cite{emz} was 
written, see \cite{ezvol} for more detail) the following asymptotic behaviour. The volumes of the stata 
(in the normalization of \cite{emz}) are conjecturally 
$${\rm vol}_{{\rm EMZ}}\,(\omoduli(m_1,\ldots,m_n)) \; \sim \; 
\frac{4}{(m_1+1)(m_2+1)\cdots (m_n+1)} \;+\; \text o(1)$$ 
as $\sum m_i =2g-2$ tends to infinity. Moreover, except for hyperelliptic components
of strata, they conjecture 
$$ \lim c_\area(\omoduli(m_1,\ldots,m_n)) = \frac{1}{2}$$
uniformly as $\sum m_i =2g-2$ tends to infinity. To avoid making this paper even
longer than it already is, we have focused on the principal stratum to prove the
two conjectures, with full asymptotic expansions in both cases. 
\par
For volumes, we are led by the Eskin-Okounkov formula to compute the asymptotics as $n \to \infty$ of the cumulants $v_n$ introduced
in~\eqref{eq:introvn}. The formula~\eqref{intro:vnGenFn} starts with a power series 
of known asymtotics (involving just factorials and Bernoulli numbers), but we
are then required to perform operations such as taking powers and compositional
inverses to arrive at $v_n$. Such a formula seems at first glance rather unsuitable 
for asymptotic calculations. However, the exact contrary is the case, by the following
mechanism of asymptotics of rapidly divergent power series.
\par
In the appendix 
we consider power series $f = \sum a_n x^n$ that have 
an asymptotic expansion of the form
\be \label{into:asym}  a_n \;\sim \; 
  n!^\a\b^nn^\g\Bigl( A_0 + \frac{A_1}{n} + \frac{A_2}{n^2} + \cdots \Bigr) \ee
for $\alpha >0$ and $\beta >0$. In all the applications to volumes and Siegel-Veech constants
we will have $\alpha = 2$. Series of this type are sometimes called of {\em Gevrey 
order~$\alpha$} in the literature. The fact that products of such power series, and hence 
positive powers, are again of Gevrey order $\alpha$ is certainly well-known, 
but even for these cases the fact that the full asymptotic expansion can be calculated is hard to find in the literature. In fact, due to the rapid growth
of $n!^\alpha$ only the first and the last terms in the formula for the product matter.
Our new observation is that if $\alpha >1$ then a similar principle holds also for the composition 
of two power series of Gevrey order~$\alpha$ and for the functional inverse
of such a series.  The proofs in both cases require a more delicate uniform estimate of the 
asyptotic growth of the Taylor coefficients of large powers of power series of Gevrey order~$>1$,
together with an application of Lagrange inversion for the case of the functional inverse.
A typical result is:
\par
\begin{itemize}
\item[$\bullet$] If  $f = \sum a_n x^n$ ($a_0=0$, $a_1=1$) has coefficients with an asymptotic expansion 
of the form~\eqref{into:asym} with $\alpha =2$, then the coefficients of the functional inverse
$f^{-1} = \sum b_n x^n$ have an asymptotic expansion of the same form, beginning  
$$ b_n \;\sim \;   n!^2\b^nn^\g\Bigl( - A_0 \,+\, \frac{\beta^{-1} a_2A_0- A_1}{n}  \+ \cdots \Bigr)\,. $$
\end{itemize}
The full statement about which Gevrey classes are closed under composition is given
in Theorem~\ref{closed}. 
\par
In Section~\ref{sec:AsHur} we apply the results on rapidly divergent series 
to the asymptotics of cumulants. For example, we compute 
(Theorem~\ref{thm:bmasy} combined with Theorem~\ref{thm:GFpsik}) that 
for $k$~fixed and $h \to \infty$ 
$$ \,\la p_{k-1} | \underbrace{p_2|\cdots|p_2}_{2h-k} \ra_L  \;\sim\; 
\frac{(-1)^h}{k\cdot 2^k} \,\frac{(2h)!^2}{h^{3/2}} 
\Bigl(\frac{2}{\pi}\Bigr)^{2h+\h}
\Bigl(1 \, -\,  \frac{2\pi^2 - 6k^2-6k-3}{48h} \+ \cdots \Bigr) 
\,.$$
For the volumes of the principal strata, it now suffices to put the pieces together.
For Siegel-Veech constants, the remaining step is to write the leading coefficients 
$c_{-1}^0(\Tr^n)$ of the generating function of Siegel-Veech constants with weight $p=-1$
and ramification
profile $\Hmu = \Tr^n$ consisting of transpositions as well in terms of cumulants, 
see Theorem~\ref{thm:asySVm1}. From this, we deduce the final result: 
\begin{itemize}
\item[$\bullet$] The Masur-Veech volume of the principal stratum is asymptotically
\bes {\rm vol}(\omoduli(\underbrace{1,\ldots,1}_{2g-2}))  \;\sim\; {\frac{4}{2^{2g-2}}} 
\Bigl(1\,-\, \frac{\pi^2}{24g} \,-\, 
\frac{\pi^4 - 60\pi^2}{1152g^2}\+\cdots\Bigr)  \ees
and the area Siegel-Veech constants of these strata have the asymptotics
\bes
c_\area(\omoduli(1^{2g-2}))\,\sim\, \frac12  \,-\, \frac1{{8}g} \,-\, 
\frac{5}{{{32}}g^2}  \,-\, \frac{4\pi^2 + 75}{{384}g^3}\+ \cdots  \,,
\ees
as $g \to \infty$.
\end{itemize}
\par
We remark that the extrapolation in Part~III from $p>0$ to $p=-1$ works at the level 
of $q$-brackets only, not at the level of shifted symmetric functions (since $T_{-1} 
\not\in \RRR$) and not at the level of cumulants either. To illustrate this, we compute
the asymptotics of the $p$-weighted variant $c_{p}^0(\Tr^n)$ in Corollary~\ref{cor:cpas}.
\par
To settle the Eskin-Zorich conjecture for all strata, one has to combine properties
of the partition function with the base change from the $f_k$-generators of~$\RRR$ 
to the $p_\ell$-generators of~$\RRR$ that appear in the partition function, 
see~\eqref{eq:fkpk} for examples with small~$k$. We plan to come back to this in a sequel to 
this paper. 
\par
\medskip
\medskip
{\bf Acknowledgements.} 
The first-named author was partially supported by NSF under the grant~1200329 and 
the CAREER award~1350396. The second-named author was partially supported by the 
ERC starting grant~257137 ``Flat surfaces" and the DFG-project MO~1884/1-1.
He would also like to 
thank the Max Planck Institute for Mathematics in Bonn, where much of this work was 
done. The authors thank Alex Eskin and Anton Zorich 
for many stimulating discussions on Siegel-Veech constants and related topics,
and also thank the referees for carefully reading 
the paper and many helpful suggestions.  
\par
\vfill
\pagebreak

%% file: Part1_BOSV.tex
\part*{Part I: Siegel-Veech constants on Hurwitz spaces} 
 
We start in Sections~\ref{sec:SVconf} and~\ref{sec:Hurwitz} with an overview of Siegel-Veech constants and Hurwitz spaces of torus coverings to set the scene.
The interpretation of Siegel-Veech constants for such Hurwitz spaces as
combinatorial objects is provided in Section~\ref{sec:SVcount}. The connections
to algebraic geometry of these invariants are given in Section~\ref{sec:Lyap}, where
area Siegel-Veech constants are expressed as boundary contributions on Hurwitz spaces, 
and in Section~\ref{sec:beta}, where the class~$\beta$ of the $\SL\RR$-foliation  
on Hurwitz spaces is expressed in terms of $\psi$-classes.
\par
Starting with Section~\ref{sec:genser} we package the Siegel-Veech constants of 
the individual
Hurwitz spaces into a generating series with respect to the degree of the coverings. The 
combinatorics of Siegel-Veech constants is then cast in the language of 
representation theory. This will be used for the proof of the 
quasimodularity Theorem~\ref{thm:prstr_minus1} at the end of Part~III.

\section{Siegel-Veech constants and configurations } \label{sec:SVconf}

\subsection{Counting problems on flat surfaces} \label{sec:counting_pb}
Let $(X,\omega)$ be a {\em flat surface}, consisting of a Riemann surface~$X$ 
and an abelian differential $\omega$ on $X$. We visualize flat
surfaces as planar polygons glued along their sides by parallel translation as 
in Figure~\ref{fig:cylinders}. The zeros of
$\omega$ are called saddles or singularities of the flat surface. With 
the billiard origin of studying flat surfaces in mind, natural
counting problems arise from that of closed geodesics under the flat metric, as well as counting saddle connections which are 
geodesics joining two given (or any two) saddles on the flat surface. 
\par
For saddle connections we can most
easily define the meaning of the counting problem. We are interested in 
properties of functions like
$$ N_{\sac}(T) \= |\{\gamma \subset X \,\, \text{a saddle connection},\, 
\ell(\gamma) \leq T\}|$$
in the limit as $T \to \infty$, where $\ell(\gamma)$ is the 
flat length of $\gamma$.

\begin{figure}[h]
\begin{centering}
\begin{tikzpicture}
   
\draw (0,2) -- (0,4) -- (4,4) -- (4,0) -- (6,0) -- (6,2) -- (0,2);
\draw (2,4) -- (2,0) -- (4,0);
\draw (0,3) -- (2,4)
      (0,3.1) -- (1.8,4)
      (0,3.2) -- (1.6,4)
      (0,3.3) -- (1.4,4)
      (0,3.4) -- (1.2,4)
      (0,3.5) -- (1,4)
      (0,3.6) -- (.8,4)
      (0,3.7) -- (.6,4)
      (0,3.8) -- (.4,4)
      (0,3.9) -- (.2,4);
\draw (0,2) -- (4,4)
      (.2,2) -- (4,3.9)
      (.4,2) -- (4,3.8)
      (.6,2) -- (4,3.7)
      (.8,2) -- (4,3.6)
      (1,2) -- (4,3.5)
      (1.2,2) -- (4,3.4)
      (1.4,2) -- (4,3.3)
      (1.6,2) -- (4,3.2)
      (1.8,2) -- (4,3.1)
      (2,2) -- (4,3);
\draw (2,1) -- (4,2)
      (2,.9) -- (4.2,2)
      (2,.8) -- (4.4,2)
      (2,.7) -- (4.6,2)
      (2,.6) -- (4.8,2)
      (2,.5) -- (5,2)
      (2,.4) -- (5.2,2)
      (2,.3) -- (5.4,2)
      (2,.2) -- (5.6,2)
      (2,.1) -- (5.8,2)
      (2,0) -- (6,2);
\draw (4,0) -- (6,1)
      (4.2,0) -- (6,.9)
      (4.4,0) -- (6,.8)
      (4.6,0) -- (6,.7)
      (4.8,0) -- (6,.6) 
      (5,0) -- (6,.5)
      (5.2,0) -- (6,.4)
      (5.4,0) -- (6,.3)
      (5.6,0) -- (6,.2)
      (5.8,0) -- (6,.1)
      (6,0) -- (6,0);

\tikzstyle{every node}=[font=\scriptsize] 
\node (1) at (4.2,3) {$1$};
\node (1) at (-0.2,3) {$1$};
\node (2) at (1.8,1) {$2$};
\node (2) at (6.2,1) {$2$};
\node (3) at (1,4.2) {$3$};
\node (3) at (1,1.8) {$3$};
\node (4) at (3,4.2) {$4$};
\node (4) at (3,-0.2) {$4$};
\node (5) at (5,2.2) {$5$};
\node (5) at (5,-0.2) {$5$};
\end{tikzpicture} 
\end{centering}
\caption{Some short cylinders on a flat surface} \label{fig:cylinders}
\end{figure}
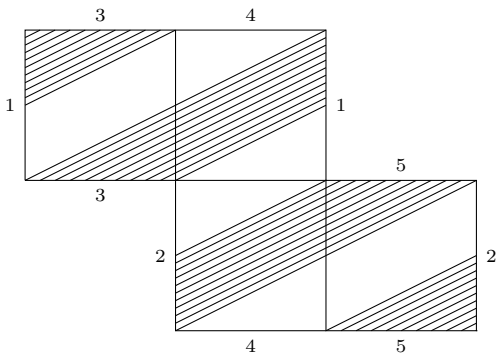
\par
Quadratic upper and lower bounds for such counting functions were established
by Masur (\cite{masur90}). Fundamental works of Veech (\cite{veech98}) 
and Eskin-Masur (\cite{eskinmasur}) showed that for almost every 
surface $(X,\omega)$ {{in the sense of the Masur-Veech measure (\cite{masur82}, \cite{veech82})}} 
there is a quadratic asymptotic, i.e. that
$$N_{\sac}(T) \, \sim \, c_{\sac}(X,\omega) T^2\,.$$ 
The constant $c_{\sac}(X,\omega)$ is the first example of a Siegel-Veech constant,
the one for (any type of) saddle connection. This notion was 
formalized by Eskin-Masur (\cite{eskinmasur}) and also by Vorobets 
(\cite{vorobets})
with the result that many natural counting functions satisfy some 
Siegel-Veech axioms and consequently have precise quadratic asymptotics. 
\par
Counting of saddle connections will however not be considered in the sequel, and we
refer to \cite{aez} for the latest results. Back to closed geodesics, note that they come in classes, homotopic to
one another. In other words, one can slide a closed geodesic transversely on
a flat surface (in both orientations) until its translate passes through
a singularity. In this way, the translates sweep out cylinders as 
in Figure~\ref{fig:cylinders}. Counting these cylinders, possibly with
weight, is the right way to interpret counting of closed geodesics. We let
\bes
 N_{\cyl}(T) \= |\{Z \subset X \,\, \text{a cylinder}, \,w(Z) \leq T\}|\,,
\ees
where $w(Z)$ is the width of the cylinder, i.e.\ the flat length
of its core curve. 
\par
Once we discuss (see Theorem~\ref{thm:intnumber}) the connection of the corresponding 
Siegel-Veech constants and intersection numbers
on moduli spaces, it will become clear that it is more natural to
count the cylinders $Z$ with weight $\area(Z)/\area(X)$, i.e.\
\be \label{eq:Narea} 
N_{\area}(T) \= \sum_{Z \subset X  \text{cylinder}, w(Z) \leq T} \frac{\area(Z)}{\area(X)}\,.
\ee
\par
The {\em Siegel-Veech constants} associated to the counting functions are 
\be \label{eq:cylplim}
c_{\cyl}(X,\omega) \= \lim_{T \to \infty} \frac{N_{\cyl}(T)}{\pi T^2}\,, \quad
c_{\area}(X,\omega) \= \lim_{T \to \infty} \frac{N_{\area}(T)}{\pi T^2}\,.
\ee
\par
In view of Section~\ref{sec:SVcount} we remark that 
there are many interesting variants of the counting functions above with 
quadratic asymptotics and which moreover satisfy the axioms
of Siegel-Veech constants in \cite{eskinmasur}. For example one could take
$$
 N_{\area, p}(T) \= \sum_{Z \subset X  \text{cylinder}, w(Z) \leq T} \frac{\area(Z)^p}{\area(X)^p},
$$
However, this does not correspond to the $p$-weighted Siegel-Veech constants
defined in Section~\ref{sec:SVcount}, which rather correspond
to the counting problem   
$$
 N_{p}(T) \= \sum_{Z \subset X  \text{cylinder}, w(Z) \leq T} 
\frac{w(Z)h(Z)^{p+2}}{\area(X)^{(p+3)/2}},
$$
where $h(Z)$ is the height of the cylinder. Note that this counting function 
$N_{p}(T)$ is not $\SL\RR$-equivariant. In particular, it does not satisfy the 
Siegel-Veech axioms. The reason for studying $N_{p}(T)$ will become
apparent in Section~\ref{sec:applSV}.
\par

\subsection{The moduli space of flat surfaces and $\SL{\RR}$ action}

We denote by $\omoduli$ the moduli space of flat surfaces of genus $g \geq 1$. 
It is the total space of the vector bundle $\pi_*(\omega_{\cXX/\moduli})$ over $\moduli$, called the \emph{Hodge bundle}. Here $\omega_{\cXX/\moduli}$ is the 
relative dualizing sheaf associated to the universal curve 
$\pi:\cXX \to \moduli$. The group $\SL{\RR}$ acts on planar polygons and
this action is well-defined also on the resulting flat surfaces. 
We may provide flat surfaces with a finite number of {\em marked points} 
$P_1,\ldots,P_n$ that may coincide with zeros of $\omega$ and vary under the action of 
$\SL{\RR}$. The space $\omoduli$ is stratified according to the number and 
multiplicities of zeros that we denote by $\omoduli(\bfm)$, where 
$\bfm = (m_1\,\ldots,m_n)$ is a partition of $2g-2$. Connected components of 
these 
strata have been classified in \cite{kz03}. There are up to three connected 
components. We will often restrict our attention to the {\em principal 
stratum} $\omoduli(1\,\ldots,1)$, the stratum where all zeros are simple, 
which is connected.
\par
The action of $\SL{\RR}$ obviously preserves the area of a flat surface. 
For this reason, whenever talking about orbit closures, volumes etc, we 
may and will tacitly assume that the invariant manifold is contained 
in the subset {\em $\oamoduli$ of flat surfaces of
area one}. We denote by $\cFF$ the foliation of $\omoduli$ by orbits of
 $\SL{\RR}$.
\par
The classification of $\SL{\RR}$-orbit closures is one of the major problems in the field, 
presently open to  large extent. Significant progress has been made recently by 
Eskin-Mirza\-khani-Mohammadi (\cite{esmi}, \cite{esmimo}) by showing that
orbit closures have a nice geometric structure, i.e.\ they are linear submanifolds of $\omoduli$. 
It has been further shown by Filip (\cite{filip}) that all linear submanifolds are algebraic varieties defined over $\overline{\mathbb Q}$. 
\par
The $\SL{\RR}$-orbit closures come with a natural $\SL\RR$-invariant
measure that we will describe in more detail below in the cases that
are relevant here. It follows from the Siegel-Veech axioms 
(see \cite{eskinmasur}) that Siegel-Veech constants for almost all
flat surfaces $(X,\omega)$ in an $\SL{\RR}$-orbit closure $M$ agree.
We call these surfaces {\em generic} (for $M$). Consequently, we
let
\be
c_\area(M) = c_\area(X,\omega)
\ee
for any $(X,\omega)$ which is generic for $M$.\footnote{It is an
interesting open problem, if $c_\area(M) = c_\area(X,\omega)$ for
any flat surface $(X,\omega)$ such that the closure of 
$\SL{\RR} \cdot (X,\omega)$ is equal to~$M$.}
\par
We will be mainly interested in the Siegel-Veech constants for strata
(since this is the most generic case) and for Hurwitz spaces, as 
introduced below, since they are combinatorially interesting,  
basically the only source of infinitely many proper closed 
$\SL{\RR}$-invariant subsets of strata for all genera and, most importantly, 
their Siegel-Veech constants approach the Siegel-Veech constants 
for strata, as we will show in Section~\ref{sec:HurtoStrata}.

\subsection{Cylinder configurations and Siegel-Veech constants for 
strata: the recursive procedure} \label{sec:cylinderconf}

Eskin-Masur-Zorich (\cite{emz}) give a recipe to calculate Siegel-Veech 
constants for strata recursively. Their result is an effective algorithm
which is nevertheless combinatorially quite involved.  We now explain their basic 
idea. Moreover we formalize the notion of cylinder configurations, which
appears for strata in~\cite{emz}, for general $\SL{\RR}$-invariant 
manifolds in order to apply it later for Hurwitz spaces.
\par
We start by recalling the notion of Siegel-Veech transform.
Let $V = V(X,\omega) \subset \RR^2$ a function that associates with a flat
surface a subset in $\RR^2$ with (real) multiplicities, satisfying the 
Siegel-Veech axioms (see Section~2 in \cite{eskinmasur}). These axioms 
are roughly the $\SL\RR$-equivariance, the quadratic growth rate of $V$, 
and an integrability condition.  The holonomies of 
all saddle connections and all closed geodesics (with multiplicity one or
with multiplicity equal to the area of the ambient cylinder) are examples 
of such functions.  Further examples come from the restriction to only 
those saddle connection vectors that belong to configurations as defined below.
For any function $\chi: \RR^2 \to \RR$
we denote by $\widehat{\chi}$ the Siegel-Veech transform with respect
to $V$, i.e.\
\be \label{eq:SVtransform}
\widehat{\chi}(X,\omega) \= \sum_{v \in V(X,\omega)} \chi(v)\,.
\ee 
Let $\nu$ be a finite $\SL\RR$-invariant measure on a subset of $\oamoduli$ 
whose support  we denote by~$H$. The fundamental results of Veech and 
Eskin-Masur (\cite{veech98}, \cite{eskinmasur}) jointly imply that for appropriate
$V$ and $\nu$ there is a constant $c(\nu,V)$, such that for all functions $\chi$
we have
\be \label{eq:SVbasic}
\frac{1}{\nu(H)} \int_{H} \widehat{\chi} d\nu \= c(\nu,V)  \int_{\RR^2} 
\chi dxdy. 
\ee
In this section  we will use $\nu$ for the Masur-Veech measure (\cite{masur82}, 
\cite{veech82}) on strata. In later sections the support of $\nu$ will be 
on Hurwitz spaces. 
Moreover, if the $\SL\RR$-orbit closure of $(X,\omega)$ is $H$ and if $V$ 
is the set of holonomy vectors of all closed geodesics with 
multiplicity one or multiplicity equal to the area of the ambient cylinder respectively, 
then $c(\nu,V) = c_\cyl(X,\omega)$ resp.\ $c(\nu,V) = c_\area(X,\omega)$ 
(\cite[Theorem~2.1]{eskinmasur}).
\par
In order to make use of~\eqref{eq:SVbasic}, one takes as test function 
$\chi_\ve$, the characteristic
function of a little disc of radius~$\ve$. The right hand side of the equation
is then $\pi\ve^2$ times the constant we are interested in. So we need to
compute the left hand side, in fact up to terms of order ${\rm o}(\ve^2)$.
\par
\medskip
Roughly speaking, a cylinder configuration is the combinatorial datum 
encoding the cylinders in a direction $\theta$ on a flat surface 
$(X,\omega)$. More precisely, a {\em cylinder configuration} (on
a genus~$g$ surface) is a closed subsurface $S \subset \Sigma_g$
together with a graph $\Gamma \subset S$ such that $\Gamma$ contains
the boundary of $S$ and such that the complementary regions, the
connected components of $S \ssm \Gamma$, are open parallel cylinders. {{In particular, boundaries of the cylinders in a cylinder configuration stay parallel and the proportions of their lengths stay fixed under the action of $\SL\RR$. }}
\par
We say that a direction $\theta$ on a flat surface $(X,\omega)$ {\em
belongs to} the cylinder configuration $\cCC = (S,\Gamma)$, if there is
a subset of the cylinders swept out by closed geodesics in the
direction $\theta$ such that the closure of these cylinders is $S$
and such that the saddle connections in $S$ form the graph $\Gamma$. 
\par
Siegel-Veech constants can be refined by counting according to the 
configuration. That is, we define
\be \label{eq:NareaC} 
N_{\area}(T,\cCC) \= \sum_{Z \subset X  \text{cylinder}, w(Z) \leq T
\atop Z \ \text{belongs to}\  \cCC} \frac{\area(Z)}{\area(X)}.
\ee
and, as above,
\be \label{eq:cylplimC}
c_{\area}(X,\omega,\cCC) \= \lim_{T \to \infty} \frac{N_{\area}(T,\cCC)}{\pi T^2}\,,
\quad  
c_{\area}(M,\cCC) \= c_{\area}(X,\omega,\cCC)
\ee
if $(X,\omega)$ is generic in $M$.
\par
Counting according to the configuration will appear in this paper as 
a technical tool. We now formalize that we want to consider only
relevant configurations and that we do not want to miss any configuration.
A {\em full set of  cylinder configurations} for an $\SL\RR$-invariant 
manifold $H$ is a finite set of cylinder configurations $\cCC_i$, $i\in I$ 
with the following properties: 
\begin{itemize}
\item[i)] For each $(X,\omega)$ and each direction $\theta$, the cylinders
in the direction $\theta$ belong to at most one of the configurations $\cCC_i$.
\item[ii)] For each $i \in I$ there exists $\ve_0>0$ such that for all $\ve$ in an interval 
$(0,\ve_0)$ there exist flat surfaces $(X,\omega)$ in $H$ that possess a cylinder 
of width $\leq \ve$ in a direction $\theta$ belonging to the cylinder configuration $\cCC_i$.
The set of such surfaces is denoted by $H^{\ve}(\cCC_i)$. 
\item[iii)] For each $i \in I$ the 
limit of $\frac{1}{\ve^2} \nu(H^{\ve}(\cCC_i))$ as $\ve \to 0$ is positive.
\item[iv)] The contributions of the configurations $\cCC_i$ sum up to
the area Siegel-Veech constants, i.e.
$$\sum_{i \in I} c_{\area}(H,\cCC_i) \= c_{\area}(H)\,.$$
\end{itemize}
{{We refer to~\cite{AMYregular} for more background and related discussion regarding the above conditions. We also remark that 
in the case of strata, a full set of configurations defined above corresponds to a complete list of configurations of homologous cylinders (or saddle connections) in the literature.}}
\par
\medskip
We now discuss Siegel-Veech constants for a stratum 
$\omoduli(\bfm)$.  Let $\nu_{\rm str} = \nu_{\omoduli(\bfm)}$ 
be the Masur-Veech measure on the stratum. Then the following formula
is a direct consequence of the definition of configuration and
the Siegel-Veech formula applied to a small disc: 
\be \label{eq:SVviaboundary}
c_\area(\omoduli(\bfm)) \= \lim_{\ve \to 0} \frac1{\pi\ve^2} \sum_\cCC 
\frac{\nu_{\rm str}(\omoduli^\ve(\bfm, \cCC))}
{\nu_{\rm str}(\omoduli(\bfm))}\,,
\ee
where the summation ranges over a full set of  
cylinder configurations for the Masur-Veech measure supported on the stratum. 
\par
To compute $c_{\area}$ for more general cases one has to apply the Siegel-Veech 
formula to several test functions, as we will explain in 
Section~\ref{sec:SVcount} when computing these constants
for Hurwitz spaces.
\par
In any case, to make this formula useful, one has to overcome two problems.
First, one has to be able to compute the volume in the numerator. It
turns out for strata that this is a sum of volumes of strata obtained
by cutting the surfaces along the core curves $\gamma_i$. This turns the computation
of \cite{emz} into a recursive formula. 
\par
Second, one has to determine a full set of configurations for a stratum.
For this purpose, recall that (\cite[Proposition~3.1]{emz}) in any stratum
two non-homologous saddle connections {{sharing the same holonomy vector}}
exist only on a set of measure zero.
This can be used to show that a full set of cylinder configurations consists
of all possibilities of embedding disjoint closed cylinders into $\Sigma_g$
such that no two core curves are homotopic, no cylinder is separating, but
any pair of cylinders is separating. For each such collection we let
$S$ be the union of the closures of the cylinders and $\Gamma$ be
their boundary curves. Combinatorially one can describe
such a cylinder configuration by the tuple of genera $(g_1,\ldots, g_s)$
of $\Sigma_g \ssm S$ up to cyclic permutation.
(See \cite[Proposition~3.1, Sections~11 and 12]{emz} for more details and
the values of many Siegel-Veech constants.)

\section{Hurwitz spaces of torus covers and their configurations} 
\label{sec:Hurwitz}

We give a short introduction to Hurwitz spaces of torus coverings and recall 
some basic notions needed in the sequel. The main result in this section 
is a combinatorial description of a full set of cylinder configurations for 
these Hurwitz spaces.

\subsection{Admissible covers and torus coverings}

Harris and Mumford (\cite{harrismumford}) came up with the notion of admissible covers to deal with degenerations of coverings of smooth curves to coverings of nodal curves. In general, denote by $p: X\to C$ a finite morphism of nodal curves such that   
\begin{itemize}
\item[i)]
The smooth locus of $X$ maps to the smooth locus of $C$ and the nodes of $X$ map to the nodes of $C$. 
\item[ii)]
Suppose that $p(s) = t$ for a node $s\in X$ and a node $t\in C$. Then there exist suitable local coordinates 
$x, y$ for the branches at $s$ and  
$u, v$ for the branches at $t$, such that 
$$u = p(x) = x^k, \quad v = p(y) = y^k \quad \text{for some} \quad
k \in \mathbb Z^{+}\,.$$ 
\end{itemize}
We say that $p$ is an \emph{admissible cover}. One useful thing to keep in mind is that adding admissible covers provides a 
natural compactification of Hurwitz spaces of ordinary branched covers, which is analogous to the Deligne-Mumford 
compactification of the moduli space of curves by adding stable nodal curves. We refer to \cite[Chapter 3.G]{HarrisMorrison} for a detailed introduction to admissible covers. 
\par
Now we specialize to torus coverings. Let $\Hmu = (\mu^{(1)}, \cdots, \mu^{(n)})$ consist of partitions $\mu^{(i)} = 
(\mu^{(i)}_{1}, \mu^{(i)}_{2}, \cdots )$ such that each entry $\mu^{(i)}_{j}$ is 
a non-negative integer and $\sum_{i,j}(\mu^{(i)}_{j} -1) = 2g-2$. We call 
such a tuple $\Hmu$ a {\em ramification profile}. 
\par
An admissible cover
$p: X \to E$ has ramification profile $\Hmu$, if it has $n$ branch points
and over the $i$-th branch point the sheets coming together form 
the partition $\mu^{(i)}$ (completed by singletons, if $|\mu^{(i)}| < \deg(p)$).
Let $\BH_{d}(\Hmu)$ (or just $\BH$ if the parameters are fixed)
denote the $n$-dimensional {\em Hurwitz space} of degree $d$, genus $g$, 
connected admissible coverings $p: X \to E$ of a curve of genus one 
with $n$ branch points and ramification profile $\Hmu$.  We use 
$H_d(\Hmu)$ for the open subset of $\BH_{d}(\Hmu)$, where $X$ is smooth.
\par
Here we fix the notation for covers parameterized by this Hurwitz space
and for counting problems. Let $\rho: \pi_1(E \ssm \{P_1,\ldots,P_n\}) \to S_d$
be the monodromy representation in the symmetric group of $d$ elements
associated with a covering in~$H_d(\Hmu)$. We use the convention that loops 
(and elements of the symmetric group) are composed from right to left. The elements
$(\alpha, \beta, \gamma_1, \cdots, \gamma_n)$ as in the left picture of
Figure~\ref{fig:pointpos} generate the fundamental group  
$\pi_1(E\ssm \{P_1,\ldots,P_n\})$ with the relation 
\be\label{eq:FRel}
\beta^{-1}\alpha^{-1}\beta\alpha \= { \gamma_n \cdots \gamma_1\,.}
\ee
\input{pic_fundgroups}

\par
Given such a homomorphism $\rho$, we let $\ual = \rho(\alpha)$, $\ube = \rho(\beta)$, $\uga_i = \rho(\gamma_i)$, and call the tuple
\be \label{eq:HT}
h \=(\ual, \ube, \uga_1, \cdots, \uga_n) \in (S_d)^{n+2}
\ee
the {\em Hurwitz tuple} corresponding to $\rho$ and the choice of generators.
Conversely, a Hurwitz tuple as in \eqref{eq:HT} satisfying \eqref{eq:FRel} 
defines a homomorphism~$\rho$ and thus a covering~$p$. 
If we are only interested in connected coverings, we require a Hurwitz 
tuple moreover to generate a transitive subgroup of~$S_d$.
\par
We say that a Hurwitz tuple {\em has profile $\Hmu$} if the conjugacy class
$[\uga_i] = \mu^{(i)}$ for $i=1,\ldots,n$. Here we use the general convention
to call two partitions of different sizes $d_1 \leq d_2$ equal, if they differ 
by $d_2 - d_1$ parts of length one. The set of Hurwitz tuples of degree~$d$ and
profile~$\Hmu$ acting transitively on~$\{1,\ldots, d\}$ is denoted by $\Cov^0_d(\Hmu)$.
\par
The covering map $p$ does not depend on the choice of the base point. Changing
the base point results in simultaneous conjugation in~$S_d$ of the Hurwitz tuple.
We call the conjugacy classes of Hurwitz tuples {\em Hurwitz classes} and refer 
to the cardinality of the set of Hurwitz classes of profile $\Hmu$
as $N^\na = N^\na_d(\Hmu)$. 
\par
The upper index ``$\na$'' indicates that no automorphisms of the coverings are taken into account.
For counting problems, in particular when studying generating series, it is more 
natural to weight any Hurwitz classes
by the factor $|\Aut(p)|^{-1}$. Such automorphisms correspond bijectively to
elements of the centralizer of $\rho(p)$. We denote the number of weighted
Hurwitz classes of profile $\Hmu$ by $N^0_d(\Hmu)$ and we have the fundamental relation
\be \label{eq:NfromCov}
N^0_d(\Hmu) \= \frac{|\Cov^0_d(\Hmu)|}{d!}.
\ee
For asymptotics on connected covers, the weighting factor $|\Aut(p)|^{-1}$
is negligible (see \cite[Section 3.1]{eo}). 
\par


We remark that for some branching profiles $\Hmu$ the space $\BH_d(\Hmu)$ can be 
disconnected, e.g.\ if the profile consists of cycles of odd length only, the
parity of the spin structure of \cite{kz03} distinguishes two components.
Whether $\BH_d(\Hmu)$ decomposes into more components than the obvious ones is a hard problem
that will not play any role in the sequel. 

\subsection{Period coordinates, invariant measure, foliations} \label{sec:PCIMF}

Denote by $\moduli[1,n]$ the moduli space of genus one curves with $n$ ordered marked points. 
Let $\omoduli[{1,n}]$ be the Hodge bundle of holomorphic one-forms over $\moduli[1,n]$. 
We introduce a coordinate system on $\omoduli[{1,n}]$ to define the 
$\SL\RR$-invariant measure $\nu$, which we have already been referring to 
in the Siegel-Veech formula, and to define foliations we will argue with 
in the sequel.
\par
We present a point $(E,\omega, P_1,\ldots,P_n)$ in $\omoduli[1,n]$ as a flat surface
as in Figure~\ref{fig:pointpos} using the unique non-zero holomorphic
one-form $\omega$ on $E$ (up to scaling). Whereas the left picture gives a basis 
of $\pi_1(E \ssm \{P_1,\ldots,P_n\})$ we indicate in the right 
picture a basis of relative homology $H_1(E, \{P_1,\ldots,P_n\},\ZZ)$.
\par 
{\em Period coordinates} are given by assigning to 
$(E',\omega', P'_1,\ldots,P'_n)$ in a neighborhood of $(E,\omega, P_1,\ldots,P_n)$ the tuple
\be \label{eq:defperco}
(z_\alpha, z_\beta, z_2,\ldots, z_n) \= \Bigl(\int_\alpha \omega',\int_\beta \omega',
\int_{\delta_1} \omega', \ldots, \int_{\delta_{n-1}} \omega'\Bigr) 
\,\in\, \CC^{n+1.}.
\ee 
It is well-known that this defines a local coordinate system on $\omoduli[1,n]$.
\par
Inside $\omoduli[1,n]$ there is a (real) hypersurface  $\oamoduli[1,n]$ of (pointed)
flat tori with $\omega$-area equal to one. Note that $\oamoduli[1,n]$ is isomorphic to an open subset of 
the symmetric space $\SL\RR \times (\RR^2)^{n-1}/\SL\ZZ \times (\ZZ^2)^{n-1}$.  Hence
by Ratner's theorem there is a unique finite $\SL\RR$-{\em invariant ergodic 
measure $\ol{\nu}_1$} on $\oamoduli[1,n]$ (up to scaling). We
will denote by $\ol{\nu}$ the push-forward of $\ol{\nu}_1$ under the quotient map
by $\SO_2(\RR)$, i.e.\ on $\moduli[1,n]$. There are two ways to construct 
$\ol{\nu}_1$ explicitly. 
\par
The first construction of $\ol{\nu}_1$ is completely analogous to the
construction that works for (the connected components of) the strata. For each
open subset $U \subset \oamoduli[1,n]$ let $C(U)$ be the cone of flat 
surfaces over $U$, i.e.\ flat surfaces $(X,\omega) \in \CC^* \cdot U$ 
with area $\leq 1$. We take $\ol{\nu}_1(U)$ to be the Lebesgue measure of $C(U)$ 
with the normalization such that the unit cube of $\ZZ[i]^{n+1} \subset 
\CC^{n+1}$ has volume one. A change of basis 
corresponds to an action of $\SL\ZZ \times 
(\ZZ^2)^{n-1}$ on period coordinates, thus preserving the integral lattice.
Consequently, the unit cube normalization is well-defined.
\par
The second construction provides a transverse measure on the following 
foliation. Denote by ${\rm REL}$ the foliation of $\omoduli[1,n]$ whose leaves are the preimages of
the forgetful map $\omoduli[1,n] \to \omoduli[1,1]$. By definition the 
leaves are $\SO_2(\RR)$-invariant. Hence the foliation descends to a foliation
on $\moduli[1,n]$ which we also denote by ${\rm REL}$. This foliation is
transversal to the foliation $\cFF$ by $\SL\RR$-orbits. The leaf of ${\rm REL}$
over $(E,P_1)$ is the $(n-1)$-fold product of $E$ minus the diagonals.
We provide $\oamoduli[1,1] = \SL\RR/\SL\ZZ$ with the Haar probability measure 
and define a transverse measure to $\cFF$ using the Euclidean volume on $E^{n-1}$,
normalized so that $\vol(E^{n-1}) = 1$. The measure $\ol{\nu}$ is obtained by the direct
integral of this transverse measure along the Haar measure on $\oamoduli[1,1]$.
\par
We let $\Omega H$ be the moduli space of pairs consisting of a covering
$(p: X \to E) \in H$ and a non-zero holomorphic one-form $\omega$ on $X$ that is a pullback from $E$ via $p$.
This is a $\CC^*$-bundle over $H$ and again we let $\Omega_1 H$ be the
hypersurface of flat surfaces $(X,\omega)$ of area one. The space
$\Omega H$ is a finite unramified cover of $\omoduli[1,n]$ and the same
holds for the restriction to $\Omega_1 H$ as well as to the variants
with constraints on the connectivity of $p$. Consequently, the
above period coordinates are local coordinates on $\Omega H$, too. Moreover, 
the measures $\ol{\nu}_1$ and $\ol{\nu}$ pull back to finite measures 
$\nu_{1}$ and $\nu$ on $\Omega_1 H$
and $H$ respectively. Finally, the foliation ${\rm REL}$ also defines
a holomorphic foliation on $\Omega H$, with leaves of codimension one.
\par

\subsection{Configurations for Hurwitz spaces}
In this section we describe a full set of cylinder configurations for 
a Hurwitz space $H=H_d(\Hmu)$ of torus coverings. 
For a given Hurwitz tuple $h$ we define the (horizontal) {\em Dehn twist}
around the  curve $\alpha$ to be the map that sends 
\begin{equation} \label{eq:newDehnaction}
h = (\ual, \ube, \uga_1, \cdots, \uga_n)
\quad  \text{to} \quad (\ual,  \ube\ual, \ual^{-1} \uga_1\ual, 
\cdots, \ual^{-1}\uga_n\ual). 
\end{equation}
\par
\begin{Prop} \label{prop:fullcylconf}
There is a natural bijection between the set of equivalence classes 
of Hurwitz tuples up to simultaneous conjugation
and Dehn twist action and a full set of cylinder configurations for the 
Hurwitz space $H$.
\end{Prop}
\par
\begin{proof} First we associate to any Hurwitz tuple $h$ a
cylinder configuration $\cCC(h)$ as follows. We order and place the branch points
on $E$ with strictly decreasing vertical coordinates, as in 
Figure~\ref{fig:pointpos}. The Hurwitz tuple defines a covering 
$p: X \to E$ with $g(X) = g$. The subsurface of the cylinder configuration is
$S = \Sigma_g$ and $\Gamma$ is the $p$-preimage of the union of closed
horizontal loops through the points $P_i$. Obviously, the resulting cylinder
configuration is unchanged under conjugation of the Hurwitz tuple
and independent of the representative in the Dehn twist orbit.
\par
Conversely, suppose that $(p: X \to E,\, \omega = p^*\omega_E)$ is a 
covering parameterized by $\Omega_1H$  and that $\theta$ is a direction 
such that no two branch points in $E$ lie on the same closed 
$\omega_E$-geodesic. (Other directions need not be taken into account, 
since aligned branch points form a measure zero subset. They do not
contribute to the Siegel-Veech constant and they do not satisfy 
the condition~iii) of a full set of cylinder configurations.)
We may assume moreover that there is a cylinder in the direction $\theta$, 
hence the $p$-image of its core curve is a closed loop on $E$ in the
direction $\theta$. We call this loop $\alpha$ and fix a base point on $\alpha$.
Next we choose a complementary direction $\theta_2$ admitting a closed 
geodesic $\beta$. We label the branch points in decreasing height
(with respect to the direction $\theta_2$) and choose loops as in 
Figure~\ref{fig:pointpos}. The monodromy of the cover defines a
Hurwitz tuple. Its equivalence class up to conjugacy and Dehn twist
action is independent of the choices we made. Finally we note that
the two constructions are inverse to each other.
\par
It remains to check that these cylinder configurations form a full
set of such configurations. Condition~i) is obvious and condition~ii) holds 
by taking the base curve $E$ of the covering sufficiently tall and thin.
In fact, $\varepsilon_0 = 1/nd$ works. Condition~iii) now follows
immediately from the preceding description of the measure $\nu_1$, 
since the location of the branch points is unconstrained except for
a set of measure zero. To check condition iv) it suffices to notice 
that we only neglected cylinder configurations that appear on a
set of $\nu_1$-measure zero.
\end{proof}
\par
Suppose the fundamental group of the punctured surface 
$E \ssm \{P_1,\ldots,P_n\}$ is given in our standard presentation
of Figure~\ref{fig:pointpos}. We remark that the core curves of the horizontal
cylinders are represented by the loops
\be \label{eq:sigma_cycles}
\sigma_0 =\alpha,\, \,\sigma_1= \alpha \,\gamma_1^{-1}, \,\,
\sigma_2  = \alpha\, (\gamma_2\gamma_1)^{-1}\, , \ldots, \,\,
\sigma_{n-1} = \alpha\,(\gamma_{n-1}\cdots\gamma_1)^{-1}.
\ee

\section{Weighted counting of Hurwitz classes} \label{sec:SVcount}

We will now count Hurwitz classes with a weight, that we call 
Siegel-Veech weight.  The aim of this section is to show that this gives 
a combinatorial way to compute the area Siegel-Veech constants of Hurwitz spaces.
\par
Let $\lambda = (\lambda_1 \geq \lambda_2 \geq \cdots \geq \lambda_k)$ with
$\lambda_i \geq 0$ be a partition. The {\em $p$-th part-length moment}
of $\lambda$ is defined as
\be \label{eq:def_pSV}
S_p(\lambda) \= \sum_{j=1}^{k} \lambda_j^p\,
\ee
for any $p \in \CC$, but only the moments $p \in \ZZ$ will be used in 
this paper. If $h =(\ual, \ube, \uga_1, \cdots, \uga_n) \in (S_d)^{n+2}$ is a 
Hurwitz tuple, we consider the $n$ permutations 
\be \label{eq:sigma_perms}
\usi_0 =\ual,\, \usi_1 = \ual \, \uga_1^{-1}, \,\,
\usi_2 = \ual \, (\uga_2\uga_1)^{-1}\, \, , \ldots, \,\,
\usi_{n-1} = \ual\,(\uga_{n-1}\cdots\uga_1)^{-1}
\ee
in  $S_d$ {{that arise from monodromy of the core curves of the horizontal cylinders as represented in \eqref{eq:sigma_cycles}}}. Define the {\em $p$-th Siegel-Veech weight of a Hurwitz 
tuple $h_j$} to be
\be \label{eq:cpHurwitz}
S_p(h_j) \= \sum_{i=0}^{n-1} S_p(\usi_i(h_j))\,.
\ee
{{Geometrically speaking, the $p$-th Siegel-Veech weight $S_p(h_j)$
encodes the sum of moduli of the horizontal cylinders on the covering surface,
each with weight given by raising to the $p$-th power.}}
This weight is obviously independent of representative of the Hurwitz tuples 
in a given Hurwitz class. We define the {\em (combinatorial) $p$-weighted 
Siegel-Veech constant} $c_p^0(d,\Hmu)$ to be the sum of the weights over all Hurwitz classes for $H_d(\Hmu)$, i.e.\ 
\begin{equation} \label{eq:cdmu} 
c^0_p(d,\Hmu) 
\= \frac{1}{n\,d!}  \sum_{j=1}^{|\Cov_d^0(\Hmu)|} S_p(h_j) \=  \frac{1}{n}  \sum_{j=1}^{N_d^0(\Hmu)} S_p(h_j) 
\end{equation}
where the two equivalent definitions are linked by~\eqref{eq:NfromCov}. 
\par
The upper zero in $c_p^0(d,\Hmu)$ refers to the fact that here all 
covers are connected and all Hurwitz tuples generate a transitive subgroup
of $S_d$ by definition. (In Section~\ref{sec:genser} we will discuss the 
passage between the connected and possibly disconnected cases.)
\par
\begin{Thm} \label{thm:SV} Fix a degree $d$ and a ramification profile $\Hmu$.
Then the combinatorial Siegel-Veech constant $c^0_{-1}(d,\Hmu)$ defined 
in \eqref{eq:cdmu} and the area Siegel-Veech constant of the
Hurwitz space $H_d(\Hmu)$ satisfy the following relation 
\be \label{eq:combforareaSV}
c_{\area}(H_d(\Hmu)) \= \frac{3}{\pi^2} \, \frac{c^0_{-1}(d,\Hmu)}{N^0_d(\Hmu)}\,.
\ee
\end{Thm}
\par
We will present two proofs of this formula. The first proof given below 
just uses the Siegel-Veech transform. It 
generalizes a combinatorial formula for Siegel-Veech constants of
\Teichmuller\ curves 
given in \cite[Appendix]{ekz}. A second proof is given in Section~\ref{sec:Lyap}, 
which is more algebraic and uses the main result of 
\cite{ekz} relating the area Siegel-Veech constant to the sum of Lyapunov exponents.
We remark that combining the two proofs gives
a new proof of the main result of \cite{ekz} in the case of Hurwitz spaces
by intersection theory only, without any reference to analytic techniques such as the determinant of the 
Laplacian etc.
\par
\begin{proof}[Proof of Theorem~\ref{thm:SV}] Let $\{\cCC_i\}$ for $i \in I$ be
a full set of cylinder configurations for the Hurwitz space $H_d(\Hmu)$.
The left hand side of~\eqref{eq:combforareaSV} is a sum of
$c_\area(H_d(\Hmu), \cCC_i)$. By Proposition~\ref{prop:fullcylconf} each cylinder 
configuration $\cCC_i$ corresponds to an orbit $\cOO_i$ of Hurwitz tuples
under the Dehn twist action and conjugation. It thus suffices to show
that $c_\area(H_d(\Hmu), \cCC_i)$ equals the contribution of $\cOO_i$ to
the numerator of the right hand side for each $i \in I$. We fix
$\cCC_i$ and $\cOO_i$ from now on and let $N_i = |\cOO_i|$.
\par
For a flat surface $(X,\omega)$ we let $V_i \subset \RR^2$ be the subset
of holonomy vectors of the core curves of cylinders belonging to the cylinder 
configuration $\cCC_i$. (We count them with multiplicity one, and area multiplicities
will be introduced through~\eqref{eq:chiwithweights} below.) Let 
$C^{(k)}$ for $k \in K$ be the cylinders of the configuration $\cCC_i$.
Since we identified cylinder configurations with equivalence classes of
Hurwitz tuples, the ratios of widths among the $C^{(k)}$ is determined by 
the configuration. In fact, the core curves of these cylinders are the $p$-preimages
of the loops $\sigma_0, \ldots, \sigma_{n-1}$ of~\eqref{eq:sigma_cycles}, so
the cylinders correspond to the parts of the partitions $\usi_0,\ldots,\usi_{n-1}$
and the widths are proportional to the cardinality of the parts. Consequently, 
we may order the cylinders increasingly by their widths $w_k = w(C^{(k)})$, 
i.e. $w_1 \leq w_2 \leq \cdots \leq w_{|K|}$.
\par 
We apply the Siegel-Veech formula~\eqref{eq:SVtransform} to two functions.  
The first function is the characteristic function $\chi_\ve$ for a small disc 
of radius $\ve$ at the origin. We evaluate
$$\frac{1}{\nu(\Omega_1H_d(\Hmu))} \int_{\Omega_1 H_d(\Hmu)} 
\widehat{\chi}_\ve d\nu_1 = {B_i} \int_{\RR^2} \chi_\ve dxdy \,,$$
where $B_i$ is the Siegel-Veech constant for $V_i$.  (In fact, it is the cylinder
Siegel-Veech constant for the configuration $\cCC_i$.)
The integrand on the left hand side is constant along the REL-foliation, and hence its value equals $N_i$ times the volume of 
an $\ve$-neighborhood of the cusp in $\omoduli[1,1]$, which is $\pi \ve^2$. The volume of $\Omega_1H$
is $N^0$ times the volume of the modular surface, which is $\pi^2/3$. Since 
the integral on the right hand side is $\pi \ve^2$, we
conclude that 
$$ B_i = \frac{3}{\pi^2} \frac{N_i}{N^0}.$$
\par
The second function we plug in the Siegel-Veech formula is the sum of characteristic 
functions for counting cylinders with fixed widths $w_k$ and (as parameter)
the tuple of heights $\he = (\he_1,\ldots,\he_{|K|})$. That is,  
for $v \in \RR^2$  we let
\ba \label{eq:chiwithweights}
\chi_{r,\he}(v,\cCC_i) = \left\{\begin{array}{lll}
0 & \text{if} & w_1||v|| \geq r \\
\frac{\he_1w_1}{d} &\text{if} & w_2||v|| \geq r > w_1||v|| \\
\cdots && \cdots \\
\frac{\he_1w_1+\cdots +\he_jw_j}{d} &\text{if} & w_{j+1}||v|| \geq r > w_j||v|| \\
\cdots && \cdots \\
\frac{\he_1w_1+\cdots +\he_{|K|}w_{|K|}}{d} & \text{if} & r > w_{|K|}||v|| \\
\end{array}
\right. 
\ea
and let $\chi_{r}$ be the function with ``average'' height, i.e.\
$\chi_{r} =  \chi_{r,(1/n,\ldots,1/n)}$.
Since 
$$ \int_{\RR^2}  \chi_{r}((x,y),\cCC_i) dxdy = \pi r^2 \frac{1}{nd}\sum_{k=1}^{|K|} w^{-1}_k,$$
we obtain using the Siegel-Veech formula again and the value of $B_i$ that  
\ba
\label{eq:intchir} \frac{1}{\nu(\Omega_1 H_d(\Hmu))} 
\int_{\Omega_1 H_d(\Hmu)} \widehat{\chi}_{r}((X,\omega),\cCC_i) d\nu_1 &\= 
\frac{3r^2}{\pi}
\frac{N_i}{N^0} \frac{1}{nd}\sum_{k=1}^{|K|} w^{-1}_k \\
&\= \frac{3r^2}{\pi} \frac{1}{N^0} \frac{1}{d} c^0_{-1}(H,\cCC_i), 
\ea
{{where an analog of~\eqref{eq:cdmu}}} was used in the last step for the configuration $\cCC_i$. 
\par
It remains to show that
\begin{equation}
\label{eq:careachi}
c_{\area}(H,\cCC_i) = d \lim_{r \to \infty} \frac{1}{\pi r^2}
\frac{1}{\nu(\Omega_1 H_d(\Hmu))} \int_{\Omega_1 H_d(\Hmu)} \widehat{\chi}_{r}((X,\omega),\cCC_i) d\nu_1.
\end{equation}
For this purpose, note that the integrand does not depend on
the location of $(X,\omega)$ within the REL-foliation, equivalently
within the fibers of the projection $\Omega_1 H_d(\Hmu) \to \oamoduli[1,1]$.
We disintegrate $\nu_1$ over this fibration as
$d\mu_{X} d\overline{\nu}_1(X,\omega)$. Then for any fixed $r$ the sum
over  $V_i(X,\omega)$ is finite and we obtain that 
$$
\begin{aligned}
 \int_{\Omega_1 H_d(\Hmu)} \widehat{\chi}_{r}((X,\omega),\cCC_i) d\nu_1 
&=  \int_{\Omega_1 H_d(\Hmu)}  \sum_{v \in V_i(X,\omega)} \chi_{r} (v ,\cCC_i) d\nu_1 \\
&= N^0 \int_{ \oamoduli[1,1]}  \sum_{v \in V_i(X,\omega)} \int_{X^{n-1}} 
\chi_{r} (v ,\cCC_i) d\mu_{X} d\overline{\nu}_1(X,\omega). \\
\end{aligned}
$$
For every covering $p: X \to E$ and every $v$ we slice the torus $E$ parallel
to $v$ and some direction $v^\perp$ given by a primitive vector in the 
lattice of $E$ which is not parallel to $v$. 
Instead of integrating over $X^{n-1}$ we will integrate over $E^{n-1}$
and take into account the degree $d$ of the covering.
Let $B = \{(\he_1,\ldots,\he_{|K|}) \in [0,1]^{|K|}: \sum_{k=1}^{|K|} \he_i = 1 \}$. 
{{Using $v$ and $v^\perp$ as a basis,}} we place the first point $P_1$ at the corner $(0,0)$. 
Integrating over the points $P_2,\ldots,P_{n}$ in $E$ can be done by placing these points
at $a_i v/||v|| + b_i v^\perp$ with $a_i,b_i \in [0,1]$ for $i=2,\ldots,n$.  The cylinders in the
direction $v$ will have height $b_{i} - b_{i+1}$ if the points are ordered by decreasing second coordinates and thus give a tuple in $B$. 
Using that $\chi_r$ is the average of the $\chi_{r,\he}$ over $B$, we obtain that 
$$
\begin{aligned}
 \int_{\Omega_1 H_d(\Hmu)} \!\!\!\!\!\!\!\widehat{\chi}_{r}((X,\omega),\cCC_i) 
d\nu_1 &\=  N^0 \int_{ \oamoduli[1,1]}  \sum_{v \in V_i(X,\omega)} 
\int_{[0,1]^{n-1}}\int_{B} \chi_{r} (v ,\cCC_i) d\mu_{X} d\overline{\nu}_1(X,\omega) \\
&\hspace{-1cm}\=  N^0 \int_{ \oamoduli[1,1]}  \sum_{v \in V_i(X,\omega)} 
\int_{[0,1]^{n-1}} \int_{\he \in B} \chi_{r,\he} (v ,\cCC_i) d\mu_{X} d\overline{\nu}_1(X,\omega) \\
&\hspace{-1cm} \= {\frac{1}{d}} N^0 \int_{\oamoduli[1,1]} N_{\area} ((X,\omega),r,\cCC_i) d\overline{\nu}_1(X,\omega).  
\end{aligned}
$$
For $r$ large, the integrand on the right hand side of the last step converges 
by~\eqref{eq:cylplim} to $\frac{1}{d}{\pi}r^2 c_{\area}(H,\cCC_i)$, 
independent of the flat surface $(X, \omega)$, where the scaling factor $\frac{1}{d}$ 
is due to that of $\chi_r$ in its definition. Recall also that the volume of $\Omega_1H$ is $N^0$ times 
the volume of the modular surface. Altogether this implies that \eqref{eq:careachi} holds.
\par
Finally, adding up the contributions from~\eqref{eq:careachi} 
using~\eqref{eq:intchir} gives the claim.
\end{proof}
\par

\section{The sum of Lyapunov exponents as a ratio of intersection numbers} 
\label{sec:Lyap}

In this section we justify geometrically  why we give preference to 
area Siegel-Veech constants over other Siegel-Veech constants. 
The first answer, given in \S~\ref{sec:nodalpush} is that they appear 
as a coefficient of the push-forward of a boundary class. The second
answer, given in \S~\ref{sec:SVtoLyap}, relates area Siegel-Veech constants
to the sum of Lyapunov exponents, which is further expressed as a ratio of intersection numbers on moduli spaces. 
We work on Hurwitz spaces throughout in this section and emphasize that the discussion is
entirely algebraic. In particular, analytic tools such as determinants of the Laplacian 
as in \cite{ekz} are not needed.
\par

\subsection{Push-forward of the nodal locus} \label{sec:nodalpush}

We fix the degree $d$ and the ramification profile $\Hmu$.
The moduli maps for the Hurwitz space and the universal family~$\cXX$ over it
give rise to the following commutative diagram
$$\xymatrix{
\cXX \ar[r]^{h} \ar[d]_{\pi}  & \barmoduli[{1,n+1}] \ar[d]^{\pi_{n+1}}  \\
\BH_d(\Hmu) \ar[r]^{f}     & \barmoduli[{1,n}]}  $$ 
where $f$ and $h$ are finite morphisms of degree $N$ and $dN$, respectively,
and where $\pi_{n+1}$ is the map forgetting the last marked point. 
Let $\delta_{\cXX}\subset \cXX$ be the (codimension two) locus of nodal singularities of the fibers.
Recall that the Deligne-Mumford boundary of $\barmoduli[{1,n}]$
consists of the divisor $\delta_{\irr}$ that parametrizes generically irreducible 
nodal rational curves and the divisors $\delta_{0,S}$ for $S$ a subset of
 $\{1,\ldots,n\}$ with $|S|\geq 2$ that parametrize generically curves with 
one separating node such that the marked  points in $S$ lie in the component of 
genus zero. We denote an undetermined linear combination of the divisors   
$\delta_{0,S}$ by $\delta_{\other}$.
\par
\begin{Thm}
\label{thm:intnumber}
 The push-forward of the nodal locus in $\cXX$ to $\barmoduli[{1,n}]$ can be evaluated using the
weighted sum of divisor classes introduced above as
\be  \label{eq:pi_sing_locus}
\pi_{n+1*}h_{*}\delta_{\cXX} \= c^0_{-1}(d,\Hmu) \, \delta_{\irr} + \delta_{\other}\,.
\ee
\end{Thm}
\par
\begin{proof}[Proof of Theorem~\ref{thm:intnumber}]
The set theoretic image of $\delta_{\cXX}$ is of course contained
in the union of boundary divisors, so the only content of the theorem
is the multiplicity of $\delta_{\irr}$. The preimage of a tubular neighborhood
of $\delta_{\irr}$ in $H$ consists of  the set of Hurwitz classes grouped
to the orbits of the Dehn twist as in \eqref{eq:newDehnaction}. (The tubular
neighborhood is determined by $\alpha$ being short.) As above, we denote 
these orbits by $\cOO_i$ for $i=1,\ldots,m$ and let $N_i = |\cOO_i|$.
Suppose that $\cOO_i$ consists of the Hurwitz classes {{$\{h_j \}$}}. 
It suffices to compare both sides of \eqref{eq:pi_sing_locus} in the
neighborhood specified by each of these orbits $\cOO_i$ separately and 
then add their contributions together. 
\par 
We want to show that the intersection number
with a test curve agree on both sides of \eqref{eq:pi_sing_locus} 
in the boundary neighborhood determined by $\cOO_i$. For this purpose we use 
the Teichm\"uller curve $C$ generated
by a square-tiled surface $(X,\omega)$ of $d n$ rectangles constructed as 
follows. Pile $n$ rectangles of width $1$ and height $1/n$ from top to bottom to produce a torus $E = \CC/(\ZZ+i\ZZ)$, and place the point~$P_l$ in the middle of the upper boundary of the $l$-th rectangle. Take 
a degree $d$ cover $p: X \to E$ with monodromy given by a Hurwitz class 
$h_j \in \cOO_i$
(using the presentation of the fundamental group as in Figure~\ref{fig:pointpos}, 
with the base point in the left part of the bottom rectangle)
and let $\omega = p^* \omega_E$. The $\SL\RR$-orbit of $(X,\omega)$
defines a \Teichmuller\ curve $\varphi:C \to \BH_d(\Hmu)$.
\par
The horizontal cylinders of the flat surface $(X,\omega)$ are in 
bijection with the union of 
the cycles $c_{s,j}$ of the permutations $\usi_s$, $s=0,\ldots,n-1$, introduced in~\eqref{eq:sigma_perms}, 
associated with a Hurwitz class $h_{j}$ in the above. We denote these cylinders by 
$C^{(k)}$ for $k \in K(j)$. These cylinders {{(possibly not maximal cylinders)}} have
height $1/n$ and width $\ell(c_{s,j})$. (Recall that the modulus $m(C^{(k)})$ 
of a cylinder is defined as the ratio ``height over width''.) 
To sum up, we have for all $j$ the relation
\begin{equation} \label{eq:cylwht}
\sum_{s=0}^{n-1} S_{-1}(\usi_s) \= \sum_{k\in K(j)} \ell(C^{(k)})^{-1} 
\= n \,\, \sum_{k\in K(j)} m(C^{(k)})\,.
\end{equation} 
\par
Let $\ell$ be the least common multiple of all the $\ell(C^{(k)})$ for 
$k\in K(j)$. The parabolic element $N(\ell) = \left(\begin{smallmatrix}
1 & n\ell \\ 0 & 1 \end{smallmatrix} \right)$ is in the affine group 
of $(X,\omega)$ and fixes the horizontal direction. Consequently, the 
corresponding diffeomorphism acts on the surface~$X$ as the product 
of Dehn twists
\be \label{eq:Dehn}
N(\ell) \= \prod_{k \in K(j)} D_{C^{(k)}}^{\ell / \ell(C^{(k)})}\,,\ee
where $D_{C^{(k)}}$ is the Dehn twist around the core curve of $C^{(k)}$.
\par
We start by determining the intersection number of the Teichm\"uller 
curve~$\varphi$ with the right hand side of~\eqref{eq:pi_sing_locus} 
in a neighborhood $U$ of  the cusp determined by the horizontal 
direction on $(X,\omega)$. On $E$ the action of $N(\ell)$ is an $\ell$-fold
Dehn twist of each of the $n$ horizontal cylinders of $E \ssm \{P_1,\ldots,P_n\}$
or, equivalently, it is an $(\ell n)$-fold Dehn twist of the unique horizontal cylinder
of $E$. In both viewpoints, the local contribution of $U$ to the intersection
$ \delta_{\irr} \cdot (f \circ \varphi)(C) $ is equal to~$\ell n$.
\par
On the other hand,  the local contribution of $U$
to the intersection $ \pi_* \delta_{\cXX}  \cdot  \varphi(C)$
is equal to  $\sum_{k \in K(j)} \ell / \ell(C^{(k)})$ by~\eqref{eq:Dehn}.
\par
Note that the Siegel-Veech weights of two Hurwitz classes related by~\eqref{eq:newDehnaction}
agree. Moreover, the (local) degree of $f$ restricted to $\varphi (U)$ is $N_i$. Comparing the two calculations above and using~\eqref{eq:cylwht}, we obtain
on $V= (f \circ \varphi)(U)$ that
$$f_*\pi_* \delta_{\cXX}|_V 
\=  \frac{N_i}{n} \sum_{k \in K(j)} \ell(C^{(k)})^{-1} \delta_{\irr}|_V 
= \frac{1}{n} \sum_{h_j\in \cOO_i}  S_{-1}(h_{j}) \delta_{\irr}|_V.$$
Summing over all the~$m$ Dehn twist orbits of Hurwitz classes~$\cOO_i$ thus completes the proof. 
\end{proof}
\par

\subsection{From Siegel-Veech to Lyapunov: an algebraic proof} 
\label{sec:SVtoLyap}

Lyapunov exponents measure the growth rate of cohomology classes 
on flat surfaces under parallel transport along the Teichm\"uller 
geodesic flow. They agree for any two flat surfaces with the same 
$\SL\RR$-orbit closure. Hence they are important invariants of
orbit closures, in particular of Hurwitz spaces and strata.
We refer e.g.\ to \cite{zorich06} and \cite{moelPCMI}
for the motivation and definition of Lyapunov exponents.
In general not much is known about number theoretic properties of 
individual Lyapunov exponents. Their sum, however, is always a rational
number. This was shown in full generality in \cite{ekz}, if one 
uses~\cite{AMYregular} to remove a technical hypothesis on regularity of $\SL\RR$-orbit closures.
The proof of~\cite{ekz} uses a large detour via Siegel-Veech constants
and many analytic tools. 
\par
On the other hand, shortly after Zorich's discovery of the
rationality behavior, Kontsevich interpreted the sum of
Lyapunov exponents as the ratio of a transverse measure~$\beta$ integrated
against two natural first Chern classes (\cite{kontsevich}). This interpretation rather
than the definition will be our starting point to compute the sum
of Lyapunov exponents. If the class $\beta$ could be interpreted
as a rational cohomology class on a suitable compactification of
an orbit closure, this would give a more conceptual proof of 
the rationality of the sum of Lyapunov exponents. Finding such
an interpretation of~$\beta$ in the case of strata is currently a
central open problem.
\par
We will identify $\beta$ for Hurwitz spaces as a rational cohomology
class. This will be stated in Theorem~\ref{thm:betainH2} below and proven in 
Section~\ref{sec:beta}. The main result in this section is a proof
of the following result, using intersection theory only. Suppose that
the smallest stratum that contains $H_d(\Hmu)$ is $\omoduli(m_1,\ldots,m_n)$.
\par
\begin{Thm}
\label{thm:hurwitzLckappa}
For the sum of Lyapunov exponents of the Hurwitz space $H_d(\Hmu)$
and the combinatorial Siegel-Veech constant $c^0_{-1}(d,\Hmu)$
we have the relation  
\be \label{eq:HurLckappa}
\lambda_1 + \cdots +\lambda_g \= \frac{c^0_{-1}(d,\Hmu)}{N^0_d(\Hmu)} + \kappa, 
\quad \text{where}\quad \kappa \= \frac{1}{12}
\sum_{i=1}^n \frac{m_i (m_i+2)}{m_i+1}\,.
\ee
\end{Thm}
\par
The proof uses Theorem~\ref{thm:intnumber} as its only ingredient
besides intersection theory. Theorem~\ref{thm:hurwitzLckappa} should 
be compared to the main result of~\cite{ekz} which states that 
for an $\SL\RR$-invariant submanifold $H$ that is minimally contained
in the stratum $\omoduli(m_1,\ldots,m_n)$ the Lyapunov exponents
and the area Siegel-Veech constant are related by
\be \label{eq:ekzmain}
\lambda_1 + \cdots +\lambda_g \= \frac{\pi^2}{3}\,c_\area(H) + \kappa\,.
\ee
Consequently, combining Theorems~\ref{thm:SV} and~\ref{thm:hurwitzLckappa} provides an algebraic proof of the
formula of Eskin-Kontsevich-Zorich in the case of Hurwitz spaces. 
\par
\medskip
We first introduce the formula for the sum of Lyapunov exponents as a 
ratio of two integrals. The projectivized Hodge bundle $\pobarmoduli$ comes with 
a tautological line bundle~$\cOO(-1)$. Its fiber over a point $(X,\omega)$ 
is the $\CC$-span of~$\omega$. The first Chern class of this line bundle 
is denoted by $\gamma_1$ in \cite{kontsevich}. We use the same notation 
for a vector bundle on the whole moduli space and its restriction to any 
algebraic $\SL\RR$-invariant submanifold $H$. A second tautological class 
is the first Chern class of the Hodge bundle, denoted by~$\lambda$.
The fiber of the Hodge bundle over a point $(X,\omega)$ is the vector
space $H^0(X,\Omega^1_X)$. (Note that $\lambda$ is denoted by~$\gamma_2$ 
in \cite{kontsevich}.)
\par
The third key player is not quite a class in cohomology, but a transverse
measure. Recall that an $\SL\RR$-invariant submanifold $H$ has a natural 
projection $\pi: H \to \PP H$, quotienting by $\SO_2(\RR)$ 
(or quotienting the ${\rm GL}(2,\RR)$-orbit closure by $\CC^*$, 
explaining the notation). Let $\cFF$ be the $\pi$-image of the 
(non-holomorphic) foliation of $H$ by $\SL\RR$-orbits. Then~$\beta$ is the 
transverse measure to the foliation $\cFF$ which is obtained by disintegrating 
the Masur-Veech measure $\nu_1$. With these notations the main formula sketched in \cite{kontsevich} becomes
\begin{equation} \label{eq:kmain}
\lambda_1 + \cdots +\lambda_g \= \frac{\int_{\PP H}\beta \wedge \lambda}
{\int_{\PP H} \beta \wedge \gamma_1}\,, 
\end{equation}
where $H$ is the $\SL\RR$-orbit closure of the flat surface whose 
Lyapunov spectrum we are interested in. A full proof of the above formula, 
stated as the ``Background Theorem'', appears in \cite[Section~3]{ekz}, 
along with references to various other cases where this formula has been 
established rigorously before. In the case of the moduli space of pointed
elliptic curves, obviously $\PP\Omega\barmoduli[1,n] = \barmoduli[1,n]$.
We will show in the next section that
on $\barmoduli[1,n]$ integration against~$\beta$ is represented by
a rational cohomology class. More precisely, we can identify the class
as follows. We define the tautological classes $\psi_i$ on $\barmoduli[{g,n}]$ 
by having the value $-\pi_*(\sigma_i^2)$ on any family of stable genus 
$g$ curves $\pi: \cXX \to C$ with sections~$\sigma_i$ corresponding to the marked points. 
\par
\begin{Thm} \label{thm:betainH2}
As elements of $H^{2n-2}(\barmoduli[1,n], \CC)$, the classes $\beta$ 
and $\psi_2\cdots \psi_n$ are proportional.
\end{Thm}
\par
Since the foliation by $\SL\RR$-orbits and the measure $\nu_1$ 
on the Hurwitz space are defined as pullbacks from $\oamoduli[{1,n}]$, 
the integration of first Chern classes of line bundles against~$\beta$ 
on $\BH_d(\Hmu)$ is proportional to the intersection product with
the class $f^*(\psi_2)\cdots f^*(\psi_n)$ where 
$f:\BH_d(\Hmu) \to \barmoduli[1,n]$
is the forgetful map.
\par

\subsection{Tautological class calculations on $\barmoduli[{1,n}]$}

To identify $\beta$ we next summarize known results on the cohomology
ring of $\barmoduli[{1,n}]$. Recall that $\lambda$ is the first Chern class 
of the Hodge bundle on $\barmoduli[{g,n}]$. Recall also the definition of 
 $\delta_{0,S}$ and $\pi_{n+1}$ from Section~\ref{sec:nodalpush} and the 
following result from \cite{AC}.
\par
\begin{Prop} \label{prop:picgen}
 $H^2(\barmoduli[{1,n}], \CC)$ is freely generated
by $\lambda$ and the boundary classes $\delta_{0,S}$ for $2 \leq |S| \leq n$.
\end{Prop}
\par
We use $\langle \mu \rangle_{1,n}$ to denote 
the degree of a given class $\mu$ in $H^{2n}(\barmoduli[{1,n}], \CC)$.
\par
As a special case of the preceding proposition, $H^2(\barmoduli[{1,1}], \CC)$ is of rank one, and in fact (see \cite[(2.46)]{witten})
\be \label{eq:caseM11}
\psi_1 \= \lambda \= \frac{1}{12}\delta_{\irr} \quad \text{and} 
\quad \langle\delta_{\irr}\rangle_{1,1} \= \frac{1}{2}. 
\ee
\par
The aim of this subsection is to deduce the following relations in
the cohomology ring of $\barmoduli[{1,n}]$ from well-known properties.
\par
\begin{Lemma} 
\label{le:psi-null} 
For any subset $S \subset \{1,\cdots, n \}$ such that $2 \leq |S| \leq n$, 
we have 
$$ \langle \delta_{0,S} \psi_2\cdots\psi_n  \rangle_{1,n} \= 0\,.$$  
\end{Lemma}
\par
Since the statement and proof is symmetric in the marked points, we may
replace here and in the subsequent lemmas 
$\psi_2\cdots\psi_n$ by any product of $(n-1)$ distinct $\psi$-classes.
\par
\begin{Lemma} 
\label{le:psi-irr}
On $\barmoduli[{1,n}]$ we have 
$$  \langle \delta_{\irr} \psi_2\cdots\psi_n  \rangle_{1,n} = \frac{(n-1)!}{2}
\quad 
\text{and} 
\quad  \langle \psi_{i} \psi_2\cdots\psi_n  \rangle_{1,n} = \frac{(n-1)!}{24}$$ 
for $1\leq i \leq n$. 
\end{Lemma}
\par
Before starting with the proofs, recall that 
\be \label{eq:psi_del}
\psi_i \delta_{0, \{i,j\}} = 0\,, \ee
\be  \label{eq:pipush}
\pi_{n+1*}\psi_{n+1} \= (2g-2 + n)[\barmoduli[{g,n}]]\,.
\ee
Equation~\eqref{eq:psi_del} follows from the fact that a $\mathbb P^1$-tail 
with two marked points has no non-trivial moduli, and~\eqref{eq:pipush} holds 
because $\psi_{n+1}$ restricted to a fiber of $\pi_{n+1}$ has degree $2g-2+n$. 
\par
Since $\psi_i = \pi^{*}_{n+1}\psi_i + \delta_{0, \{i, n+1\}}$ for $i\neq n+1$, 
by \eqref{eq:psi_del} and the projection formula we obtain that 
\be \label{eq:projection}
\pi_{n+1*}(\psi_1^{a_1}\cdots\psi_{n}^{a_n}\psi_{n+1}^{a_{n+1}}) \= 
\pi_{n+1*}(\psi_{n+1}^{a_{n+1}})(\psi_1^{a_1}\cdots\psi_{n}^{a_n})\,.
\ee
As a special case when $a_{n+1} = 1$, by \eqref{eq:pipush} we obtain the {\em dilaton equation} (see \cite[(2.45)]{witten}) 
\be \label{eq:dilaton}
 \langle\prod_{i=1}^n \psi_i^{a_i}\psi_{n+1}\rangle_{g,n+1} \= 
(2g-2 + n) \,\langle \prod_{i=1}^n \psi_i^{a_i}\rangle_{g,n}\,.
\ee
\par
\begin{proof}[Proof of Lemma~\ref{le:psi-null}]
In the case when $|S| = 2$ or $n = 2$, the result follows 
from~\eqref{eq:psi_del}. Suppose it holds for all $S$ 
on $\barmoduli[{1,k}]$ with $k < n$ and for $S$ on $\barmoduli[{1,n}]$ with $|S| < j$. 
Without loss of generality, assume that $n \in S$ 
and let $S' = S\ssm\{n\}$. Since
$\pi_{n}^{*} \delta_{0, S'} = \delta_{0, S'} + \delta_{0, S}$, we
obtain that 
\begin{align*}
 \langle \delta_{0,S} \psi_2\cdots\psi_n  \rangle_{1,n}  
& \= \langle \delta_{0,S'} \pi_{n*}(\psi_2\cdots\psi_n)  \rangle_{1,n}   \\
& \=  (n-1)  \langle \delta_{0,S'} \psi_2\cdots\psi_{n-1}  \rangle_{1,n-1} \quad \=  0\,, 
\end{align*}
using \eqref{eq:pipush}, \eqref{eq:projection} and the induction hypothesis. 
\end{proof}
\par
\begin{proof}[Proof of Lemma~\ref{le:psi-irr}] 
To prove the first formula we use that  $\pi_{n}^{*}\delta_{\irr} = \delta_{\irr}$ 
and hence 
$$\langle \delta_{\irr} \psi_2\cdots\psi_n  \rangle_{1,n} \= (n-1) \langle \delta_{\irr} \psi_2\cdots\psi_{n-1}  \rangle_{1,n-1} $$
by the projection formula and \eqref{eq:pipush}. 
The result follows by induction and from~\eqref{eq:caseM11}. 
\par
For the second formula we can assume without loss of generality that $i=1$ or $i=2$. 
The dilation equation \eqref{eq:dilaton} implies that 
$$ \langle \psi_{i} \psi_2\cdots\psi_n  \rangle_{1,n} 
\= (n-1) \langle \psi_{i} \psi_2\cdots\psi_{n-1}  \rangle_{1,n-1}\,. $$
For $i = 1$, the result follows by induction and from \eqref{eq:caseM11}. 
{{For $i = 2$, note that 
$$\langle  \psi_2^2 \rangle_{1,2} = \langle  \psi_1^2 \rangle_{1,2} = \langle  \psi_1 \rangle_{1,1} = \frac{1}{24}\,,$$
which can be seen by using the relation $\psi_1 = \pi_2^{*}\psi_1 + \delta_{0, \{ 1,2\}}$ and the projection formula 
for the map $\pi_2: \barmoduli[{1,2}]\to \barmoduli[{1,1}]$. Then 
the result follows similarly by induction. 
}}
\end{proof}
\par
For the proof of Theorem~\ref{thm:hurwitzLckappa} we also need the following 
statement.
\par
\begin{Lemma}
\label{le:psi-omega}
Let  $\omega_{\pi_{n+1}}$ be the first Chern class of the relative dualizing 
sheaf associated to $\pi_{n+1}$. Then $  \langle \pi_{n+1*}(\omega_{\pi_{n+1}}^2) \psi_2\cdots\psi_n  \rangle_{1,n} = 0$.
\end{Lemma}
\par
\begin{proof}
From the relation (see e.g. \cite{Logan})  
$$ \psi_{n+1} \= \omega_{\pi_{n+1}}+ \sum_{i= 1}^n\delta_{0,\{i,n+1\}}\, $$
in the tautological ring, we deduce 
\begin{align*}
\pi_{n+1*}(\psi^2_{n+1}) & \=  \pi_{n+1*} \left(\psi_{n+1}\Big(\omega_{\pi_{n+1}}+ \sum_{i= 1}^n\delta_{0,\{i,n+1\}}\Big)\right) \\  
                                         & \=  \pi_{n+1*}\left(\psi_{n+1}\omega_{\pi_{n+1}} \right) \\ 
                                         & \=  \pi_{n+1*}(\omega_{\pi_{n+1}}^2) + \sum^n_{i=1} \psi_i\, 
 \end{align*}
where we used $\omega_{\pi_{n+1}}\delta_{0,\{i,n+1\}} = - \delta^2_{0,\{i,n+1\}}$, 
{{$\pi_{n+1*}(\delta^2_{0,\{i,n+1\}}) = - \psi_i$ (see e.g. \cite[Table 1]{Logan})}} and~\eqref{eq:psi_del} in the above. It follows that 
\begin{align*}
   \langle \pi_{n+1*}(\omega_{\pi_{n+1}}^2) \psi_2\cdots\psi_n  \rangle_{1,n} & \=  \langle \pi_{n+1*}(\psi^2_{n+1})\psi_2\cdots\psi_n  \rangle_{1,n} - \sum_{i=1}^n \langle \psi_i \psi_2\cdots\psi_n  \rangle_{1,n}\\ 
& \=  \langle \psi_{n+1} \psi_2\cdots\psi_{n+1}  \rangle_{1,n+1} - \sum_{i=1}^n \frac{(n-1)!}{24} \quad \= 0, 
\end{align*}
 where we applied \eqref{eq:projection} and Lemma~\ref{le:psi-irr} in the last two steps. 
\end{proof}

Assuming Theorem~\ref{thm:betainH2} for the moment, we can prove 
formula~\eqref{eq:ekzmain} and thus the rationality of the sum of 
Lyapunov exponents for Hurwitz spaces using intersection theory only.
\par
\begin{proof}[Proof of Theorem~\ref{thm:hurwitzLckappa}]
By Theorem~\ref{thm:betainH2} and Kontsevich's formula~\eqref{eq:kmain}
we need to evaluate the quotient 
$$L \= \frac{\langle\lambda (f^{*}\psi_2)\cdots (f^{*}\psi_n)\rangle_{\BH}}
{\langle(f^{*}\lambda)(f^{*}\psi_2)\cdots (f^{*}\psi_n)\rangle_{\BH}}\,, $$
{{where the class $\gamma_1$ in~\eqref{eq:kmain} is $f^{*}\lambda$ in this case, since the generating differentials on the covering curves are pulled back from the target elliptic curves and on $\barmoduli[1,1]$ the Hodge bundle is a line bundle with first Chern class $\lambda$.}}
By the projection formula, the denominator is equal to 
\begin{eqnarray*}
N^0 \langle\lambda\psi_2\cdots\psi_n\rangle_{1,n} & = 
& N^0 \langle(\pi_n^{*}\lambda) \psi_2\cdots\psi_n\rangle_{1,n} \\ 
                                                    & = & N^0 \langle\lambda \pi_{n*}(\psi_2\cdots\psi_n)\rangle_{1,n-1} \\ 
                                                    & = & N^0 (n-1) \langle \lambda \psi_2\cdots\psi_{n-1}\rangle_{1,n-1}  \quad \= \cdots \\ 
                                                    & = & \frac{N^0(n-1)!}{24}
\end{eqnarray*}
by recursion.
Next, we evaluate the numerator. Noether's formula states that
$12\lambda \= \pi_{*}(\delta_{\cXX} + \omega^2_{\pi})$ 
where~$\delta_{\cXX}$ is the class of the nodal locus in the universal curve $\cXX$ over the Hurwitz space. 
Hence the numerator is equal to 
$$ \langle (f_{*}\lambda) \psi_2\cdots \psi_n \rangle_{1,n}
\=  \frac{ \langle (\pi_{n+1*}h_{*}\delta_{\cXX} + \pi_{n+1*}h_{*}(\omega^2_{\pi}))
\psi_2\cdots \psi_n\rangle_{1,n}}{12}\,. $$
Using Lemmas~\ref{le:psi-null},~\ref{le:psi-irr}, and 
Theorem~\ref{thm:intnumber}, we obtain that 
$$\langle(\pi_{n+1*}h_{*}\delta_{\cXX})\psi_2\cdots \psi_n\rangle_{1,n}
\= \frac{(n-1)!}{2}\,c_{-1}^0(d,\Hmu)\,. $$
\par
For the other term involving $\omega^2_{\pi}$, we apply the Riemann-Hurwitz 
formula 
$$ \omega_{\pi} = h^{*}\omega_{\pi_{n+1}} + \sum_{i,j} m_{ij} \Gamma_{ij}\,, $$
where $\Sigma_i$ is the section of the $i$-th branch point and 
$\Gamma_{ij}\subset \cXX$ is the section of ramification order $m_{ij}$ 
in the inverse image of $\Sigma_i$. Consequently,  
$$h_{*}(\omega_{\pi}^2) \= h_{*}(h^{*}\omega_{\pi_{n+1}})^2 + 2 \sum_{i,j}m_{ij} 
(h_{*}\Gamma_{ij})\omega_{\pi_{n+1}} + \sum_{i,j}m_{ij}^2 h_{*}(\Gamma^2_{ij})\,. $$
Using the relations 
$$h^{*}\Sigma_i \= \sum_{j}(m_{ij}+1)\Gamma_{ij}, \quad  h_{*}\Gamma_{ij} \= 
N^0\Sigma_i, \quad \text{and} \quad  \Gamma_{ij}\Gamma_{kl} \= 0$$ 
for $(i,j) \neq (k,l)$, we obtain that 
$$ h_{*}(\Gamma^2_{ij}) \= \frac{1}{m_{ij}+1}(h^{*}\Sigma_i)\Gamma_{ij} 
\= \frac{N^0}{m_{ij}+1}\Sigma_{i}^2\,. $$
Moreover, we have 
$$ \omega_{\pi_{n+1}}\Sigma_i \= - \Sigma_i^2, \quad 
h_{*}(h^{*}\omega_{\pi_{n+1}})^2 \= dN^0 \omega^2_{\pi_{n+1}}.$$
Using these equalities, we obtain that
$$ h_{*}(\omega^2_{\pi}) \=  dN^0 \omega^2_{\pi_{n+1}} - 
N^0 \Big( \sum_{i,j}\frac{m_{ij}(m_{ij}+2)}{m_{ij}+1}\Sigma_i^2   \Big), $$
$$ \pi_{n+1*} h_{*}(\omega^2_{\pi}) \=  dN^0 \pi_{n+1*}(\omega^2_{\pi_{n+1}}) 
+ N^0 \Big( \sum_{i,j}\frac{m_{ij}(m_{ij}+2)}{m_{ij}+1}\psi_i \Big). $$
Applying Lemmas~\ref{le:psi-irr} and \ref{le:psi-omega}, we conclude that 
\bas
 \langle\pi_{n+1*} h_{*}(\omega^2_{\pi})\psi_2\cdots\psi_n \rangle_{1,n} 
&\= \frac{N^0(n-1)!}{24}\,\Big( \sum_{i,j}\frac{m_{ij}(m_{ij}+2)}{m_{ij}+1}\Big)  \\
& \= \frac{N^0(n-1)!}{2}\,\kappa\,.  
\eas
Assembling all the ingredients we computed, we thus obtain the desired equality. 
\end{proof}
\par

\section{Identifying the $\beta$-class} \label{sec:beta}

The first aim of this section is to justify, as we claimed in the previous section,
that the integration against the transverse measure~$\beta$ used to define 
the sum of Lyapunov exponents is proportional to the cup product with
a rational cohomology class. We treat the case of the
$\SL\RR$-invariant manifold $\oamoduli[{1,n}]$. The proof is to some 
extent parallel to that in~\cite{Ba07L}. However in our situation, periods cannot
be used at every point to provide coordinates of the locus. The use of
cross-ratio coordinates is a new ingredient here. Both the proofs here 
and in~\cite{Ba07L} rely on the fact that the REL-foliation is of complex 
codimension one, transverse to the foliation of $\SL\RR$-orbits. Such an 
$\SL\RR$-invariant manifold is called of {\em rank one} and presumably, 
the identification of $\beta$ as a multiple of a rational cohomology class
can be achieved for all rank-one $\SL\RR$-invariant manifolds. 
\par 
Recall that $\SL\RR$-invariant manifolds $H$ have a natural projection $\pi: H \to \PP H$ by modulo 
$\CC^{*}$. For such a manifold $\PP H$ the disintegration along the image of
the $\SL\RR$-foliation of the $\pi$-pushforward of
the Masur-Veech measure $\nu_1$ (and here, for $\PP \oamoduli[1,n] = 
\moduli[1,n]$, even more concretely, the symmetric space measure $\ol{\nu}_1$, 
see \S~\ref{sec:PCIMF}) can be made explicit. In general, let~$M$
be a manifold with a measure $\nu$ and a foliation~$\cFF$ whose leaves 
are Riemannian manifolds. For a $p$-form~$\omega$ we define a 
function $||\omega||_{\cFF}$ by 
\bes 
||\omega||_{\cFF} \= \sup_{v_1,\ldots,v_p \in T\cFF} \frac{\omega(v_1,\ldots,v_p)}
{||v_1||\cdots||v_p||}
\ees
and we let 
\be \label{eq:defnormFF}
\int_\cFF \omega \= \int_M ||\omega||_\cFF d\nu\,.
\ee
\par
We first apply this definition to $M = \PP  \oamoduli[1,n]$, the
push-forward of $\nu_1$, and the image foliation $\cFF$ of $\SL\RR$-orbits.
Its leaves are quotients of~$\HH$, provided with the Poincar\'e metric. 
It follows from a local calculation and the definition 
of $\nu_1$ that on $2$-forms the functionals 
$\omega \mapsto \int_\cFF \omega$ and $\omega \mapsto \int_{\moduli[1,n]} 
\beta \wedge \omega$ are proportional.
\par
\begin{Prop} \label{prop:betaascurrent}
The integration along $\cFF$, i.e.\ the map $\omega \mapsto \int_\cFF \omega$, 
defines a closed current of dimension~$2$ on $\barmoduli[1,n]$.
\end{Prop}
\par
By slight abuse of notation and suppressing the proportionality constant 
we denote the current defined by integration along $\cFF$ by~$\beta$.
\par
The second aim of this section is the proof of Theorem~\ref{thm:betainH2}.
In view of Proposition~\ref{prop:picgen} and Lemma~\ref{le:psi-null}, it is equivalent to show the 
following proposition.
\par
\begin{Prop} \label{prop:betadeltaS}
Let $S\subset \{1,\ldots,n \}$ with $|S| \geq 2$. Then $\langle \delta_{0,S} \beta \rangle_{1,n}=0$.
\end{Prop}
\par
\medskip
We prepare for the proof of Proposition~\ref{prop:betaascurrent} and
recall Mumford's notion of forms of Poincar\'e growth. For this purpose
we provide open sets isomorphic to $(\Delta^*)^k \times \Delta^n$ with a 
metric~$\rho$ by putting the Euclidean metric on the $\Delta$-factors and the 
Poincar\'e metric on the $\Delta^*$-factors. We say that a $p$-form~$\omega$ 
on a manifold $X$ has {\em Poincar\'e growth} with respect to a divisor
$D$, if $X$ can be covered by polydiscs $V_\alpha \cong \Delta^n$ such that
$U_\alpha = V_\alpha \cap (X \ssm D)  \cong (\Delta^*)^k \times \Delta^{n-k}$ 
and $||\omega||_\rho$ is bounded on each of the $U_\alpha$.
\par
Since the volume form on $\barmoduli[1,n]$ has Poincar\'e growth 
with respect to the divisor $\delta_{\irr}$, the following is the 
main step towards proving Proposition~\ref{prop:betaascurrent}.
\par
\begin{Lemma} \label{le:Poincarebounded}
For any $2$-form $\omega$ on $\barmoduli[1,n]$  of Poincar\'e growth 
with respect to the divisor $\delta_{\irr}$, the norm $||\omega||_{\cFF}$
is bounded.
\end{Lemma}
\par
\begin{proof}[Proof of Lemma~\ref{le:Poincarebounded}]
For each boundary point of $\barmoduli[1,n]$ let $U \cong \Delta^n$
be a sufficiently small open neighborhood such that 
$U \cap \moduli[1,n] \cong (\Delta^*)^{r} \times \Delta^{n-r}$. Recall that on 
$U$ we consider the metric $\rho$ as the product of the Poincar\'e
metrics on the $\Delta^*$-factors and the Euclidean metric on the 
$\Delta$-factors. It suffices to check that
$||v||_\rho/||v||_{\cFF}$ is bounded for any vector field $v$ on $U$.
Since $\cFF$ has complex dimension one it suffices to check for 
any vector field tangent to $\cFF$ that each of the factors contributing
to $||v||_\rho$ is bounded.
\par
We first consider a neighborhood of a generic point in $\delta_{\irr}$.
As coordinates in $\omoduli[1,n]$  we use the period coordinates
$(z_\alpha, z_\beta, z_2,\ldots,z_n)$ as defined in~\eqref{eq:defperco}.
We choose the representative of our point in $\moduli[1,n] = \PP 
\omoduli[1,n]$ to have $z_\alpha = 1$. We take~$v$ to be the tangent
vector field to the action of the diagonal subgroup of $\SL\RR$ 
given by the matrices $a_t = {\rm diag}(e^{-t/2},e^{t/2})$. Then, in
terms of the coordinates $\tau_1 = z_\beta/z_\alpha,\,\, v_2 = z_2/z_\alpha,\ldots
\, , v_n = z_n/z_\alpha$, the action is given by   
$$a_t(\tau_1,v_2,\ldots,v_n) \= (\Re(\tau_1) +i e^t \Im(\tau_1), \,\,\,
\Re(v_2)+ie^{t}\Im(v_2), \ldots,\Re(v_n)+ie^{t}\Im(v_n)). $$
Consequently, a unit tangent vector field is given by
\begin{equation}
\label{eq:unittv}
 v \= i\Im(\tau_1) \frac\partial{\partial \tau_1} \+ 
\sum_{j=2}^n \, i \Im(v_j) \frac\partial{\partial v_j}.
\end{equation}
In the polydisc coordinates $q_1 \= e^{2\pi i \tau_1}, \,\, v_2,\ldots, v_n$
around $\delta_{\irr}=\{q_1=0\}$ the tangent vector is 
\begin{equation}
\label{eq:qunittv}
v \=  q_1 \log |q_1| \frac\partial{\partial q_1} \,+\, \sum_{j=2}^n  
\, i \Im(v_j) \frac\partial{\partial v_j}.
\end{equation}
The first summand is bounded by definition of the Poincar\'e metric and  
the boundedness is obvious for all the remaining summands.
\par
\smallskip
Next we consider a neighborhood of a generic point $Q$ in $\delta_{0,S}$.
We introduce the following convenient coordinate system. 
Denote by~$s$ the cardinality of $S$ and relabel the marked points so that 
$S = \{ n-s+1, \ldots, n \}$.  The point $Q$ parameterizes a flat 
surface $(E,\omega)$ of genus one with the marked points 
$P_1,\ldots,P_{n-s}$ and a rational tail with the marked points 
$P_{n-s+1},\ldots,P_{n}$ attached to $Q$ at a point $K$. Let
$(\tau_1^{(0)}, v_{2}^{(0)},\ldots, v_{n-s}^{(0)}, v_K^{(0)})$
be the period coordinates of $(E,\omega,P_1,\ldots,P_{n-s},K)$, 
normalized as above such that $z_{\alpha}^{(0)}=1$. Smooth surfaces in a 
neighborhood of~$Q$ are represented by flat surfaces $(E,\omega,
P_1,\ldots,P_n)$ such that the normalized coordinates 
$(\tau_1, v_2,\ldots, v_n)$ 
have the following properties. The coordinates $\tau_1$
and $v_i$ for $i=2,\ldots,n-s$ are close to their initial values (denoted by an upper index $(0)$), the coordinate 
$v_{n-s+1}$ is close to $v_K^{(0)}$, and 
$v_{n-s+j}$ is close to $v_{n-s+1}$ for $j=2, \ldots, s$. 
Let $t_{n-s+j} = v_{n-s+j} - v_{n-s+1}$ for $j = 2, \ldots, s$ and let 
$u_{n-s+j} = t_{n-s+j}/t_{n-s+2}$ for $j = 3, \ldots, s$. Here $u_{n-s+j}$ measures the approaching rate 
of $z_{n-s+j}$ to $z_{n-s+1}$ with respect to that of $z_{n-s+2}$ to $z_{n-s+1}$. 
\par
With this normalization, the cross-ratio coordinate system on a polydisc neighborhood 
around $Q$ we use is $(\tau_1, v_2,\ldots,v_{n-s}, v_{n-s+1}, t_{n-s+2}, u_{n-s+3},\ldots,u_n)$.
In this coordinate system, $t_{n-s+2}$ measures the distance from the boundary $\delta_{0,S}$ 
and the corresponding disc is provided with the Poincar\'e metric, while
all the other discs are provided with the Euclidean metric. Relabeling these
points in $S$ and using that $Q$ is generic in $\delta_{0,S}$ we may assume 
moreover that $u_{n-s+j}$ is bounded near $Q$.
\par
We use the action of the diagonal flow $a_t$ as above. In our chosen 
coordinates a unit tangent vector is
\bas \label{eq:Sunittv}
v &\= i \Im(\tau_1) \frac\partial{\partial \tau_1} \,+\, \sum_{j=2}^n  
\, i \Im(v_j) \frac\partial{\partial v_j} \\
& \=  i \Im(\tau_1) \frac\partial{\partial \tau_1} \,+\, \sum_{j=2}^{n-s+1}  
\, i \Im(v_j) \frac\partial{\partial v_j} + i \Im(t_{n-s+2})\frac\partial{\partial t_{n-s+2}} + \sum_{k=n-s+3}^n f_k
\frac\partial{\partial u_{k}}\,,
\eas 
where 
\bes
f_k  \= \frac{i \Im(t_{n-s+j}) t_{n-s+2} - i \Im(t_{n-s+2})t_{n-s+j}}{t_{n-s+2}^2}\,
\ees
for $k = n-s+j$ and $j\geq 3$. 
From this it is clear that $||v||_\rho$ is bounded near $Q$.
\par
\smallskip
The case that the boundary point lies in the intersection of several boundary
divisors directly follows from the combination of these calculations, 
since $\rho$ is defined as the product metric.
\end{proof}
\par
We will be brief in the remaining steps, following \cite{Ba07L}. 
The preceding lemma and the finite total volume show that
for any two-form $\omega$ of Poincar\'e growth along $\delta_{\irr}$ we have 
$\int_{\cFF} |\omega| < \infty$
and hence integration over $\cFF$ defines a current $\beta$ on 
$\barmoduli[1,n]$. (Details are given in loc.\ cit., Corollary~8.4.)
\par
The final step in the {\em proof of Proposition~\ref{prop:betaascurrent}}
consists of showing that the current is closed. To achieve this we need
to show that $\int_{\cFF} d\eta = 0$ for any smooth one-form $\eta$.
This follows as in \cite[Theorem~8.1]{Ba07L}, by an application of
Stokes' theorem from the following existence statement of suitable
cusp neighborhoods. Let $N \subset \SL\RR$ be the subgroup of
upper triangular matrices and $H$ the horocycle subgroup. 
\par
\begin{Lemma}
For any $\epsilon >0$ there is a closed $H$-invariant neighborhood $W$ 
of $\delta_{\irr}$ such that $\vol(W) < \epsilon$ and such that $\partial W$ 
is transversal to $\cFF$.
\end{Lemma}
\par
Here two submanifolds are called transversal if the sum of their tangent spaces 
generates the whole tangent space at every point of their intersection. Orbits of~$N$ are 
of course both contained in $W$ and $\cFF$.
\par
\begin{proof}
As in \cite{Ba07L} we take a decomposition of $(E,P_1,\ldots,P_n)$ into horizontal
cylinders $C_i$ and let $f((E,P_1,\ldots,P_n))  = \sum \rho({\rm height}(C_i))$
where $\rho: \RR \to \RR$ is a bump function to make the function $f$ smooth near zero.
\par
Let $W_\ell = f^{-1}(\ell,\infty)$. Since the height is $N$-invariant, 
$W_\ell$ is also $N$-invariant. 
Since the total $\nu_1$ volume of $\oamoduli[1,n]$ is finite, the
volume of  $W_\ell$ as $\ell \to 0$ is eventually smaller than $\epsilon$.
\end{proof}
\par
\medskip
We now come to the {\em proof of Proposition~\ref{prop:betadeltaS}}. 
Morally, this is due to the fact that the foliation can be extended to 
$\barmoduli[1,n] \ssm \delta_{\irr}$ and that $\delta_{0, S}$ is a leaf 
of the foliation, which gets no mass from a transverse measure. 
A precise argument, inspired by \cite{Ba07L}, will be given in the remainder 
of this section. 
\par
By definition of the intersection product of the cohomology class
of a two-current and a divisor we have to show that 
$\int_\cFF {\rm PD}(\delta_{0,S}) = 0$,  where ${\rm PD}(\delta_{0,S})$ 
is the two-form Poincar\'e dual to the divisor $\delta_{0,S}$. The idea of
the proof is that this Poincar\'e dual two-form can be represented by
a smooth form with compact support on a tubular neighborhood~$N$ 
of~$\delta_{0,S}$.
\par
We use the coordinates around $\delta_{0,S}$ as in the proof of
Lemma~\ref{le:Poincarebounded}, in particular $t= t_{n-s+1}$ measures the distance 
to the boundary. By \cite[Proposition~6.24 b) and p.~70]{botttu},  
the Poincar\'e dual of $\delta_{0,S}$ is represented by a smooth compactly supported
two-form $\Psi = d(\rho(t)\psi)$ where $\psi$ is a one-form (constructed by
patching angular forms) and where $\rho: [0,\infty) \to \RR$ is a bump function, 
identically one near zero and with support in $[0,1]$.
\par
We cut off the integration over $N$ on two sets. The first set is
$C_n=\pi_1^{-1}(C)$, where $C$ is a neighborhood of the cusp in 
$\barmoduli[1,1]$ bounded by the horocycle $H$ and $\pi_1$ is the morphism forgetting all markings but the first one. 
The second set is a small neighborhood $N_s$ of~$\delta_{0,S}$ given by $t \leq s$. The horocycle flow
defines a foliation $\cFF_H$ whose leaves are contained in the leaves of $\cFF$.
The boundary $H_n$ of $C_n$ and the boundary $B_s$ of $N_s$ are both
foliated by horocycles.
\par
Recall the definition of foliated integrals from~\eqref{eq:defnormFF}. 
By Stokes' theorem
$$ \int_{\cFF} \Psi \= \int_N ||\Psi||_\cFF \= \int_{C_n \cup N_s} ||\Psi||_\cFF  
\= \int_{H_n} ||\rho\,\psi||_{\cFF_H} + \int_{B_s} ||\rho \,\psi||_{\cFF_H}.$$
Since $\Psi$ is smooth, in particular of Poincar\'e growth, its $\cFF$-norm is 
bounded by Lemma~\ref{le:Poincarebounded}. We can thus estimate the last two integrals 
by a constant depending only on $\Psi$ times the length $\ell(H)$ and times
the volume $\nu(B_s)$, respectively. These contributions can both be made arbitrarily
small by shrinking $\ell(H)$ and $s$. The following two lemmas consequently
conclude the proof of Proposition~\ref{prop:betadeltaS}.
\par
\medskip
\begin{Lemma} \label{le:inthoro}
There exists a constant $A_1$, depending only on $n$, such that
$$ \int_{H_n} ||\rho\,\psi||_{\cFF_H} < A_1 \, \ell(H)\,.$$
\end{Lemma}
\par
\begin{Lemma} \label{le:intBs}
There exists a constant $A_2$, depending only on $n$, such that
$$ \int_{B_s} ||\rho\,\psi||_{\cFF_H} < A_2\, s^2\,.$$
\end{Lemma}
\par
\begin{proof}[Proof of Lemma~\ref{le:inthoro}] As in \cite[Lemma~2.4]{Ba07L}, 
one shows using the local coordinates $(q_1 = e^{2\pi i \tau_1},v_2,\ldots,v_{n-s+1},
t_{n-s+2},u_{n-s+3},\ldots,u_{n})$ of Lemma~\ref{le:Poincarebounded}
that for any smooth one-form $\eta$ compactly supported on $N$ the
norm $||\eta||_{\cFF_H}$ is bounded. Next, one shows as in \cite[Lemma~2.5]{Ba07L}
that $||\psi||_{\cFF_H}$ is bounded on compact subsets of~$N$ by compensating 
the singularities with another one-form of bounded $\cFF_H$-norm. Both
calculations happen essentially in the two variables $(q_1, t_{n-s+2})$ as
in loc.~cit, and the other variables are irrelevant.
\par
It now suffices to show that there is a constant $C(\rho)$ such that
\bes
\int_{H_n} {\rm supp}(\rho) d\mu \leq C(\rho) \ell(H)\,
\ees
where $\mu$ is the product of the arc length measure on the horocycle 
and the transverse measure. This is an exercise in hyperbolic geometry 
that is solved in \cite[Lemma~2.6]{Ba07L}.
\end{proof}
\par
\begin{proof}[Proof of Lemma~\ref{le:intBs}]
The claim follows from the boundedness of $||\psi||_{\cFF_H}$
shown in the previous lemma and $\nu(B_s) = \pi s^2 \vol(\moduli[1,1])$. 
\end{proof}
\par

\section{Generating series for counting problems} \label{sec:genser}

The standard procedure to count connected Hurwitz numbers is
to first count all covers (a problem for which functions involved are
nice, e.g.\ shifted symmetric), then to pass to covers without unramified
components (which involves taking $q$-brackets), and finally to apply inclusion-exclusion
to reduce to the connected case. We show in this section that this procedure 
applies in principle also to the counting problems with Siegel-Veech weight, 
if one takes into account that the Siegel-Veech weight is {\em additive}
on a disjoint product of permutations, in contrast to the constant weight~$1$
which is {\em multiplicative}.
\par
We provide first examples of all these generating series and state at the end 
of the section in Theorem~\ref{thm:prstr_minus1} one of our
main results, the quasimodularity of generating functions of Siegel-Veech 
constants. 
\par
For the application to Siegel-Veech asymptotics for strata, we will often restrict to the ramification profile where each $\mu^{(i)}$ is 
a {\em cycle}~$\mu_i$, i.e. there is only one ramification point in each fiber 
over $P_i$.
\par
\subsection{Counting connected and possibly disconnected coverings} 
So far, we have imposed the connectivity constraint on the coverings.
We remove the upper index zero, if we take all coverings (of profile $\Hmu$)
into consideration. As technical intermediate notion we will also consider 
coverings without unramified components and reflect this in the notation by
a prime. Consequently, we define $\Cov_d(\Hmu)$ to be the set of all
Hurwitz tuples~$h \in S_d^{n+2}$ (without the transitivity hypothesis)
and we let $\Cov'_d(\Hmu)$ be the subset of Hurwitz tuples~$h =(\ual, \ube, 
\uga_1, \cdots, \uga_n)$ in $\Cov_d(\Hmu)$ where the action of the
subgroup $\langle \uga_1, \cdots, \uga_n \rangle$ is non-trivial on
every $\langle h\rangle$-orbit. We denote by $N_d(\Hmu)$ and $N'_d(\Hmu)$
the number of the corresponding Hurwitz classes including the usual weight of $1/\Aut(p)$, i.e. 
\be \label{eq:NdCovd}
N_d^{*}(\Hmu) \= \frac{|\Cov_d^{*}(\Hmu)|}{d!} \quad \text{for} \quad 
* \in \{', 0, \emptyset\}.
\ee
\par
To express the passage between these counting problems we work with the 
generating series
$$ N(\Hmu) \= \sum_{d=0}^\infty N_d(\Hmu) q^d, \quad N'(\Hmu) \= \sum_{d=0}^\infty 
N_d'(\Hmu) q^d, \quad N^0(\Hmu) \= \sum_{d=0}^\infty N^0_d(\Hmu) q^d$$
for all (resp.\ without unramified components, resp.\ connected) coverings.
For the empty branching profile, we drop the argument $\Hmu$, in particular, 
$$ N() \= (q)_\infty^{-1} = \sum_{\lambda} q^{|\lambda|}  \= 
1 + q + 2q^2 + 3q^3 + 5q^4 + 7q^5 + \cdots $$
is the partition function, {{where 
$(q)_{\infty} = \prod_{n\geq 1}(1 - q^n)$}}. From   
$$ |\Cov_d(\Hmu)| = \sum_{j=0}^d \binom{d}{j}\, |\Cov'_j(\Hmu)|\,|\Cov_{d-j}()| $$
we derive the passage between the generating functions 
\begin{equation} \label{eq:NNN}
N'(\Hmu) = N(\Hmu)/N()\,, 
\end{equation}
{{see e.g.~\cite{eo}. }}
\par
Next, we recall the passage from $N'(\Hmu)$ to $N^0(\Hmu)$. We denote by 
$\PPP(n)$ or $\PPP(N)$ the set of partitions of the set $N = \{1,\ldots,n\}$. 
Recall also the notation $\Part(n)$ which is the set of partitions of $n$ (not of the set $N$). 
We now use our assumption that 
each $\mu^{(i)}$ is a cycle, i.e.\ there is only one ramification point 
in each fiber over the branch point $P_i$, which is sufficient for later 
applications in the paper. Under this assumption any covering $p$ without 
unramified components induces a partition $\alpha \in \PPP(n)$ 
corresponding to the ramification points of the connected components
of the covering. This implies
\begin{equation} \label{eq:NprimeNN}
N'(\Hmu) \= \sum_{\alpha \in \PPP(n)} \prod_{j=1}^{\ell(\alpha)} 
N^0(\Hmu_{\alpha_j})\,,
\end{equation}
where $\Hmu_{\alpha_k}$ is the subset of the ramification profile corresponding
to the indices appearing in the $k$-th subset $\alpha_k$ of $\alpha$. We are
rather interested in expressing $N^0(\Hmu)$ in terms of $N'(\Hmu_{\alpha})$. It
follows from~\eqref{eq:NprimeNN} and M\"obius inversion that
\be \label{eq:NNfromNpr}
N^0(\Hmu) \= \sum_{\alpha \in \PPP(n)} (-1)^{\ell(\alpha)-1} 
(\ell(\alpha)-1)!\, \prod_{j=1}^{\ell(\alpha)} N'(\Hmu_{\alpha_j})\,.
\ee
\par
Finally, we recall the classical Burnside Lemma  (see e.g.\ 
\cite[Theorem~A.1.10]{LanZvon}) that 
the number of coverings with ramification
profile $\Hmu$ and {{any permutation $\mu^{(i)}$}}
is given by
\be \label{eq:Burnside}
 N_d(\Hmu) \= \sum_{\lambda \in \Part(d)} \prod_{i=1}^n f_{\mu^{(i)}}(\lambda)\,, 
\ee
where a conjugacy class $\sigma$ is completed with singletons
to form a partition of $|\lambda|$ and where
\be \label{eq:defssf}
f_{\sigma}(\lambda) \= z_{\sigma} \chi^\lambda(\sigma)/\dim \chi^\lambda\,. 
\ee
Here $z_\sigma$ denotes the size of the conjugacy class of $\sigma$ and 
$\dim \chi^\lambda$ is the dimension of representation $\lambda$. We
also write $f_k$ for the special case that $\sigma$ is a $k$-cycle.
\par
\medskip
We specialize now even further for the case of simply branched coverings, 
i.e.\ $\mu_i$ being the class $\Tr$ of a transposition for all~$i$. In 
this case the number of branch points $n=2k = 2g-2$ is even. 
For small values of $k$ the generating series are
\begin{align*}
N(\Tr^2) &\= 2q^2 + 18q^3 + 80q^4 + 258q^5 + \cdots \\
N'(\Tr^2) &\= 2q^2 + 16q^3 + 60q^4 + 160q^5 + \cdots \\
N^0(\Tr^2) &\=  N'(\Tr^2)   \\
\\
N(\Tr^4) &\= 2q^2+162q^3+2624q^4+21282q^5 + \cdots \\
N'(\Tr^4) &\= 2q^2 + 160q^3 + 2460q^4 + 18496q^5 + \cdots  \\
N^0(\Tr^4) &\= 2q^2 + 160q^3 + 2448q^4 +18304q^5 +  \cdots  \\
\end{align*}

\subsection{Generating series for Siegel-Veech counting} 

Recall from \eqref{eq:cdmu} the combinatorial definition of the $p$-weighted Siegel-Veech constant $c^0_{p}(d,\Hmu)$
for connected covers. In the same way as \eqref{eq:cdmu} we can define the $p$-weighted Siegel-Veech constants 
$c_{p}(d,\Hmu)$ for all covers and $c'_p(d,\Hmu)$ for covers without unramified components, by taking the Hurwitz tuples 
ranging over all covers and over covers without unramified components, respectively. 
\par
As in the classical counting case, we introduce for counting
with Siegel-Veech weight the generating series
\ba \label{eq:def:cpseries}
c_p(\Hmu) = \sum_{d \geq 0} c_{p}(d,\Hmu) q^d, \quad\! c'_p(\Hmu) = \sum_{d \geq 0}
 c'_{p}(d,\Hmu) q^d, \quad\! c^0_p(\Hmu) = \sum_{d \geq 0} c^0_{p}(d,\Hmu) q^d 
\ea
for counting all (resp.\ without unramified components, resp.\ connected)
covers with $p$-weighted Siegel-Veech constants and study the passage between them.
\par
We first simplify the sum~\eqref{eq:cpHurwitz} by reducing 
from~$n$ terms per Hurwitz tuple to just one summand.
\par
\begin{Lemma} \label{le:cntimesalpha}
For $* \in \{', 0, \emptyset\}$ and any ramification profile $\Hmu$, 
we have
\begin{equation} \label{eq:cpsimple}
c^*_{p}(d,\Hmu) \=   \sum_{j=1}^{N^*_d(\Hmu)} S_p(\alpha^{(j)})\,,
\end{equation}
where $\alpha^{(j)}$ is the first element of the Hurwitz tuple $h_j$.
\end{Lemma}
\par
\begin{proof}
If $(\alpha, \beta, \gamma_1, \cdots, \gamma_n)$ is a Hurwitz tuple of profile $\Hmu$, 
i.e.\ satisfying the relation 
$$[\beta^{-1},\alpha^{-1}] \= \beta^{-1}\alpha^{-1}\beta\alpha \= 
\gamma_n \cdots \gamma_1\, $$ 
then
$$ [\beta^{-1},\gamma_1\alpha^{-1}]
 \=  (\beta^{-1} \gamma_1 \beta)\cdot \gamma_n \cdots \gamma_2$$
gives rise to a Hurwitz tuple 
$(\alpha\gamma_1^{-1},\beta, \gamma_2, \cdots, \gamma_n, (\beta^{-1} \gamma_1 \beta))$
of the profile~$\Hmu' = (\mu^{(2)}, \ldots, \mu^{(n)}, \mu^{(1)})$. 
This map is a bijection between Hurwitz tuples, which is equivariant
with respect to simultaneous conjugation. On the other hand, 
$$[\beta^{-1},\alpha^{-1}] \= (\gamma_{n} \gamma_{n-1} \gamma_{n}^{-1}) 
\gamma_{n} \gamma_{n-2} \cdots \gamma_1 $$ 
is a  Hurwitz tuple with the same $(\alpha,\beta)$ and with the profile
where the last two points are swapped. Iterating the use of such 
transforms in the profile gives a bijection between Hurwitz tuples of profile~$\Hmu$ and $\Hmu'$
that preserves $(\alpha,\beta)$. The combination of the two observations
shows that the sums over all Hurwitz tuples of the contribution of
$\sigma_0 = \alpha $ to~\eqref{eq:cdmu} and the contribution of 
$\sigma_1 = \alpha \gamma_1^{-1}$ coincide. Iterating this comparison~$n$ 
times for all $\sigma_i$ proves the claim.
\end{proof}
\par
The passage from $c_p(\Hmu)$ to $c'_p(\Hmu)$ in the following proposition  
uses essentially that Siegel-Veech weights are additive on disjoint cycles in the sense that
$S_p(\lambda) = \sum_{i \geq 0} S_p(\lambda_i)$ for a partition 
$\lambda = (\lambda_1,\lambda_2,\ldots)$.
\par
\begin{Prop} \label{prop:cpconversion}
Let $\Hmu = (\mu_1, \ldots, \mu_n)$ be a ramification profile with $n$ branch 
points and each $\mu_i$ being a cycle. Then for any $p$ the 
generating series for Siegel-Veech counting without unramified
components and for Siegel-Veech counting with connected coverings are related by
\be \label{eq:cmupcmu0}
 c'_p(\Hmu) \= \sum_{\sigma \in \PPP(n)} \, \sum_{k=1}^{\ell(\sigma)}  
c_p^0(\Hmu_{\sigma_k}) \prod_{j=1, j\neq k}^{\ell(\sigma)} N^0(\Hmu_{\sigma_j})
\ee
where $\sigma = (\sigma_1,\ldots, \sigma_{\ell(\sigma)})$ and  
$\Hmu_{\sigma_k} = (\{\mu_i\}_{i \in \sigma_k})$.
The generating series for Siegel-Veech counting without unramified
components and for Siegel-Veech counting of all coverings are 
related by 
\be \label{eq:cmucmup}
c_p(\Hmu) \= c'_p(\Hmu) N() + N'(\Hmu) c_p(). 
\ee
\end{Prop}
\par
\begin{proof}
For the first relation, suppose a covering without unramified components corresponds to the partition 
$\sigma \in \PPP(n)$ of the $n$ branch points. Such a covering is given by the data 
$$(\alpha^{(k)}, \beta^{(k)}, \{ \gamma_i \}_{i \in \sigma_k})_{k = 
1,\ldots,\ell(\sigma)}\,.$$ 
By Lemma~\ref{le:cntimesalpha} its contribution to the left hand side is 
$S_p(\alpha^{(1)}\cdots \alpha^{(\ell(\sigma))})$, whereas each summand 
of the interior sum on the right hand side gives a contribution 
of $S_p(\alpha^{(k)})$. Additivity of the function $S_p$ on disjoint cycles   
implies that these contributions are equal.
\par
The second relation follows from the same argument, by decomposing a covering
into its unramified components and into the remaining components.
\end{proof}
\par
\smallskip
The proposition below uses representation theory to reduce the 
computation of Siegel-Veech counting from a sum over all Hurwitz 
tuples to just a sum over pairs of partitions. This expression
will be simplified further in Part~II. 
We emphasize for future use that the following proposition does
not require the additional hypothesis that each $\mu_i$ is a cycle. 
\par
\begin{Prop} \label{prop:SVTpformula}
If $\Hmu =  (\mu_1,\ldots,\mu_n)$ with $\mu_i \in \Part(d)$ any 
partitions, then 
\ba \label{eq:cpclosed}
 c_{p} (d, \Hmu) & \=   
\sum_{\lambda \in \Part(d)} \prod_{i=1}^nf_{\mu_i}(\lambda)  
\frac 1{d!}\sum_{\tau\in \Part(d)} z_{\tau}  S_p(\tau) \chi^\lambda(\tau)^2. 
\ea
\end{Prop}
\par
\begin{proof} We start by recalling the proof of the Burnside Lemma to 
count coverings. If we want to count all factorizations $\prod_{i=1}^{n+2} \gamma_i = 1$ with $\gamma_i \in S_d$
belonging to a fixed conjugacy class $C_i$, then the number of such factorizations is
$$|\Cov_d(C_1,\ldots,C_{n+2})| \= \sum_{\lambda \in \P(d)} \frac{(\dim \chi^{\lambda})^2}{d!}\, \prod_{i=1}^{n+2} f_{C_i}(\lambda),$$
which can be checked by comparing the trace of the action of $\sum_{g \in C_i}{g} 
\in \CC[S_d]$ on the decomposition of $\CC[S_d]$ into irreducible 
representations (e.g.\ \cite[Theorem~A.1.9]{LanZvon}). We will apply this 
to $C_i = \mu_i$ for $i=1,\ldots,n$, for
$\gamma_{n+1} = \alpha$ belonging to any conjugacy class, and for 
$\gamma_{n+2} = \beta \alpha^{-1} \beta^{-1}$ being a conjugate of $\alpha^{-1}$. 
Since there are $d!/z_\alpha$ elements that conjugate a given $\alpha^{-1}$ 
into a given element $\alpha'^{-1} =  \beta \alpha^{-1} \beta^{-1}$, we deduce that
the number of factorizations $[\alpha,\beta] = \prod_{i=1}^n \gamma_i$ 
with $\gamma_i$ in the conjugacy class $\mu_i$ is
$$ |\Cov_d(\Hmu)| \= \sum_{\lambda \in \P(d)} \prod_{i=1}^nf_{\mu_i}(\lambda) \Bigl( \sum_{\alpha \in \P(d)} 
z_\alpha \chi^\lambda(\alpha)^2 \Bigr). $$
If we count with Siegel-Veech weight, using Lemma~\ref{le:cntimesalpha} we see that the innermost 
bracket is $\sum_{\alpha \in \Part(d)} z_\alpha S_p(\alpha) \chi^\lambda(\alpha)^2$ 
instead. Finally recall the relation $N_d^{*}(\Hmu) = |\Cov_d^{*}(\Hmu)| / d!$ and similarly for the Siegel-Veech count, thus proving 
the desired formula. 
\end{proof}
\par
Our initial motivation for the analysis of $q$-brackets in Part~II 
is to prove the following theorem, as one of our main results. 
The special case that $\mu_i = \Tr$ will be analyzed in detail in this paper, since it corresponds to the
counting problems for the principal stratum in genus $g = k+1 = \tfrac n2 + 1$.
\par
\begin{Thm} \label{thm:prstr_minus1}
For each $\mu_i$ being a cycle, the two counting functions $c^0_p(\mu_1,\ldots,\mu_n)$ 
and $c'_p(\mu_1,\ldots,\mu_n)$ are quasimodular forms 
of mixed weight $\leq \sum_{i=1}^n (|\mu_i| + 1) + p + 1$ for $\SL\ZZ$
and for any $p \geq -1$.
\par
If $\mu_i = \Tr$ for all $i$, then the counting functions 
$c^0_p(\Tr^{n})$ and $c'_p(\Tr^{n})$  
are quasimodular forms 
of pure weight $3n + p + 1$ for $\SL\ZZ$ and for any $p \geq -1$.
\end{Thm}
The proof of this theorem will be completed in Section~\ref{sec:applSV}.

\subsection{Examples of the Siegel-Veech counting functions} 

We specialize to the case $\mu_i = \Tr$, the class of a transposition, 
and give examples of the series introduced above.
\par
For $p=1$, the Siegel-Veech counting function $c_1^*(\Hmu) = D(N^*(\Hmu))$
is just the $D=q \frac{\partial}{\partial q}$-derivative for 
$* \in \{\emptyset, ', 0\}$. If $\Hmu = \Tr^n$, then by the Riemann-Hurwitz formula 
$n=2k = 2g-2$ has to be even. 
For small $k$ and for $p=-1$ the first several series are
\ba \label{eq:cm1princEX}
c_{-1}(\Tr^2) &\= \tfrac 52 q^2 + \tfrac{49}2 q^3 + 121q^4 +\tfrac{2593}6 q^5 
+\cdots \\
c_{-1}'(\Tr^2) &\=  \tfrac 52 q^2 + 20q^3 + 75q^4 + 200q^5 + \cdots \\
c_{-1}^0(\Tr^2) &\=  c_{-1}'(\Tr^2) \\ 
\\
c_{-1}(\Tr^4) & \= \tfrac 52 q^2 + \tfrac{441}2 q^3 + 3764q^4 + 
\tfrac{194107}{6}q^5 + \cdots  \\
c_{-1}'(\Tr^4) &\= \tfrac 52 q^2 + 216 q^3 + 3378 q^4 + 25664 q^5 + \cdots \\
c_{-1}^0(\Tr^4) &\= \tfrac 52 q^2 + 216 q^3 + 3348q^4  + 25184q^5 + \cdots
\ea
\par
It was shown by Eskin-Okounkov in \cite{eo} based on
work of \cite{blochokounkov} that $N'(\Hmu)$ and $N^0(\Hmu)$
are quasimodular forms for any $\Hmu$. We will recall these
notions and results in Part~II. 
\newpage

%% file: pic_fundgroups.tex
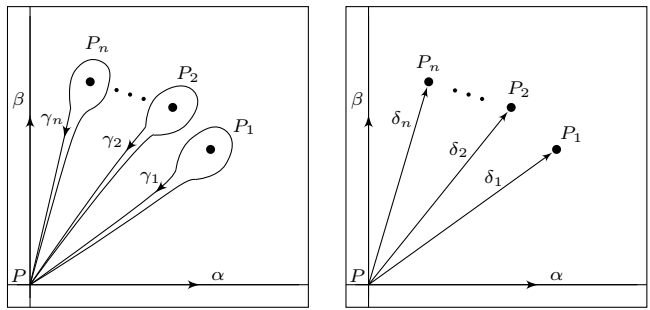
\begin{figure}[h]
\begin{centering}
\begin{tikzpicture}
\tikzset{
	>=latex',
	firstarrow/.style={
		->,
		shorten >=2pt,},
	secondarrow/.style={
		postaction={decorate},
		decoration={markings,mark=at position .65 with
			{\arrow[line width=.7pt]{>}}}}, 
	thirdarrow/.style={
		postaction={decorate},
		decoration={markings,mark=at position .90 with
			{\arrow[line width=.7pt]{<}}}} 
}    

\draw (0,0) node(1){} -- (0,4) node(2){} -- (4,4) node(3){} -- (4,0) node(4){} -- cycle;
\draw (4.5,0) node(5){} -- (4.5,4) node(6){} -- (8.5,4) node(7){} -- (8.5,0) node(8){} -- cycle;
\draw (0,0.3) node(9){} -- (4,0.3) node(10){};
\draw (0.3,0) node(11){} -- (0.3,4) node(12){};
\draw (4.5,0.3) node(13){} -- (8.5,0.3) node(14){};
\draw (4.8,0) node(15){} -- (4.8,4) node(16){};

\draw (9) -- (10) [secondarrow];
\draw (11) -- (12) [secondarrow];
\draw (13) -- (14) [secondarrow];
\draw (15) -- (16) [secondarrow];

\draw (4.8,0.3) -- (5.6,3) [firstarrow];

\begin{scope}[rotate around={-8:(4.8,.3)}]
\draw (4.8,0.3) -- (6.35,2.9) [firstarrow];
\fill (6.35,2.9) circle (1.7pt);+
\end{scope} 

\draw (4.8,0.3) -- (7.3,2.1) [firstarrow];

\fill (1.1,3) circle (1.7pt)
(2.7,2.1) circle (1.7pt)
(5.6,3) circle (1.7pt)
(7.3,2.1) circle (1.7pt);

\draw (0.3,0.3) .. controls (.77,2.32) and (.81,2.49) .. (.84,2.65) [thirdarrow];
\draw (0.3,0.3) .. controls (.82,2.15) and (.9,2.53) .. (1.1,2.63); 
\draw plot [smooth, tension=1] coordinates { (.84,2.65) (.9,3.15) (1.28,3.25) (1.32,2.87) (1.1,2.63)};

\begin{scope}[rotate around={-8:(.3,.3)}]
\draw (0.3,0.3) .. controls (1.3,2.32) and (1.49,2.49) .. (1.53,2.6) [thirdarrow];
\draw (0.3,0.3) .. controls (1.3,2.15) and (1.63,2.53) .. (1.65,2.5); 
\draw plot [smooth, tension=1] coordinates { (1.53,2.6) (1.65,3.05) (2.1,3.15) (2.06,2.7) (1.65,2.5)};
\fill (1.85,2.9) circle (1.7pt);
\end{scope} 

\draw (0.3,0.3) .. controls (2,1.52) and (2.19,1.69) .. (2.24,1.83) [thirdarrow];
\draw (0.3,0.3) .. controls (2,1.35) and (2.33,1.73) .. (2.50,1.7); 
\draw plot [smooth, tension=1] coordinates { (2.24,1.83) (2.45,2.25) (2.9,2.35) (2.9,1.9) (2.50,1.7)};

\tikzstyle{every node}=[font=\scriptsize] 
\node (0) at (-.15,-.15) {$~$};
\node (P) at (0.15,0.42) {$P$};
\node (P) at (4.65,0.42) {$P$};
\node (Pn) at (1.2,3.5) {$P_n$}; 
\node (P2) at (2.4,3.1) {$P_2$}; 
\node (P1) at (3.15,2.4) {$P_1$};
\node (Pn) at (5.6,3.25) {$P_n$}; 
\node (P2) at (6.8,2.9) {$P_2$}; 
\node (P1) at (7.5,2.3) {$P_1$}; 
\node (b) at (0.15,2.75) {$\beta$};
\node (b) at (4.65,2.75) {$\beta$};
\node (a) at (2.8,0.42) {$\alpha$};
\node (a) at (7.3,0.42) {$\alpha$};
\node (gn) at (.6,2.5) {$\gamma_n$};
\node (g2) at (1.4,2.2) {$\gamma_2$};
\node (g1) at (1.9,1.72) {$\gamma_1$};
\node (dn) at (5.25,2.5) {$\delta_n$};
\node (d2) at (6,2.1) {$\delta_2$};
\node (d1) at (6.47,1.72) {$\delta_1$};

\fill (1.45,2.89) circle (.8pt)
(1.63,2.83) circle (.8pt)
(1.81,2.772) circle (.8pt);
\fill (5.95,2.885) circle (.8pt)
(6.13,2.83) circle (.8pt)
(6.31,2.772) circle (.8pt);

\end{tikzpicture} 
\end{centering}
\caption{Standard presentation of $\pi_1(E\ssm \{P_1,\ldots,P_n\})$ 
and standard choice of relative periods}
\label{fig:pointpos}
\end{figure}

%% file: Part2_BOSV.tex
\part*{Part II: Bloch-Okounkov correlators and 
their growth polynomials} 
  
The point of departure for Part~II is a beautiful 
theorem of Bloch and Okounkov saying that the $q$-bracket (a certain weighted average) of any 
``shifted symmetric polynomial" on~the set~$\P$ of all partitions is a quasimodular form. In the first 
section of this part we review some of the many ways to describe elements of~$\P$ and the definition
of shifted symmetric polynomials. Section~\ref{sec:QMofqb} contains the statement of the
Bloch-Okounkov theorem and various complementary results, as well as a review of the definitions 
and main properties of quasimodular forms.  The following section shows how to associate to each 
quasimodular form a ``growth polynomial" that contains information both about the growth of the 
function near cusps and about the growth of its Fourier coefficients.  This notion is essentially
equivalent to one used by Eskin and Okounkov in~\cite{eo}, but since these polynomials can also
be useful in other contexts in the theory of modular forms we give a different and considerably more 
detailed presentation, including alternative descriptions and other basic properties of growth polynomials.

The new results of this chapter are contained in the last three sections.  The main fact is that the 
growth polynomials of the quasimodular forms defined by the Bloch-Okounkov theorem, unlike these forms 
themselves, can be given in terms of explicit generating functions.  This was discovered by Eskin and 
Okounkov in~\cite{eo} in terms of the so-called ``$n$-point correlators".  In Section~\ref{sec:growthBO} 
we give their formula with a different and simpler proof, as well as a second formula 
in terms of an all-variable generating function that we show can be represented by a formal Gaussian integral 
vaguely reminiscent of the path integrals of quantum field theory.  This is then applied in Section~\ref{sec:gensercumu}
to give a new formula for certain special combinations of $q$-brackets called ``cumulants", which are the
expressions that we will need for the applications to the calculation of invariants of moduli 
spaces and Siegel-Veech constants. A result of this type was also given in~\cite{eo}, but here we find a direct 
proof and thus as a corollary a much simpler proof of their result.  Finally, in Section~\ref{sec:OneVariable}
we show how to express the main quantities of interest to us for the geometric applications
in terms of some special power series in one variable, related to the Hurwitz zeta functions, whose 
Taylor coefficients are simple multiplies of Bernoulli numbers.

\section{Partitions and shifted symmetric polynomials} \label{sec:Partitions} 

Let $\P$ denote the set of all partitions.  We use $\l$ to denote a generic element of~$\P$ and $\l^\vee$ 
to denote the dual partition.  The {\it size} of~$\l$ (i.e. the number of which it is a partition) will 
be denoted by~$|\l|$, and $\P(d)$ denotes the set of all partitions of~$d$.

There are (at least) six elementary ways to view a partition, all of which will be used in the sequel. 

\begin{itemize}
\item[(a)] {\em Parts.} We write $\l = (\l_1, \l_2,\ldots)$, 
with $\l_1 \geq \l_2 \geq \cdots$ and $\sum_{i=1}^\infty \l_i = |\l|$.
If $k$ is the largest index such that $\l_k >0$, we call 
$k = \ell(\l)$ the {\it length} of $\l$.
\item[(b)] {\em Multiplicities.} Let $r_1,r_2,r_3,\cdots$ be non-negative integers, almost all equal to~0,
and write $\l = 1^{r_1}2^{r_2}3^{r_3}\ldots$, so that $r_m = |\{j \geq 1: \l_j = m\}|$.
In these coordinates the size of $\l$ is given by $\sum mr_m$ and its length by $\sum r_m$.
\item[(c)] {\em Young diagram.} To any $\l\in\P$ we associate the Young diagram
$$Y_\l = \{(i,j): 1 \leq j \leq k, 1 \leq i \leq \l_j\} \subset \NN^2\,.$$
This clearly gives a bijection between $\P$ and the set of finite subsets of~$\NN^2$ that are closed 
under making either coordinate smaller.  The set~$Y_\l$ is usually denoted pictorially by replacing
the elements of $\NN^2$ by boxes of unit size, oriented so that increasing $i$ moves one to the right
and increasing~$j$ moves one downwards. The Young diagram of $\l^\vee$ is the transpose of~$Y_\l$.
\item[(d)] {\em Frobenius coordinates.} We encode a partition~$\l\in\P(d)$ by a collection of numbers 
\be\label{FrobCoords} (s;\, a_1 \geq \cdots \geq a_s \geq 0,\; b_1 \geq \cdots \geq b_s \geq 0) \ee 
with $a_i, b_i \in \Z$ and $\sum_{i=1}^s (a_i + b_i + 1) = d$. 
These are given in terms of the Young diagram by setting $s$ equal to the length of the main
diagonal of~$Y_\l$ (i.e. the largest $i$ with $(i,i)\in Y_\l$) and by defining $a_i$ and $b_i$ to 
be the number of boxes of $Y_\l$ to the right of or below the diagonal box~$(i,i)$, i.e.\ 
$a_i = \l_i -i$ and $b_i = \l^\vee_i - i= |\{j: \l_j \geq i\}| - i.$ 
\item[(e)] {\em Semibounded subsets.} We let $X_\l = \{ \l_i - i+\h \mid i \geq 1\} \subset \Z+\h$.
Subsets that arise in this way are bounded above (by $\l_1 -\frac12$) and have a complement in 
$\Z+\h$ that is bounded below (by $\h-k$, where $k$ is the length of $\l$). They have the 
further property that the number of positive elements of $X_\l$ is equal to the number of
negative elements of $X_\l^c=(\Z+\h)\smallsetminus X_\l$.  This leads to the last description:
\item[(f)] {\em Balanced subsets.}  There is a bijection between~$\P$ and the set of all finite 
subsets $C\subset\Z+\h$ with $\sum_{c \in C} \sgn(c) = 0$. The set $C_\l$ associated to~$\l$
under this bijection is given in terms of the Frobenius coordinates~\eqref{FrobCoords} of~$\l$
by $C_\l = \{ a_i + \tfrac12 \} \cup  \{- b_j-\tfrac12\}$, and conversely we recover the Frobenius
coordinates by defining the $a_i$ to be the non-negative elements of $C-\tfrac12$ and the 
$-b_i$ to be the non-positive elements in $C + \frac12$. In terms of the set~$X_\l$ of~(e),
we have $C_\l=(\Z+\h)_{>0}\cap X_\l\;\cup\; (\Z+\h)_{<0}\smallsetminus X_\l\,$.
\end{itemize}

\smallskip
For each integer $\ell\ge0$ we define the {\it power sum} function $P_\ell:\P\to\QQ$ by
\be \label{eq:defpk}
P_\ell(\l) \= \sum_{c\in C_\l}\sgn(c)\,c^\ell\= \sum_{i=1}^s\,\bigl[\,(a_i+\h)^\ell
\m(-b_i-\h)^\ell\bigr]\,, \ee
where we have used the descriptions (f) and (d). The first three values are  
$$P_0(\l) = 0, \quad P_1(\l) = |\l|, \quad P_2(\l) = 2\,f_{\Tr}(\l) 
= 2\,z_\Tr\, \frac{\chi^\l(\Tr)}{\chi^\l(1)} $$
with the notations as in Section~6. 
Using the correspondence between~(e) and~(f) we can express $P_\ell(\l)$ in terms of the elements of~$X_\l$ in the form
\be \label{eq:def2pk}  P_\ell(\l) \= \sum_{i=1}^\infty \left( (\l_i -i +\h)^\ell - (-i + \h)^\ell \right) \,. \ee
This shows that $P_\ell$ is an element of the algebra $\Lambda^*$ of {\em shifted symmetric functions} (\cite{KerOls},
\cite{OkOls}), which is defined as $\Lambda^* = \varprojlim \Lambda^*(n)$ where $\Lambda^*(n)$ is the algebra of symmetric polynomials in the $n$ variables \hbox{$\l_1-1$,\,}$\ldots,\l_n-n$ and the projective limit 
is taken with respect to the homomorphisms setting the last variable to zero. In fact, a theorem of Okounkov 
and Olshanski (\cite{OkOls}) states that the algebra $\Lambda^*$ is freely generated by the $P_\ell$ with~$\ell\ge1$.

We will also work with a differently normalized set of functions $Q_k: \P \to \QQ$ 
that are related to the power sum functions by
 \be \label{PtoQ}  Q_0(\l)\=1\,, \qquad Q_k(\l) \= \frac{P_{k-1}(\l)}{(k-1)!}+\b_k \;\quad \text{if $k\ge1\,$,}  \ee
where the constants $\beta_k \in \QQ$, with $\b_0=1$, $\b_1=0$, $\b_2=-\frac1{24}$,\,\dots 
are defined by the power series expansion
\be  \label{eq:defbeta}
B(z) \,:=\, \frac{z/2}{\sinh(z/2)} \=\, \sum_{k=0}^\infty \b_k\,z^k\= 1\,-\,\frac{1}{24}\,z^2 \+ \frac{7}{5760}\,z^4
\+ \cdots 
\ee
The somewhat unnatural-looking definition~\eqref{PtoQ} can be explained by noting that 
$\ell!\,\b_{\ell+1}$ equals
$(1-2^{-\ell})\,\zeta(-\ell)$, which is the natural regularization of the divergent sum 
$\sum_{i=1}^\infty(-i+\h)^\ell$ in~\eqref{eq:def2pk}, so that 
$\ell!\,Q_{\ell+1}$ can be thought of as the regularization of the divergent sum
 $\sum_{x\in X_\l}x^\ell$.  Another way to understand 
the relationship between the $P$'s and the $Q$'s is in terms of generating functions: 
if we set 
 \be\label{defw} w_\l^0(t)\= \sum_{c\in C_\l} \text{sgn}(c)\,t^c\,,\qquad w_\l(t)\=\sum_{x\in X_\l} t^x 
\qquad(|t|>1)\,, \ee
then the functions $W^0_\l(z):=w^0_\l(e^z)$ and $W_\l(z):=w_\l(e^z)$ have 
Laurent series expansions given by
 \be\label{defW} W^0_\l(z)\= \sum_{\ell=0}^\infty P_\ell(\l)\,\frac{z^\ell}{\ell!}\,, 
\qquad W_\l(z) \= \sum_{k=0}^\infty Q_k(\l)z^{k-1} \ee
and the relationship between $X_\l$ and $C_\l$ described above implies that $w_\l(t) \= w^0_\l(t) \+ \frac{\sqrt t}{t-1}$
or $W_\l(z)=W^0_\l(z)+\frac{1/2}{\sinh(z/2)}$.
\par
For later purposes we also introduce yet a third normalization, namely
\be \label{eq:smallpk}
p_\ell(\lambda) \= P_\ell(\lambda) \+ (1-2^{-\ell})\,\zeta(-\ell) 
\= \ell!\,Q_{\ell+1}(\lambda)\,.
\ee
This then agrees with the notation in~\cite{eo} (whereas in~\cite{blochokounkov} 
the symbol $p_\ell$ is used for what we call~$P_\ell$) 
and will be used in Sections~\ref{sec:growthBO} and~\ref{sec:gensercumu}.
\par
Now let $\RRR$ be the ring $\QQ[Q_1,Q_2,\dots]$, with the grading $\RRR=\bigoplus \RRR_k$ 
given by assigning to~$Q_k$ the weight~$k$. (It is in order to define this grading that 
we work with the $Q_k$ rather than the~$P_\ell$.)
To any element $f\in \RRR$ we associate a function on~$\P$, denoted by the same letter, 
by setting $f(\l)=f(Q_1(\l),Q_2(\l),\dots)$.
By the result quoted above, this function lies in~$\Lambda^*$ and all elements of~$\Lambda^*$ arise this way.  
(The ring $\Lambda^*$ is isomorphic to the quotient $\RRR_*/Q_1 \RRR_*$.)

\section{Quasimodular forms and the  Bloch-Okounkov theorem} \label{sec:QMofqb}

Let $f:\P\to\QQ$ be an arbitrary function on the set $\P$ of all partitions.
 Motivated by the averaging operators encountered in classical statistical physics,
Bloch and Okounkov associate to~$f$ the formal power series
\be \label{defqbrac} 
  \sbq f \= \frac{\sum_{\l\in\P} f(\l)\,q^{|\l|}}{\sum_{\l\in\P} q^{|\l|}}\;\,\in\;\QQ[[q]]\,,
\ee
which we will call the {\em $q$-bracket}, and prove that this $q$-bracket is 
a quasimodular form whenever~$f$ belongs to~$\Lambda^*$.  More precisely, their theorem says:

\begin{Thm} [Bloch-Okounkov]\label{thm:bo}
If $f$ is a shifted symmetric function of weight~$k$, then $\sbq{f}$ is a quasimodular form of weight $k$.
\end{Thm} 
\noindent In view of the description of the grading given in the previous section, this says that if $f$ is a weighted
homogeneous polynomial of degree~$K$ in the functions $Q_k$, where $Q_k$ has weight~$k$, then $\sbq f\in\wM_K$.  
To calculate these $q$-brackets, it clearly suffices to calculate them for monomials $Q_{k_1}\cdots Q_{k_n}$.
We therefore introduce the generating Laurent series
\be  \label{eq:defW}
W(z) \= \sum_{k=0}^\infty Q_k\,z^{k-1} \quad\;\in\;\RRR[z^{-1},z]] 
\ee
corresponding to the function $W_\l(z)$ in~\eqref{defW}, and define the {\it $n$-point correlator}
\ba \label{eq:defcorrelator}  F_n(z_1,\dots,z_n) &\= \langle W(z_1)\cdots W(z_n)\rangle_q \\
   &\= \sum_{k_1,\dots,k_n\ge0} \sbq{Q_{k_1}\cdots Q_{k_n}}\,z_1^{k_1-1}\cdots z_n^{k_n-1} \ea
for each $n$. (Here the dependence on $\tau$ and $q=e^{2\pi i\tau}$ has been omitted from the notation on the left, 
and the subscript~$n$ could also be omitted, since it is simply equal to the number of variables.)  Bloch and Okounkov 
give a beautiful identity to compute the functions $F_n$ in terms of the Jacobi theta series
 \be\label{eq:defTh} \th(z) \= \th_\t(z) 
  \= \sum_{\nu \in\Z + \tfrac12} (-1)^{[\nu]}\,e^{\nu z}\,q^{\nu^2/2}\quad\in\;q^{1/8}\,\QQ[[q]][[z]]\,, \ee
the first three cases of this formula being given (with $G_2$ as in~\eqref{defEis}) by 
\ba\label{F1F2F3}  F_1(z_1) &\=\frac{\theta'(0)}{\theta(z_1)}\,, \qquad
 F_2(z_1,z_2)\=\frac{\theta'(0)}{\theta(z_1+z_2)}\,\text{Sym}_2\biggl(\frac{\th'}\th(z_1)\biggr)\,, \\
 F_3(z_1,z_2,z_3) & \=\frac{\theta'(0)}{\theta(z_1+z_2 + z_3)}\,\text{Sym}_3\biggl(\frac{\th'}\th(z_1)\,\frac{\th'}\th(z_1+z_2)
   -\frac{\th''}{2\th}(z_1)-G_2\biggr)\;, \ea 
where ``$\,\text{Sym}_n$" denotes complete symmetrization of a function of~$n$ variables. 
\par
An elementary and very short proof of Theorem~\ref{thm:bo}  is given in \cite{zagBO}, together with several complementary 
results concerning the correlators~$F_n$.  Since several of these will be useful for us later, we list some of them here briefly. 
First, however, we begin by reviewing the definition and main properties of quasimodular forms.
\par
We recall first that a modular form of weight~$k$ on the full modular group $\G=\SL\Z$ is a holomorphic function $\ph$ from the complex
upper half-plane~$\HH$ to~$\CC$ satisfying $\ph\bigl(\frac{a\t+b}{c\t+d}\bigr)=(c\t+d)^k\phi(\t)$
for all $\g=\bigl(\smallmatrix a&b\\c&d\endsmallmatrix\bigr)\in\G$, an example being the Eisenstein series
\be \label{defEis}  G_k(\t)
\= -\,\frac{B_k}{2k}\;+\;\sum_{n=1}^\infty \sigma_{k-1}(n) \,q^n
\qquad(\text{$k>0$ even)},  \ee
if $k\ge4$. (Here $B_k$ denotes the $k$-th Bernoulli number { and
$\sigma_{k-1}(n) = \sum_{d|n}d^{k-1}$}.)
We denote by $M_k$ the space of all modular forms of weight $k$ on~$\G$
and by $M_*=\bigoplus_k M_k$ the corresponding graded ring.
{
For $k \geq 4$ we have $M_k = \CC G_k \oplus S_k$, where $S_k$ is the
subspace of cusp forms (modular forms with no $q^0$ term).
For $k$ odd we have $M_k = 0$ and set $G_k = 0$.}
\par
A quasimodular form is defined by imposing only a weaker transformation law under the action of~$\G$, 
typical examples being $G_2$ and the derivatives of modular forms, but we can omit the intrinsic definition since it is known that 
the algebra $\wM_*$ of all quasimodular forms on~$\G$ is freely generated over $M_*$ by the quasimodular form $G_2$ of weight~2.
More explicitly, using Ramanujan's convenient notations $P$, $Q$ and, $R$ (also often denoted by $E_2$, $E_4$, and $E_6$) for
$-24G_2=1-24q-\cdots\in\wM_2$, $240G_4=1+240q+\cdots\in M_4$, and $-504G_6=1-504q-\cdots\in M_6$, we have
\be \label{structure}  M_* \= \QQ[Q,R]\,,\qquad \wM_*\= M_*[P] \=\QQ[P,\,Q,R]\,. \ee
Examples of the first equality are $G_8=Q^2/480$ and $G_{12}=(441Q^3+250R^2)/65520$. 
A basic fact is that the ring $\wM_*$ is closed under the differentiation operator 
$$ D \= \frac1{2\pi i}\,\frac d{d\t} \= q\,\frac d{dq}\,, $$
as can be seen either directly from the definition that we have omitted or else from the 
structure theorem~\eqref{structure} together with Ramanujan's formulas
\be\label{RamDeriv}  D(P)\,=\,\tfrac1{12}\,(P^2-Q), \;\quad D(Q)\,=\,\tfrac13\,(PQ-R), \;\quad D(R)\,=\,\tfrac12\,(PR-Q^2)\,.\ee
The operator~$D$ acts on~$\wM_*$ as a derivation of degree~$+2$ (i.e. it raises the weight of a quasimodular form by~2).
Another important operator is the derivation $\fd$ of degree~$-2$ defined in terms of the isomorphisms~\eqref{structure}
as $12\,\p/\p P$.  (There is also an intrinsic definition.)  Together 
with~$D$ and the { weight operator
$\bfW$} sending $f$ to $kf$ for $f\in\wM_k$ they span a Lie algebra of derivations 
of~$\wM_*$ isomorphic to $\mathfrak s\mathfrak l_2$, namely,  
\be \label{Lie} [\bfW,\,D]\,=\,2D\,, \qquad [\bfW,\,\fd]\,=\,-2\fd\,, 
\qquad [\fd,\,D]\,=\,\bfW\,, \ee
where the first two equations simply say that $D$ and~$\fd$ have degree~2 and~$-2$.
\par
A collection of examples of quasimodular forms that will be important for us is given by the Taylor expansion
\be \label{eq:thetainH}  \th(z) \=  \th'(0)\,\sum_{n=0}^\infty H_n(\t)\,z^{n+1} \ee
of the Jacobi theta series~\eqref{eq:defTh}, in which $\th'(0)=\eta(\t)^3$, where $\eta$ is the Dedekind eta function defined by 
$\eta(\t)=q^{1/24}\prod(1-q^n)$ or by $1728\eta(\t)^{24}=Q^3-R^2$. We have $H_n\in\wM_n$ for all~$n$, the 
first few values being $ H_0=1$, $H_2=P/24$, $H_4=(5P^2-2Q)/5760$ (and of course $H_n=0$ for~$n$ odd), and the later values
being computable recursively by the formula 
\be\label{Hrecur} 4n(n+1)H_n \= 8D(H_{n-2})\+PH_{n-2}\,. \ee

The expansion~\eqref{eq:thetainH} is at the base of the proof of Theorem~\ref{thm:bo} given in \cite{zagBO}.  More precisely, it is 
shown there by a very simple combinatorial argument that 
 \be\label{identity} \sbq{\th(\pp)g}\= 0 \qquad\text{for all}\quad g\in\QQ[Q_2,Q_3,\dots]\subset\RRR\,, \ee 
where $\pp:\RRR\to\RRR$ is the derivation sending $Q_k$ to $Q_{k-1}$; then the quasimodularity of the Taylor coefficients 
of $\theta$ together with an easy induction on the weight suffices to prove the quasimodularity of $\sbq f$ for all~$f\in\RRR$.  The 
identity~\eqref{identity} is also used to prove the inductive formula (equivalent to the explicit formulas in Bloch-Okounkov)
 \be  \label{eq:BORec}
\sum_{J \subseteq N} (-1)^{n-|J|}\,\th^{(n-|J|)}\bigl(z_J\bigr)\,
   F_{|J|}\bigl(\fZ_J \bigr) \= 0 \qquad  (n\ge1)  \ee
for the correlators, where $N=\{1,\dots,n\}$ and for any subset~$J\sse N$ (including the empty set) we denote by $\fZ_J$ and~$z_J$ 
the set of $z_j$ with~$j\in J$ and the sum of these elements, respectively. Yet a third equivalent version given in~\cite{zagBO} 
is the following axiomatic characterization of the correlators, in which the symbol $[G]^+$ in the final axiom denotes 
the strictly-positive-exponent part of a Laurent series~$G$ in several variables: 
\begin{Thm} \label{thm:BOviaAxioms} The Bloch-Okounkov correlators  $F_n(z_1,\dots,z_n)$ $(n\ge0)$ 
are the unique Laurent series satisfying:
\begin{itemize}
\item[(i)] $F_0(\;)\,=\,1\,$. 
\item[(ii)] $F_n(z_1,\dots,z_n)$ is symmetric in all~$n$ arguments.
\item[(iii)] $F_n(z_1,\dots,z_n)\,=\,\dfrac1{z_n}\,F_{n-1}(z_1,\dots,z_{n-1})\+\text{\rm O}(z_n)\;$ as $\,z_n\to0\,$.
\item[(iv)] $\bigr[\th(z_1+\cdots+z_n)\,F_n(z_1,\dots,z_n)\bigr]^+=\,0\,$ for all~$\,n\ge0\,$.
\end{itemize}
\end{Thm}
\par
An important aspect of the Bloch-Okounkov map $\,\sbq{\;\,}: \RRR_* \to \wM_*$, which
will be used several times in the sequel, is its relation to the 
$\mathfrak s\mathfrak l_2$-action on $\wM_*$ as defined above.
For two of the generators of $\mathfrak s\mathfrak l_2$ this is easy: since 
$\,\sbq{\;\,}: \RRR_* \to \wM_*\,$ preserves the grading we 
have $\bfW\sbq{f} = \sbq{Ef}$, where $E = \sum Q_k \partial/\partial Q_k$ is the Euler
operator, and we also have the formula
\be\label{Daction}  D\,\sbq f \= \sbq{Q_2\,f} \+ \frac P{24}\,\sbq f \qquad (f \in \RRR)
\ee
as an immediate consequence of the definition~\eqref{defqbrac} and the formulas $Q_2(\l)=|\l|-\frac1{24}$ and $D(\eta)=P\eta/24$. The action of the third generator~$\fd$
of~$\mathfrak s\mathfrak l_2$, which is much harder to compute, was found in~\cite{zagBO}
and is given as follows.
\par
\begin{Prop} \label{prop:fd} 
The action of the derivation $\fd:\wM_*\to\wM_{*-2}$ as defined above on $q$-brackets 
is given by
\be\label{E2deriv}  \fd\,\sbq f \= \bq{\h\bigl(\Delta\,-\,\pp^2\bigr)\,f} \qquad (f \in \RRR)\,,\ee
where $\Delta: \RRR\to \RRR$ is the second order differential 
operator of degree~$-2$ defined by 
\be\label{defDelta} \Delta \= \sum_{k,\,\ell\,\ge\,0} \binom{k+\ell}k \,Q_{k+l}\,\frac{\p^2}{\p Q_{k+1}\,\p Q_{\ell+1}}  \ee
and $\partial$ the derivation defined in~\eqref{identity}.
Moreover, the actions of $\Delta$ and $\partial$ commute.
\end{Prop}
\par
\begin{proof} The first statement is Theorem~3 of \cite{zagBO}
and the second is an easy consequence of the definitions  of~$\partial$ and~$\Delta$ using
$\binom{k+\ell}k = \binom{k+\ell-1}k + \binom{k+\ell-1}{\ell}$.
\end{proof}
\par
A corollary of this proposition (\cite[Theorem~2]{zagBO}) is that, 
if we define the ``top coefficient" $\bold T(F)$ of a quasimodular 
form~$F$ of weight~$2n$  as the coefficient of~$P^n$ in the expression of~$F$ as 
a polynomial in $P$, $Q$, and~$R$, then 
\be\label{Top}  \bold T(\sbq f) \= - \frac{(2n-3)!!}{(-12)^n}\,\mu(f) \qquad\text{for all $f\in \RRR_{2n}$,}  \ee
where $(2n-3)!!:=1\times3\times\cdots\times(2n-3)$ (resp.~$(-1)!!=1$, $(-3)!!=-1$) and $\mu:\RRR\to\QQ$ is the ring homomorphism sending 
$Q_n$ to $(1-n)/n!$ for every~$n\ge0$. 
\par
\smallskip
To illustrate the statement of the Bloch-Okounkov theorem we end this section by giving a short list of the $q$-brackets of all
monomials in the~$Q$'s of even weight~$\le6$: 
\bas &\sbq{Q_2}\,=\,\frac{-P}{24}\,, \qquad \sbq{Q_2^2}\,=\,\frac{-P^2+2Q}{576}\,, \qquad \sbq{Q_4}\,=\,\frac{5P^2+2Q}{5760}\,,  \\ 
 &\sbq{Q_2^3}\,=\,\frac{-3P^3+18QP-16R}{13824}\,, \qquad \sbq{Q_2Q_4}\,=\,\frac{15P^3-6QP-16R}{138240}\,, \\
 &\sbq{Q_3^2}\,=\,\frac{5P^3-3QP-2R}{25920}\,, \qquad \sbq{Q_6}\,=\,\frac{-35P^3-42QP-16R}{2903040}\,. \eas
A slightly longer list, up to weight~8, can be found in~\cite{zagBO}.

\section{The growth polynomials of quasimodular forms} \label{sec:growthpoly}

In this section we introduce a polynomial (actually two polynomials, related to each other by a
simple transformation) that describes the growth of a quasimodular form $F(\t)$ near $\t=0$ and at the same
time the average growth of its Fourier coefficients.  In the following section this polynomial will
be computed for the image of the Bloch-Okounkov map.  The latter calculation is equivalent to a
result of Eskin and Okounkov~(\cite{eo}), of which we will then be able to give simpler alternative
proofs, and the idea of considering the asymptotic growth of quasimodular forms near the origin is
already contained in their work, but not explicitly worked out in this generality.  Since the construction
is very natural and will undoubtedly be useful also in other situations involving quasimodular forms,
we present it here in fair detail, including some further properties.
The map assigning to a quasimodular form its growth polynomial is a ring homomorphism that can be thought
of as a kind of polynomial evaluation map, and we will denote the two versions of this map by the symbols~$\Ev$ and~$\ev$.

For a quasimodular form $F\in\wM$ we write $F(\infty)$ ($:=\lim_{\t \to i \infty} F(\t) = a_0(F)$, 
where $F(\t)=\sum_{n=0}^\infty a_n(F)\,q^n$ is the Fourier expansion of~$F$) for the constant term.
 We write $E_k\in\wM_k$ for the normalized Eisenstein series 
$G_k/G_k(\infty)$ for $k\in 2\NN$ (so $E_2=P$, $E_4=Q$, and $E_6=R$ are the generators of
the algebra~$\wM_*$), and set $E_0=1$ and $E_k=0$ for~$k$ odd.  As before we write $Df$
or $f'$ for the derivative $\frac1{2\pi i}\frac{df}{d\t}$ of $f\in\wM_*$ and use the
notations $f^{(r)}$ and $D^r(f)$ interchangeably.
The space $\wM_*$ of quasimodular forms with coefficients in~$\QQ$ is the direct sum of the
subspace $\DE$ spanned by all derivatives of all Eisenstein series~$E_k$ and the subspace $\DS$ spanned 
by all derivatives of all cusp forms.
We can therefore define a linear map $\evX:\wM_*\to\QQ[X]$ by setting
$$ \text{$\evX[F]\,=\,0$ for $F\in\DS$,} \qquad 
 \evX\bigl[E_{2\ell}^{(r)}\bigr](X) \,=\, \begin{cases} \qquad\quad\delta_{r,0} & \text{if $\ell=0$,} \\
 (r+1)!\,X \+ 12\,r! & \text{if $\ell=1$,} \\
  \quad \frac{(r+2\ell-1)!}{(2\ell-1)!}\,X^\ell & \text{if $\ell\ge2$.} \end{cases}$$

Presented like this, the definition looks somewhat unnatural, but in fact the map~$\evX$ has very 
nice properties, as given in the next five propositions.
\begin{Prop} \label{prop:ev1st}
The map $\evX$ is the algebra homomorphism from $\wM_*$ 
to~$\QQ[X]$ sending $E_2$ to~$X+12$, $E_4$ to~$X^2$, and $E_6$ to~$X^3$. 
\end{Prop}
\par
\begin{proof}
 It is clear by induction on~$r$ that $\evX$ is characterized axiomatically by the three properties
\begin{itemize}
\item[(i)] $\evX[f](X)\,=\,a_0(f)X^k$ for $f\in M_{2k}\,$;
\item[(ii)] $\evX[E_2](X)\,=\,X\+12\,;$
\item[(iii)] $\evX[DF] \,=\, \bigl(X\frac d{dX}+k\bigr)\evX[F]$ for $F\in \wM_{2k}\,$.
\end{itemize}  
It therefore suffices to show that the algebra homomorphism $\Phi:\wM_*\to\QQ[X]$ defined
by $E_2\mapsto X+12$, $E_4\mapsto X^2$, $E_6\mapsto X^3$ has the same three properties.  The 
first one is obvious since it holds for the generators $E_4$ and $E_6$ of the ring~$M_*$
and since $f\mapsto a_0(f)X^{\text{wt}(f)/2}$ is a ring homomorphism, and the second is 
true by definition. For the third, we have to check the commutativity of the diagram
$$\xymatrix{\wM_* \ar[r]^{\evX\hphantom{xx}} \ar[d]_{D-H} & \QQ[X] \ar[d]^{X\frac d{dX}}\\
\wM_*  \ar[r]^{\evX\hphantom{xx}}   &\QQ[X]  \\}  $$
where $H:\wM_*\to\wM_*$ is the operator sending $F\in\wM_{2k}$ to~$kF$. 
This commutativity follows for the generators $P,\,Q,\,R$
from Ramanujan's formulas~\eqref{RamDeriv}, since
\bas 
& P \,\mapt{D-H}\,\frac{P^2-Q}{12}-P\,\,\, \mapt\Phi \frac{(X+12)^2-X^2}{12}\m(X+12)\,=\,X \,=X\frac d{dX}\Phi(P)\,, \\
& Q \,\mapt{D-H}\,\frac{PQ-R}3-2Q \mapt\Phi \frac{(X+12)X^2-X^3}3\m2X^2\,=\,2X^2 \,=X\frac d{dX}\Phi(Q)\,, \\
& R \,\mapt{D-H}\,\frac{PR-Q^2}2-3R \mapt\Phi \frac{(X+12)X^3-X^4}2\m3X^3\,=\,3X^3 \,=X\frac d{dX}\Phi(R)\,,
\eas
and then holds in general because the
horizontal maps in the diagram are ring homomorphisms and the vertical maps 
are derivations.
\end{proof}
\par
The next proposition expresses the map $\evX:\wM_*\to\Q[X]$ explicitly in terms of the action of the Lie algebra 
$\mathfrak s\mathfrak l_2=\la \bfW,D,\fd\ra$ on~$\wM_*$ described
{ in~\eqref{Lie}}  and the 
constant term map $a_0:F\mapsto F(\infty)$ from $\wM_*$ to~$\Q$.
\begin{Prop} \label{prop:evLie}
 For any $F\in\wM_*$ we have $\evX[F](X)= a_0\bigl(X^{\bfW/2}e^\fd F\bigr)$.
\end{Prop}
\begin{proof} This is a corollary of Proposition~\ref{prop:ev1st} since the maps $e^\fd$, 
$X^{\bfW/2}$, and $a_0$ are all algebra homomorphisms (because $\fd$ and~$\bfW$ 
are derivations) and since the identity in question holds by inspection for the 
generators $P$, $Q$, and~$R$ of~$\wM_*$.
Note that $X^{\bfW/2}e^\fd$ can also be written as $e^{\fd/X} X^{\bfW/2}$.
\end{proof} 
\par
The third proposition relates $\evX[F]$ directly to the behavior of~$F(\t)$ as $\t\to0$.
\par
\begin{Prop} \label{prop:evAsy}
For $F\in\wM_{2k}$ the polynomial $\evX[F](X)$ describes the asymptotic behavior of $F(\t)$ near the
cusp~$\t=0$.  More precisely, we have 
\be\label{evAsy} F\bigl(i\ve) \= \frac1{(2\pi\ve)^k}\;\evX[F]\Bigl(-\frac{2\pi}\ve\Bigr) \;+\;\text{\rm (small)} \qquad\bigl(\ve\searrow0\bigr), \ee
where {\rm``small"} means terms that tend exponentially quickly to~$0$ as~$\ve$ tends~to~$0$. 
\end{Prop}
\par
\begin{proof} 
Suppose first that $F=f^{(r)}$ with $f\in M_{2\ell}$. The modular transformation property $f(-1/\t)=\t^{2\ell}f(\t)$ gives
$$  f(i\ve) \= \frac{(-1)^\ell}{\ve^{2\ell}}\,\sum_{n=0}^\infty a_n(f)\,e^{-2\pi n/\ve} $$
and hence
\begin{flalign}
 F(i\ve) & \= \Bigl(-\frac1{2\pi}\,\frac d{d\ve}\Bigr)^r\,f(i\ve) 
   \= \frac{(-1)^\ell}{(2\pi)^r}\,\frac{(r+2\ell-1)!}{(2\ell-1)!}\,\frac{a_0(f)}{\ve^{2\ell+r}} \;+\;\text O\bigl(\ve^{-2\ell-r}e^{-2\pi/\ve}\bigr) &\nonumber \\
  & \= \frac{(r+2\ell-1)!}{(2\ell-1)!}\,a_0(f)\,\frac{(-2\pi/\ve)^\ell}{(2\pi\ve)^{r+\ell}} \;+\;\text{\rm (small)}\,, \nonumber 
\end{flalign}
confirming the statement in this case.  If $F=E_2^{(r)}$, then the modular transformation property 
$E_2(-1/\t)=\t^2E_2(\t)+6\t/\pi i$ gives
$$  E_2(i\ve) \= -\,\frac1{\ve^2}\+\frac6{\pi\ve} \+ \frac{24}{\ve^2}\,\sum_{n=1}^\infty \sigma_1(n)\,e^{-2\pi n/\ve} $$
and hence
$$ F(i\ve) \= \Bigl(-\frac1{2\pi}\,\frac d{d\ve}\Bigr)^rE_2(i\ve) \= (r+1)!\,\frac{-2\pi/\ve}{(2\pi\ve)^{r+1}}
   \+ \frac{12\,r!}{(2\pi\ve)^{r+1}}\;+\;\text{\rm (small)}\,,$$
again in accordance with the statement of the proposition.  Since $\wM_*$ is spanned by the derivatives of modular
forms and of~$E_2$, this completes the proof.
\end{proof}
Note that Proposition~\ref{prop:evAsy} also gives an alternative proof of Proposition~\ref{prop:ev1st}, since sending a 
function to its asymptotic development near~0 is obviously a ring homomorphism.  We nevertheless preferred to give an independent 
and purely algebraic proof to emphasize the axiomatic description of the map~$\evX$ and in particular its relation to differentiation.

As already mentioned, for the applications to the quasimodular forms coming from the Bloch-Okounkov theorem, and
for the results of Eskin and Okounkov that we will reprove and generalize in the next section, it is convenient to
work with a different normalization of the growth polynomials that we now introduce. 
Replace the variable $\ve$ in the proposition by $\hbar=h/2\pi$ (where the letters~$h$ and~$\hbar$ are meant to suggest 
Planck's constant and quantum mechanics).  Then if we define 
  \be\label{defevh} \evh[F](h) \= \frac1{h^k}\,\evX[F]\Bigl(-\frac{4\pi^2}h\Bigr) \;\in\;\QQ[\pi^2][1/h] \ee
{{for $F\in  \wM_{2k}$}}, the statement of Proposition~\ref{prop:evAsy} is that $F(\t)$ equals $\evh[F](h)$ plus exponentially small
terms as $q=e^{2\pi i\t}=e^{-h}$ tends to~1.  Although the two polynomials $\evX[F]$ and $\evh[F]$ are equivalent, 
it is useful to retain both versions because each has convenient properties: the former because it has rational 
coefficients and no extraneous powers of the variable, and the latter because it describes the growth of~$F(\t)$ 
near~$\t=0$ directly.  We will refer to both $\evX[F]$ and $\evh[F]$ as the {\it growth polynomials} of the quasimodular
form~$F$. This terminology is justified not only by Proposition~\ref{prop:evAsy}, relating these polynomials to the growth
of~$F(\t)$ near~$\t=0$, but also by the following result, which says that their leading terms
determine the average asymptotic growth of the Fourier coefficients of~$F$.

\begin{Prop} \label{prop:evgrowth}
Let $F$ be a homogeneous element of~$\wM_*$ satisfying $\evh[F]=Ah^{-p}+\text{\rm O}(h^{1-p})$ as $h\to0$ for some integer~$p\ge0$
and constant $A\ne0$.  Then the sum of the first~$N$ Fourier coefficients of~$F$ has the asymptotic behavior
\be\label{coeffsum} \sum_{n=1}^N a_n(F) \= A\,\frac{N^p}{p!}\;+\;\text{\rm O}\bigl(N^{p-1}\log N\bigr)\qquad (N\to\infty)\,.\ee
\end{Prop}
\begin{proof}  Let the weight of~$F$ be~$2k$. If $k=0$ then there is nothing to prove, since $a_n(F)=0$ for all~$n\ge1$. If~$k\ge1$, then we can
write $F$ as a linear combination of derivatives $D^{k-\ell}G_{2\ell}$ and $D^{k-\ell}f_\ell$ with $1\le \ell\le k$, where $G_{2\ell}$ is the
Eisenstein series and $f_\ell$ a cusp form of weight~$2\ell$.  Since we are assuming that $\evh[F]$ is not identically zero,
there is at least one Eisenstein contribution, and since the degree of $\evh(D^{k-\ell}G_{2\ell})$ in~$h^{-1}$ is $k+\ell$, it follows
that $k+1\le p\le2k$.  The terms $D^{k-\ell}f_\ell$ do not affect the estimate~\eqref{coeffsum}, since by a result of Hafner and Ivi\'c~(\cite{HI})
we have $\sum_{n\le N} a_n(f_\ell)=\text O(N^{\ell-\frac16})$ (the weaker estimate $\,\text O(N^\ell\log N)$  would be enough for our purposes) 
and by partial summation we deduce that $\sum_{n\le N} a_n(D^{k-\ell}f_\ell)=\sum_{n\le N} n^{k-\ell}a_n(f_\ell)=\text O(N^{k-\frac16})$, which 
gets absorbed into the error term in~\eqref{coeffsum} since~$k\le p-1$. It therefore suffices to consider the case $F=G^{(r)}_{2\ell}$
with $\ell\ge1$, $r\ge0$, $k=\ell+r$.  For this form we have from the original definition of the growth polynomial the formula
 \bas   \evX\bigl[G^{(r)}_{2\ell}\bigr] &\= -\,\frac{B_{2\ell}}{4\ell}\,\frac{(r+2\ell-1)!}{(2\ell-1)!}\,X^\ell \m \frac{r!}2\,\delta_{\ell,1} \\
  &\= (r+2\ell-1)!\,\frac{\zeta(2\ell)}{(2\pi i)^{2\ell}}\,X^\ell \m \frac{r!}2\,\delta_{\ell,1}\,, \eas
which we can rewrite in terms of~$\evh$ as
 \be\label{evhDrG} \evh\bigl[G^{(r)}_{2\ell}\bigr] \= (r+2\ell-1)!\,\frac{\z(2\ell)}{h^{r+2\ell}} \m r!\,\frac{\delta_{\ell,1}}{2h^{r+1}}\,, \ee
and on the other hand
 \bas \sum_{n=1}^Na_n\bigl(G^{(r)}_{2\ell}\bigr) &\= \sum_{n=1}^N n^r\,\sigma_{2\ell-1}(n) 
  \= \sum_{{a,\,b\ge1\atop ab\le N}} a^{r+2\ell-1}b^r  \\
  &\= \sum_{b=1}^N b^r\,\biggl(\frac{(N/b)^{r+2\ell}}{r+2\ell}\+\text O\bigl((N/b)^{r+2\ell-1})\bigr)\biggr) \\
  & \= \z(2\ell)\,\frac{N^{r+2\ell}}{r+2\ell} \+ \text O\bigl(N^{r+2\ell-1}\log N\bigr)\,. \eas
(Here the ``$\log N$" factor is needed only for~$\ell=1$.)  This confirms~\eqref{coeffsum} in this case and hence also in general. 
\end{proof}

Our final statement about the growth polynomials associated to a quasimodular form~$F$ is that the number of monomials
they contain equals the number of poles of the meromorphic continuation of the $L$-series
$L(s,F)=\sum_{n=1}^\infty a_n(F)n^{-s}$, with the corresponding exponents and coefficients corresponding to the
positions and residues of these poles. 
\begin{Prop} \label{prop:Lseries}  Let $F$ be a quasimodular form of weight~$2k$.  Then the $L$-series of~$F$ has a meromorphic
continuation to the whole complex plane, with at most simple poles at $s=k,\dots,2k$ as its only singularities, and the growth 
polynomial $\evh[F]$ of~$F$ is given in terms of the residues of~$L(s,F)$ by the formula
\be \label{evhLFs} \evh[F](h) \= \sum_{m=k}^{2k}  (m-1)!\,{\rm Res}_{s=m}\bigl[L(s,F)\bigr]\,h^{-m}\,. \ee
\end{Prop}
\begin{proof} Again we verify this by looking at the cases of derivatives of cusp forms and of Eisenstein series separately.
For the first case the assertion is trivial, since if $F=f^{(r)}$ for some cusp form~$f$ then $L(s,F)=L(s-r,f)$ extends to
an entire function of~$s$ and the polynomial~$\evX(F)$ vanishes identically.  If $F=G^{(r)}_{2\ell}$, then we have
$L(F,s)= L(s-r,G_{2\ell})= \zeta(s-r)\,\zeta(s-r-2\ell+1)$, which extends to a meromorphic function having only simple
poles, one at $s=r+2\ell$ with residue $\zeta(2\ell)$ and a second one at $s=r+1$ with residue $-\frac12$ if~$\ell=1$,
so that equation~\eqref{evhLFs} agrees with equation~\eqref{evhDrG}.   \end{proof} 

It is perhaps worth noting that an alternative proof of Proposition~\ref{prop:Lseries} could be given using Proposition~\ref{prop:evAsy},
since if $F(it) = \sum_{m=1}^Mc_mt^{-m}+\text O(t^N)$ for $t$ small, where $M$ is fixed and $N$ can be chosen arbitarily large, then we have
(initially for $s$ with sufficiently large real part)
\bas & (2\pi)^{-s}\G(s)\,L(s,F) \= \int_0^\infty \bigl(F(it)-a_0(F))\,t^{s-1}dt \\
&\qquad \= \int_0^{t_0} \Biggl(\sum_{m=1}^Mc_mt^{-m}-a_0(F)+\text O(t^N)\Biggr)\,t^{s-1}\,dt \+\int_{t_0}^\infty \text O(e^{-2\pi t})\,t^{s-1}dt \\
&\qquad \= \sum_{m=1}^M\frac{c_m}{s-m} \m \frac{a_0(F)}s \+ \bigl(\text{holomorphic for $\Re(s)>-N$}\bigr)\,, \eas
giving a meromorphic continuation of~$L(F,s)$ to the whole complex plane with simple poles of residue 
$(2\pi)^mc_m/(m-1)!$ at integers~$m\ge1$ and no other poles.  Proposition~\ref{prop:Lseries} also explains
where Proposition~\ref{prop:evgrowth} comes from, using the standard expression for $\sum_{n=1}^Na_n(F)$ as
$\frac1{2\pi i}\,\int_{C-i\infty}^{C+i\infty} L(F,s)\,N^s\,\frac{ds}s$ for $C$ sufficiently large and then shifting
the path of integration to the left to pick up a residue from the rightmost pole of~$L(F,s)$ and using the functional
equation of the $L$-series and the Phragm\'en-Lindel\"of theorem to estimate the integrand on the shifted contour.  We omit the details.

We end this section with a simple illustrative example.

\begin{Prop} \label{prop:EvH}  The growth polynomial of the quasimodular form $H_{2k}\in\wM_{2k}$
defined by~\eqref{eq:thetainH} is given by
\be\label{evXHn}  \evX[H_{2k}](X) \= \sum_{m,\,n\ge0\atop m+n=k}\frac{1}{2^mm!}\,\frac{(X/4)^n}{(2n+1)!}\,. \ee
\end{Prop}
\begin{proof} We give two proofs of equation~\eqref{evXHn}, to illustrate the use of the different properties of growth polynomials. 
Write $h_{2k}$ for $\evX[H_{2k}]$.  Then the recursion~\eqref{Hrecur} and the differentiation 
property~(iii) in the proof of Proposition~\ref{prop:ev1st} give 
 $$  k\,(2k+1)\,h_{2k}(X) \=  \Bigl(X\,\frac d{dX} \+ k-1\+ \frac{X+12}8\Bigr)\,h_{2k-2}(X)\,,$$
and~\eqref{evXHn} follows easily by induction on~$k$ starting with the value $h_0(X)=1$.  Alternatively, from
equation~\eqref{eq:thetainH} and Proposition~\ref{prop:evAsy} we have
$$  \frac{\th_{i\ve}(z)}{z\,\th'_{i\ve}(0)}  \= \sum_{k=0}^\infty H_{2k}(i\ve)\,z^{2k} 
\=\sum_{k=0}^\infty h_{2k}\Bigl(-\frac{2\pi}\ve\Bigr)\,\frac{z^{2k}}{(2\pi\ve)^k} \;+\;\text{(small)}\,, $$
where ``small" denotes terms decreasing faster than any power of~$\ve$, and since from the
modular transformation property of $\th$ we have 
$$  \frac{\th_{i\ve}(z)}{z\,\th'_{i\ve}(0)}  \= e^{z^2/4\pi\ve}\,\frac{\th_{i/\ve}(iz/\ve)}{iz/\ve\cdot\th'_{i/\ve}(0)}
 \=  e^{z^2/4\pi\ve}\,\frac{\sin(z/2\ve)}{z/2\ve}\;+\;\text{(small)}\,, $$
we obtain a second proof of~\eqref{evXHn} by comparing the coefficients of $z^{2k}$. 
\end{proof}  
\par
The second proof above gives the generating series for the~$h_{2k}$ explicitly: 
\begin{flalign}  \label{eq:evXtheta}
 & \; \sum_{k=0}^\infty h_{2k}(X)\,z^{2k+1}  \= \evX\Bigl[\frac{\th(z)}{\th'(0)}\Bigr](X) \= e^{z^2/2}\, \frac{\sinh(z\sqrt{X}/2)}{\sqrt{X}/2}\\
& \=  z  \+  \Bigl(\frac X4+ 3\Bigr)\frac{z^3}{3!} \+ \Bigl(\frac{X^2}{16} + \frac{5X}2 + 15\Bigr)\frac{z^5}{5!} \nonumber
\+ \Bigl(\frac{X^3}{64} + \frac{21X^2}{16} + \frac{105X}4\Bigr)\frac{z^7}{7!} \+ \cdots\,. \nonumber 
\end{flalign}
\bigskip

\section{The growth polynomials of $q$-brackets} 
\label{sec:growthBO}

In this section we will consider the growth polynomials of $q$-brackets, for which we 
use the notations $\sbrX f:=\evX[\sbq f](X)$ and $\sbrh f:=\evh\bigl[\sbq f\bigr](h)$ ($f\in\RRR$)
and the terminology {\it $X$-brackets} and {\it $h$-brackets}, respectively.  It turns out that,
whereas there is no really practical ``closed formula" for the $q$-brackets of arbitrary elements 
of~$\RRR$, there {\it is} such a formula for their growth polynomials.  In fact, there are two, each 
in terms of a suitable generating function.  One of them, which is due to Eskin and Okounkov~(\cite{eo}) but 
of which we will give a simpler proof and also a slight refinement, 
gives the growth polynomial $F(z_1,\dots,z_n)_X:=\evX[F(z_1,\dots,z_n)](X)$
of the correlator function~\eqref{eq:defcorrelator} for each integer~$n\ge1$.  The other, which we will 
state as Theorem~\ref{thm:hbracketbfU} and which is the principal result of this section, gives all of the
$X$-brackets simultaneously as a single generating function in infinitely many variables (``partition
function") that we express as a one-dimensional formal Gaussian integral.

To motivate these formulas, we first look at small values of~$n$.  For $n=1$ we
find from the first of equations~\eqref{F1F2F3} together with equation~\eqref{eq:evXtheta} the result
$$ F_1(z)_X \= \frac{x\,e^{-z^2/2}}{\sinh xz}\qquad(x\,:=\,\sqrt X/2)\,, $$
and similarly for $n=2$ the second of equations~\eqref{F1F2F3} together with~\eqref{eq:evXtheta} and the
addition law for the hyperbolic sine function give
\bas F_2(z_1,z_2)_X &\= \frac{x\,e^{-(z_1+z_2)^2/2}}{\sinh x(z_1+z_2)} \,\biggl(z_1 + \frac x{\tanh xz_1}+z_2 + \frac x{\tanh xz_2}\biggr) \\
  &\= e^{-z_{12}^2/2}\,\biggl(\frac{xz_{12}}{\sinh xz_{12}} \+ \frac x{\sinh xz_1}\,\frac x{\sinh xz_2}\biggr)\qquad(z_{12}:=z_1+z_2)\,, \eas
while for~$n=3$ a similar calculation using the third of equations~\eqref{F1F2F3} gives
\bas    F_3(z_1,z_2,z_3)_X &\=  \frac{x\,e^{-z_{123}^2/2}}{\sinh xz_{123}}\, 
   \text{Sym}_3\biggl[\Bigl(z_1 + \frac x{\tanh xz_1}\Bigr)\Bigl(z_{12}+\frac x{\tanh xz_{12}}\Bigr)  \\ 
   & \hphantom{ \frac{xe^{-z_{123}^2/2}}{\sinh xz_{123}}\,\text{Sym}_3} \qquad
      -\,\Bigl(\frac{1+x^2+z_1^2}2\+\frac{xz_1}{\tanh xz_1}\Bigr) \+ \frac{4x^2+12}{24}\,\biggr]  \\
  &\= e^{-z_{123}^2/2}\,\biggl(\frac{xz_{123}^2}{\sinh xz_{123}} 
    \+ \frac x{\sinh xz_1}\,\frac{xz_{23}}{\sinh xz_{23}} \+ \frac x{\sinh xz_2}\,\frac{xz_{13}}{\sinh xz_{13}} \\ 
  &\qquad\qquad\qquad \+ \frac x{\sinh xz_3}\,\frac{xz_{12}}{\sinh xz_{12}}
      \+ \frac x{\sinh xz_1}\,\frac{x}{\sinh xz_2}\,\frac x{\sinh xz_3}\biggr)   \eas
with $z_{123}:=z_1+z_2+z_3$ etc. These special cases suggest the following result 
in which, as in Section~\ref{sec:genser}, $\PPP(n)$ denotes the 
set of unordered partitions of the set $\{1,\ldots,n\}$.
\par
\begin{Thm}[\cite{eo}, Theorem~4.7] \label{thm:EO} The $X$-evaluation of 
the $n$-point Bloch-Okoun\-kov correlator is given by
\be \label{evXofEO}
 F_n\left(z_1,\ldots,z_n\right)_X \= e^{-z_N^2/2} \sum_{\alpha \in \PPP(n)} \,
\prod_{A\in\alpha} \,\frac{z_A^{|A|-1}\sqrt X/2}{\sinh(z_A\sqrt X/2)}\,, 
\ee
where $N=\{1,\ldots,n\}$ and $z_A = \sum_{a \in A} z_a$ for $A\subseteq N$.
\end{Thm}
\par
Theorem~\ref{thm:EO}, which we will prove below,  gives a formula for the growth polynomial 
of the Bloch-Okounkov correlator functions 
$F_n(z_1,\ldots,z_n)$ defined in~\eqref{eq:defcorrelator}, and thus for the $q$-bracket of a product of~$n$ 
generators~$Q_k$ of~$\RRR$ for a fixed value of~$n$.  It turns out that these formulas can be expressed more 
simply, and in a way that is better suited for our applications, if we organize them into a different kind
of generating function, namely the {\it partition function}
\begin{flalign} \Phi(\bfu)_q & \= \Bigl\langle \exp\Bigl(\sum_{\ell \geq 1} p_\ell u_\ell\Bigr)\,\Bigr\rangle_q
  \=\sum_{\bfn\geq 0}\,\langle\underbrace{p_1,\ldots,p_1}_{n_1},\underbrace{p_2,\ldots,p_2}_{n_2},\ldots\rangle_q \,\frac{\bfu^\bfn}{\bfn!} 
   \nonumber  \\   \label{eq:defPhi}   
  &\=\sum_{n=0}^\infty \frac1{n!} \sum_{\ell_1,\ldots,\ell_n \geq 1} \langle p_{\ell_1} \cdots p_{\ell_n} \rangle_q \, u_{\ell_1}\cdots u_{\ell_n}\;. 
\end{flalign} 
Here $p_\ell=\ell!\,Q_{\ell+1}$ as in~\eqref{eq:smallpk} and we have used standard multi-variable notation:
${\bf u} = (u_1,u_2,\cdots)$ denotes a tuple of countably  many independent variables~$u_i$
and $\bfn \geq 0$ denotes a multi-index ${\bf n} = (n_1,n_2,\cdots)$ of non-negative integers~$n_i$,
with ${\bf u}^{\bf n} = \prod_{m\ge0}u_m^{n_m}$ and ${\bf n}! = \prod_{m\ge0}n_m!\,$. 
The definition~\eqref{eq:defPhi} can be compared to Witten's generating
function for intersection numbers of $\psi$-classes on the moduli spaces $\barmoduli[g,n]$: 
$$ \Phi_{\rm Witten}(u_0,u_1,\ldots)   \= \sum_{n=0}^\infty \frac1{n!} \sum_{m_1,\ldots,m_n \geq 0} 
\langle \tau_{m_1} \cdots \tau_{m_n} \rangle \, u_{m_1}\cdots u_{m_n} $$
with 
$$  \langle \tau_{m_1} \cdots \tau_{m_n} \rangle \= \int_{\barmoduli[{g,n}]} \psi_1^{m_1}\cdots \psi_{n}^{m_n}\,, $$
in which the formal variables $u_i$ are also attached to the number of occurences of~$\psi_i$ in 
the product rather than to the index of the marked point.

Our main result below gives an explicit formula for the growth polynomial generating function
$\Phi(\bfu)_X:=\evX[\Phi(\bfu)_q](X)$.  To state it, we introduce an auxiliary generating function defined by
\be \label{defBuyX} \cBB(\bfu,y,X) \= \sum_{\genfrac{}{}{0pt}{2}{ \bfa > 0 }{r \geq 0}}
(a_1+2a_2+3a_3+\cdots)!\,\, \beta_{2-r+w(\bfa)} \sqrt{X}^{2-r + w(\bfa)}\, \frac{\bfu^\bfa}{\bfa!}\frac{y^r}{r!}\,,\ee
with $\beta_m$ as in~\eqref{eq:defbeta} and $w(\bfa) = a_2 + 2a_3 + 3a_4 + \cdots$.  (Alternative and
simpler expressions for~$\cBB$ 
are given in equations~\eqref{newdefBuyX} and~\eqref{Bviat} below.)  Note that the exponents of~$X$ are all non-negative
and integral, since $\b_k=0$ for $k<0$ or $k$~odd, and also that the coefficient of each monomial in $X$ and~$y$ contains
only finitely many monomials in the~$u_i$, so that $\cBB(\bfu,y,X)$ belongs to~$\QQ[\bfu][[y,X]]$.
\par
\begin{Thm} \label{thm:hbracketbfU}
The growth coefficient polynomial of the generating function $\Phi(\bfu)_q$ can be expressed as the formal Gaussian integral
\be \label{eq:PhiGauss} \Phi(\bfu)_X  \= \frac1{\sqrt{2\pi}}\int_{-\infty}^\infty e^{-y^2/2 \+ \cBB(\bfu,iy,X)}\,dy\;. \ee
\end{Thm}

\noindent Note that the expression on the right hand side of~\eqref{eq:PhiGauss} is purely algebraic and
does not really involve integration, since we can state the theorem equivalently as
\be \label{eq:PhiGauss2}   \Phi(\bfu)_X \= \frakI\bigl[e^{\cBB(\bfu,y,X)}\bigr] \quad\in\;\,\Q[X][[\bfu]]\,, \ee
where $\frakI$ is the functional on power series in~$y$ defined on monomials by
\be \label{deffrakI} \frakI\bigl[y^n\bigr] \= \begin{cases} (-1)^{n/2}(n-1)!!\, & \text{for $n$ even,} \\
  \qquad\quad 0 & \text{for $n$ odd.}  \end{cases}  \ee
Equation~\eqref{eq:PhiGauss2}  makes sense because the terms of~$\cBB$ all have strictly positive degree in the $u_i$
and the coefficient of any monomial $u_1^{\ell_1}u_2^{\ell_2}\cdots$ in~$\cBB$, and hence also in~$e^{\cBB}$, is 
a polynomial in $X$ and $y$ to which the functional~$\frakI$ can be applied to get a polynomial in~$X$.
\par
We now prove Theorems~\ref{thm:EO} and~\ref{thm:hbracketbfU}. Our proof of the former will use the axiomatic characterization 
of~$F_n(z_1,\dots,z_n)$ given in Theorem~\ref{thm:BOviaAxioms}. Theorem~\ref{thm:hbracketbfU} will then be deduced from Theorem~\ref{thm:EO}, the 
argument being a purely formal one in the sense that if we replaced the power series~$B(z)=\frac{z/2}{\sinh z/2}$ 
in equation~\eqref{evXofEO} by any other even power series with constant coefficient~1, and replaced the numbers 
$\b_m$ in the definition~\eqref{defBuyX} by the coefficients of this power series, then the new equation~\eqref{evXofEO} 
would imply the new equation~\eqref{eq:PhiGauss} in exactly the same way.
\par
\begin{proof}[Proof of Theorem~\ref{thm:EO}] From the axiomatic description of $F_n(z_1,\dots,z_n)$ given in
Theorem~\ref{thm:BOviaAxioms} and the fact that $\,\evX:\wM_*\to\Q[X]$ is a ring homomorphism it follows
immediately that there is a similar axiomatic description of $F_n(z_1,\dots,z_n)_X$ in which the function
$\th(z)$ in (iv) is replaced by its $X$-evaluation as given in~\eqref{eq:evXtheta} and everything else is unchanged.
We thus need to verify that the right hand side of~\eqref{evXofEO}, which we denote by $F_n^*(z_1,\dots,z_n)_X$ in the proof below, 
satisfies these modified axioms. First, note that $F^*$ is indeed a Laurent series in the variables $z_1,\dots,z_n$,
since the exponent of $z_A=\sum_{a\in A}z_a$ in~\eqref{evXofEO} is negative only if $|A|=1$.  
The property~(i) and the symmetry in the arguments stated in~(ii) 
are immediate for~$F^*$ from its definition. For~(iii) (with $n$ replaced by~$n+1$ and $z_{n+1}$ by~$z$), we observe 
that, since any element of $\PPP(n+1)$ is obtained from a unique element $\alpha$ of~$\PPP(n)$ either by adding the 
one-element set $\{n+1\}$ to~$\alpha$ or by replacing some element of~$\alpha$ by its union with~$\{n+1\}$, we have
\bas & F_{n+1}^*(z_1,\dots,z_n,z)_X \= e^{-(z_N+z)^2/2} \sum_{\alpha\in\PPP(n)}
 \biggl[\frac{\sqrt X/2}{\sinh z\sqrt X/2 } \,\prod_{A\in\alpha}\frac{z_A^{|A|-1}\sqrt X/2}{\sinh z_A\sqrt X/2} \\
& \hphantom{F_{n+1}^*(z_1,\dots,z_n,z)_X \= } \qquad
  \+ \sum_{B\in\alpha}\frac{(z_B+z)^{|B|}\sqrt X/2}{\sinh((z_B+z)\sqrt X/2)} \,
  \prod_{A\in\alpha\smallsetminus\{B\}} \frac{z_A^{|A|-1}\sqrt X/2}{\sinh(z_A\sqrt X/2)}\biggr]  \\
&\quad \=  e^{-z_N^2/2}\Bigl(1-zz_N+\text O(z^2)\Bigr) \sum_{\alpha\in\PPP(n)}\Bigl(\frac1z\+\sum_{B\in\alpha}z_B\+\text O(z)\Bigr)
 \prod_{A\in\alpha}\frac{z_A^{|A|-1}\sqrt X/2}{\sinh z_A\sqrt X/2} \\
&\quad \= \frac1z\,F_n^*(z_1,\dots,z_n)_X \+ \text O(z) \qquad\text{as $z\to0$}\,, \eas
as desired. Finally, to show~(iv) we multiply the right hand side of~\eqref{evXofEO} with~\eqref{eq:evXtheta} (with
$z$ replaced by~$z_N$). The positive part of this expression is zero if we can show that the positive part of 
$\sinh(s_N)\prod_{A\in\alpha}\frac{s_A^{|A|-1}}{\sinh s_A}$ is~0 for each $\alpha \in \PPP(n)$ individually, 
where we set $s_A = z_A\sqrt X/2$.  But since $s_N=\sum_{A\in\alpha}s_A$, we have
$$ \sinh(s_N) \= \sum_{\b\subseteq\a\atop\text{$\ell(\b)$ odd}} \prod_{A\in\a\ssm\b}\cosh s_A \cdot \prod_{A\in\b} \sinh s_A $$
and hence 
$$ \sinh(s_N)\,\prod_{A\in\alpha}\frac{s_A^{|A|-1}}{\sinh s_A} 
\= \sum_{\b\subseteq\a\atop\text{$\ell(\b)$ odd}} \prod_{A\in\b} s_A^{|A|-1}\cdot\prod_{A\notin\b}\frac{s_A^{|A|-1}}{\tanh s_A} \,. $$
We want to show that the coefficient of $z_1^{r_1}\cdots z_n^{r_n}$ in each summand of this expression vanishes if all of the~$r_i$
are strictly positive. Since each $i$ belongs to only one of the sets in~$\a$, and since every set~$\b$ occurring has odd cardinality
and hence is non-empty, it suffices to prove this statement for a single term $s_A^{|A|-1}$~($A\in\b$), and this is obvious since a 
homogeneous polynomial of degree~$|A|-1$ cannot contain every variable $z_a$ ($a\in A$) to a strictly positive power. \end{proof}
\par
\begin{proof}[Proof of Theorem~\ref{thm:hbracketbfU}] Both for this proof and for use later in the paper,
it is convenient to define a linear map $\Om_n$ from the space of Laurent polynomials or 
Laurent series in $n$ variables $z_1,\dots,z_n$ to the space of polynomials or power series in 
infinitely many variables $u_1,u_2,\dots$ by the formula
\be\label{defOmn} 
{ \Om_n\bigl[z_1^{\ell_1}\,\cdots\, z_n^{\ell_n}\bigr] \=
  \begin{cases}\dfrac{\ell_1!\,\cdots\,\ell_n!}{n!}\,u_{\ell_1}\cdots u_{\ell_n} & \text{if $\ell_1,\dots,\ell_n\ge1$,} \\
   \hphantom{\ell_1!\cdots \ell_n!}\quad 0 & \text{otherwise.} \end{cases} 
}\ee
With this notation, we can compute our two generating functions~\eqref{eq:defPhi} and~\eqref{eq:defcorrelator},
or their $X$-bracket versions, by
\be \label{TwoGenFn} \Phi(\bfu)_q \,=\, \sum_{n=0}^\infty \Om_n\bigl[F_n(z_1,\dots,z_n)\bigr]\,,\quad
   \Phi(\bfu)_X \,=\, \sum_{n=0}^\infty \Om_n\bigl[F_n(z_1,\dots,z_n)_X\bigr]\,. \ee
(Recall that $p_\ell=\ell!\,Q_{\ell+1}$ for $\ell\ge1$.)  On the other hand, if $\a\in\PPP(n)$ is a partition of~$N = \{1,\dots,n\}$ and
if to each $A\in\a$ we have associated a power series $G_A(z)$ in one variable, then from the multinomial theorem we find that
$ \Om_n\bigl[\prod_{A\in\a}G_A(z_A)\big]$ equals the product over all $A\in\a$ of $G_A(d/dt)(U(t)^{{|A|}})|_{t=0}$, where 
$z_A=\sum_{a\in A}z_a$ as before and where
\be \label{defUt}  U(t) \= u_1t \+ u_2t^2 \+\cdots  \ee
is the generating power series of the~$u_\ell$. In particular, if $G_A(z)=G_{|A|}(z)$ depends 
only on the cardinality of~$A$, then 
\be\label{G-identity} \Om_n\Biggl[\sum_{\a\in\PPP(n)\atop \ell(\a)=m}
  \prod_{A\in\a} G_{|A|}(z_A)\Biggr] \=  \frac1{m!}\,
  \sum_{s_1,\dots,s_m\ge1\atop s_1+\cdots + s_m=n}\frac{\g_{s_1}(\bfu)}{s_1!}\cdots\frac{\g_{s_m}(\bfu)}{s_m!}\, \ee 
with $\g_s(\bfu):=G_s(d/dt)(U(t)^s)|_{t=0}$, because if $\a$ is a partition~ of~$N=\{1,\dots,n\}$ with~$m$ parts, 
then we can order them in precisely~$m!$ ways (since they are non-empty and distinct), and there are 
$\frac{n!}{s_1!\cdots s_m!}$ ordered partitions $N=A_1\sqcup\cdots\sqcup A_m$ of given sizes $s_1,\dots,s_m\ge1$.
Summing~\eqref{G-identity} over~$m$ and then over~$n$ gives
\be \label{eq:OmSumPart}
 \sum_{n=0}^\infty \Om_n\Biggl[\,\sum_{\a\in\PPP(n)}\prod_{A\in\a} G_{|A|}(z_A)\Biggr] \= 
  \exp\Biggl(\,\sum_{s\ge1} \frac{\g_s(\bfu)}{s!}\Biggr)\,. 
\ee
On the other hand, observing that for any $z$ independent of $y$ we have from~\eqref{deffrakI}
 \be \label{eq:Ih}   e^{-z^2/2} \= \sum_{\ell=0}^\infty \frac{(-1/2)^\ell}{\ell!}z^{2\ell} \= \frakI\bigl[e^{zy}\bigr]\;, \ee
we can rewrite~\eqref{evXofEO} as
$$ F_n(z_1,\dots,z_n)_X \=  \frakI\Biggl[ \,\sum_{\a\in\PPP(n)}\prod_{A\in\a}\Bigl(z_A^{|A|-2}B\bigl(z_A\sqrt X\bigr)e^{z_Ay}\Bigr) \Biggr]  $$
with $B(z)$ defined as in~\eqref{eq:defbeta}. We insert this into the second equation of~\eqref{TwoGenFn} and note that the 
maps~$\Om_n$ and $\frakI$ commute.  We apply~\eqref{eq:OmSumPart} with 
\bes
G_s(z) \= z^{s-2} \bigl(B(z\sqrt{X})e^{zy} { \, - \, 
\delta_{s,1}\bigr)\,.}
\ees
{
(Here we are allowed to delete the pole term $1/z$ for $s=1$ because negative
powers in~\eqref{defOmn} are discarded.)} This gives
equation~\eqref{eq:PhiGauss2} with~$\cBB$ defined by
\be\label{newdefBuyX}  \cBB(\bfu,y,X) \= \sum_{k,r\ge0\atop k+r\ge2} \b_k\,X^{k/2}\,
   \frac{y^r}{r!}\,\sum_{s\ge1} \frac{d^{k+r+s-2}}{dt^{k+r+s-2}}\Bigl(\frac{U(t)^s}{s!}\Bigr)\Bigr|_{t=0}\,,  \ee
which is easily seen to be equivalent to the definition~\eqref{defBuyX}.
\end{proof}

\medskip 

The inner sum in the formula~\eqref{newdefBuyX} used above can be calculated in a more explicit form using the Lagrange inversion 
theorem.  This leads to the following two propositions, special cases of which will 
be used in Section~\ref{sec:OneVariable}
to write down various one-variable power series that will be needed for the asymptotic calculations in Part~IV.
\begin{Prop}  \label{Prop10.4}  Let 
$$ T(y) \= T(\bfu,y) \= \frac{u_1}{1-u_1}\,y \+ \frac{u_2}{(1-u_1)^3}\,y^2 
\+  \frac{2u_2^2 +(1-u_1)u_3}{(1-u_1)^5}y^3 \+ \cdots$$
be the solution of $T=U(y+T)$, with $U(t)$ as in~\eqref{defUt}. Then 
\be\label{Bviat}  \cBB(\bfu,y,X) \= \int_0^yT(y')dy' \+ \sum_{k\ge2}\beta_{k} \, T^{(k-1)}(y)\,X^{k/2}\,. \ee
\end{Prop}
\begin{proof} This follows (independently of the definition of the coefficients~$\b_k$) directly from 
equation~\eqref{newdefBuyX} together with the formula
$$ T\,=\,U(y+T) \quad\Longleftrightarrow\quad T \,=\, \sum_{s=1}^\infty \frac1{s!}\,\frac{d^{s-1}}{dy^{s-1}}\,U(y)^s\,, $$
which is one of the forms of the Lagrange inversion theorem.
\end{proof}
\noindent{\bf Remark.} 
From either~\eqref{defBuyX} or~\eqref{Bviat} we see that the function $\cBB(\bfu,y,X)$ 
has a very special form: if we denote by $c_n(\bfu)$ the coefficient of~$y^n$ in $T(\bfu,y)$, then 
$$ \cBB(\bfu,y,X) \= c_1(\bfu)\Bigl(\frac{y^2}2-\frac X{24}\Bigr)\+ c_2(\bfu)\Bigl(\frac{y^3}3-\frac {Xy}{12}\Bigr)
   \+c_3(\bfu)\Bigl(\frac{y^4}4-\frac {Xy^2}8+\frac{7X^2}{960}\Bigr)\+\cdots $$ 
in which the ratio of the coefficients of $X^iy^j$ and $y^{2i+j}$ for any $i,j\ge0$ is independent of~$\bfu$.
\begin{Prop}\label{prop:dBdul}  Let $t(y)=t(\bfu,y)$ be the inverse power series of $y=t-U(t)$. Then for all $\ell\ge1$ we have
  \be\label{dBdul}  \frac{\p\cBB(\bfu,y,X)}{\p u_\ell} 
    \=  \sum_{k=0}^\infty \beta_k\,\frac{\p^k}{\p y^k}\biggl(\frac{t(\bfu,y)^{\ell+1}}{\ell\+1}\biggr)\,X^{k/2}\,. \ee
\end{Prop}
\begin{proof} We first observe that the power series $t(y)$ of this proposition is related to the $T(y)$ of the previous
proposition by $t(y)=y+T(y)$. Then  
$$ 0 \= \frac{\p}{\p u_\ell}\Bigl(T(\bfu,y)\m U\bigl(y+T(\bfu,y)\bigr)\Bigr)\= (1-U'(t))\,\frac{\p T}{\p u_\ell}\m t^\ell$$
or
$$ \frac{\p T}{\p u_\ell} \= \frac{t^\ell}{1-U'(t)} \= \frac{t^\ell}{\p  y/\p  t} \= \frac \p{\p y}\biggl(\frac{t^{\ell+1}}{\ell+1}\biggr)\,.$$
Combining this with~\eqref{Bviat}, we find
$$ \frac{\p^2\cBB(\bfu,y,X)}{\p y\,\p u_\ell} \= \frac{\p}{\p u_\ell}\Biggl(\sum_{k=0}^\infty \beta_k \, \frac{\p^kT}{\p y^k}\,X^{k/2}\Biggr)
  \= \sum_{k=0}^\infty \beta_k \, \frac{\p^{k+1}}{\p y^{k+1}}\biggl(\frac{t^{\ell+1}}{\ell+1}\biggr)\,X^{k/2}\,, $$
and~\eqref{dBdul} follows by integrating with respect to~$y$, the constant term being~0.
\end{proof}
\par
\smallskip 
We next present a result that gives a small refinement of Theorem~\ref{thm:EO}.
At the end of Section~\ref{sec:QMofqb} we introduced two differential operators
$\partial$ and~$\Delta$ on the ring of shifted symmetric polynomials 
and explained their relationship to $q$-brackets (Propostion~\ref{prop:fd}).
The following proposition describes their surprisingly simple
action on the $n$-point generating function $W(z_1)\cdots W(z_n)$ {{(see~\eqref{defW} for the notation $W(z)$)}}. 
\par
\begin{Prop} We have 
\bes \label{eq:partial}
g(\partial) \bigl(W(z_1)\cdots W(z_n)\bigr) \= g(z_N) \, W(z_1)\cdots W(z_n)
\ees
for any power series $g(t) \in \CC((t))$, and 
\bes \label{eq:Delta}
e^{\ve \Delta/2} \bigl(W(z_1)\cdots W(z_n)\bigr) =  \sum_{\alpha \in \PPP(n)} \,
\prod_{A\in\alpha} \,
(\ve z_A)^{|A|-1}\, W(z_A)\,.
\ees 
\end{Prop}
\par
\begin{proof} The definition of $\partial$ implies that $W(z_1)\cdots W(z_n)$
is an eigenvector of~$\partial$ with eigenvalue~$z_N$. This gives the
first formula. For $\Delta$ we find, using $\dfrac{\partial  W(z_i) }{\partial Q_{k+1}}
= z_i^k$,  
\ba
\Delta\Biggl(\prod_{i=1}^n W(z_i)\biggr) 
&\=  \sum_{1\leq i \neq j \leq n} 
\Biggl(\sum_{k,\ell \geq 0} \binom{k+\ell}{k} Q_{k+\ell}\, z_i^k z_j^\ell \Biggr) 
\prod_{1 \leq h \leq n \atop h\neq i,\,j} W(z_h)  \\
& \= 2 \!\sum_{1\leq i< j \leq n} 
(z_i + z_j)W(z_i + z_j)  \prod_{1 \leq h \leq n \atop h\neq i,\,j} W(z_h)\,.
\ea
%
%
%
By induction we obtain a formula for the action of $\Delta^r$ on 
$W(z_1)\cdots W(z_n)$, and then multiplying by $(\ve/2)^r/r!$ and
summing over~$r$ we obtain the claim.
\end{proof}
\par
Now  take $g=e^{-\ve t^2/2}$ in the first formula of the proposition 
and then replace $\ve$ and $z_i$ by $1/X$ and $z_i\sqrt{X}$, respectively, in 
both formulas. Then from the two assertions of Proposition~\ref{prop:fd} 
we obtain {(with $\bfW$ and $\fd$ as in~\eqref{Lie})}
\begin{flalign}
X^{\bfW/2} e^\fd \bq{W(z_1)\cdots W(z_n)} 
&\= e^{\fd/X}\, \bq{\sqrt{X}\,W\bigl(z_1\sqrt{X}\bigr) \;\cdots \;\sqrt{X}\,
W\bigl(z_n\sqrt{X}\bigr)}  \nonumber \\
& \= X^{n/2}\, \bq{ e^{\Delta/2X} \,e^{-\partial^2/2X}\, 
\bigl(W\bigl(z_1\sqrt{X}\bigr)\cdots W\bigl(z_n\sqrt{X}\bigr)\bigr)} 
\nonumber \\
& \= e^{-z_N^2/2} \, \sum_{\alpha \in \PPP(n)} {
 \bq{\prod_{A\in\alpha} \,
z_A^{|A|-1}\sqrt X W(z_A \sqrt{X})}}\,. \label{lastequation}
\end{flalign}
This is the above-mentioned strengthening of Theorem~\ref{thm:EO}, 
since~\eqref{evXofEO} follows immediately from~\eqref{lastequation} 
{{and Proposition~\ref{prop:evLie}}} by applying the 
ring homomorphism {
$\RRR_*  \xrightarrow[]{\sbq{\cdot}} \wM_* \xrightarrow[]{a_0} \QQ$ 
 which sends $Q_k \mapsto \beta_k$ and $W(z)$ to $1/2\sinh(z/2)$.}
\par
\smallskip
We end this section by giving a statement about the ``degree drop" of the growth polynomials of certain $q$-brackets.
It says, for instance, that the $X$-bracket $\sbrX{Q_3^{2n}}$, which {\it a priori} could have degree up to $3n$
in~$X$ since the weight of $Q_3^{2n}$ is~$6n$, in fact has degree at most~$2n$.  A related and even stronger statement for
the ``connected brackets" studied in the next section will lead to the definitions of cumulants that
will be crucial for the asymptotic calculations given in Part~IV.

\begin{Prop} The degree of the growth polynomial of an element of~$\RRR$ of weight $2k$ that is a product of $2n$ elements
of odd weight is at most $k-n$.  In particular, the degree of the $X$-bracket of a monomial $p_1^{r_1}p_2^{r_2}\cdots$
in the~$p_i$ is bounded by $r_1+r_2+2(r_3+r_4)+3(r_5+r_6)+\cdots\,$.
\end{Prop}
\begin{proof}  We will in fact prove the second statement of the proposition, which is clearly equivalent to the first.
To any monomial $\bfu^\bfa=u_1^{a_1}u_2^{a_2}\cdots$ we associate the invariants $w(\bfa)=a_2+2a_3+3a_4+\cdots$
(as in~\eqref{defBuyX}), $s(\bfa)=a_1+a_2+a_3+\cdots$ (= the $s$ of \eqref{newdefBuyX}), $K(\bfa)=2a_1+3a_2+4a_3+\cdots$
(the modular weight), $O(\bfa)=a_2+a_4+a_6+\cdots$ (corresponding to the number of occurrences of $p_\ell$ of odd weight),
and $\ve(\bfa)=0$ or~1 depending on whether $O(\bfa)$ is even or odd. They are related by $K=2s+w$, $s\ge O\ge0$, and 
$w\equiv\ve\!\pmod2$. If a monomial $u_1^{a_1}u_2^{a_2}\cdots\,X^d$ occurs in~\eqref{defBuyX}, then we 
have $s\ge1$ and $2d=2+w-r\le 2+w-\ve$, because $r\ge0$ and $r$~must be strictly positive if $w(\bfa)$ is odd
since $\b_k$ vanishes for $k$~odd.  It follows that $K-2d\ge2s-2+\ve$, which is always~$\ge O$. (If $O=0$ then
$2s-2+\ve\ge2s-2\ge0$; if $O=1$ then $2s-2+\ve\ge2s-1\ge1$, and if $O\ge2$ then $2s-2+\ve\ge2s-2\ge2O-2\ge O$.)
Thus the $X$-degree of the monomial in question is always $\le\frac12(K-O)$, and since both the $X$-degree and
the invariants $K$ and $O$ are additive, it follows that the same estimate is true for any monomial occurring
in any power of $\cBB$, and hence also for every monomial occurring in our formula~$\frakI\bigl[e^\cBB\bigr]$
for $\Phi(\bfu)_X$.
\end{proof}

\section{The generating series of cumulants}  \label{sec:gensercumu}

In this section we study the connected $q$-brackets and cumulants of~\cite{eo},
which encode the counting functions for counting of connected covers and their 
leading terms. Our main result is Theorem~\ref{thm:psi}, which gives
an expression for the generating series of cumulants as the value of the function 
$\cBB(\bfu,y)+y^2$ of the previous section at a stationary point. The proof relies on the
principle of least action applied to the formal Gaussian integral formula for~$\Phi(\bfu)_X$. 
 
We begin by describing a general algebraic formalism that is relevant 
in many geometric counting problems when we pass from disconnected to connected objects. 
Let $R$ and $R'$ be two commutative $\Q$-algebras with unit and $\la\;\,\ra:R\to R'$
a linear map sending~1 to~1.  (Of course the cases of interest to us will be when
$R$ is the Bloch-Okounkov ring~$\RRR$ and $\la\;\,\ra$ is the $q$-, $X$-, or $h$-bracket 
to $R'=\wM_*$, $\Q[X]$, or $\Q[\pi^2][h]$, respectively.)  Then we extend $\la\;\,\ra$ to a 
multi-linear map $R^{\otimes n}\to R'$ for every $n\ge1$, the image of $f_1\otimes\cdots\otimes f_n$ 
being denoted by either $\la f_1|\cdots|f_n\ra$ or $\la|f_1\otimes\cdots\otimes f_n|\ra$,
that are defined { by the formula
\be  \label{slash}
\langle f_1|\ldots|f_n\rangle \= \sum_{\alpha \in \PPP(n)} (-1)^{\ell(\a)-1} 
(\ell(\a)-1)!\, \prod_{A\in\a} \Bigl\la \prod_{a\in A} f_a \Bigr\ra
\ee
}(cf.~\eqref{eq:NNfromNpr}), where $\ell(\a)$ denotes the length (cardinality) of the partition~$\a$.
For instance, for $n=2$ and $n=3$ we have
\bas \la f|g\rangle &\= \la fg\ra \,-\, \la f\ra \la g\ra\,, \\
  \la f|g|h\rangle &\= \la fgh\ra \,-\, \la f\ra \la gh\ra \,-\, \la g\ra \la fh\ra 
       \,-\, \la h\ra \la fg\ra \+ 2 \,\la f\ra\la g\ra\la h\ra\,. \eas      
Following~\cite{eo}, we call $\la f_1|\cdots|f_n\ra$ the {\em connected bracket} 
of the functions $f_1,\ldots,f_n$ corresponding to the original bracket~$\la\;\,\ra$. Note that the 
connected bracket is symmetric, so defines a map from $S^n(R)$ to~$R'$, and that it vanishes if any~$f_i$ 
equals~1, so in fact descends to a map $S^n(R/\Q)\to R'$, and also that the definition can be inverted
to express all brackets in terms of connected ones, e.g.
 \bas \la fg\ra &\= \la f|g\rangle \+ \la f\ra \la g\ra\,,  \\
  \la fgh\ra  &\= \la f|g|h\rangle  \+ \la f\ra \la g|h\ra \+ \la g\ra\la f|h\ra 
       \+ \la h\ra \la f|g\ra \+ \la f\ra\la g\ra\la h\ra     \eas
and in general 
\be  \label{inverseslash}  \langle f_1\cdots f_n\rangle 
 \= \sum_{\alpha \in \PPP(n)} \prod_{A\in\a} \bigl\la\bigl| \otimes_{a\in A} f_a \bigr|\bigr\ra \,.
\ee
This formula is a special case of Proposition~\ref{prop:connectedbq} below.
\par
Perhaps the most important property of connected brackets is their appearance in the logarithm of the
original bracket applied to an exponential:
\bas \log\bigl(\bigl\la e^{f_1+f_2+f_3+\cdots}\bigr\ra\bigr)
   &\= \log\Bigl(1 \+ \sum_i\la f_i\ra \+ \frac1{2!}\sum_{i,j}\la f_if_j\ra
     \+ \frac1{3!}\sum_{i,j,k}\la f_if_jf_k\ra \+\cdots\Bigr) \\
  &\= \sum_i\la f_i\ra \+ \frac1{2!}\sum_{i,j}\la f_i|f_j\ra 
   \+ \frac1{3!}\sum_{i,j,k}\la f_i|f_j|f_k\ra \+\cdots \;, \eas
which explains by a well-known principle why the connected brackets correspond to the 
counting functions of connected objects.  This gives us yet a third definition 
of the connected bracket $\la f_1|\cdots|f_n\ra$, as the coefficient of the monomial
$x_1\cdots x_n$ in $\bigl\la \exp(x_1f_1+\cdots+x_nf_n)\ra$. 
Applying it to the rings $R=\RRR$ and $R'=\wM_*$ and the $q$-bracket $\la\;\,\ra_q$, we find
that the generating series of the connected $q$-brackets is equal to the logarithm of the
partition function $\Phi(\bfu)_q$ defined in~\eqref{eq:defPhi}:
\ba \label{eq:psiq}
\Psi(\bfu)_q &:=\;\sum_{n>0}^\infty \frac1{n!} \sum_{\ell_1,\ldots,\ell_n \geq 1} 
   \la p_{\ell_1}|\cdots|p_{\ell_n} \ra_q\, u_{\ell_1}\cdots u_{\ell_n}\\ 
& \= \sum_{\bfn > 0}\la\underbrace{p_1|\cdots|p_1}_{n_1}|\underbrace{p_2|\cdots|p_2}_{n_2}|\cdots\ra_q\,\frac{\bfu^\bfn}{\bfn!}
  \= \log \Phi(\bfu) _q\,,
\ea
and similarly 
\be\label{eq:psiX}
\Psi(\bfu)_X:=\evX[\Psi(\bfu)_q]=\log\Phi(\bfu)_X
\ee
for the generating function of the connected $X$-brackets $\la p_{m_1}|\cdots|p_{m_n} \ra_X\,$.
\par
\medskip
Our main concern is the $X$-evaluation of connected brackets. The first result, which is due to Eskin and Okounkov (\cite[Theorem~6.3]{eo}) but will also follow from our proof of Theorem~\ref{thm:psi} below, is that the degree of the connected $X$-brackets
as a polynomial in~$X$, which for the original $X$-bracket was at most half of the weight, drops by one for every $\,|\,$-insertion.
\par
\begin{Prop} \label{prop:degdrop} Let $f_i\in\RRR_{k_i}$ $(i=1,\dots,n)$ be homogeneous elements of the ring~$\RRR$ and
$k=k_1+\cdots+k_n$ the total weight.  Then $\deg(\la f_1|\cdots|f_n\ra_X)\le1-n+k/2$.
\end{Prop}
\par 
Motivated by this, we define the {\em leading coefficient} of the growth polynomial of $\sbq{f_1|\cdots|f_n}$ 
for $f_i$ and $k$ as in the proposition by
\bes \la f_1|\cdots|f_n\ra_L \,=\, [X^{1-n +k/2}] \,\la f_1|\cdots|f_n\ra_X 
\,=\,\lim_{X\to\infty} \frac{\evX[\langle f_1|\cdots|f_n\rangle_q](X)}{X^{1-n+k/2}}\,. \ees 
We will be especially interested in the case when each of the $f_i$ is one of the
standard generators $p_\ell=\ell!\,Q_{\ell+1}$ of~$\RRR$. We define the rational numbers
\be \label{eq:defcumu}
 \lda  \ell_1,\cdots,\ell_n\rda_\QQ \= \la p_{\ell_1}|\cdots|p_{\ell_n} \ra_L\qquad(\ell_1,\dots,\ell_n\ge1)\,,
\ee 
which we call {\em rational cumulants},\footnote{We avoid powers of $\pi$ here. The real
cumulants $\lda \ell_1,\ldots,\ell_s \rda\in\Q[\pi]$ will be defined in Part~IV.} with the corresponding generating series
\ba \label{eq:psi1}
\Psi(\bfu)_L &\= \sum_{\bfn\ge0}\lda\underbrace{1,\ldots,1}_{n_1},\underbrace{2,\ldots,2}_{n_2},
\ldots\rda_\QQ\frac{\bfu^\bfn}{\bfn!} \\
&\= \sum_{n=0}^\infty \frac1{n!} \sum_{\ell_1,\ldots,\ell_n \ge 1} 
\lda {\ell_1}, \ldots, {\ell_n} \rda_\QQ\, u_{\ell_1}\cdots u_{\ell_n} \;.
\ea
\par
Our main result in this section is a formula for this generating function that will be used in Section~\ref{sec:OneVariable}
and in Part~IV. Its statement uses the function $\cBB(\bfu,y,X)$ defined in~\eqref{defBuyX}. We write $\cBB(\bfu,y)$ for 
the polynomial $\cBB(\bfu,y,1)$ and denote by $\cBB'(\bfu,y,X)$ and $\cBB'(\bfu,y)$ the derivatives of $\cBB(\bfu,y,X)$ 
and $\cBB(\bfu,y)$ with respect to~$y$.

\begin{Thm} \label{thm:psi} The generating series of cumulants is given by
\be \label{eq:psiformula}
\Psi(\bfu)_L = \Bigl. \cBB(\bfu,y_0)   \;+\; \frac{y_0^2}{2} \,,
\ee  
where $y_0=y_0(\bfu)\in\Q[[\bfu]]$ is the unique power series satisfying $\cBB'(\bfu,y_0)+y_0=0\,$.
\end{Thm}
\par
\begin{proof}[Proof of Proposition~\ref{prop:degdrop} and Theorem~\ref{thm:psi}]  We first note that
there is a unique power series $y=y_0(\bfu,X)$ as solution of the equation $\cBB'(\bfu,y,X)+y=0$, as one can
see either by Newton's method or by iterating $y\mapsto-\cBB(\bfu,y,X)$ (starting in either case with $y=0$), 
or alternatively by noting that the latter map is a contraction and hence has a unique fixed point.  
The special case $y_0(\bfu,1)$ is the function $y_0(\bfu)$ occurring in the theorem, and in fact the two
functions are equivalent because from its definition $\cBB(\bfu,y,X)$ has the homogeneity property
\be \label{Bhomogen} \cBB(\bfu,ty,t^2X) \= t^2 \,\cBB(t\!\circ\!\bfu,y,X), \ee
 where $t\!\circ\!\bfu:=(u_1,tu_2, t^2u_3,\ldots)$, and therefore $y_0(\bfu,X)=X^{1/2}y_0(X^{1/2}\!\circ\!\bfu)$.
From the beginning of the Taylor expansion of~$\cBB$, as given either by its definition or by 
Proposition~\ref{Prop10.4}, we find that the expansion of $y_0(\bfu,X)$ begins with 
$$ y_0(\bfu,X)\= \frac{u_2}{12(1-u_1)^2}X \,-\,\Bigl(\frac{u_2^3}{9(1-u_1)^6}
    +\frac{u_2u_3}{8(1-u_1)^5}+\frac{7u_4}{240(1-u_1)^4}\Bigr)\,X^2 \+ \cdots \,,$$
in which the homogeneity property just mentioned is reflected in the fact that the coefficient of~$X^n$
is homogeneous of weight $2n-1$ for each~$n$, where $u_i$ has weight~$i-1$ (or equivalently, if $y_0$ is thought
of as an element of $\Q[X][[\bfu]]$, that the coefficient of any monomial $\bfu^\bfa$ is a multiple of $X^{1+w(\bfa)/2}$).

We now use formula~\eqref{eq:psiX} together with Theorem~\ref{thm:hbracketbfU}, which expresses $\Phi(\bfu)_X$
as a formal Gaussian integral. To evaluate the logarithm of this integral, guided by the {\em principle 
of least action}, we shift the integration variable $y$ by $y_0(\bfu,X)$ so that the exponent of the integrand has
no linear term.  The procedure is justified because the translational invariance of the Gaussian integral (or a simple 
combinatorial calculation using the formal definition~\eqref{deffrakI}) gives the transformation property
  \be\label{GaussTransl}  \frakI[F(y+z)] \= e^{z^2/2}\,\frakI[e^{-yz}F(y)] \ee
for polynomials $F(y)$ (equation~\eqref{eq:Ih} is the special case~$F=1$ of this), and we can also apply this when 
$F(y)$~is a power series in~$y$ with coefficients in~$\bfu$ so long as the coefficient of the monomial $\bfu^\bfa$ 
vanishes for $\bfa=0$ and is a polynomial in~$y$ for~$\bfa>0$.  We apply it with $z=y_0(\bfu,X)$, using the Taylor expansion
$$ \cBB(\bfu,y_0+y,X)+\frac{(y_0+y)^2}2 \=  \cBB(y_0)+\frac{y_0^2}2 
    \+\bigl(\cBB''(y_0)+1)\,\frac{y^2}2 \+ \cBB'''(y_0)\,\frac{y^3}6 \+ \cdots \;, $$
where $\cBB^{(n)}(y_0)$ is shorthand for $\frac{\p^n\cBB}{\p y^n}(\bfu,y_0(\bfu,X),X)$. Here the coefficient of~$y$ is~0 by 
the definition of~$y_0$ and the other terms have expansions beginning with 
\bas & \cBB(\bfu,y_0,X)+\frac{y_0^2}2 \= \frac{u_1}{24(1-u_1)}\,X \+ \Bigl(\frac{u_2^2}{90(1-u_1)^5}+\frac{7u_3}{960(1-u_1)^4}\Bigr)\,X^2 \+ \cdots\,,   \\
    & \cBB''(\bfu,y_0,X)+1 \= \frac1{1-u_1} \m \Bigl(\frac{u_2^2}{3(1-u_1)^5}+\frac{u_3}{4(1-u_1)^4}\Bigr)\,X \+ \cdots\,,  \\
    & \cBB'''(\bfu,y_0,X) \= \frac{2u_2}{(1-u_1)^3} 
         \m \Bigl(\frac{4u_2^3}{3(1-u_1)^7} +\frac{9u_2u_3}{2(1-u_1)^6} +\frac{u_4}{(1-u_1)^5}\Bigr)\,X \+ \cdots\,,  \eas
in which the coefficient of $X^k$ in $\cBB^{(n)}(y_0)$ is homogeneous of weight~$2k+n-2$ in~$\bfu$.
Making the substitution $\bfu\mapsto X^{-1/2}\circ\bfu$ and using this homogeneity property, we therefore find
 $$  \Phi(X^{-\h}\!\circ\!\bfu)_X \,=\, e^{\bigl(\cBB(y_0)+\h y_0^2\bigr)X}\,
   \frakI\biggl[\exp\biggl(\frac{\cBB''(y_0)}2y^2+\frac{\cBB'''(y_0)}{6\sqrt X}y^3+\frac{\cBB^{\rm iv}(y_0)}{24X}y^4+\cdots\biggr)\biggr] $$
(now with $\cBB^{(n)}(y_0)=\frac{\p^n\cBB}{\p y^n}(\bfu,y_0(\bfu))$), and expanding the first few terms of this by~\eqref{deffrakI} and taking logarithms gives
  \bas  \Psi(X^{-1/2}\circ\bfu)_X & \= \Bigl(\cBB(y_0)+\frac12y_0^2\Bigr)\,X \m \frac12\log\bigl(1+\cBB''(y_0)\bigr) \\
  &\qquad \+ \biggl(\frac{B^{\rm iv}(y_0)}{8(1+\cBB''(y_0))^2}\m\frac{5B'''(y_0)^2}{24(1+\cBB''(y_0))^3}\biggr)\,\frac1X \+ \cdots\;.  \eas
The fact that this Laurent series has no powers $X^{>1}$ implies Proposition~\ref{prop:degdrop},
the fact that the coefficient of~$X$ is $\cBB(y_0)+\h y_0^2$ gives Theorem~\ref{thm:psi}, and the further
terms of the expansion give as many subleading terms of $\Psi(\bfu)_X$ as desired.
\end{proof}

Equation~\eqref{eq:psiformula} gives an effective way to evaluate cumulants, since $y_0$ is given as a fixed
point and can be computed rapidly by iteration.  The next proposition, which is also suitable for practical 
calculations, gives an alternative formula for the generating series~$\Psi(\bfu)_L$, 
reminiscent of the formula for~$\cBB(\bfu,y,X)$ in Proposition~\ref{Prop10.4}.
\begin{Prop} \label{prop:psiassum}
 The generating series of rational cumulants is given by
\ba \label{eq:psiassum}
\Psi(\bfu)_L &\=  \cBB(\bfu,0) + \sum_{m=2}^\infty \frac{(-1)^{m-1}}{m(m-1)}\,\bigl[y^{m-2}\bigr](\cBB'(\bfu,y)^m)\;.
\ea
\end{Prop} 
\par
\begin{proof} We need to prove the identity 
\be \label{eq:Bstationary}
\cBB(0) + \sum_{m=2}^\infty \frac{(-1)^{m-1}}{m(m-1)}\, [y^{m-2}](\cBB'(y)^m) \= \cBB(y_0) + \frac{y_0^2}{2}\,,
\ee
where $\cBB(y)=\cBB(\bfu,y)$ and $y_0$ is the solution of $\cBB'(y_0) = - y_0$. Write $z=y/\cBB'(y)$ and 
expand $y^2$ in powers of $z$, i.e. we define $a_m$ by $y^2 = \sum_{m\ge2} a_m z^m$. Then 
\bas 
& \bigl[y^{m-2}\bigr](\cBB'(y)^m) \=  \Res_{y=0}\,\Bigl(\frac{\cBB'(y)^m}{y^{m-1}} dy \Bigr)  
\= \Res_{y=0}\, \Bigl( z^{-m}\, d\bigl(\frac{y^2}{2}\bigr) \Bigr) \\ &\qquad \= -\frac12  \Res_{z=0}\, \Bigl(y^2 d(z^{-m}) \Bigr) 
\=\frac m2\Res_{z=0}\,\Bigl(\frac{y^2}{z^{m+1}} dz \Bigr) \=\frac m2 a_m.
\eas
Let $S$ be the left hand side of the expression in \eqref{eq:Bstationary}. Then
\begin{alignat}{2}
S- \cBB(0) &\= \frac12 \sum_{m = 2}^\infty \frac{(-1)^{m-1}}{m-1} a_{m}
& &\ = \frac12 \int_{0}^{-1} \frac{y(z)^2}{z^2} dz \nonumber \\
&\= \frac12 \int_{z=0}^{z=-1} y^2 \frac{d}{dy}\Bigl(\frac{-1}{z(y)}\Bigr) dy 
& &\= \frac12  \int_{z=0}^{z=-1} \bigr(2\cBB(y) - y\cBB'(y)\bigl)'dy  \nonumber\\
&\= \bigl(\cBB(y) - \frac y 2 \cBB'(y)\bigr)\Bigr|_{z(y)=0}^{z(y)=-1} 
& &\= \cBB(y_0) \+ \frac{y_0^2}2 \m \cBB(0)\,, \nonumber 
\end{alignat}
since  $z = 0$ corresponds to $y=0$ and $z=-1$ to $y = y_0$.
\end{proof}
\par
\medskip
Finally, just as Theorems~\ref{thm:EO} and~\ref{thm:hbracketbfU} in the previous section, 
one also has a version of the formula for cumulants with a fixed number of variables, i.e. for the generating function 
\ba\label{eq:corrL}
C_n(z_1,\ldots,z_n) &\= \la|W(z_1)\otimes\cdots\otimes W(z_n)|\ra_L \\
&\= \sum_{k_1,\ldots,k_n \geq 0} \la Q_{k_1}|\ldots|Q_{k_n} \ra_L\; z_1^{k_1-1} \cdots z_n^{k_n-1} \\
& \= \frac{\delta_{n,1}}{z_1} \+ \sum_{\ell_1,\ldots,\ell_n \ge1} \lda \ell_1,\cdots,\ell_n \rda\; \frac{z_1^{\ell_1} \cdots z_n^{\ell_n}}{\ell_1!\,\cdots\,\ell_n!}\,.
\ea 
(Here the last equality holds because $Q_0=1$, $Q_1=0$, and $Q_{\ell+1}=p_\ell/\ell!$ for~\hbox{$\ell\ge1$} and 
because all connected brackets having some argument equal to~1 vanish except for $\lda 1 \rda =1$.)  This formula, which can be
deduced from Theorem~\ref{thm:EO}, is equivalent to~\cite[Theorem~6.7]{eo}, where it is stated in a
somewhat different form, but here we will deduce it instead from Proposition~\ref{prop:psiassum}.
\par
\begin{Prop} The generating function~\eqref{eq:corrL} is given by
\be \label{eq:cumnfixed}
C_n(z_1,\ldots,z_n) \=  \sum_{\alpha \in \PPP(n)} \,(-1)^{\ell(\a)-1} \,z_N^{\ell(\a)-2}
\prod_{A\in\alpha} \,\frac{z_A^{|A|}/2}{\sinh(z_A/2)}\;, 
\ee
where $N$ and $z_A$ for $A\subseteq N$ have the same meaning as in Theorem~\ref{thm:EO}.
\end{Prop}
\par
\begin{proof}  We use the same formalism and notations as in the proof of Theorem~\ref{thm:hbracketbfU}. In view of 
equations~\eqref{eq:psi1}, \eqref{eq:corrL}, and~\eqref{defOmn} we have $\Om_n[C_n(z_1,\dots,z_n)]=\psi(\bfu)^{(n)}_L$, 
the degree~$n$ part of $\psi(\bfu)_L$, so if we denote by $\psi(\bfu)^{(n,m)}_L$ ($1\le m\le n$) the degree~$n$ part of
the $m$-th term in~\eqref{eq:psiassum} and by $C_{n,m}(z_1,\dots,z_n)$ the subsum of the right hand
side of~\eqref{eq:cumnfixed} corresponding to partitions~$\a\in\PPP(n)$
with $\ell(\a)=m$, then it suffices
to prove that $\Om_n[C_{n,m}]=\psi(\bfu)^{(n,m)}_L$ for each~$m$. 
 Instead of \eqref{eq:Ih} we now use that
$z^{m-2} = (m-2)!\, [y^{m-2}] e^{zy}$ to get
$$ C_{n,m}(z_1,\ldots,z_n) \= (-1)^{m-1}(m-2)!\,\bigl[y^{m-2}\bigr] \sum_{\a\in\PPP(n) \atop \ell(\a)=m}\,
    \prod_{A\in\a} \Bigl(z_A^{|A|-1}\,B(z_A)\,e^{z_Ay}\Bigr) $$
for $m\ge2$, with $B(x)$ as in~\eqref{eq:defbeta}.  Then using~\eqref{G-identity} and the fact that the 
operations~$\Om_n$ and $[y^{m-2}]$ commute, we find
$$ \Om_n\bigl[C_{n,m}\bigr] \= \frac{(-1)^{m-1}}{m(m-1)}\,\Biggl(\bigl[y^{m-2}\bigr] \,
 \Biggl(\sum_{s\ge1}\frac{\g_{s}(\bfu)}{s!}\Biggr)^m\Biggr)^{(n)}\, $$
where $\g_{s,y}(\bfu):=G_{s,y}(d/dt)(U(t)^s)|_{t=0}$ with $G_{s,y}(z)=z^{s-1}B(z)e^{yz}$. But
$$ \sum_{s\ge1}\frac{\g_{s}(\bfu)}{s!} \= \sum_{s\ge1} \,\Biggl( \,\sum_{k,r\ge0\atop k+r\ge1} \b_k \,
  \frac{y^r}{r!}\, \frac{d^{k+r+s-1}}{dt^{k+r+s-1}}\Bigl(\frac{U(t)^s}{s!}\Bigr)\Bigr|_{t=0} \Biggr) \= \cBB'(\bfu,y) $$
by~\eqref{newdefBuyX} with~$X=1$. This completes the proof of the cases $m\ge2$. The  case $m=1$ is similar but easier, using
$$ \sum_{n\ge1}\Om_n\bigl[C_{n,1}\bigr]
  \= \sum_{n\ge1} \,\Biggl( \,\sum_{k\ge2} \b_k \, \frac{d^{k+n-2}}{dt^{k+n-2}}\Bigl(\frac{U(t)^n}{n!}\Bigr)\Bigr|_{t=0}\Biggr) 
\= \cBB(\bfu,0)\,. $$
\end{proof}
\par

Theorem~\ref{thm:psi} or either of the last two propositions can let us compute the leading terms of 
connected brackets whose arguments are single generators~$p_\ell$. For the leading terms of more general connected brackets, 
we need a formula that expresses mixed brackets, involving both products and slashes, as products of connected brackets 
of single variables. A special case of this formula is~\eqref{inverseslash} above, and a simple mixed example is  
$$\la f|gh\ra \= \la f|g|h\ra\+\la g\ra\la f|h\ra \+\la h\ra\la f|g\ra\,.$$
The general result is stated, for arbitrary rings and brackets, in the following 
proposition. Certain versions of the result were known before (e.g. it is equivalent to \cite[Proposition~4.3]{Speed}; cf.~\cite[Chapter~6]{McCullagh}, p.~279), but the proof is not easy to find in the literature, and hence we give a short one here. For the formulation we need some terminology.
If~$\alpha$ and~$\beta$ are partitions of a finite set~$N$, we denote by $\alpha \vee \beta$
the finest partition coarser than both (i.e. if we think of partitions as equivalence
relations,  the equivalence relation generated by~$\a$ and~$\b$).
We denote by $\one$ the one-element partition~$\{N\}$. If $\alpha \vee \beta = \one$, 
then it is easy to see that $|\alpha| +|\beta| \leq |N|+1$. If equality holds, then 
the partitions are called {\em complementary}. The pairs with  $\alpha \vee \beta = \one$
will play a role in the following proposition, while complementary partitions appear in the corollary
concerning leading terms.
\par
\begin{Prop} \label{prop:connectedbq} 
Let $f_1,\dots,f_n$ be elements of~$\RRR$. Then for any partition $\beta$ of $N =\{1,\ldots,n\}$  we have
\be\label{eq:connectedbq} \bigl\la\bigl| \tens{B \in \beta} f_B \bigr|\bigr\ra  
\= \sum_{\alpha \in \PPP(n) \atop \a\vee\b=\one} 
\prod_{A\in\a} \bigl\la\bigl| \tens{a\in A} f_a \bigr|\bigr\ra\;, \ee
where  $f_B=\prod_{b\in B}f_b$ for $B\subseteq N$.
\end{Prop}
\par
\begin{proof}
We first recall the generalized {\it M\"obius inversion formula} for partially ordered sets in the
special case of the lattice of partitions of~$N$, ordered by \hbox{$\alpha \leq \beta$} if~$\alpha$ is finer 
than~$\beta$ (cf.~\cite{Rota}, especially\ Example~$1$ of Section~7). 
If $g$ is any function on~$\PPP({n})$
and $G$ is the associated {\it cumulative function}
$$G(\a)=\sum_{\b\le\a}g(\b), \quad \text{then}\quad g(\beta) = \sum_{\alpha \leq \beta}  
\mu(\alpha,\beta) G(\alpha)$$
with the {\it M\"obius function} $\mu(\alpha,\beta)$ 
given by $\prod_{B\in \beta} (-1)^{|\alpha_B|-1}({|\alpha_B|-1})!$, where $\alpha_B$ for $B \subseteq N$
is the partition on~$B$ induced by~$\alpha$. In this notation the definition~\eqref{slash} of the 
connected bracket can be written as 
\bes
\langle |\, \otimes_{i \in N} f_i\,|\rangle \= \sum_{\alpha \in \PPP(N)} \mu(\alpha,\one)\, 
\prod_{A\in\a} \bigl\la f_A \bigr\ra\,,
\ees
whose M\"obius inversion is~\eqref{inverseslash}. We apply this with~$N$ 
replaced by~$\beta$, noting that $\PPP(\beta)$
can be identified with $\{\gamma \in \PPP(N) \mid \gamma \geq \beta\}$, to obtain
$$ \la|\, \otimes_{B \in \beta} f_B \,|\ra  \= \sum_{\gamma \geq \beta}  \mu(\gamma,\one)
 \prod_{C\in\gamma}\la f_C \ra\,. $$
We now apply~\eqref{inverseslash} to each factor on the right hand side and
identify $\prod_{C \in \gamma} \PPP(C)$ with $\{\alpha \in \PPP(N)\mid \alpha \leq \gamma \}$  to obtain
$$ \la|\, \otimes_{B \in \beta} f_B \,|\ra  \= \sum_{\alpha} \Bigr( \sum_{\gamma \geq \alpha \vee \beta} \mu(\gamma,\one)
\Bigr) \prod_{A\in\a}\la|\, \otimes_{a\in A} f_a\,|\ra\,. $$
The proposition follows since $\sum_{\gamma \geq \alpha \vee \beta} \mu(\gamma,\one) = \delta_{\one, \alpha \vee \beta}$.
\end{proof}
\par
Proposition~\ref{prop:connectedbq} can be used in particular with $f_i=p_{\ell_i}\in\RRR$ for integers 
$\ell_i\in\NN$ to compute arbitrary connected $q$- or $X$-brackets in terms of those whose
arguments are single~$p_\ell$'s.  In the case of the $X$-brackets, we see from Proposition~\ref{prop:degdrop} that
the total degree drop on the left hand side of~\eqref{eq:connectedbq} (i.e. the minimal difference between the 
degree of this polynomial with respect to~$X$ and $K/2$, where $K=\sum(\ell_i+1)$ is the total weight) is $|\b|-1$, 
while the degree drop for the $\a$-th term on the right is $\sum_{A\in\a}(|A|-1)=n-|\a|$, which is strictly smaller
than $|\b|-1$ unless $\a$ and~$\b$ are complementary.  We therefore obtain the following expression for the 
leading terms of arbitrary connected $X$-brackets in terms of rational cumulants.

\begin{Cor} \label{cor:connectedbq} Let $\ell_1,\dots,\ell_n$ be natural numbers and for $B\subseteq N=\{1,\dots,n\}$
set $f_B=\prod_{b\in B}p_{\ell_b}\in\RRR$. Then for any partition $\b=\{B_1,\dots,B_s\}$ of~$N$ we have
\be\label{eq:leading}  
  \bL{f_{B_1}|\cdots|f_{B_s}} \= \sum_{\alpha} \prod_{A\in \alpha} \lda p_{\ell_a},\;a \in A \rda_\QQ\,, \ee
where the sum is over all partitions $\alpha$ of $N$ that are complementary to~$\b$.
\end{Cor}
\par
\par
We single out one important special case of this corollary.
If $|\beta| = n-1$, so that $\beta$ has the form $\{\{1,2\},\{3\},\ldots,\{n\}\}$, 
then the partitions~$\alpha$ with $\alpha\vee\beta=\one$ are the one-set partition~$\one$
and the two-set partitions $\{A_1,A_2\}$ with $1 \in A_1$ and $2 \in A_2$, with all but the first
of these being complementary to~$\b$. Therefore Corollary~\ref{cor:connectedbq} in this case
tells us that for any $f,g,h_i \in \RRR$ we have
$$\bL{fg |h_1|\cdots|h_m}\= \sum_{I \sqcup J = \{1,\ldots,m\}} \bL{|\,f \otimes \prod_{i\in I}
h_i\,|} \bL{|\,g \otimes \prod_{j\in J} h_j\,|}\,. 
$$
In particular, for any $n_1,n_2,\ldots \geq 0$ we have
\bas
&\phantom{\=} \bL{fg |\underbrace{p_1|\cdots|p_1}_{n_1}|\underbrace{p_2|\cdots|p_2}_{n_2}|\cdots} 
\\
&\= \sum_{\bfn = \bfn' + \bfn''} \frac{\bfn!}{\bfn'!\,\bfn''!}\,
\bL{f|\underbrace{p_1|\cdots|p_1}_{n'_1}|\underbrace{p_2|\cdots|p_2}_{n'_2}|\cdots}\,
\bL{g|\underbrace{p_1|\cdots|p_1}_{n''_1}|\underbrace{p_2|\cdots|p_2}_{n''_2}|\cdots} \,.
\eas
Making a generating series, we obtain the following proposition.
\par
\begin{Prop} \label{prop:PhiHOMO}
The map $\RRR \to \QQ[\bfu]$ defined by 
\ba \label{eq:defPsiUU}
\Psi(f;\bfu) \= \sum_{\bfn \geq 0} 
\bL{f|\underbrace{p_1|\cdots|p_1}_{n_1}|\underbrace{p_2|\cdots|p_2}_{n_2}|\cdots}\,\frac{\bfu^n}{\bfn!}
\ea
is a homomorphism of $\QQ$-algebras.
\end{Prop}
\noindent
Note that the generating series~\eqref{eq:defPsiUU} for $f= p_\ell$ takes the value
\be \label{Psiul}
  \Psi(p_\ell;\bfu) \= \sum_{\bfn \geq 0} 
\bL{\underbrace{p_1|\cdots|p_1}_{n_1}|\cdots|\underbrace{p_\ell|\cdots|p_\ell}_{n_{\ell}+1}|\cdots}
  \,\frac{\bfu^n}{\bfn!} \= \frac{\partial \Psi(\bfu)_L}{\partial u_\ell}\,,
\ee
and since $\RRR$ is generated by the $p_\ell$, this also gives the general values.  A more
explicit formula for $\Psi(p_\ell;\bfu)$ will be given in equation~\eqref{Psipl} below.
\par
\medskip

\section{One-variable generating series for cumulants}  \label{sec:OneVariable}
The main generating series identities of the last two sections were expressed in terms of a 
multi-variable~$\bfu=(u_1,u_2,\dots)$. For our main applications to the calculations of volumes 
and Siegel-Veech constants, we will be particularly interested in the specialization to the
case when this multi-variable has the special form $(0,u,0,0,\dots)$ for a single variable~$u$.
The basic invariants here are the special cumulants 
  \be\label{defvn}    v_n \=\frac1{n!}\, \lda\underbrace{2,\ldots,2}_n \rda_\QQ
    \= \frac1{n!}\,\la \underbrace{p_2|\cdots|p_2}_n \ra_L \qquad(n>0) \ee
involving only 2's (corresponding to coverings of a torus having only simple branch
points), which will be used in Part~IV for the computation of the volume of the principal stratum of abelian differentials, and their generating series
 \ba\label{defpsi}  \psi(u) &  \= \Psi(0,u,0,0,\dots)_L  \= \sum_{n=2}^\infty v_n\,u^n  \\
  & \=  \frac1{90}u^2 \m \frac7{162}u^4 \+ \frac{377}{810}u^6 \m \frac{23357}{2430}u^8
  \+ \frac{16493303}{51030}u^{10} \m \cdots\;.  \ea
Note that $v_n$ in~\eqref{defvn} vanishes unless $n$ is even, and then corresponds to genus $g$ coverings of a torus,
 where $n=2g-2$. To take into account genus~0 and~1, it turns out to be appropriate 
to extend~\eqref{defvn} to all~$n$ by defining
  \be\label{specialv} v_{-2} \,=\, v_0 \,=\, -\,\frac1{24}\,, \;\quad v_n\,=\,0\text{ for $n$ odd or $n<-2\,$.} \ee
The next most important numbers for us are the mixed cumulants defined by 
  \be\label{defvnk}   v_{n,k} \,=\, \frac k{n!}\, \lda\underbrace{2,\ldots,2}_n,k-1\rda_\QQ
    \= \frac{k!}{n!}\,\la \underbrace{p_2|\cdots|p_2}_n|Q_k\ra_L \ee
for $n\ge0,\,k\ge1$  and by $v_{n,0}=\delta_{n,0}$ if~$k=0$ (which agrees 
with~\eqref{defvnk} in that case since~$Q_0=1$), with corresponding generating series
   \be\label{defpsik} \psi_k(u)\,=\,\sum_{n=0}^\infty v_{n,k}\,u^n \= k!\,\Psi(Q_k;0,u,0,0,\dots)\,, \ee
where $\Psi(F;\bfu)$ is the power series associated to $F\in\RRR$ in Proposition~\ref{prop:PhiHOMO}. The
values of $v_{n,k}$ for $k \in \{0,1,2,3\}$ and $n\ge0$ are given in terms of~$v_n$ by
  \be\label{psi0123} v_{n,0}=\delta_{n,0},\quad v_{n,1}=0,\quad v_{n,2}=(4n+2)v_n,\quad v_{n,3}=(3n+3)v_{n+1}\,, \ee
and the first numerical values of $v_{n,k}$ for $4\le k\le6$ are given by
  \bas  
     \psi_4(u) &\= \hphantom{-\,}\frac7{240} \m \frac{5u^2}{18} \+ \frac{259u^4}{54}
        \m \frac{110773u^6}{810} \+ \frac{2220941u^8}{378} \m \cdots  \,, \\
     \psi_5(u) &\= -\,\frac{13u}{126} \+ \frac{179u^3}{81} \m \frac{33415u^5}{486} 
        \+ \frac{26367046u^7}{8505} \m \frac{29692284359u^9}{153090} \+ \cdots \,, \\
     \psi_6(u) &\= -\,\frac{31}{1344} \+ \frac{587u^2}{720} \m \frac{38525u^4}{1296} \+
       \frac{84696203u^6}{58320} \m \frac{12981245593u^8}{136080} \+ \cdots  \;. \eas
Finally, we want to study the particular combinations of cumulants defined by
  \be\label{defkn}  \kappa_n \= \sum_{k=0}^n 2^k\,v_{n-k,k}\=
  \delta_{n,0}\+\sum_{k=1}^n \frac{2^k k}{(n-k)!}\, \lda\underbrace{2,\ldots,2}_{n-k},k-1 \rda_\QQ \,,  \ee
which will be related in Part~IV to the area Siegel-Veech constants $c_\area(\omoduli(1^{2g-2}))$,
and the corresponding generating function 
  \ba\label{defKu}  K(u) &\= \sum_{n=0}^\infty \kappa_n\,u^n \= \sum_{k=0}^\infty (2u)^k\,\psi_k(u)  \\
  & \= 1 \m \frac13 u^2 \+ \frac{13}9 u^4 \m \frac{445}{27}u^6 \+ \frac{142333}{405}u^8 -\frac{975203}{81}u^{10} \+ \cdots\;. \ea

In this section, which uses all of the results proved in Part~II, we give explicit formulas allowing
for the numerical calculation of the coefficients of each of the generating series $\psi$, $\psi_k$, and $K$
(and also, as we will see in Part~IV, for the asymptotic evaluation of these coefficients).
It turns out that all of these generating functions can be expressed by a single sequence of Laurent series,
which we now introduce.

We begin by defining a polynomial of degree and parity~$n$ for each integer~$n\ge0$ by 
$\Br_n(X)=B_n(X+\h)$, the $n$-th Bernoulli polynomial with its argument shifted by one-half. 
Its first values are $1$, $X$, $X^2-\frac1{12}$,
$X^3-\frac X4$, and $X^4-\frac{X^2}2\+\frac7{240}$, and its expansion for general~$n$ is given by
\be\label{BnPoly}  \Br_n(X) \= \sum_{k=0}^n (n)_k\,\b_k\,X^{n-k}\qquad(n=0,1,2,\dots), \ee
where $\b_k$ is as in~\eqref{eq:defbeta} and $(n)_k=n(n-1)\cdots(n-k+1)$ is the descending 
Pochhammer symbol. We extend this definition to arbitrary complex values of~$n$ by setting
\be\label{BnSer}  \Br_n(X) \= \sum_{k=0}^\infty (n)_k\,\b_k\,X^{n-k} \;\,\in\, X^n\CC[X^{-1}]\qquad(n\in\CC)\,, \ee
a shifted Laurent series whose expansion begins with 
 $$ \Br_n(X) \= X^n \m \frac{n(n-1)}{24}X^{n-2} \+ \frac{7n(n-1)(n-1)(n-3)}{5760}X^{n-4} \m \cdots \;. $$
The fact that this series is divergent for all~$n\notin\Z_{\ge0}$ is not important for us, since we will use it only as a formal 
series, but it is worth mentioning that $\Br_n(X)$ can be defined as an actual function of $n$ and~$X$ by the formula 
 \be\label{HurwitzZ} \Br_n(X) \= -\,n\,\zeta(1-n,X+\h)  \qquad(n\in\,\CC,\;X\,\in\,\CC\smallsetminus(-\infty,-\h]),  \ee
where $\zeta(s,\a)$ denotes the Hurwitz zeta-function, defined by the convergent series $\sum_{m=0}^\infty(m+\a)^{-s}$
for $\a\in\CC\smallsetminus(-\infty,0]$ and $\Re(s)>1$ and then for all~$s$ by meromorphic continuation.
This new function $\Br_n(X)$ is entire in~$n$ (since $\zeta(s,\a)$ has a simple pole at~$s=1$ as its only singularity),
reduces to the previous definition if $n$ is a non-negative integer, and has the asymptotic expansion~\eqref{BnSer}
for all~$n\in\CC$, as one can see for instance for $\Re(n)<0$ from the integral representation
$\frac1{\Gamma(-n)}\int_0^\infty\frac{t^{-n}e^{-tX}dt}{2\sinh t/2}$ valid in that case. From the formula~\eqref{HurwitzZ},
or from the definition~\eqref{BnSer} and a simple calculation with Bernoulli numbers, we see that $\Br_n(X)$ 
satisfies the functional equation
 \be\label{FE}  \Br_n(X+\h) \m \Br_n(X-\h) \= n\,X^{n-1}  \ee
for all~$n$, and for $n\notin\Z_{\ge0}$ this property characterizes $\Br_n(X)$ uniquely as an element of $X^n\CC[X^{-1}]$,
giving us an alternative and less computational definition.  

For our purposes we need only the cases $n\in\Z_{\ge0}$, 
where~$\Br_n(X)$ is a polynomial, \hbox{and~$n\in\Z_{\ge0}-\h$}, the first three cases here being
\bas \Br_{-1/2}(X) &\= X^{-1/2} \m \frac1{32}\,X^{-5/2} \+ \frac{49}{6144}\,X^{-9/2} \m \frac{341}{65536}\,X^{-13/2} \+ \cdots\;, \\
\Br_{1/2}(X) &\= X^{1/2} \+ \frac1{96}\,X^{-3/2} \m \frac7{6144}\,X^{-7/2} \+ \frac{31}{65536}\,X^{-11/2} \m \cdots\;, \\
\Br_{3/2}(X) &\= X^{3/2} \m \frac1{32}\,X^{-1/2} \+ \frac7{10240}\,X^{-5/2} \m \frac{31}{196608}\,X^{-9/2} \+ \cdots\;. \eas
We can now state our final formulas for the generating functions $\psi$, $\psi_k$, and~$K$.
\par
\begin{Thm} \label{thm:GFvn} Let the Laurent series $X(u) = (4u)^{-1} + \cdots$ be defined by
\be \label{eq:defXu} 
X\=X(u) \quad \Longleftrightarrow \quad \dfrac1{2\sqrt u} \= \Br_{1/2}(X)\,.
\ee
Then the numbers $v_n$ defined by equations~\eqref{defvn} and~\eqref{specialv} are given either 
by the generating series 
   \be \label{GFvn1}   \sum_{n=-2}^\infty (4n+2)v_n\,u^{n+1} \= X(u)  \ee
or by the generating series 
   \be \label{GFvn2} \sum_{n=-2}^\infty (3n+3)\,v_n\,u^{n+1/2} \= \Br_{3/2}(X(u)) \,.\ee
\end{Thm}
\begin{Thm} \label{thm:GFpsik}  
Define $X=X(u)$ as in Theorem~\ref{thm:GFvn}.  Then the generating series $\psi_k$
defined by~\eqref{defvnk} and~\eqref{defpsik} is given for all $k\ge0$  by
  \be \label{GFpsik} \psi_k(u) \= \frac1{(2u)^k}\,\sum_{m=0}^k(-1)^m\,
\binom km\,(4u)^{m/2}\,\Br_{m/2}(X(u))\;.   \ee
\end{Thm}
\begin{Thm} \label{thm:GFkappa}  Let $X$ and $u$ be as above. Then the generating series~\eqref{defKu} is given by
   \be  \label{GFkappa}  2\,u^{1/2}\,K(u) \= \Br_{-1/2}(X(u)) \,. \ee
\end{Thm}
\par
\bigskip We make a few remarks on these theorems before giving their proofs.  
\par
\smallskip 
{\bf1.}  By taking a linear combination of equations~\eqref{GFvn1} 
and~\eqref{GFvn2} we can also obtain the explicit, though not very attractive, closed formula
  \be\label{ugly} \psi(u) \= \frac2{3\sqrt u}\,\Br_{3/2}(X(u)) \m \frac {X(u)}{2u} \+ \frac1{24} \+\frac1{24u^2} \ee 
for the original generating series~$\psi(u)$ defined in~\eqref{defpsi}.

{\bf2.} The right hand side of equation~\eqref{GFpsik} reduces to $1$ and to
$\frac1{2u}(1-\sqrt{4u}\,\Br_{1/2}(X))$ for $k=0$ and $k=1$, respectively, so Theorem~\ref{thm:GFpsik} gives 
the correct values $\psi_0(u)=1$ and $\psi_1(u)=0$ in these two cases.  In fact, if we wished we could rewrite 
the whole theorem as the assertion that there is {\it some} Laurent series $X=X(u)=\frac1{4u}\+\cdots$ 
such that~\eqref{GFpsik} holds for all~$k\ge0$, since then the special case~$k=1$ combined with the fact 
that $\psi_1$ vanishes identically would force the relation $\sqrt{4u}\,\Br_{1/2}(X)=1$.

{\bf3.} Similarly, using that $\Br_1(X)=X$ we find that equation~\eqref{GFpsik} for $k=2$ and $k=3$
reduces to $4u^2\psi_2(u)=-1+4uX$ and $8u^3\psi_3(u)=-2+12uX-8u^{3/2}\Br_{3/2}(X)$, respectively,
in agreement with equations~\eqref{psi0123}, \eqref{GFvn1}, and~\eqref{GFvn2}. 
  
{\bf 4.} The individual terms on the right hand side of~\eqref{GFpsik} have poles of order~$k$ 
in~$u$, but all the negative powers of~$u$ cancel in the sum because the coefficient of~$u^i$ in 
$(4u)^m\Br_{m/2}(X)$ is a polynomial of degree~$i$ in~$m$ for all~$i\ge0$ and the $k$-th difference of such
a polynomial vanishes if~$i<k$.   
  
{\bf 5.} This same delicate cancellation means that one cannot deduce equation~\eqref{GFkappa}
from equation~\eqref{GFpsik} simply by plugging the latter into~\eqref{defKu}, because in the double series 
obtained by this substitution one cannot interchange the order of summation.

\bigskip
For the proof of Theorems~\ref{thm:GFvn}--\ref{thm:GFkappa} we use the formalism of the previous two sections
and in particular the power series $\cBB(\bfu,y)$ and its specialization 
\be \label{defBuySimple}  \cBB(u,y) \,=\, \cBB((0,u,0,0,\dots),y) 
   \,=\, \sum_{k\ge0} \b_k\,\sum_{a\ge1,\,r\ge0 \atop a-r = k-2} \frac{(2a)!}{a!r!}\,u^a y^r\,.  \ee
to $\bfu=(0,u,0,0,\dots)$.  It is obvious from the definition that the specialization of~$\cBB(u,y)$ to~$y=0$ 
can be expressed in terms of the function $\Br_{3/2}(X)$ by
  $$ \cBB(u,0) \=\sum_{k=3}^\infty \frac{(2k-4)!}{(k-2)!}\,\b_k u^{k-2} 
  \= \frac2{3\sqrt u}\Br_{3/2}\Bigl(\frac1{4u}\Bigr)-\frac1{12u^2}+\frac1{24}\,.$$
What is more surprising is that the {\it whole} two-variable function $\cBB(u,y)$ can be expressed in terms of the 
one-variable function $\Br_{3/2}(X)$, as stated in the following proposition. This is the reason why the whole story works.
\par
\begin{Prop} \label{Prop12.2} The two-variable function~$\cBB(u,y)$ defined by~\eqref{defBuySimple} can be
expressed in terms of the one-variable function~$\Br_{3/2}(X)$ by the formula
  \be\label{Identity}
  \cBB(u,y) \, = \, \frac2{3\sqrt u}\,\Br_{3/2}\Bigl(\frac{1-4uy}{4u}\Bigr) \m \frac{y^2}2
  \+ \frac{y}{2u} \m  \frac{1}{12u^2} \+ \frac1{24}\;.  \ee 
\end{Prop} 
\begin{proof} We have
\bas \cBB(u,y) &\= \sum_{k\ge0}\,\b_k\,\sum_{a\ge1,\,r\ge0 \atop a-r = k-2} \frac{(2a-1)!!}{r!}\,(2u)^a y^r \\
& \= \frac43\;\sum_{k\ge0}\,(3/2)_k\,\b_k\,(4u)^{k-2} \,\sum_{r\ge0 \atop k+r\ge3} \binom{r+k-\frac52}r\,(4uy)^r\,.
\eas
The proposition then follows since the internal sum is equal to $(1-4uy)^{3/2-k}$ by the binomial theorem in
all cases except  $k=2$ and $k=0$, where we must subtract one or three monomials corresponding to $0\le r\le2-k$.
(The term~$k=1$ does not enter since $\b_k=0$ for $k$~odd.)
\end{proof}  
\par

We observe that the above proposition and its proof are just the specialization of Proposition~\ref{Prop10.4} to
$\bfu=(0,u,0,0,\dots)$, since in that case the function $U(t)$ reduces to $ut^2$, the solution $T(y)=T(u,y)$ of $T=U(y+T)$ is given by
$$ T(y) \= \frac{1-2uy-\sqrt{1-4uy}}{2u} \,, $$
 and the integral and derivatives of $T(y)$ are given by
\bas \int_0^y T(y')dy' &\= \frac{(1-4uy)^{3/2} - (1-6uy+6u^2y^2)}{12u^2}\,, \\
 T^{(k-1)}(y) &\= -\,\delta_{k,2} \+ \frac2{3\sqrt u}\,(3/2)_k\, \biggl(\frac{4u}{1-4uy}\biggr)^{k-3/2}\qquad(k\ge2)\,. \eas

\medskip
Using Proposition~\ref{Prop12.2} we can now give the proofs of all three theorems above.   
\begin{proof}[Proof of Theorem~\ref{thm:GFvn}]   We calculate $\psi(u)$ using Theorem~\ref{thm:psi}.
Differentiating~\eqref{Identity} and using the obvious formula $\Br_n'(X)=n\Br_{n-1}(X)$ for any~$n$, we find that
$$ \cBB'(u,y) \+ y \= \frac1{2u} \m \frac1{\sqrt u}\,\Br_{1/2}\biggl(\frac{1-4uy}{4u}\biggr)\,, $$
which vanishes if $X=\frac{1-4uy}{4u}$ is related to~$u$ by $\Br_{1/2}(X)=\frac1{2\sqrt u}$. Thus the function 
$y_0=y_0(u)$ occurring in Theorem~\ref{thm:psi} is related to the function~$X(u)$ defined in Theorem~\ref{thm:GFvn}
by $X(u)=\frac{1-4uy_0(u)}{4u}$.  Substituting this into Theorem~\ref{thm:psi} and using Proposition~\ref{Prop12.2}
again gives equation~\eqref{ugly} after a short computation, and differentiating this equation and using the
definition of $y_0(u)$ once again lets us then deduce the nicer formulas~\eqref{GFvn1} and~\eqref{GFvn2}. 
\end{proof}  

\begin{proof}[Proof of Theorem~\ref{thm:GFpsik}] From equations~\eqref{Psiul} and~\eqref{eq:psiformula} we have, for any~$\ell\ge1$,
  \begin{flalign} \Psi(p_\ell,\bfu) &\= \frac{\p\Psi(\bfu)_L}{\p u_\ell}
    \= \frac{\p}{\p u_\ell}\Bigl(\cBB\bigl(\bfu,y_0(\bfu)\bigr)+\frac12y_0(\bfu)^2\Bigr)  \nonumber \\
   &\=\frac{\p\cBB(\bfu,y)}{\p u_\ell}\Bigr|_{y=y_0(\bfu)}
    \+\Bigl(\cBB'\bigl(\bfu,y_0(\bfu)\bigr)+y_0(\bfu)\Bigr) \frac{\p y_0(\bfu)}{\p u_\ell} \nonumber \\
  \label{Psipl}  &\= \frac{\p\cBB(\bfu,y)}{\p u_\ell}\Bigr|_{y=y_0(\bfu)}   \;. \end{flalign}
On the other hand, specializing~\eqref{dBdul} to $\bfu=(0,u,0,0,\dots)$ and $X=1$ we get
\bes  n\,\frac{\p\cBB(\bfu,y)}{\p u_{n-1}}\Bigl|_{\bfu=(0,u,0,0,\dots)}
  \=\sum_{k=0}^\infty\b_k\,\frac{\p^k}{\p y^k}\bigl( t(u,y^n)\bigr)\,.  \ees
Substituting into this the formula 
\bas \frac{\p^k}{\p y^k}\bigl( t(u,y)^n\bigr) 
  &\= \frac{\p^k}{\p y^k}\biggl(\Bigl(\frac{1-\sqrt{1-4uy}}{2u}\Bigr)^n\biggr) \\
  &\= \frac{\p^k}{\p y^k}\biggl(\frac1{(2u)^n}\,\sum_{m=0}^n(-1)^m\binom nm (1-4uy)^{m/2}\biggr)  \\
  &\=  \frac1{(2u)^n}\,\sum_{m=0}^n(-1)^m\binom nm\,(m/2)_k\,(4u)^k(1-4uy)^{m/2-k}  \eas
and interchanging the order of summation we get (after changing~$n$ to~$k$)
$$  k\,\frac{\p\cBB(\bfu,y)}{\p u_{k-1}}\Bigl|_{\bfu=(0,u,0,0,\dots)}
  \= \frac1{(2u)^k}\,\sum_{m=0}^k(-1)^m\binom km\,(4u)^{m/2} \,\Br_{m/2}\biggl(\frac{1-4uy}{4u}\biggr) \,,  $$
and now combining this with~\eqref{Psipl} and remembering that $\frac{1-4uy_0(u)}{4u}=X(u)$ we obtain the
formula~\eqref{GFpsik} for the power series $\psi_k(u)=k\Psi(p_{k-1};0,u,0,0,\dots)$.
\end{proof}

\begin{proof}[Proof of Theorem~\ref{thm:GFkappa}] Just as in Remark {\bf 4.} above, equation~\eqref{GFpsik} gives
$$ 2^k\,v_{n-k,k} \= [u^n]\bigl((2u)^k\psi_k(u)\bigr) \= \sum_{m=0}^k(-1)^m\binom kmP_n(m) $$
for any $n\ge k\ge0$, where $P_n(m)$ is the coefficient of $u^n$ in $(4u)^{m/2}\Br_{m/2}(X(u))$, which is a 
polynomial of degree~$\le n$ in~$m$.  Summing over~$0\le k\le n$, we find
$$  \kappa_n\= \sum_{0\le m\le k\le n}(-1)^m\binom kmP_n(m) 
  \= \sum_{m=0}^n(-1)^m\binom{n+1}{m+1}P_n(m)  \= P_n(-1)\,,$$
where the last equality holds because the $(n+1)$st difference of a polynomial of degree~$\le n$ vanishes.
It follows that the generating function $K(u)=\sum\kappa_nu^n$ equals $(2u)^{-1/2}\Br_{-1/2}(X(u))$, as asserted.
\end{proof} 
\par
\vfill
\pagebreak

%% file: Part3_BOSV.tex
\part*{Part III: The hook-length moment $T_p$} 

The heros of Part~III are the hook-length moments
\be \label{eq:defTp}
T_p(\lambda) \= \sum_{\xi \in Y_\lambda} h(\xi)^{p-1} \qquad (p > 0 
\text{ odd})\,, 
\ee
where $Y_\lambda$ is the Young diagram of $\lambda$ and $h(\xi)$ is
the hook-length of the cell~$\xi$. 
We will show that these functions belong to the ring of shifted
symmetric polynomials and study their 
effect on $q$-brackets of functions on partitions.  
\par
In Section~\ref{sec:elemTp} we show that $T_p$ appears naturally in a spectral 
decomposition of Schur's orthogonality relation and as a natural function whose 
$q$-brackets are Eisenstein series. The $q$-brackets are linear, but are far 
from ring homomorphisms. In 
Section~\ref{sec:tpop1} we give a remarkable formula that expresses the 
multiplicative effect of $T_p$ inside a $q$-bracket only in terms of
Eisenstein series  and a collection of differential operators. 
The proof of these formula relies on a two-step
recursive expression for the Bloch-Okounkov functions that we will give 
in Section~\ref{sec:tpop2}.
Finally, in Section~\ref{sec:applSV} we apply our knowledge about $T_p$, 
which we extend to include $T_{-1}$, to prove the quasimodularity of 
Siegel-Veech 
generating series that appeared at the end of Part~I.

\section{From part-length moments to hook-length moments} 
\label{sec:elemTp}
 
The function  $T_p$ defined in~\eqref{eq:defTp} has three
remarkable properties that we discuss in this section. 
{{(We continue to use the notations for partitions given 
at the beginning of Section~\ref{sec:Partitions}.)}} The first property 
appears in the problem of decomposing Schur's orthogonality
relation, which states that  for $\lambda_1,\lambda_2 \in \P(\pd)$
$$\frac{1}{\pd!} \sum_{\mu \in \Part(\pd)} z_\mu \left(\sum_{m=1}^\infty 
mr_m(\mu)\right) 
\chi^{\lambda_1}(\mu)\chi^{\lambda_2}(\mu) \=  \pd \delta_{\lambda_1,\lambda_2},$$
where $z_\mu = \pd! \cdot (\prod_{m=1}^{\infty} m^{r_m(\mu)} \prod_{m=1}^\infty r_m(\mu)!)^{-1} $ 
is the size of the conjugacy class of the partition $\mu$. What is the
contribution, if we fix $m$ in the inner sum?
\par
To give the answer, we denote by $h(\xi)$ the hook-length of 
a cell $\xi\in Y_\lambda$ in the Young diagram of $\lambda$ and define
the hook-length count and the related counting polynomial to be
$$  N_m(\lambda) \= |\{\xi\in Y_\lambda \mid h(\xi)=m\}|\,, \quad 
H_\l(t) \= \sum_{m=1}^\infty N_m(\l)\,t^m \= \sum_{\xi\in Y_\l}t^{h(\xi)}\;\in\;t\,\ZZ[t]\,.$$
\par
\begin{Thm} \label{thm:hooklengthformula}
For each $\pd \in \NN$, $\lambda \in \Part(\pd)$, and $m \in \NN$, we have the identity
\begin{equation} \label{eq:HLdiag}
 \frac{1}{\pd!}\sum_{\mu \in \Part(\pd)} z_\mu\, m r_m(\mu) \chi^\lambda(\mu)^2 
\= N_m(\l)
\end{equation}
\end{Thm}
\par
We define, as in Part~I,  the $p$-th weight $S_p(\lambda) = \sum_{j=1}^k \lambda_j^p$ of a
partition $\lambda = (\lambda_1, \ldots, \lambda_k)$. Multiplying~\eqref{eq:HLdiag} by $t^m$, summing over all $m \geq 1$, and taking
the $(p-1)$-st moment, {{that is, applying $p-1$~times the differential operator 
$D = z\frac{\partial}{\partial z}$ and substituting $z=1$, }}
we thus obtain the following statement, which is the original
motivation for this section and will be used crucially in Section~\ref{sec:applSV}.
\par
\begin{Cor} \label{cor:SumSpTp}
For every $\lambda \in \Part(\pd)$
\bes 
\frac{1}{\pd!}\sum_{\mu \in \Part(\pd)} S_p(\mu)\, z_\mu\, \chi^\lambda(\mu)^2 
\= T_p(\l)\,.
\ees
\end{Cor}
\par
We remark that one can introduce a transformation $f \mapsto {\bf M}f$ on
the functions $\P \to \QQ$ to be
$${\bf M}f(\lambda) \=   \frac1{d!} \sum_{\mu \in \P(d)} z_\mu f(\mu) \chi^\lambda(\mu)^2 $$ 
for $\lambda \in \P(d)$. This transformation has the feature
 $\bq{f} = \bq{{\bf M}f}$ and, by the preceding results, 
the image of $S_p$ under the transformation 
${\bf M}$ is $T_p$. This was used in~\cite{zagBO} as one of several examples
to point out that the set of functions with quasimodular $q$-brackets is much
larger than the ring of shifted symmetric functions.
\par
The proof of Theorem~\ref{thm:hooklengthformula} will actually give a more
general formula. For two partitions $\sigma$ and $\lambda$ with $\sigma_i \leq \lambda_i$
we define a {\em skew Young diagram} $\lambda/\sigma$ by removing the cells of $Y_{\sigma}$ from 
the cells of $Y_{\lambda}$. We call $\lambda/\sigma$ a {\em border strip} or {\em rim hook}, 
if it is connected (through edges of boxes, not only through vertices) and if it does not
contain a $2\times 2$ block. We also write $\lambda \ssm \gamma$ for the smaller partition $\sigma$ after removing 
the rim hook $\gamma$ from $\lambda$. The {\em height} $\hgt(\gamma)$ of a rim hook $\gamma$ is
the number of its rows minus one. There is an obvious bijection between hooks and
rim hooks that fixes the end-points of the hook. For $m \leq \pd$ we define 
a $|\P(d-m)|\times |\P(\pd)|$ matrix by
$$ (D^\pd_m)_{\sigma,\lambda} =  \left\{\begin{array}{ll} 
(-1)^{\hgt(\gamma)} & \text{if}\,\,\, \lambda/\sigma = \gamma \,\,\, \text{is a rim hook} \\
0 & \text{otherwise,} \\
\end{array}\right.$$
where $\lambda \in \P(\pd)$ and $\sigma \in \P(m)$. Theorem~\ref{thm:hooklengthformula}
then follows from the result below.
\par
\begin{Prop} \label{prop:removestrips}
For each pair $\lambda_1, \lambda_2 \in \P(\pd)$, we have 
\begin{equation} \label{eq:weightedSchur}
 \frac{1}{\pd!}\sum_{\mu \in \Part(\pd)} z_\mu\, m r_m(\mu) \chi^{\lambda_1}(\mu) 
\chi^{\lambda_2}(\mu) = \left( (D^\pd_m)^T D^\pd_m \right)_{\lambda_1,\lambda_2}.
\end{equation}
\end{Prop}
\par
Note that the matrix $(D^\pd_m)^T D^\pd_m$ does not depend on the choice of 
ordering the elements in $\P(d-m)$ that we used to form the matrix $D^\pd_m$.
\par
\begin{proof}[Proof of Proposition~\ref{prop:removestrips}] 
The proof will be based on the Murnaghan-Nakayama rule. To recall this, we say 
that $\alpha = (\alpha_1,\alpha_2,\ldots)$ is a {\em composition}
of $\pd$, if $\alpha_i \in \NN$ and $\sum_{i=1}^\infty \alpha_i = \pd$. (A partition is
thus a composition with weakly decreasing $\alpha_i$.) Let  $\alpha \ssm \alpha_1$
denote the composition $(\alpha_2,\alpha_3,\ldots)$ of $\pd-\alpha_1$. 
The Murnaghan-Nakayama rule states that if $\lambda \in \P(\pd)$ 
and $\alpha$ is a composition of $\pd$, then
$$ \chi^\lambda(\alpha) = \sum_{|\gamma| = \alpha_1} (-1)^{\hgt(\gamma)} \chi^{\lambda \ssm \gamma}(\alpha \ssm \alpha_1), $$
where the sum is over all rim hooks $\gamma$ of $\lambda$ with $\alpha_1$ cells.
\par
The left hand side of~\eqref{eq:weightedSchur} is a sum over $\mu \in \P(\pd)$ and only 
those with a part of length $m$ contribute. We may thus use a composition 
$\mu = (m, \alpha_2(\mu), \alpha_3(\mu), \ldots)$ to evaluate the left hand side.
Let $\mu' = (\alpha_2(\mu), \alpha_3(\mu), \ldots)$ and use $\gamma_i$
to denote rim hooks of $\lambda_i$ below. Then we have   
\begin{align*}
& \frac{1}{\pd!} \sum_{\mu \in \Part(\pd)}  z_\mu mr_m(\mu) \chi^{\lambda_1}(\mu) \chi^{\lambda_2}(\mu)  \\
& = \sum_{\mu \in \Part(\pd)} \frac{z_\mu}{\pd!}  mr_m(\mu) \left( \sum_{|\gamma_1|=m} (-1)^{\hgt(\gamma_1)}
\chi^{\lambda_1\ssm \gamma_1}(\mu')\right) \left( \sum_{|\gamma_2|=m} (-1)^{\hgt(\gamma_2)}
\chi^{\lambda_2\ssm \gamma_2}(\mu')\right) \\
& = \sum_{\mu' \in \Part(\pd-m)} \frac{z_{\mu'}}{(\pd-m)!}  \left( \sum_{|\gamma_1|=m} (-1)^{\hgt(\gamma_1)}
\chi^{\lambda_1\ssm \gamma_1}(\mu')\right) \left( \sum_{|\gamma_2|=m} (-1)^{\hgt(\gamma_2)}
\chi^{\lambda_2\ssm \gamma_2}(\mu')\right) \\
& =  \sum_{|\gamma_1|=m} \sum_{|\gamma_2|=m}  (-1)^{\hgt(\gamma_1)+ \hgt(\gamma_2)} \sum_{\mu' \in \Part(\pd-m)} 
\frac{z_{\mu'}}{(\pd-m)!} \chi^{\lambda_1\ssm \gamma_1}(\mu') \chi^{\lambda_2\ssm \gamma_2}(\mu') \\
& =  \sum_{|\gamma_1|=m} \sum_{|\gamma_2|=m}  (-1)^{\hgt(\gamma_1)+ \hgt(\gamma_2)} 
\delta_{\lambda_1\ssm \gamma_1, \lambda_2\ssm \gamma_2}.
\end{align*}
This agrees with the right hand side by the definition of $D^\pd_m$.
\end{proof}
\par
The second remarkable property is that the $q$-brackets of $T_p$ are 
Eisenstein series.
\par
\begin{Prop} \label{prop:qofTp}
For all $p \in \ZZ$
$$ \langle T_p \rangle_q  \= \sum_{\pd \geq 1} \sigma_p(\pd)\, q^\pd.$$
\end{Prop}
\par
Note in particular, that $\langle T_{-1} \rangle_q = -\log((q)_\infty)$ is
not a quasimodular form, but almost, in the sense that its derivative is quasimodular.
This statement can also be deduced from Corollary~\ref{cor:SumSpTp}, and below we give an
elementary proof.
\par
\begin{proof}  For any given $\pd \in \NN$,
the multiset of hook-lengths of all partitions of $d$ is equal to the
multiset that is the union over all $|\lambda|=d$ of $\lambda_i$ repeated
$\lambda_i$ times. This fact appears in many guises in the combinatorics
literature, e.g.\ in \cite{bachermanivel}. In our notation 
$$\sum_{\pd=1}^\infty \sum_{|\lambda| = \pd} H_\lambda(z) q^\pd \= 
\sum_{\lambda} \sum_{j \geq 0} { \lambda_j}
z^{\lambda_j} q^{|\lambda|}\,.$$
In the right hand side of this expression, the coefficient in front of $z^m$ equals 
\begin{flalign}  \sum_\lambda { m} r_m(\lambda) q^{|\lambda|} &\= 
q^m \sum_{j=1}^\infty \Bigl(\sum_{\lambda_1 \geq \cdots \geq \lambda_{j-1} \geq m} \!\!\!\!
q^{\lambda_1+\cdots +\lambda_{j-1}} \Bigr) 
\Bigl( \sum_{m \geq \lambda_{j+1} \geq \lambda_{j+2} \geq \cdots} \!\!\!\!
q^{\lambda_{j+1} +\lambda_{j+2}+\cdots } \Bigr) &\nonumber \\
&\=  q^m \prod_{d \geq m} (1-q^d)^{-1} \prod_{d\leq m}(1-q^d)^{-1} \= 
\frac{q^m }{ (q)_\infty (1-q^m)}. \nonumber
\end{flalign}
Consequently,
$$(q)_\infty \sum_{\pd=1}^\infty \sum_{|\lambda| = \pd} H_\lambda(z) q^\pd
\= \sum_{m \geq 1} \frac{q^m}{1-q^m} z^m  
\= \sum_{d,m \geq 1} q^{md} z^m.$$
The claim follows by taking
the {{$(p-1)$-st moment}}, that is, applying $p-1$~times the differential operator 
$D = z\frac{\partial}{\partial z}$ and plugging in $z=1$.  
\end{proof}
\par

Finally we define $\wT_p:\P\to\QQ$ by 
$$ \wT_p(\l) \= \begin{cases}  T_p(\l) \+ \tfrac12 \z(-p) & \text{for $p\ge1$ odd} \\
\qquad\qquad\; 0 &\text{for $p$ even,} \\ 
\end{cases} $$
or equivalently by the generating function
\be \label{eq:Hlwithrec}
\frac{1}{z^2} \+  2\sum_{p=1}^\infty \wT_p(\l)\,\frac{z^{p-1}}{(p-1)!} \=
H_\l(e^z) \+ H_\l(e^{-z}) \+ \frac1{4\,\sinh^2(z/2)} 
\quad\in\;z^{-2}\QQ[[z^2]]\;. \
\ee
The following result describes these functions in terms of the basic invariants~$Q_k(\l)$.
\par
\begin{Thm} \label{thm:TpisinLambda}
The function $T_p:\P\to\ZZ$ belongs to the 
ring $\Lambda^*$ of shifted symmetric functions for every odd~$p\ge1$.
Explicitly, the function $\wT_p:\P\to\QQ$ is the homogeneous element 
of weight~$p+1$ given by
\be  \label{eq:TpinQk}
\frac{\wT_p(\lambda)}{(p-1)!} \= \frac12\, \sum_{k=0}^{p+1}(-1)^k Q_k(\l)\,Q_{p+1-k}(\l) 
\qquad(p\ge 1)\,.  
\ee
\end{Thm}
\par
In terms of generating functions, we can restate formula~\eqref{eq:TpinQk} as
\be  \label{eq:TpWW}
\frac{1}{z^2} \+  2\sum_{p=1}^\infty \wT_p(\l)\,\frac{z^{p-1}}{(p-1)!} 
\= -\,W_\lambda(z)\,W_\lambda(-z) 
\quad\bigl(\,=\, W_\lambda(z)W_{\lambda^\vee}(z)\,\bigr)
\ee
where $W_\lambda(z)$ is defined as in~\eqref{defW}. 
\par
\begin{proof} Denote by $H_\l^{(1)}(t)$, $H_\l^{(2)}(t)$, and $H_\l^{(3)}(t)$ 
the contributions to~$H_\l(t)$ coming from
the cells $s=(i,j)\in Y_\l$ with $1\le i,\,j\le r$, $1\le i\le r<j$, 
and $1\le j\le r<i$, respectively, where $(r;a_1,\dots,a_r;b_1,\dots,b_r)$ 
are the Frobenius coordinates of~$\l$.  If $1\le i,\,j\le r$, then
(see Figure~\ref{cap:hookFrob}, left picture) 
\begin{figure}
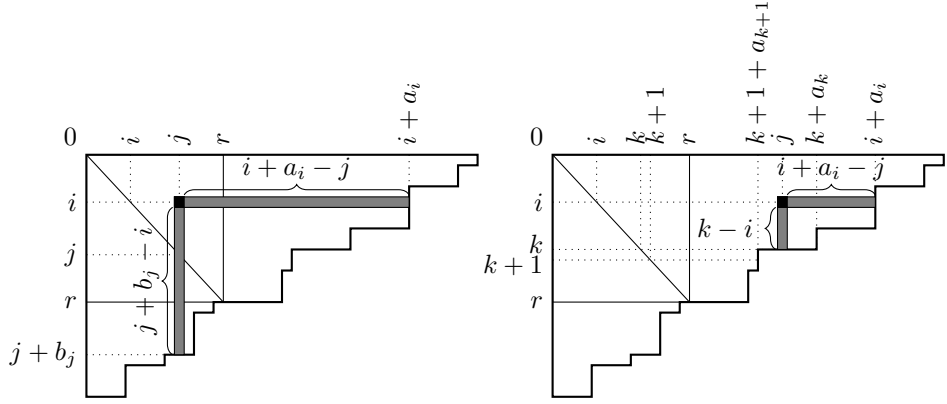

\begin{center}
  \tikzpicture[xscale=0.13,yscale=-0.14]
  \def\hki{4}
  \def\hkj{9}
  \def\hkai{24}
  \def\hkbj{15}
  \draw[white] (0,-1) -- (0,25);
  \draw[thick] (0,0) -- ++(40,0) \foreach \i in 
  { -2, 0, -5, 0, 0, 0, -6, 0, -6, 0, -1, 0, 0, -7, -2, 0, 0, 0, -3, -4, 0, 0, -4}
  { -- ++(0,1) -- ++(\i,0) } -- cycle;
  \draw (0,0) rectangle (14,14); 
  \draw (0,0) -- (14,14); 
  \draw[thin,dotted] (0, \hki+0.5) -- (\hkj+\hkai,\hki+0.5);
  \draw[thin,dotted] (0, \hkj+0.5) -- (\hkj+0.5,\hkj+0.5);
  \draw[thin,dotted] (\hki+0.5, 0) -- (\hki+0.5, \hki+0.5);
  \draw[thin,dotted] (\hkj+0.5, 0) -- (\hkj+0.5, \hkj+0.5);
  \draw[fill=gray] (\hkj, \hki) rectangle ++(1, \hkbj);
  \draw[fill=gray] (\hkj, \hki) rectangle ++(\hkai, 1);
  \draw[fill=black] (\hkj,\hki) rectangle (\hkj+1,\hki+1);
  \draw[thin,dotted] (\hkj+\hkai, 0) -- (\hkj+\hkai, \hki);
  \draw[thin,dotted] (0, \hki+\hkbj) -- (\hkj, \hki+\hkbj);
  \node[anchor=south east] at (0,0) {$0$};
  \node[anchor=east] at (0,\hki+0.5) {$i$};
  \node[anchor=east] at (0,\hkj+0.5) {$j$};
  \node[anchor=east] at (0,14) {$r$};
  \node[anchor=east] at (0,\hki+\hkbj) {$j+b_j$};
  \node[anchor=west,rotate=90] at (\hki+0.5, 0) {$i$};
  \node[anchor=west,rotate=90] at (\hkj+0.5, 0) {$j$};
  \node[anchor=west,rotate=90] at (14, 0) {$r$};
  \node[anchor=west,rotate=90] at (\hkj+\hkai, 0) {$i+a_i$};
  \draw [decorate,decoration={brace,amplitude=4pt},xshift=0pt,yshift=-5pt] 
  (\hkj+1,\hki) -- (\hkj+\hkai,\hki) node [midway,yshift=+10pt] {$i+a_i-j$};
  \draw [decorate,decoration={brace,amplitude=4pt},xshift=-5pt,yshift=0pt] 
  (\hkj,\hki+\hkbj) -- (\hkj,\hki+1) node [midway,xshift=-10pt,rotate=90] {$j+b_j-i$};
  \endtikzpicture
\!\!\!
 \tikzpicture[xscale=0.13,yscale=-0.14]

  \def\hki{4}
  \def\hkj{23}
  \def\hkai{10}
  \def\hkbj{5}
  \draw[white] (0,-1) -- (0,25);
  \draw[thick] (0,0) -- ++(40,0) \foreach \i in 
  { -2, 0, -5, 0, 0, 0, -6, 0, -6, 0, -1, 0, 0, -7, -2, 0, 0, 0, -3, -4, 0, 0, -4}
  { -- ++(0,1) -- ++(\i,0) } -- cycle;
  \draw (0,0) rectangle (14,14); 
  \draw (0,0) -- (14,14); 
  \draw[thin,dotted] (0, \hki+0.5) -- (33,\hki+0.5);
  \draw[thin,dotted] (\hkj+0.5, 0) -- (\hkj+0.5, 9);
  \draw[thin,dotted] (\hki+\hkbj, 0) -- (\hki+\hkbj, \hki+\hkbj);
  \draw[thin,dotted] (\hki+\hkbj+1, 0) -- (\hki+\hkbj+1, \hki+\hkbj+1);
  \draw[thin,dotted] (\hki+0.5, 0) -- (\hki+0.5, \hki+0.5);
  \draw[thin,dotted] (27, 0) -- (27, \hki+\hkbj);
  \draw[thin,dotted] (21, 0) -- (21, \hki+\hkbj);
  \draw[fill=gray] (\hkj, \hki) rectangle ++(1, \hkbj);
  \draw[fill=gray] (\hkj, \hki) rectangle ++(\hkai, 1);
  \draw[fill=black] (\hkj,\hki) rectangle (\hkj+1,\hki+1);
  \draw[thin,dotted] (\hkj+\hkai, 0) -- (\hkj+\hkai, \hki);
  \draw[thin,dotted] (0, \hki+\hkbj) -- (\hkj, \hki+\hkbj);
  \draw[thin,dotted] (0, \hki+\hkbj+1) -- (21, \hki+\hkbj+1);
  \node[anchor=south east] at (0,0) {$0$};
  \node[anchor=east] at (0,\hki+0.5) {$i$};
  \node[anchor=east] at (0,14) {$r$};
  \node[anchor=east,yshift=2pt] at (0,\hki+\hkbj) {$k$};
  \node[anchor=east,yshift=-2pt] at (0,\hki+\hkbj+1) {$k+1$};
  \node[anchor=west,rotate=90] at (\hki+0.5, 0) {$i$};
  \node[anchor=west,rotate=90] at (\hkj+0.5, 0) {$j$};
  \node[anchor=west,rotate=90] at (14, 0) {$r$};
  \node[anchor=west,rotate=90] at (\hkj+\hkai, 0) {$i+a_i$};
  \node[anchor=west,rotate=90] at (21,0) {$k+1+a_{k+1}$};
  \node[anchor=west,rotate=90] at (27,0) {$k+a_{k}$};
  \node[anchor=west,rotate=90,yshift=2pt] at (\hki+\hkbj,0) {$k$};
  \node[anchor=west,rotate=90,yshift=-2pt] at (\hki+\hkbj+1,0) {$k+1$};
  \draw [decorate,decoration={brace,amplitude=4pt},xshift=0pt,yshift=-5pt] 
  (\hkj+1,\hki) -- (\hkj+\hkai,\hki) node [midway,yshift=+10pt] {$i+a_i-j$};
  \draw [decorate,decoration={brace,amplitude=4pt},xshift=-5pt,yshift=0pt] 
  (\hkj,\hki+\hkbj) -- (\hkj,\hki+1) node [midway,xshift=-5pt,anchor=east] {$k-i$};
  \endtikzpicture
\end{center}
\caption{Hooks in Frobenius coordinates, for a cell inside (left) and
outside (right) the central square} \label{cap:hookFrob}
\end{figure}
the hook from $s$ has end-points $(i,i+a_i)$ and $(j+b_j,j)$,
so $h(s)=(i+a_i-j)+(j+b_j-i)+1=a_i+b_j+1$.  Hence
$$H_\l^{(1)}(t) \= \sum_{i,\,j=1}^r t^{a_i+b_j+1} \= \sum_{ c,\,c'\in C_\l \atop 
c>0>c'} t^{c-c'}\,, 
\quad H_\l^{(1)}(t) \+ H_\l^{(1)}(1/t) \= \sum_{ c,\,c'\in C_\l \atop cc'<0} t^{c-c'}\;. $$
If $1\le i\le r<j$, then (see Figure~\ref{cap:hookFrob}, right picture) 
the end-points of the hook from 
$s$ are at $(i,i+a_i)$ and~$(k,j)$, where $i\le k\le r$ is the unique 
index with $k+1+a_{k+1}<j\le k+a_k$ (resp.~$r<j\le r+a_r$ if~$k=r$),
so here $h(s)=(i+a_i-j)+(k-i)+1=a_i+k-j+1$. Hence
\bas & H_\l^{(2)}(t) \= \sum_{i=1}^r t^{a_i+1} \biggl(\sum_{k=i}^{r-1}\frac{t^{-a_{k+1}-1}\m t^{-a_k}}{t\m 1} 
  \+ \frac{1\m t^{-a_r}}{t\m 1}\biggr) \\ &\qquad \= -\sum_{1\le i\le k\le r}t^{a_i-a_k} \+ \sum_{i=1}^r\frac{t^{a_i+1}\m1}{t\m1} 
   \=  -\sum_{ c,\,c'\in C_\l \atop c\ge c'>0} t^{c-c'} \+ 
\sum_{ c\in C_\l \atop c>0} \frac{t^{c+\h}-1}{t-1}   \,, \\
& H_\l^{(2)}(t) \+ H_\l^{(2)}(1/t) \= -\sum_{c,\,c'\in C_\l \atop c,\,c'>0} t^{c-c'} 
\+  \sum_{c\in C_\l \atop c>0} \frac{t^c\m t^{-c}}{t^{1/2}\m t^{-1/2}}   \;. 
\eas
Similarly, or by interchanging the roles of the $a_i$ and $b_j$ 
(i.e. replacing~$\l$ by $\l^\vee$), 
$$ H_\l^{(3)}(t) \+ H_\l^{(3)}(1/t) \= -\sum_{c,\,c'\in C_\l \atop c,\,c'<0} 
t^{c-c'} \m  \sum_{ c\in C_\l \atop c<0}  \frac{t^c\m t^{-c}}{t^{1/2}\m t^{-1/2}} \;.$$
Adding all three formulas we get
\bas
H_\l(t) \+ H_\l(1/t) &\= -\sum_{c,\,c'\in C_\l} \sgn(cc')\, t^{c-c'}
 \+  \sum_{c\in C_\l } \sgn(c)\, \frac{t^c\m t^{-c}}{t^{1/2}\m t^{-1/2}} \\
&\= -\,w_\l^0(t)\,w_\l^0(1/t) \+ \frac{t^{1/2}}{t\m1}\, 
\bigl(w_\l^0(t)\+ w_\l^0(1/t)\bigr) \\
  &\= -\,w_\l(t)\,w_\l(1/t) \m \frac t{(t\m1)^2}\,, 
\eas
where $w_\l^0(t)$ and $w_\l(t)$ are defined in~\eqref{defw}. In view of~\eqref{eq:Hlwithrec} the above identity is equivalent to equation~\eqref{eq:TpWW}.
\end{proof}
\par

\section{A formula for  $q$-brackets involving $\wT_p$} \label{sec:tpop1}

With the applications to Siegel-Veech constants in mind, the  most important
among the functions~$T_p$ is the case $p=-1$. Here $T_p$ is {\em not} a shifted 
symmetric function and the Bloch-Okounkov theorem does not apply. The motivation for this 
section is to isolate the $p$-dependence outside the $q$-brackets and to interpolate the
quasimodularity proven for $p\geq 1$ to $p =-1$.  This is achieved by discovering a general formula
for the $q$-bracket of the product of $T_p$ ($p\geq 1$ odd) with an arbitrary shifted symmetric function.
\par
The basic observation, first made experimentally, is that the $q$-brackets $\sbq{\wT_p\,f}$
for a fixed element $f \in \RRR$ and varying odd numbers~$p$ is a linear combination of derivatives of
Eisenstein series with coefficients that are independent of~$p$, i.e.\ 
\be  \label{eq:coarseeffofTp}
\bq{\wT_p\,f} \= \sum_{i,\,j\ge0} \rho_{i,j}(f)_q \,G^{(j)}_{p+i+1}\qquad\text{for all odd $p\ge1$}\,, 
\ee
where $G^{(j)}_{k} := D^j G_k$ and $\rho_{i,j}(f)_q \in \wM_*$. Notice that the
quasimodular forms $\rho_{i,j}(f)_q$ are uniquely determined by this for $i$~even
(since~$p$ takes on infinitely many values), while those for $i$~odd are completely
free (since $G_k \equiv 0$ for $k$~odd). We then find that the quasimodular forms 
$\rho_{i,j}(f)_q$ have natural lifts from $\wM_*$ to~$\RRR$, i.e.~there exist 
linear operators $\rho_{i,j}$ from the Bloch-Okounkov ring to itself such that 
\ba  \label{eq:effofTp}
\bq{\wT_p\,f} \= \sum_{i,\,j\ge0} \bq{\rho_{i,j}(f)} \,G^{(j)}_{p+i+1}\qquad\text{for all odd $p\ge1$}\,.
\ea
In view of the formula for $\wT_p$ as a quadratic polynomial in the $Q_k$'s given
in the previous section ({equation~\eqref{eq:TpWW}}),
we can rewrite~\eqref{eq:effofTp} in 
terms of the $Q_k$-generating series $W(z)$ as
\ba \label{eq:rhoijgen}
& F(u,-u,z_1,\ldots,z_n) + \frac1{u^2}F(z_1,\ldots,z_n) \\
\= & 
-2\, \sum_{i,\,j\ge0 \atop p \geq 1\,{\rm odd}} \bq{\rho_{i,j}\bigl(W(z_1)\cdots W(z_n)\bigr)} \, 
G^{(j)}_{p+i+1}\, \frac{u^{p-1}}{(p-1)!}\,,
\ea
   {{where $W$ and $F$ are defined in~\eqref{eq:defW}
       and~\eqref{eq:defcorrelator}. }}
It is in this form that we will prove in Section~\ref{sec:tpop2}. For this purpose, however, 
we need to know explicit formulas for the maps $\rho_{i,j}$. We remark that 
finding these formulas required a combination of numerical computation, 
interpolation, and guesswork, because the $q$-bracket
from~$\RRR$ to $\wM_*$ is far from injective and~\eqref{eq:effofTp}
gives only the $q$-brackets, not the  maps $\rho_{i,j}$ themselves. 
It eventually turned out that there is a natural lift. The maps $\rho_{i,j}$
admit two quite different-looking descriptions, one as differential operators
on the ring~$\RRR$ and one via a closed formula for $\rho_{i,j}\bigl(W(z_1)\cdots W(z_n)\bigr)$
for each fixed value of~$n$, analogous to the two types of generating
functions (correlators and partition functions) used in Section~\ref{sec:growthBO}.
\par
We begin with some  preliminary observations. For compatibility with the weight we require 
that $\rho_{i,j}$ has weight $-i-2j$. We also require the initial values
\be \label{eq:veryspecialcases}
\rho_{i,0} = \delta_{i,0} \cdot \text{\rm Id}\,, \qquad \rho_{0,1}=\p_2\,,
\ee
where $\p_2$ is the derivation of degree $-2$ on $\L_*$ sending $Q_k$ to $Q_{k-2}$.  Next, for 
compatibility with~\eqref{Daction} we require that $[\rho_{i,j},Q_2] = \rho_{i,j-1}$, or equivalently, that 
\be \label{eq:rhoforQ2}
\rho_{i,j} (Q_2 f) \= Q_2\, \rho_{i,j}(f) \+ \rho_{i,j-1}(f)
\ee
for all $f\in\Lambda^*$ and $i,\,j\ge0$, where $\rho_{i,j-1}(f)=0$ if~$j=0$.  Finally, for the effect
of $\rho_{i,j}$ on powers of the generator $Q_3$ (which are the only important ones for the case of the 
principal stratum, since $f_2=\h P_2 = Q_3$) we find the simple formula 
\ba  \label{eq:rhoijforQ3}
\rho_{i,j}\Bigl(\frac{Q_3^n}{n!}\Bigr) \= 
\begin{cases} \dfrac{Q_3^{n-j}Q_{j-i}}{2^i(i+1)!\,(n-j)!} &\text{if $0\le i\le j\le n$} \\ 
      \qquad\quad\; 0 &\text{otherwise.}\end{cases} 
\ea
(with $Q_k =0$ for $k<0$) which together with~\eqref{eq:rhoforQ2} already describes the action 
of $\rho_{i,j}$ on~$\QQ[Q_1,Q_2,Q_3]\subset\RRR$. 
\par
We now observe that equation~\eqref{eq:rhoijforQ3} can be rewritten as 
\be  \label{eq:rho3diff}
\rho_{i,j} \Bigl|_{\QQ[Q_3]} \= \frac{Q_{j-i}}{2^i(i+1)!} \,\frac{\partial^j}{\partial Q_3^j} \,.
\ee
This suggests that $\rho_{i,j}$ may be expressed as a differential
operator on~$\RRR$, and further experiments suggest that it is linear in the generators~$Q_k$, 
but polynomial in the derivations $\tfrac{\partial}{\partial Q_k}$. We therefore write 
\be \label{rijdiffop}
\rho_{i,j} \= \sum_{k=0}^\infty  Q_k\, \rho_{i,j}^{(k)} \Bigl(\frac{\partial}{\partial p_1}, 
\frac{\partial}{\p p_2},\ldots  \Bigr)\,,
\ee
where to simplify later formulas we have used $Q_k$ for the linear part (including $Q_0 =1$), but
$p_\ell = \ell!\,Q_{\ell+1}$ for the derivations. 
Here the polynomial $\rho_{i,j}^{(k)}$ {{in the variables $u_\ell$}}
has weight $i+2j+k$ and degree $j$, where $u_\ell$ has degree~$1$ and 
weight~$\ell+1$ (and therefore, since $Q_k$ has weight~$k$ and degree $0$~or~$1$ depending on whether $k=0$ or $k>0$, 
that the full operator $\rho_{i,j}$ has weight~$-i-2j$ and mixed degree~$-j$ and~$1-j$). Because of this bi-homogeneity 
property, there is no loss of information if we consider only the power series $\rho^{(k)} = \sum_{i,j}\rho_{i,j}^{(k)}$.
In this language, equation~\eqref{eq:veryspecialcases} says that the constant and linear terms of $\rho^{(k)}$
are $\delta_{k,0}$ and $u_{k+1}/(k+1)!$, respectively; the differentiation property~\eqref{eq:rhoforQ2} translates 
into the property  $\rho^{(k)}(u_1,u_2,\ldots) = e^{u_1}\rho^{(k)}(0,u_2,\ldots)$; and equation~\eqref{eq:rho3diff} 
says that $\rho^{(k)} (0,u,0,0,\ldots) = 2^ku^{k-1}e^u$ for~$k>0$.
\par
To find the full formula, the key observation is that $\rho^{(k+1)}=\dd\rho^{(k)}(\bfu)$ for all~$k$, 
where $\dd$ is the derivation $\sum_{i=0}^\infty (i+1) \,u_{i+1} \p/\p u_i$ on $\QQ[[\bfu]]$. It follows that
$\rho^{(k)}=\dd^k\rho^{(0)}$ for all~$k\ge1$. We were not able to recognize the coefficients of the power 
series~$\rho^{(0)}$ directly, but the next case $\rho^{(1)}$ turned out to be easy to recognize, since if we 
made the choice
 \be\label{rho1prelim} \rho^{(1)}(\bfu) \= 2\,\exp(u_1+u_2+u_3+\cdots) \ee
and then defined the other $\rho^{(k)}$ as $\dd^{k-1}\rho^{(1)}$ (meaning in the case of~$k=0$ that we 
have to integrate once with respect to~$\dd$), then we obtained operators  having the right properties. 
To get the~$k=0$ term, we note that, since we are free to choose the operators $\rho_{i,j}$ for~$i$ odd 
in any way we want, we can replace~\eqref{rho1prelim} by its odd part
 \be\label{rho1final} \rho^{(1)}(\bfu) \= \exp(u_1+u_2+u_3+\cdots) \m \exp(u_1-u_2+u_3-\cdots)\,. \ee

This can now be integrated to give the formula $\rho^{(0)}(\bfu)= \int_0^1 e^{U(t) - U(t-1)}\,dt$, 
where  $U(t) = \sum u_n t^n$ as in~\eqref{defUt}, because from $\dd(U(t)) = U'(t)$ and $U(0)=0$ we obtain
$$ \dd\biggl(\int_0^1 e^{U(t) - U(t-1)}\,dt\biggr) \=\int_0^1 d\Bigl(e^{U(t) - U(t-1)}\Bigr) 
\= e^{U(1)} \m e^{-U(-1)}\,. $$
Now applying powers of~$\dd$ to get formulas for the higher~$\rho^{(k)}$, we are led to the
following final formulation of the experimentally obtained expression for the operators~$\rho_{i,j}$, 
which includes all of the special cases discussed above:
\begin{Thm} \label{thm:rhoijDiffOp} Define power series $\rho^{(k)}(\bfu)$ for~$k\ge0$ by the generating series
\be \label{eq:Thm14.1}
\sum_{k=0}^\infty \rho^{(k)}(\bfu)\, \frac{v^k}{k!}  \=  \int_{v}^{v+1} e^{U(t) - U(t-1)}\,dt\,, 
\ee
with $U(t) = \sum u_n t^n$ as in~\eqref{defUt}, and let $\rho_{i,j}^{(k)}$ for $i,\,j\ge0$ be the
part of~$\rho^{(k)}$ of degree~$j$ and weight~$i+2j+k$. Then equation~\eqref{eq:effofTp} holds with $\rho_{i,j}$
defined by~\eqref{rijdiffop}.
\end{Thm}
\par
This theorem can also be expressed  as a formula for the action of the maps $\rho_{i,j}$
on the  generating function $\Phi(\bfu)=\exp(p_1u_1+p_2u_2+\cdots)$ whose $q$- and $X$-brackets 
$\Phi(\bfu)_q$ and $\Phi(\bfu)_X$ were studied in Section~\ref{sec:growthBO}.
For the reasons of  weight and degree explained above, it is enough to specify the action of the
total operator $\rho = \sum_{i,j} \rho_{i,j}$ on~$\Phi$. In view of~\eqref{eq:smallpk}, this action 
is given simply in terms of a {\em first-order} differential operator in the~$u$'s
$$\rho(\Phi(\bfu)) \=  \biggl(\int_{0}^{1} e^{U(t) - U(t-1)}\,dt
\+ \sum_{\ell = 1}^\infty \frac{1}{\ell!} \frac{d^\ell}{dt^\ell} \Bigl( e^{U(t) - U(t-1)}\Bigr)\biggr|_{t=0}^{t=1}
\,\frac{\partial}{\partial u_\ell}\,\biggr)\;  \Phi(\bfu)\,.
$$
\par
\smallskip
In the rest of this section we give a proof of the following for the action
of $\rho_{i,j}$ on products $W(z_1)\cdots W(z_n)$, which is what we need for~\eqref{eq:rhoijgen}.
\par
\begin{Thm} \label{thm:rhoijCorr}
The effect of the operator $\rho_{i,j}$ defined in Theorem~\ref{thm:rhoijDiffOp}
on monomials $Q_{k_1}\cdots Q_{k_n}$ of fixed length~$n$  is given in terms of the
generating function $W(z) = \sum Q_k z^{k-1}$ by
\be \label{eq:rhoW}
\rho_{i,j}\bigl(W(z_1)\cdots W(z_n)\bigr) \= 
  \sum_{J\subset N\atop|J|=j} W(z_J)\,R_{i}(\fZ_J)\,\prod_{\nu\in N\ssm J}W(z_\nu)
\ee
where $N=\{1,\dots,n\}$, $z_J = \sum_{j\in J} z_j$, $\fZ_J = \{z_j,\, j\in J\}$,
and the polynomials $R_i(\fZ_J)$ are given by the generating function
\bes
\sum_{i=0}^\infty R_i(\fZ_J)\,t^i \= \frac{e^{tz_J} -1}{t} \,\prod_{\nu \in J}\frac{1-e^{-tz_\nu}}{t}
\= \frac{\sinh(tz_J/2)}{t/2}  \,\prod_{\nu \in J}\frac{\sinh(tz_\nu/2)}{t/2}\,.
\ees 
\end{Thm}
\par
Note that formula~\eqref{eq:rhoW} makes sense, even though $W(z)$ is a Laurent series
beginning with $1/z$, because the polynomial $R_{i}(\fZ_J)$ is divisible by~$z_J$. 
Notice also that the formula implies $ \rho_{i,j}\bigl(W(z_1)\cdots W(z_n)\bigr)=0$ if~$n<j$.
\par
\begin{proof} Write the polynomials  $R_i(\fZ_J)$ as $R_{i,j}(\fZ_J)$ ($j=|J|$) for clarity, 
and for $k \geq 0$ set 
$$ R_{i,j}^{(k)}(\fZ_J) \= z_J^{k-1}\, R_{i,j}(\fZ_J)\,,$$
which is a homogeneous polynomial of degree $i+j+k$ (even for~$k=0$, as just pointed out).
In view of the definition~\eqref{rijdiffop}, the equation to be proved is equivalent to
\be \label{eq:keinlabel}
\rho_{i,j}^{(k)}\bigl(W(z_1)\cdots W(z_n)\bigr) \= 
  \sum_{J\subset N\atop|J|=j} R_{i,j}^{(k)}(\fZ_J)\,\prod_{\nu\in N\ssm J}W(z_\nu)\,.
\ee 
\par
To prove~\eqref{eq:keinlabel} we will use the linear map $\Om_j:\Q[z_1,\dots,z_j]\to\Q[\bfu]$ 
defined in~\eqref{defOmn}. This map satisfies the general formula 
$$ \Om_j(R)\Bigl(\frac\p{\p p_1},\frac\p{\p p_2},\dots\Bigr)\bigl(W(z_1)\cdots W(z_n)\bigr)
 \= \sum_{J\subset N\atop|J|=j} R(\fZ_J)\prod_{\nu\in N\ssm J}W(z_\nu)$$
for any symmetric function~$R$ in~$j$ variables, because
$$ \frac{\p^j\bigl(W(z_1)\cdots W(z_n)\bigr)}{\p p_{\ell_1}\cdots\p p_{\ell_j}}
 \;= \sum_{1\le i_1,\dots,i_j\le n\atop \text{$i_1,\dots,i_j$ distinct}}
   \frac{z_{i_1}^{\ell_1}\cdots z_{i_j}^{\ell_j}}{\ell_1!\,\cdots\,\ell_j!}\,
    \prod_{1\le\nu\le n\atop \nu\notin\{i_1,\dots,i_j\}}W(z_\nu)\;. $$
by induction on~$j$ (since $\p W(z)/\p p_\ell=z^\ell/\ell!$). Therefore~\eqref{eq:keinlabel} 
will follow if we show that  
\be\label{viaOmega}
  \rho_{i,j}^{(k)}(u_1,u_2\dots) \= \Omega_j\bigl(R_{i,j}^{(k)}(z_1,\dots,z_j)\bigr)\,.
\ee
This is true for $k=1$ because the definition of $R_i(\fZ_J)$ via generating functions
can be expanded as 
\be\label{Rij1}   R_{i}(z_1,\dots,z_j) \= \bigl(1+(-1)^i\bigr) \, 
\sum_{n_1,\dots,n_j\ge1\atop n_1+\cdots+n_j=i+j+1} 
\frac{z_1^{n_1}\cdots z_j^{n_j}}{n_1!\,\cdots\,n_j!}     \ee
or, in view of the definition of $\Omega_j$, as 
$$ \Omega_j(R_{i})(\bfu) 
  \= \frac{1+(-1)^i}{j!}\,\sum_{n_1,\dots,n_j\ge1\atop n_1+\cdots+n_j=i+j+1} u_{n_1}\cdots u_{n_j}
\,, $$  
which agrees with $ \rho_{i,j}^{(1)}(\bfu)$ by virtue of either~\eqref{rho1prelim} 
or~\eqref{rho1final}. 
The case $k \geq 1$ then follows because $\rho_{i,j}^{(k)} = \dd \rho_{i,j}^{(k-1)}$ 
and because the map $\Omega_j$ satisfies
$$ \Omega_j\bigl(z_J\,R(\fZ_J)\bigr) \= \dd \bigl(\Omega_j R(\fZ_J)\bigr) $$
for any polynomial~$R$, as one verifies easily. The case $k=0$ follows from the
same observation together with the fact that the representation of a function of~$\bfu$
as $\Omega_j(R)$ is unique if~$R$ is assumed to be symmetric in its arguments
and divisible by their product.
\end{proof}
\par
It is perhaps amusing to note that the polynomials $R_{i,j}^{(k)}$ are virtually
impossible to recognize numerically, whether they are written in the variables~$z_\nu$
or in their elementary symmetric polynomials, unless $k=1$, which is the one
case that one cannot find experimentally, because the coefficient~$Q_1$ in~\eqref{rijdiffop} 
vanishes identically. In practice, we found expressions in terms of elementary symmetric 
polynomials, such as
\bes
 R_{i,2}^{(k)}(z_1,z_2) \= \frac{2}{(i+2)!}\,
\sum_{a+2b=i}\frac{(-1)^b\,(a+b+1)!}{(a+1)!\,(b+1)!}\,(z_1+z_2)^{k+a}(z_1z_2)^{b+1} 
\ees
for $j=2$ and a much more complicated expression for $j=3$, and then worked backwards from there.

\section{Correlators with two distinguished variables}
\label{sec:tpop2}

The information we need to calculate the effect of $T_p$ on $q$-brackets will all follow from 
the theorem below and its corollary. To formulate this theorem, we let $Z_\ell(u)$ ($\ell\ge0$) be 
the functions  defined by
$$ \frac{\th(u+v)\,\th'(0)}{\th(u)\,\th(v)} \= \frac1v\+\sum_{\ell=0}^\infty 
Z_\ell(u)\,\frac{v^\ell}{\ell!}\,.  $$
By \cite[equation~(15)]{zagier_periodJacobi}, these functions are given by
\ba \label{Zag01}
Z_0(u) & \= \frac{\th'(u)}{\th(u)} \=\zeta(u) \= \frac1u \m 2\sum_{r \geq 0} G_{r+1} \,\frac{u^r}{r!}\,, 
\\ 
Z_\ell(u) &  \= -\, 2\,\sum_{r \geq 0} \,G^{(\min(r,\ell))}_{|r-\ell|+1} \,\frac{u^r}{r!}\qquad (\ell \geq 1)\,. 
\ea

\begin{Thm} \label{thm:mainconjuv} 
A Bloch-Okounkov correlator involving two distinguished variables $u$ and $v$
can be written as a linear combination of products of a correlator involving
only $u+v$ and a function $Z_\ell$ involving only one of the variables $u$ and~$v$.
More precisely, we have 
\ba \label{eq:mainconj}
 F(u,v,\fZ_N) \= \sum_{J\sse N} F(u+v+z_J,\,\fZ_{J^c})\,
\sum_{I\sse J} (-1)^{|J\ssm I|} \bigl(Z_{|J|}(u+z_I)\+ Z_{|J|}(v+z_I)\bigr)\,.
\ea
\end{Thm}
\par
This will be proved at the end of the section. 
\par
\begin{Cor}\label{cor:uminusu}
A correlator with two variables $u$ and $-u$ that
add up to zero can be expressed  in terms of the nearly-elliptic functions $Z_j$ 
and correlators not involving~$u$ by the formula
\be  \label{eq:uminusu}
 F(u,-u,\fZ_N) \= \sum_{J\sse N}  F(z_J,\,\fZ_{J^c})\,M(u,\fZ_J)\,,
\ee
where $M(u,\fZ_J)$ is defined as $\zeta'(u)$ if $J =\emptyset$ and by
\be \label{eq:Mdef}
M(u, \fZ_J) = 
\sum_{I\sse J} (-1)^{|J\ssm I|} \bigl(Z_{|J|}(z_I+u)\+ Z_{|J|}(z_I-u)\bigr)
\ee
if $|J| \geq 1$.
\end{Cor}
\par
\begin{proof}
{
The terms with $J \neq \emptyset$ in~\eqref{eq:uminusu}  are obtained
from~\eqref{eq:mainconj} by specializing to~$v=-u$. For the
$J = \emptyset$ term, we use Theorem~\ref{thm:BOviaAxioms}~(iii) to obtain
\bes 
\lim_{v \to -u} F(u,v,\fZ_N) \=  F(\fZ_N)\,
\lim_{\varepsilon\to 0} \frac{Z_0(u) -Z_0(u-\varepsilon)}{\varepsilon}
\=  F(\fZ_N)\,M(u)\,, 
\ees
because $Z_0'(u) = \zeta'(u) = M(u)$.}
\end{proof}
\par
We remark that for the following proof of  Theorem~\ref{thm:rhoijDiffOp}
we only need
this corollary, but its statement seems not to 
allow an inductive proof (since after applying the recursion~\eqref{eq:BORec}
we are left
with correlators involving the variable $u$ just once), so that we are forced
to show the more general result~\eqref{eq:mainconj}.
\par
\begin{proof}[Proof of Theorem~\ref{thm:rhoijDiffOp}] 
In view of Theorem~\ref{thm:rhoijCorr} we have to prove~\eqref{eq:effofTp} with $\rho_{i,j}$ defined
by~\eqref{eq:rhoW}. Applying the $q$-bracket to the latter
and using the definition of correlators we obtain 
\bes
\bq{\rho_{i,j} \bigl( W(z_1)\cdots W(z_n)\bigr)} \= [t^{i+j+1}]
  \sum_{J \subseteq N \atop |J| = j}
(1-e^{-tz_J})\,\cdot \,\prod_{\nu \in J} (e^{tz_\nu}-1)\,\cdot\,
 F(z_J,\,\fZ_{J^c}) \,.
\ees
Substituting this and~\eqref{eq:uminusu} into~\eqref{eq:rhoijgen}, we see that the formula
to be proved reduces to the two identities
\bes
M(u) + \frac{1}{u^2} \= -2 \sum_{p \geq 1\,{\rm odd}} G_{p+1}\,\frac{u^{p-1}}{(p-1)!} 
\ees
and
\be \label{eq:M2ndcase}
M(u, \fZ_J) \= -2 \sum_{i \geq 0 \atop i \, \text{even}} \sum_{p \geq 1 \atop p \, \text{odd}} 
G^{(j)}_{p+i+1} \,\cdot\, \frac{u^{p-1}}{(p-1)!} \,\cdot\,
[t^{i+1}] (1-e^{-tz_J})\left(\prod_{\nu \in J} \frac{e^{tz_\nu} -1}{t} \right)
\ee
for $|J| = j \geq 1$. The first of these follows from~\eqref{Zag01} and
the second follows by noting that 
\bes
 [t^{i+1}] \left(\prod_{\nu \in J} \frac{e^{tz_\nu} -1}{t} \right) = 
[t^{i+j+1}] \sum_{I \subseteq J} (-1)^{|J|-|I|} e^{tz_{I}} = \sum_{I \subseteq J} (-1)^{|J|-|I|} 
\frac{z_I^{i+j+1}}{(i+j+1)!}\,,
\ees
\bes
 [t^{i+1}] (-e^{-tz_J})\left(\prod_{\nu \in J} \frac{e^{tz_\nu} -1}{t} \right) = 
 \sum_{I \subseteq J} (-1)^{|J|-|I|+i} 
\frac{z_I^{i+j+1}}{(i+j+1)!}, 
\ees
and then calculating 
\bas
\text{RHS of~\eqref{eq:M2ndcase}}
 &\= -4\, \sum_{k \geq 2 \atop k \, \text{even}} G_{k}^{(j)} \sum_{I \subseteq J} (-1)^{j-|I|}
\sum_{i+p = k-1 \atop i,p \geq 0, \,\, i \, \text{even}} \frac{z_I^{i+j+1}}{(i+j+1)!} \frac{u^{p-1}}{(p-1)!} \\
&\= -2\,  \sum_{I \subseteq J} (-1)^{j-|I|}  \sum_{k \geq 2 \atop k \, \text{even}}
 G_{k}^{(j)} \frac{(z_I +u)^{k+j-1} + (z_I -u)^{k+j-1}}{(k+j-1)!} \,.
\eas
Now the claim follows from~\eqref{eq:Mdef} and~\eqref{Zag01} together with the fact that 
$$ \sum_{I\subseteq J} (-1)^{|I|} P(z_I) = 0 $$ 
for any polynomial $P$ of degree smaller than $|J| = j$. 
\end{proof}
\par
\begin{proof} [Proof of Theorem~\ref{thm:mainconjuv}]
We define $\wF(\fZ_N) = \frac{\th(z_N)}{\th'(0)}\,F(\fZ_N)$.
We can change~$F$ to~$\wF$ everywhere in the theorem without affecting the truth 
of the statement, since the sum of the arguments of~$F$ is the same in all terms. 
We denote by $G(u,v,\fZ_N)$ the right hand side of~\eqref{eq:mainconj}, so that we have to 
show that 
$G(u,v,\fZ_N)=F(u,v,\fZ_N)$, or equivalently 
that $\wG(u,v,\fZ_N):=\th(u+v+z_N)G(u,v,\fZ_N)/\th'(0)$ $=\wF(u,v,\fZ_N)$.
We will do this by comparing poles and elliptic transformation properties.
\par
It is easy to see that the residue at $z_n=0$ of the function
$F(u,v,z_1,\dots,z_n)$ equals $F(u,v,z_1,\dots,z_{n-1})$ and that the residue 
at $z_n=0$ of $G(u,v,z_1,\dots,z_n)$ equals $G(u,v,z_1,\dots,z_{n-1})$, 
so by induction on~$|N|$ the difference $\wF-\wG$ has no poles at $z_n=0$. 
 For the poles at~$u=0$ the calculation is even
easier: the residue of $F(u,v,z_1,\dots,z_n)$ at~$u=0$ is $F(v,z_1,\dots,z_n)$, 
and the residue of $G(u,v,z_1,\dots,z_n)$ at~$u=0$ is easily seen to have 
the same value. In fact, only the term~$J=\es$ in the definition of $G$
contributes, since $Z_j$ is holomorphic for~$j\ne0$.  Finally, 
if $u+v+z_J=0$ for some $J\sse N$ (which we can assume is unique, 
since we can assume that all the variables are generic), then the left hand
side of~\eqref{eq:mainconj} has no pole and the right hand side is also 
non-singular because the terms for~$I$ and $I^c=J\ssm I$ cancel
since $(-1)^{|I^c|}Z_{|J|}(v+z_I)=(-1)^{|J\ssm I|}Z_{|J|}(-u-z_{I^c}) = -(-1)^{|I|}Z_{|J|}(u+z_{I^c})$. 
\par
For the elliptic transformation properties, we recall that Bloch and 
Okounkov have shown in \cite{blochokounkov} the elliptic transformation law
\be \label{eq:diffeqF}
\wF(z_1+\t,z_2,\dots,z_n) \= \sum_{1\in J\sse N}(-1)^{|J|-1}\wF(z_J,\fZ_{Jc})\,.
\ee
We introduce the difference operator 
$\D_x$ which associates to any function $f(x)$, possibly depending on other 
variables, the difference $(\D_xf)(x)=f(x+\t)-f(x)$.
Then~\eqref{eq:diffeqF} says that $\D_{z_1}\wF$ equals the right hand side 
of~\eqref{eq:diffeqF} with the term~$J=\{1\}$ omitted, while for~$Z_\ell$ we have
\be \label{eq:DuZ}
\D_uZ_\ell(u) \= \sum_{k=1}^\ell(-1)^k\binom\ell k\,Z_{\ell-k}(u) \+ \frac{(-1)^{\ell+1}}{(\ell+1)!}\;.
\ee
\par
Since both sides in~\eqref{eq:mainconj} are symmetric in the variables 
$z_1,\ldots,z_n$ and also in $u$ and $v$, is suffices to show that the 
differences of 
$\wF$ and $\wG$ with respect to (say)~$z_1$ and $u$ agree, in which case
$\wF-\wG$ is periodic and holomorphic in all variables, hence is a constant.
\par
We start with the variable~$u$. We have 
$$ \D_u\wF(u,v,\fZ_N) \= \sum_{\es\ne J\sse N}(-1)^{|J|}\wF(u+z_J,v,\fZ_{J^c}) \m 
  \sum_{H\sse N}(-1)^{|H|}\wF(u+v+z_H,\fZ_{H^c})\,. $$
We can compute the first summand using the identity we claim, 
which is true by induction on~$|N|$, since $|J^c|<|N|$ for $J\ne\es$. This gives
\bas 
\sum_{\es\ne J\sse N}(-1)^{|J|}\wF(u+z_J,v,\fZ_{J^c}) & \= \!\!\!\sum_{\es\ne J\sse K \sse H \sse N} \!\!(-1)^{|H| + |K\ssm J|} \wF(u+v+z_H,\,\fZ_{H^c})\cdot \\ 
& \cdot \bigl(Z_{|H\ssm J|}(u+z_K)\+ Z_{|H\ssm J|}(v+z_K-z_J)\bigr)\,. 
\eas
Combining the terms we obtain 
\bes
\D_u\wF(u,v,\fZ_N)=\sum\limits_{H\sse N}(-1)^{|H|}\,\a(H)\,\wF(u+v+z_H,\,\fZ_{H^c})\,,
\ees 
with
\bas
\a(H) & \= (-1) \+ \sum_{{ I\sse H} \atop {|H\ssm I|\le\l<|H|}} (-1)^{|H|-\l + |I|}
\binom{|I|}{|H|-\l}\,Z_\l(u+z_I) \\
& \+ \sum_{I\sse H \atop |I|\le\l<|H|} (-1)^{|I|}\binom{|H|-|I|}{\l-|I|}\,Z_\l(v+z_I)\,. 
\eas
On the other hand, using~\eqref{eq:diffeqF} and~\eqref{eq:DuZ}, we obtain 
\bes
\D_u\wG(u,v,\fZ_N)=\sum\limits_{H\sse N}(-1)^{|H|}\b(H)\wF(u+v+z_H,\,\fZ_{H^c})
\ees
with 
\bas \b(H) &\= \sum_{I\sse J\sse H}(-1)^{|I|}
    \Biggl(\sum_{k=0}^{|J|} (-1)^k\binom{|J|}k\,Z_{|J|-k}(u+z_I)
\+Z_{|J|}(v+z_I)\Biggr)  \\ 
&\qquad \+ \sum_{I\sse J\sse H}(-1)^{|I|} \frac{(-1)^{|J|+1}}{(|J|+1)!}  \\
   &\qquad-\,\sum_{I\sse H}(-1)^{|I|}\bigl(Z_{|H|}(u+z_I)\+Z_{|H|}(v+z_I)\bigr) \\
   &\= \sum_{ I\sse H\atop 0\le\l < |H|}(-1)^{|I|}\Biggl[\sum_{|I|\le n\le|H|}
        (-1)^{n-\l}\binom n\l\,\binom{|H|-|I|}{|H|-n}\Biggr]\,Z_\l(u+z_I) \\
& \qquad\+ \sum_{ I\sse H\atop |I|\le\l<|H|} (-1)^{|I|}\binom{|H|-|I|}{\l-|I|}\,Z_\l(v+z_I) 
 \\
& \qquad\+\sum_{J\sse H} \frac{(-1)^{|J|+1}}{(|J|+1)!} 
\sum_{I\sse J}(-1)^{|I|} \;.  
\eas
Notice that the expression in the square brackets equals 
\be
\label{eq:binomialtrick}
(-1)^{|H|-\l}\binom{|I|}{|H|-\l} \= \bigl[x^{|H|-\l}\bigr] 
\Bigl( \frac{1}{(1+x)^{\l+1}} \,
(1+x)^{|H|-|I|} \Bigr)\,,
\ee
which is zero when $\l < | H \ssm I |$, and the summation of the constant terms is equal to $-1$ with the only non-zero contribution coming from $J = \emptyset$. 
It follows that the formulas for~$\a(H)$ and~$\b(H)$ agree. 
\par
\smallskip
The difference with respect to the variable $z_1$ behaves similarly. We have 
\begin{flalign}
&\phantom{\=} \Delta_{z_1} \wF(z_1, z_2,\ldots, z_n,u,v) & \nonumber\\
& \=  \sum_{\{1\} \subsetneq J \subseteq N} (-1)^{|J| - 1} 
\wF(z_J,u,v, \fZ_{J^c})  \+ \sum_{1 \in J \subseteq N} (-1)^{|J|+1}
\wF(z_J + u+v, \fZ_{J^c})  & \nonumber\\
& \phantom{\=}  \+ \sum_{1 \in J \subseteq N} (-1)^{|J|} \bigl(
\wF(z_J+u, v, \fZ_{J^c}) \+
\wF(z_J+v, u, \fZ_{J^c}) \bigr) & \nonumber \\
& \= \sum_{\{1\} \subsetneq J \subseteq N}  \sum_{I \subseteq J_2 \subseteq J^c} (-1)^{|J|+|I| +|J_2|+1}  
\wF(z_{J_2} + u + v, z_J, \fZ_{(J\cup J_2)^c}) \, \cdot & \nonumber\\
& \phantom{\sum_{\{1\} \subsetneq J \subseteq N}  \sum_{I \subseteq J_2 \subseteq J^c} (-1)^{|J|+|I| +|J_2|+1} } \cdot\, 
\bigl( Z_{|J_2|}(z_I + u) + Z_{|J_2|}(z_I + v)\bigr) & \nonumber\\
& \phantom{\=}  \+ \sum_{1 \in H \subseteq N}
(-1)^{|H|} \wF(z_H + u + v, \fZ_{H^c})\, \alpha(H) & \nonumber 
\end{flalign}
with 
\begin{flalign}
 \alpha(H) 
& \=  \sum_{1\not\in I \subseteq K \subsetneq H } (-1)^{|I|} 
\bigl( Z_{|K|}(z_I + u) + Z_{|K|}(z_I + v) \bigr) & \nonumber\\
& \+ (-1)  & \nonumber\\
& \+ \sum_{I,J \subseteq  H,\,\, I \cap J = \emptyset \atop 1 \in J\cup I,\, J\neq \emptyset} (-1)^{|I|} 
\bigl( Z_{|H \ssm J|}(z_I + z_J + u) + Z_{|H \ssm J|}(z_I + z_J + v) \bigr)
& \nonumber \\
& \=  \sum_{1 \not\in I  \subseteq H} \sum_{\ell = |I|}^{|H|-1} 
\alpha_0(\ell,I,H) \bigl(Z_{\ell}(z_I + u) 
+ Z_{\ell}(z_I + v)\bigr) & \nonumber \\
& \+ \sum_{1 \in I \subseteq H} \sum_{\ell = |H|-|I|}^{|H|-1} 
\alpha_1(\ell,I,H) \bigl(Z_{\ell}(z_I + u) 
+ Z_{\ell}(z_I + v)\bigr) & \nonumber \\
& \+ (-1) & \nonumber
\end{flalign}
where, taking $I \cup J$ as the new $I$ and $\ell = |H \ssm J|$ for 
the transformation from the second summand to the fourth,
$$ \alpha_0(\ell,I,H) = (-1)^{|I|} \, \binom{|H| - |I|}{\ell - |I|}$$
and 
$$ \alpha_1(\ell,I,H) = (-1)^{|I|+|H|-\ell} \, \binom{|I|}{|H|-\ell}\,.$$
\par
Computing  $\Delta_{z_1}$ of the right hand side we distinguish the
case $1 \in J^c$ which is the only one where terms of the form  
$\wF(z_{J} + u + v, z_{J_2}, \fZ_{(J\cup J_2)^c})$ appear, the case $1 \in I$
and the remaining case $ 1 \in J \ssm I$. After simplifying, we obtain that 
\begin{flalign}
& \phantom{\= i} \Delta_{z_1} \wG(z_1, z_2,\ldots, z_n,u,v) & \nonumber \\
& \= \sum_{I \subseteq J \subseteq N \atop 1 \not\in J\,} \sum_{\{ 1\} \subsetneq J_2 \subseteq J^c} 
(-1)^{|J|+|I|+|J_2|+1} \wF(z_{J} + u + v, z_{J_2}, \fZ_{(J\cup J_2)^c}) \,\cdot 
& \nonumber\\
& \phantom{\= \sum_{I \subseteq J \subseteq N, \, 1 \not\in J\,} \sum_{\{ 1\} \subsetneq  J_2 \subseteq J^c} 
(-1)^{|J|+|I|+|J_2|+1}} \cdot \bigl( Z_{|J|}(z_I + u) + Z_{|J|}(z_I + v)\bigr) & \nonumber\\
& \phantom{\=}  \+ \sum_{1 \in H \subseteq N}
(-1)^{|H|} \wF(z_H + u + v, \fZ_{H^c}) \, \beta(H) & \nonumber 
\end{flalign}
with 
\begin{flalign}
\beta(H) & \= \sum_{1\not\in I \subseteq J \subseteq H} (-1)^{|I|} 
\bigl( Z_{|J|}(z_I + u) + Z_{|J|}(z_I + v)\bigr) & \nonumber \\
& \+ \sum_{I\subseteq H} (-1)^{|I|+1} \bigl( Z_{|H|}(z_I + u) + Z_{|H|}(z_I + v)\bigr) & \nonumber \\
& \+ \sum_{1 \in I \subseteq J \subseteq H} (-1)^{|I|} 
\Biggl[ \sum_{k=1}^{|J|} (-1)^k {|J|\choose k}
    Z_{|J|-k}(z_I + u) + Z_{|J|-k}(z_I + v) + 2\cdot \frac{(-1)^{|J|+1}}{(|J|+1)!} \Biggr]
& \nonumber \\
& \+ \sum_{1 \in I \subseteq J \subseteq H} (-1)^{|I|} 
\bigl(Z_{|J|}(z_I + u) + Z_{|J|}(z_I + v)\bigr)\,.
& \nonumber 
\end{flalign}
Notice that the summation of $2\cdot \frac{(-1)^{|I| + |J|+1}}{(|J|+1)!}$ ranging over $1\in I \subseteq J \subseteq H$ 
equals $-1$, where the only non-zero contribution comes from $J = \{1\}$ (for fixed~$J$ and for varying~$I$). In addition, set $\ell = |J| - k$ in the summation of the terms with subscript $|J|-k$, and apply~\eqref{eq:binomialtrick} to simplify. We conclude that
\begin{flalign}
\beta(H) & \= 
 \sum_{1 \not\in I \subseteq H} \sum_{\ell = |I|}^{|H|-1} 
\beta_0(\ell,I,H) \bigl(Z_{\ell}(z_I + u) 
+ Z_{\ell}(z_I + v)\bigr) & \nonumber \\
& \+ \sum_{1 \in I \subseteq H} \sum_{\ell = |H|-|I|}^{|H|-1} 
\beta_1(\ell,I,H) \bigl(Z_{\ell}(z_I + u) 
+ Z_{\ell}(z_I + v)\bigr) & \nonumber \\
& \+ (-1) & \nonumber
\end{flalign}
where
$$ \beta_0(\ell,I,H) = (-1)^{|I|} \binom{|H|-|I|}{\ell - |I|}$$
and 
$$ \beta_1(\ell,I,H) = (-1)^{|H|+|I|+\ell} \binom{|I|}{|H|-\ell}.$$
We see that $\alpha_0 = \beta_0$ and $\alpha_1 = \beta_1$, hence $\alpha(H) = \beta(H)$. 
\end{proof}
\par

\section{Applications to $T_{-1}$ and 
to Siegel-Veech constants} \label{sec:applSV}

We saw in Proposition~\ref{prop:qofTp} that $\bq{T_{-1}}$ is not a quasimodular form, but
the $\tau$-derivative is. In this section we use the formula~\eqref{eq:effofTp}
on the effect of $T_p$ to deduce that a certain linear combination of
brackets involving $T_{-1}$ is indeed quasimodular. We apply this to prove
the quasimodularity of area Siegel-Veech constants. 
\par
\begin{Thm} \label{thm:QMofbqs}
For all $f\in\L_k$ the modified $q$-bracket 
\be \label{eq:modsbqs}
\sbqs f \= \sbq{T_{-1}\,f} \m \sbq{T_{-1}}\,\sbq f 
\m \frac1{24}\,\sbq{\partial_2(f)} 
\ee
is a quasimodular form of weight $k$. More precisely, we have
\bes  
\bqs f \= \sum_{i\ge2,\,j\ge0} G_i^{(j)}\,\bq{\rs_{i,j}(f)}\,, 
\ees
where $\rs_{i,j}\,=\,\rho_{i,j}+\delta_{i,2}\,\rho_{0,j+1}\,$.  
\end{Thm}
\par
\begin{proof}
{
From~\eqref{eq:effofTp} we get, for $p>0$ odd, 
$$  \bq{T_p\,f} \=  \sum_{i\ge0,\,j\ge1} 
\bq{\rho_{i,j}(f)}\,\Bigl(G^{(j)}_{p+i+1} \,-\,\delta_{i+j,0}\frac{\z(-p)}2\Bigr)\,.
$$
Using $\rho_{i,0}=0$ for $i>0$ and
$T_p(\lambda) = \sum_m m^{p-1} N_m(\lambda)$,
where $\l\mapsto N_m(\l)$ is the hook-length counting function
of Section~\ref{sec:elemTp}), we can rewrite this as
\bes
\sum_{m>0} m^{p-1} \bq{N_m f} \= \sum_{i,j\ge0 \atop \text{$i$ even}} 
\bq{\rho_{i,j}(f)} \,\sum_{n=1}^\infty n^j \sigma_{p+i}(n)\,q^n\,.
\ees
Since a function of the form $p\mapsto\sum_m a_mm^p$ on $\{p \in \NN,\,\text{$p$
odd}\}$ determines all the~$a_m$ uniquely, we deduce
\bas
\bq{N_m f} &\= \sum_{i,\,j\ge0 \atop \text{$i$ even}}
\biggl(\,\sum_{r>0} m^{i+1}(mr)^jq^{mr}\biggr)\, \bq{\rho_{i,j}(f)} \,.
\eas
Therefore, 
\begin{flalign} 
\bq{T_{-1}\,f} &\= \sum_{m=1}^\infty m^{-2}\, \bq{N_m f} 
\= \sum_{i,\,j\ge0  \atop \text{$i$ even}}
\biggl(\,\sum_{m,\,r>0} m^{i-1}(mr)^jq^{mr}\biggr)\,\bq{\rho_{i,j}(f)} 
& \nonumber \\
&\= \sum_{i\ge2,\,j\ge0} G^{(j)}_i\,\bq{\rho_{i,j}(f)}
\+ \Biggl(\sum_{m,r> 0}\frac{q^{mr}}{m}\Biggr)\,\bq{\,f}
\nonumber\\
& \phantom{\=} \+ \sum_{j\ge1}\bigl(G^{(j-1)}_2
+\frac1{24}\delta_{j,1}\bigr)\,\bq{\rho_{0,j}(f)}
\,. & \nonumber
\end{flalign}
In view of the formula $\rho_{0,1}=\p_2$ and
Proposition~\ref{prop:qofTp} this gives the claim. }
\end{proof}
\par
We can now prove Theorem~\ref{thm:prstr_minus1} of Part~I. Recall that 
Eskin and Okounkov have shown (\cite{eo}) the quasimodularity of the generating function of Hurwitz numbers
\be \label{eq:QMHur}
N'(\Pi) \in \wM_{\leq \wgt(\Hmu)}, \quad N^0(\Pi) \in \wM_{\leq \wgt(\Hmu)}
\ee
where $\wgt(\Hmu) = \sum \wgt(\mu_i)$ for $\Hmu = (\mu_1, \ldots, \mu_n)$ and the weight of
a $b$-cycle is defined to be $b+1$. This is a consequence
of the Burnside formula~\eqref{eq:Burnside} and the Bloch-Okounkov
Theorem~\ref{thm:bo} using the formula~\eqref{eq:NNN}
and the fact (\cite{KerOls}) that the character functions $f_k$ 
in~\eqref{eq:defssf} are shifted symmetric functions. (We give
examples in Section~\ref{sec:HurtoStrata}.) The formula~\eqref{eq:NprimeNN} provides
the passage to the connected case.  Moreover, since $f_2 = Q_3$ is
a shifted symmetric function of pure weight three,  
the modular forms $N'(\Tr^n)$ and $N^0(\Tr^n)$ are pure of 
weight $3n$. 
\par
\begin{proof}[Proof of Theorem~\ref{thm:prstr_minus1}] We start with the
case $\mu_i = \Tr$ for all $i$, that is $\Hmu = \Tr^n$. 
Combining the passage from counting all covers to counting covers without
unramified components in~\eqref{eq:cmucmup}, the Siegel-Veech analog 
of the Burnside formula~\eqref{eq:cpclosed}, and Corollary~\ref{cor:SumSpTp}, 
we deduce that 
\be \label{eq:cpviaTp}
c_p'(\Tr^n) = \bq{T_p f_2^n} - \bq{f_2^n} \bq{T_p} 
\ee
and the preceding remarks together with~\eqref{eq:TpWW}
imply that for $p$ positive $c'_p(\Tr^n)$ is quasimodular of
weight $3n + p+1 =  6g-6+p+1$. Moreover,  $\partial_2(Q_3^n) = 0$
implies that 
\be \label{eq:cm1viastar}
c_{-1}'(\Hmu) \= \bqs{f_2^n}
\ee
and this is a quasimodular form of weight $3n = 6g-6$ by 
Theorem~\ref{thm:QMofbqs}.
\par
For all odd $p \geq -1$ we can use~\eqref{eq:cmupcmu0} to
recursively conclude that the generating functions $c_{p}^0(\Tr^n)$ for
counting connected covers with Siegel-Veech weight are also quasimodular
forms of weight $3n+p+1$.
\par
In the general case, the same argument works, except that now $f_k$ is
not pure, but a linear combination of shifted symmetric functions of weight
$\leq k+1$. Since $\partial_2$ also decreases weight, we conclude that
$c_p^0(\mu_1,\ldots,\mu_n)$  is quasimodular of mixed weight $\leq p+1+
\sum_{i=1}^n (|\mu_i| +1)$.
\end{proof}
\newpage

%% file: Part4_BOSV.tex
\part*{Part IV: Volumes and Siegel-Veech 
constants for large genus} 

We return here to the geometric set-up around Siegel-Veech 
constants in Part I. Using all of the results of Parts I$-$III, 
we find closed formulas for both the Masur-Veech volumes and 
the Siegel-Veech constants of the principal stratum in terms
of generating functions related to Hurwitz zeta functions.
\par
While the focus in Part~I was on Hurwitz spaces, 
we show in Section~\ref{sec:HurtoStrata} that the large degree 
asymptotics also provide the Siegel-Veech constants for strata.
In addition, this section contains a short digression on interpreting
the non-varying phenomenon for the sum of Lyapunov exponents in terms of
our quasimodularity results. 
\par
Finally, Sections~\ref{sec:AsHur} and~\ref{sec:AsyMVSV} prove the 
Eskin-Zorich conjecture for the
large genus asymptotics of Masur-Veech volumes and Siegel-Veech constants
for the case of principal stratum.

\section{From Hurwitz spaces to strata}  
\label{sec:HurtoStrata}

We have worked out in Part~I a combinatorial formula
for Siegel-Veech constants and proved in Part~III the quasimodularity
of their generating functions. We now show that we can determine
the area Siegel-Veech constants of strata (i.e.\ of any
generic flat surface in a stratum $\omoduli[g](m_1,\ldots,m_n)$)
as limits of Siegel-Veech constants of Hurwitz spaces.
In this context, the non-varying phenomenon for sums of Lyapunov exponents
(or, equivalently, for area Siegel-Veech constants) discovered
in \cite{chenmoeller} turns out to be just a proportionality of two quasimodular forms.
We will discuss this in the second part of this section.
\par
To determine Siegel-Veech constants of strata we use Hurwitz spaces 
with ramification profile $\Hmu = (\mu_1,\ldots,\mu_n)$ where each $\mu_i$
is an $m_i$-cycle.
\par
\begin{Prop} \label{prop:HurStrLim}
For any stratum $\omoduli[g](m_1,\ldots,m_n)$ 
the normalized combinatorial area Siegel-Veech constants
converge to the area Siegel-Veech constant 
of a generic surface $(X,\omega)$ in that stratum, i.e.\
\be \label{eq:HurStrL}
\frac{3}{\pi^2} \frac{\sum_{d=1}^D c_{-1}^0(d,\Hmu)}{\sum_{d=1}^D N^0_d(\Hmu)}
\,\to\, c_\area(X,\omega) \quad \text{for} \quad D \to \infty.
\ee
\end{Prop}
\par
The proof is an adaptation of the argument of Eskin written
for the case of arithmetic \Teichmuller\ curves in \cite[Appendix]{chenrigid}.
\par
\begin{proof} We abbreviate ${\bf m} = (m_1,\ldots,m_n)$ and 
${\rm vol} = \nu_{\rm str}(\oamoduli[g]({\bf m}))$. We 
let $V = V(X,\omega) \subset \RR^2$ be the weighted subset of
holonomy vectors of core curves of cylinders on $(X,\omega)$ with multiplicity 
equal to the area of each cylinder. We denote by $\widehat{f}$
the Siegel-Veech transform (cf.~\eqref{eq:SVtransform}) of a compactly 
supported function
$f: \RR^2 \to \RR$ with respect to $V$. Then~\eqref{eq:SVbasic} applied to
the stratum and the Hurwitz spaces gives
\be \label{eq:SVforHur}
\frac{1}{{\rm vol}} 
\int_{\omoduli[g]({\bf m})} \widehat{f}(X)\, d\nu_{\rm str} (X)
\= c_\area(\oamoduli[g]({\bf m})) \int_{\RR^2} f\,dxdy
\ee
and
\be\label{eq:SVforStr}
\frac{1}{\nu_{1}(\Omega_1 H_d(\Hmu))} 
\int_{\Omega_1 H_d(\Hmu)} \widehat{f}(X)\, d\nu_{1} (X)
\= c_\area(H_d(\Hmu)) \int_{\RR^2} f\,dxdy
\ee
for any fixed generic flat surface $(X_d,\omega_d)$ in 
the Hurwitz space $H_d(\Hmu)$. 
\par
The key step is that the uniform density of
rational lattice points in period coordinates implies by the
arguments in \cite[Section~3.2]{eo} that for
every pointed elliptic curve $E$ in $\moduli[1,n]$
\be \label{eq:limHurdense}
\lim_{D \to \infty} \frac{\sum_{d=1}^D\,\widehat{f}_d(\Hmu)}{\sum_{d=1}^D N^0_d(\Hmu)} 
\= \frac{1}{{\rm vol}} 
\int_{\oamoduli[g]({\bf m})} \widehat{f}(X)\, d\nu_{\rm str} (X)
\ee
where $$\widehat{f}_d(\Hmu) \= \sum_{{\pi: X \to E}
\atop {\in \Cov^0_d(\Hmu)/\sim}} \widehat{f}(X)$$ and
where $\pi: X \to E$ is the covering topologically specified 
by the equivalence class of a Hurwitz tuple in $\Cov^0_d(\Hmu)$ (i.e. up to simultaneous conjugation on the Hurwitz tuples). 
We use the claim and~\eqref{eq:SVforHur} together with an extra averaging over 
$\Omega_1\moduli[1,n]$ and interchange limit and integral
by dominated convergence to obtain that 
\begin{flalign*}
\frac{1}{{\rm vol}} \int_{\oamoduli[g]({\bf m})} \!\!\!\!\!\!\! \!\!\!
\widehat{f}(X)\, d\nu_{\rm str} (X)
& \= \lim_{d \to \infty} \frac{1}{\ol{\nu}_1(\Omega_1\moduli[1,n])}
\int_{\Omega_1\moduli[1,n]} \!\! \frac{1}{N^0_d(\Hmu)} \!\!\! \sum_{{\pi: X \to E}
\atop {\in \Cov^0_d(\Hmu)/\sim}} \!\!\!\widehat{f}(X) \,d\ol{\nu}_1(X) &\\  
& \= \lim_{d \to \infty} \frac{1}{\nu_1(\Omega_1 H_d(\Hmu))}
\int_{\Omega_1 H_d(\Hmu)} \widehat{f}(X) \,d\nu_1(X) \\ 
& \= \lim_{d \to \infty} c_\area(H_d(\Hmu)) \int_{\RR^2} f\,dxdy\,. &
\end{flalign*}
The proposition now follows from Theorem~\ref{thm:SV} by comparing 
the preceding equality to~\eqref{eq:SVforStr}. 
\end{proof}
\par
In the remainder of this section we relate the non-varying phenomenon
for strata in low genus and the quasimodularity theorem for Siegel-Veech constants.
In \cite{chenmoeller} we called a connected component of a
stratum $\omoduli(\bfm)$ {\em non-varying} if for every Teichm\"uller
curve $C$ generated by a Veech surface in that component the
sum of Lyapunov exponents for $C$ is the same as the sum of Lyapunov
exponents for the whole component. Since the main theorem of
\cite{ekz}, as recalled in~\eqref{eq:ekzmain}, holds for all 
$\SL\RR$-invariant submanifolds and since $\kappa$ depends on the
stratum only, we may replace ``sum of Lyapunov exponents'' by
``area Siegel-Veech constant'' in the definition of non-varying.
\par
The non-varying phenomenon holds for a number of connected components of strata
in low genus and was discovered experimentally by Kontsevich and Zorich.
It was first proved in \cite{chenmoeller} by exhibiting geometrically defined divisors in the
moduli spaces of (pointed) stable curves that are disjoint from
Teichm\"uller curves in a given stratum. Later on another proof 
was given by Yu and Zuo in \cite{yuzuo} using filtrations of the Hodge bundle over 
Teichm\"uller curves. 
\par
For a connected stratum $\omoduli(\bfm)$ non-varying implies that
the quasimodular forms $N^0(\mu_1,\ldots,\mu_n)$ and $c_{-1}^0(\mu_1,\ldots,\mu_n)$,
where $\mu_i$ is a cycle of length $m_i + 1$, are simply {\em proportional}. In fact, the Hurwitz
spaces $H_d(\Hmu)$ considered in this paper contain a dense set of
Teichm\"uller curves and the argument of Proposition~\ref{prop:HurStrLim}
in the form of \cite{chenrigid} implies the claim. Conversely, we expect that 
the non-varying phenomenon restricted to the class of arithmetic Teichm\"uller curves 
can be shown by extending the quasimodularity theorem to Hurwitz spaces
with more than one ramification point in the fiber over a branch point. Note that the case of non-arithmetic
Teichm\"uller curves is not in the scope of the discussion here, because they do not arise from a covering construction. 
\par 
We present examples for all strata in genus two and three.
To compute volumes and Siegel-Veech constants using the formulas in
the preceding sections, we first need to express the functions $f_i$
defined in~\eqref{eq:defssf} as polynomials in our standard generators of
the ring of shifted symmetric functions. This goes back to work
of Kerov and Olshanski (\cite{KerOls}). Explicit formulas have been compiled
e.g.\ by Lassalle (\cite{lassalle}). The first few of these functions are
\begin{flalign}
f_1  & \= p_{1}
 + \frac{1}{24}  \qquad \qquad \qquad \qquad \qquad \,
f_2   \=  \frac{1}{2} p_{2}  &\nonumber \\
f_3  & \= \frac{1}{3} p_{3} - \frac{1}{2} p_{1}^2
 + \frac{3}{8} p_{1} + \frac{9}{640}   \qquad \quad
f_4   \= \frac{1}{4} p_{4} -p_{2} p_{1}
 + \frac{4}{3} p_{2} \label{eq:fkpk} \\
f_5  & \= \frac{1}{5} p_{5} -p_3p_1 -\frac{1}{2} p_{2}^2 +\frac{5}{6} p_{1}^3
 - \frac{175}{48} p_{1}^2   +\frac{25}{8} p_{3} + \frac{2375}{1152} p_{1}
  + \frac{40625}{580608} \,.
&\nonumber
\end{flalign}
\par
The counting functions with and without Siegel-Veech weight for
the principal stratum in genus two and three have been given in~\eqref{eq:cm1princEX}.
By Theorem~\ref{thm:prstr_minus1} we can now confirm that 
$$c_{-1}^0(\Tr^2) \= \frac54 N^0(\Tr^2) \= \frac54 \frac{1}{25920}(5P^2 -3PQ-2R)\,.$$
The modular forms
$$ N^0(\Tr^4) \= \frac{-6P^6 + 15QP^4 + 4RP^3 - 12Q^2P^2 - 12RQP + 
7Q^3 + 4R^2}{1492992} $$
and 
$$ c_{-1}^0(\Tr^4) \= \frac{-34P^6 + 87QP^4 + 20RP^3 
- 72Q^2P^2 - 60RQP + 39Q^3 + 20R^2}{5971968}$$
are not proportional, but since the principal stratum in genus three does 
not have the non-varying property, we did not expect them 
to be proportional, either. 
\par
In the stratum $\omoduli[2](2)$ we let $\Hmu$ be
a single $3$-cycle $\sigma_3$. The Siegel-Veech constant
is given as the ratio of
\begin{flalign}
N^0(\sigma_3) & \= \bq{f_3} \=  \frac{1}{384}P^2 - \frac{1}{960}Q - \frac1{64}P 
+ \frac{9}{640} & \nonumber \\  
& \= 3x^3 + 9x^4 + 27x^5 + 45x^6 + 90x^7 + 135x^8 + 201x^9 + \cdots & \nonumber
\end{flalign}
and
\begin{flalign}
c_{-1}^0(\sigma_3) 
& \= \bq{T_{–1}f_3} - \bq{T_{-1}}\bq{f_3}  \=  \frac{10}9  N^0(\Hmu) & \nonumber\\
&\=  \frac{10}{3}x^3 + 10x^4 + 30x^5 + 50x^6 + 100x^7 + 150x^8 + \frac{670}{3}x^9 
+ \cdots & \nonumber
\end{flalign}
confirming the proportionality expected by the non-varying property.
\par
Similarly, in the stratum $\omoduli[3](3,1)$ we let $\Hmu$ consist of
a $4$-cycle $\sigma_4$ and a $2$-cycle $\Tr$. As expected we find the proportionality of
$$N^0(\sigma_4, \Tr)  \= \frac{1}{272160} \left(-35P^4 + 140P^3 + 
42QP^2 -84Q + 8RP -15Q^2 - 56R\right)$$ 
and
\begin{flalign}
c_{-1}^0(\sigma_4, \Tr) & \= \bq{T_{-1}f_4f_2} -\bq{T_{-1}}\bq{f_4f_2} = \frac{21}{16} N^0(\Hmu)\,.  & \nonumber
\end{flalign}
\par
In the stratum $\omoduli[3](2,1,1)$, the non-varying phenomenon is again confirmed by
\begin{flalign}
N^0(\sigma_3,\Tr,\Tr) &\= \bq{f_3f_2^2} - \bq{f_3}\bq{f_2^2}  & \nonumber\\
&\=\frac{1}{55296} \left(-P^5 + P^4 + 2QP^3 -3 2QP^2 - Q^2P + Q^2\right) & \nonumber
\end{flalign}
and
\begin{flalign}
c_{-1}^0(\sigma_3,\Tr,\Tr) & \= \bq{T_{-1}f_3f_2^2} - \bq{T_{-1}}\bq{f_3f_2^2}  
& \nonumber\\
& \phantom{\=} -N^0(\Tr^2)c_{-1}^0(\sigma_3) - c_{-1}^0(\Tr^2)N^0(\sigma_3) \=  
\frac{49}{36} N^0(\sigma_3,\Tr,\Tr) \,. & \nonumber
\end{flalign}
\par
The stratum $\omoduli[3](4)$ has two connected components. Both are non-varying, 
with area Siegel-Veech constants $7/5$ and $6/5$, respectively. However, 
the quasimodular forms $N^0(\sigma_5)$ and $c_{-1}^0(\sigma_5)$ are not proportional, 
since 
\begin{flalign}
N^0(\sigma_5) &\= \frac{-875 P^3 + 13125P^2 + 714Q - 49875P -3570Q -144R + 40625}
{580608} 
& \nonumber \\
c_{-1}^0(\sigma_5) &\= 
\frac{-3875P^3 + 58125P^2 + 3102Q - 219375P -15510Q -592R + 178125
}{2073600} 
\,.& \nonumber 
\end{flalign}
This is not a contradiction, since the volumes of the two components are 
not equal and our definition of Siegel-Veech constant only gives the total contribution.
\par
The same happens in the stratum $\omoduli[3](2,2)$. Again the stratum has 
two connected components, both non-varying, with different Siegel-Veech constants, and 
the quasimodular forms $N^0(\sigma_3, \sigma_3)$ and $c_{-1}^0(\sigma_3, \sigma_3)$ 
are not proportional.
\par
In \cite{eop} the volumes of the connected components of strata have been calculated
individually. The generating functions are quasimodular forms for
the subgroup $\Gamma_0(2)$ of $\SL\ZZ$. It seems likely that the
counting functions with Siegel-Veech weight $c_{-1}^0$ for these components
are also quasimodular forms for $\Gamma_0(2)$.

\section{Asymptotics of series related to Hurwitz zeta functions} \label{sec:AsHur}

In this section we apply the general results about asymptotics proved in 
the appendix 
to the special one-variable generating series that were introduced in Section~\ref{sec:OneVariable}.  
Specifically, we will prove the following asymptotic formulas for the 
coefficients of the power series $uX(u)$ and $(4u)^{m/2}\Br_{m/2}(X(u))$ ($m \in \ZZ_{\geq -1}$) 
occurring in Theorems~\ref{thm:GFvn}$-$\ref{thm:GFkappa}.

\begin{Thm} \label{thm:vnasy}
The coefficients $v_n$\ $(n \geq -2$ even$)$ defined by~\eqref{GFvn1}
have the asymptotic expansion
\be \label{eq:vnasy}
v_n\;\sim\; (-1)^{\frac n2-1}\,\frac{n!}{8\sqrt{2n}}\,\Bigl(\frac{2}{\pi}\Bigr)^{n+\frac52}\,
\Bigl(1 \,-\, \frac{2\pi^2+3}{24\,n} \+  \frac{4\pi^4-36\pi^2+9}{1152\,n^2} + \cdots \Bigr)\,, 
\ee
where the last factor is a (divergent) power series in $1/n$
with coefficients in $\QQ[\pi^2]$.
\end{Thm}
\par
\begin{Thm} \label{thm:bmasy}
For $m \in \ZZ_{\geq -1}$ the coefficients $b_m(h)$ defined by 
\bes
(4u)^{m/2}\Br_{m/2}(X(u))\= \sum_{h=0}^\infty b_m(h) \,u^{2h}
\ees
have asymptotics given by
\bes \label{eq:asymbm1}
b_{-1}(h) \; \sim \;  \,(-1)^h \,\frac{(2h)!}{h^{5/2}}\,\Bigl(\frac{2}{\pi}\Bigr)^{2h+\h}\,
\Bigl(1 \,-\, \frac{2\pi^2 + 15}{48\,h} \+ \frac{4\pi^4 + 12\pi^2 - 207}{4608\,h^2} \+ \cdots \Bigr) 
\ees
for $m =-1$ and by
\par
\be \label{eq:asymbmpos}
b_{m}(h) \; \sim \; (-1)^h \,\frac{(2h)!}{h^{3/2}} 
\Bigl(\frac{2}{\pi}\Bigr)^{2h+\h}\,
\Bigl(A_0(m) \+  \frac{A_1(m)}{h} \+  \frac{A_2(m)}{h^2} 
\+ \cdots \Bigr) 
\ee
for $m \geq 0$, where each $A_i(m)$ belongs to $\QQ[\pi^2]$.
The coefficient $A_i(m)$ has the form $A_i(m)=(-1)^im(P_i(m)-\ve_i(m))$
with $P_i(m)\in\QQ[\pi^2][m]$ and a correction term $\ve_i(m)$ 
that is non-zero only for $m \in \{1,3,\ldots,2i+1 \}$, as illustrated in
Table~\ref{cap:Peps}.
\end{Thm}

\begin{figure}[h]
$$ \begin{array}{|c|c|c|c|c|}
\hline  &&&& \\ [-\halfbls] 
i & P_i(m) & \ve_i(3) & \ve_i(5) & \ve_i(7) \\
[-\halfbls] &&&& \\ 
\hline\hline &&&& \\ [-\halfbls]
0 & \frac1{2^2} & \text{---} & \text{---} & \text{---} \\
[-\halfbls] &&&&\\
\hline &&&& \\ [-\halfbls]
1 & \frac{P-3}{2^6} & \frac{1}{2^4} & \text{---} & \text{---} \\
[-\halfbls] &&&&\\
\hline &&&& \\ [-\halfbls]
2 & m(m-5)\frac{P}{2^8} + \frac{P^2+2P+25}{2^{11}} 
  & \frac{P-15}{2^8} & \frac{3}{2^6} & \text{---} \\
[-\halfbls] &&&&\\
\hline &&&& \\ [-\halfbls]
3 & m(m-5)\frac{P^2-35P}{2^{12}}  + \frac{\frac13P^3+61P^2-735P-105}{2^{15}}
& \frac{P^2-70P+385}{2^{13}} & \frac{3P-105}{2^{10}} & \frac{15}{2^8} \\
[-\halfbls] &&&&\\
\hline
\end{array}
$$
\captionof{table}[foo]{Coefficients in the expansion of $b_m(h)$. Here $P=2\pi^2/3$.}
\label{cap:Peps}
\end{figure}
\par
 We observe that the first of these two theorems is a special case of the second,
since by Theorem~\ref{thm:GFvn} we can write $v_{n}$ not only as
the coefficient $(4n+2)u^{n+1}$ in $X(u)$ but also as the
the coefficient $24(n+1)u^{n+2}$ in $(4u)^{3/2}\Br_{3/2}(X(u))$.
We have stated it as a separate theorem, not only because it is
the most important case for our applications (to volumes of strata), but also because it
must be proved separately  and then used for the proof of
Theorem~\ref{thm:bmasy}.  The case~$m=2$ of Theorem~\ref{thm:bmasy} also includes 
Theorem~\ref{thm:vnasy}, because $\Br_1(X) \equiv X$. Besides the cases $m=2$ and $m=3$, we
also note the special cases $m=0$ and $m=1$ where $(4u)^{m/2}\Br_{m/2}(X(u))$ is identically~$1$ 
and all coefficients of the asymptotic expansion~\eqref{eq:asymbmpos} vanish.  Because of 
the latter observation, we have omitted the values of $\ve_i(1)=P_i(1)$ from Table~\ref{cap:Peps}. We also wrote the asymptotic formula for $b_{-1}(h)$ separately in Theorem~\ref{thm:bmasy} 
because this case is of special interest to us as the one giving the coefficients of the 
power series $K(u)$ in~Theorem~\ref{thm:GFkappa} related to the area Siegel-Veech constants,
and also because the asymptotic expansion in this case has a different leading power of~$h$, compared 
to the case for~$m\geq 0$.

\begin{proof} The proof consists of successive applications of the
rules for operation with power series of Gevrey class $\alpha =2$, 
as given in the appendix, using in each case the explicit values for
the small orders of which the first few were listed there.  There is one important
preliminary point.  The series $\Br_{n}(X)$ for $n \in \h\ZZ$ is a Laurent series 
in~$X^{-1/2}$, but up to a factor $X^{n}$ it is actually an even Laurent series in~$X^{-1}$.  
We therefore make the substitution $x=X^{-2}$, writing $\Br_n(X)$ as $X^n\br_n(x)$ where
$\br_n(x) = \sum (n)_{2k} \beta_{2k} x^k$, a power series  in~$x$.  This is important
because the effect of replacing $X^{-1}$ by its square root is to change the series in
question from even power series of Gevrey order~1 to power series of Gevrey order~2,
to which the results about composition and functional inverse apply. We must therefore
work with {\em three} variables $x$, $X$, and~$u$, related by 
$X = X(u) = \frac{1}{4u}  - \frac{u}{12}+\cdots $ and $x =  X^{-2}\,$.
\par
We first note that the number~$\beta_k$ equals $2/(2\pi i)^k$ to all orders for $k$ even.  (The two 
numbers differ by a factor $(1-2^{1-k})\zeta(k) = 1+\text{O}(2^{-k})$.) It follows from Stirling's 
formula that $(n)_k\beta_k$, the coefficient of~$x^{k/2}$ in $\br_n(x)$, has the asymptotic expansion  
$$ (n)_k \beta_k \; \sim \; \frac {k^{-n-1}}{\Gamma(-n)}\, \frac{2k!}{(2\pi i)^k}\,
  \Bigl(1 \+ \frac{n(n+1)}{2k}\+ \frac{n(n+1)(n+2)(3n+1)}{24k^2} \+ \cdots \Bigr) \\
$$
to all orders in~$h$ as $k=2h\to \infty$ with $n$~fixed. Note that the right hand side vanishes 
identically if $n$ is a non-negative integer, which is as it  should be since  $\br_n(x)$ is a polynomial
of degree~$n$ in this case. Note also that we can use Stirling's formula again to replace the asymptotic 
expansion on the right by one involving $h!^2$ rather than~$(2h)!$, making explicit the fact that the 
power series $\br_n(x)$ has Gevrey class~2, but the expression in terms of~$(2h)!$ is simpler and more 
convenient for the applications. We will be concerned only with the case when $n=m/2\ge-1/2$ is 
half-integral, since these are the cases occurring in Section~\ref{sec:OneVariable}, 
and the different behavior of the coefficients $A_i(m)$ for even and odd~$m$ 
is a direct consequence of this remark.
\par
Specializing the above to the case $n=1/2$ and applying the rules for reciprocals $f^{-1}$ 
from the appendix, 
we obtain the asymptotics of the coefficients of 
$16u^2 = x/\br_{1/2}(x)$ as an invertible power series in~$x$. Applying to this the rule 
for the functional inverse we obtain the asymptotics of the expansion coefficients of $x$
as an even power series in~$u$, and then applying again the rule for powers $f^\lambda$, this
time for $\lambda=-1/2$, we obtain the asymptotics of the coefficients of $X=x^{-1/2}$ as an
odd power series in~$u$. They are as given in Theorem~\ref{thm:vnasy}.

Exactly the same type of calculation gives the proof for Theorem~\ref{thm:bmasy}. 
Since we now have the asymptotics of the coefficients of both power series $\br_n(x)$ 
and $x=x(u)$, we obtain the asymptotics of the coefficients of $\br_n(x(u))$ by applying 
the rule for the composition of power series of Gevrey class~2, the asymptotics for
the series $(x/16u^2)^{-m/4}=(4uX)^{m/2}$ by applying the rules for powers to either
of the monic power series $x(u)/16u^2$ or $4uX(u)$, and the asymptotics for their product
$(4u)^{m/2}\Br_{m/2}(X(u))=(16u^2/x)^{m/4}\br_{m/2}(x)$ by applying the rule for products.
The results of these computations are the ones given in the theorem.  The difference
between the cases of odd and even~$m$, as already noted, comes from the fact that 
the power series $\br_{m/2}(x)$ terminates in the former case, so that when we apply
the rule for composition to the two series $\br_{m/2}$ and $x(u)=4u+\cdots$, the
``last" contributions in~\eqref{compos} (in the terminology explained there) all
vanish and we get only the ``first" ones.  These lead to the polynomial part $P_i(m)$
of the expansion coefficients $A_i(m)$.  For $m$ odd (and also for non-integral values
of~$m$, which we are not considering), one also has to include the ``last" contributions
in~\eqref{compos}  as well, and for a fixed odd value of~$m$ this gives a second 
infinite expansion in powers of~$1/h$ contributing to~\eqref{eq:asymbmpos}.  This
second expansion starts a little later than the first one, which is why for each 
value of~$i$ there are only finitely many odd values of $m$ for which the term~$\ve_i(m)$
is non-zero.
\end{proof}

\section{Asymptotics of Masur-Veech volumes and Siegel-Veech constants} 
\label{sec:AsyMVSV}

In this section we prove two conjectures of Eskin and Zorich on
the large genus asymptotics of the Masur-Veech volumes and the 
area Siegel-Veech constants for the principal stratum.
Our strategy, based on the results of the previous two sections, 
gives not only the top terms of the asymptotics conjectured
by Eskin and Zorich, but all terms (or as many as one is willing to
compute).
\par
We start with a discussion on the normalizations of the measure.
The Masur-Veech measure of a subset $S$ of $\oamoduli[g](m_1,\ldots,m_n)$ 
is the volume in the $N$-dimen\-sional Lebesgue measure ($N = 2g-1+n$) in period 
coordinates of the cone under $S$ in $\omoduli[g]$. The viewpoint
adopted in \cite{eo} is to define the unit cube in the lattice 
$\ZZ[i]^N \subset \CC^N$ to have volume one. We denote by  
${\rm vol}(\oamoduli[g](m_1,\ldots,m_n))$ the volumes with respect 
to this normalization.
\par
An alternative normalization (used in the key reference \cite{emz}
for Siegel-Veech constants) is to compute for $t \in \RR$
the function $\vol(S,t)$ giving the volume of the cone over
$S$ intersected with the set $\{{\rm area}(X,\omega) \leq t\} 
\subset \oamoduli[g](m_1,\ldots,m_n)$ and then to declare 
$2 \tfrac{\partial}{\partial t} \vol(S,t)$ to be the 
Masur-Veech volume of $S$. This definition mimics the relation
between the area and volume of a sphere in $\CC^N$. We denote by  
${\rm vol}_{\rm EMZ}(\oamoduli[g](m_1,\ldots,m_n))$ the volumes with 
respect to this normalization. This normalization is discussed
in \cite{zoSQ} and it is shown there that
\bes
{\rm vol}_{\rm EMZ}(\oamoduli[g](m_1,\ldots,m_n)) \= 2N
{\rm vol}(\oamoduli[g](m_1,\ldots,m_n))\,.
\ees
\par
We follow the idea of Zorich and Eskin-Okounkov (\cite{eo}) to compute volumes 
by counting lattice points with finer and finer mesh size.
It will be convenient to introduce cumulants that
involve the appropriate powers of $\pi$. Hence we define
$\lda \ell_1,\ldots, \ell_s \rda$ as the leading term (in $1/h$) of an
$h$-evaluation. More precisely, let
$$ \evh[ \langle p_{\ell_1}|\cdots |p_{\ell_s|}\rangle] \=  
\frac{1}{h^{1+\sum_{i=1}^s (\ell_i+1)}}\,\lda \ell_1,\ldots,\ell_s \rda \,(1 + {\rm O}(h))\,,$$
so that by Proposition~\ref{prop:degdrop} and~\eqref{defevh}
\be \label{cumucumu}
\lda \ell_1,\ldots,\ell_s \rda \= (-4\pi^2)^{1+\sum_{i=1}^s (\ell_i-1)/2} \lda \ell_1,\ldots, 
\ell_s \rda_\QQ\,. 
\ee
The  volumes and the cumulants for small genera are listed in Table~\ref{cap:VolAsym}, 
taken from work of Eskin and Okounkov. 
\par
\begin{figure}[h]
$$ \begin{array}{|c|c|c|c|c|c|}
\hline  &&&&& \\ [-\halfbls] 
n = 2g-2 & 2 & 4 & 6 & 8 & 10 \\
[-\halfbls] &&&&& \\ 
\hline\hline &&&&& \\ [-\halfbls]
{\rm vol }
& \frac{1}{1350}  \pi^4 &  \frac{1}{87480} \pi^6& \frac{29}{134719200} \pi^8 & \frac{23357\, \pi^{10}}{5359129776000} & 
\frac{16493303\, \pi^{12}}{179616593572416000} \\
[-\halfbls] &&&&&\\
\hline &&&&& \\ [-\halfbls]
{\rm vol_{{\rm EMZ}}} 
& \frac{1}{135}  \pi^4 &  \frac{1}{4860} \pi^6& \frac{377 }{67359600}\pi^8 
& \frac{23357\, \pi^{10}}{157621464000} &  
 \frac{16493303\, \pi^{12}}{4276585561248000} \\
[-\halfbls] &&&&&\\
\hline &&&&& \\ [-\halfbls]
\lda \underbrace{2,\ldots,2}_n \rda &\frac{16}{45}\pi^4 &\frac{1792}{27} \pi^6
&\frac{ 772096}{9} \pi^8
& \frac{10715070464\, \pi^{10}}{27}& \frac{43236204216320\, \pi^{12}}{9}\\
\hline
\end{array}
$$
\captionof{table}[foo2]{Masur-Veech volumes of the principal stratum}
\label{cap:VolAsym}
\end{figure}
\par
\begin{Prop} \label{prop:volfromcumu}
The volume of the principal stratum can be expressed
in terms of cumulants as 
\be \label{eq:cumutovol}
(4n+2)\, {\rm vol}\,(\omoduli(1^{n})) \=  {\rm vol_{{\rm EMZ}}}\,(\omoduli(1^{n})) \= 
\frac{\lda \overbrace{2,\ldots,2}^n \rda}{2^{n-1}\, (2n)! } \,.
\ee
\end{Prop}
\par
\begin{proof} The definition of connected brackets in~\eqref{slash},
and hence the definition of cumulants as their leading terms,  
are made to reproduce the passage from counting covers without unramified
components to counting connected covers in~\eqref{eq:NNfromNpr}. 
Consequently, the combination the definitions~\eqref{eq:Burnside}, \eqref{eq:NdCovd}, and~\eqref{eq:NNN}
gives $N'(\Tr^n) = \bqs{f_2^n}$ and together with $f_2 = \tfrac 12 p_2$ this implies
\be \label{eq:cumuleadN}
\evh(N^0(\Tr^n)) \= \frac{\lda \overbrace{2,\ldots,2}^n\rda}{2^n} 
\,h^{-(2n+1)}\,(1+{\rm O}(h))\,. 
\ee
The volume of the stratum can be computed as the limit as $D \to \infty$
of the number of points with period coordinates in $\ZZ[D^{-1}]$. The precise
version of this idea is the following formula by Eskin and Okounkov 
(\cite[Formula~3.2]{eo})
\be \label{eq:EOvol}
{\rm vol}\,(\omoduli(1^{n}))\= \lim_{D \to \infty} D^{-(2n+1)} \sum_{d=1}^D N^0_d(\Tr^n)\,.
\ee
The proposition now follows from Proposition~\ref{prop:evgrowth}. 
\end{proof}
\par
On the basis of numerical values obtained from the
algorithms in \cite{eo}, Eskin and Zorich made the following conjecture.
\par
\begin{Conj}[\cite{ezvol}] \label{conj}
Let 
$$V(\bfm) \= \frac{(m_1+1)(m_2+1)\cdots (m_n+1)}{4} \,\,
{\rm vol_{{\rm EMZ}}}\,(\omoduli(m_1,\ldots,m_n))\,.$$
Then $V(\bfm) = 1 + {\rm o}(1)$ as $\sum m_i =2g-2$ tends to infinity.
\end{Conj}
\par
\begin{Thm} \label{thm:volasy}
Conjecture~\ref{conj} holds for the principal stratum.
\end{Thm}
\par
\begin{proof} By~\eqref{eq:cumutovol}, the conversion~\eqref{cumucumu} 
from cumulants to rational cumulants and via~\eqref{defvn} to~$v_n$, and 
the asymptotics of $v_n$ given in Theorem~\ref{thm:vnasy} we have
\bes V(\underbrace{1,\ldots,1}_{2g-2})  \;\sim\; 
\Bigl(1\,-\, \frac{\pi^2}{24g} \,-\, 
\frac{\pi^4 - 60\pi^2}{1152g^2}\+\cdots\Bigr)  \ees
as $g \to \infty$.
\end{proof}
\par
\medskip
We now discuss the large genus asymptotics of the area Siegel-Veech constants 
$c_\area(\omoduli(1^{2g-2}))$, again restricted to the case of the principal stratum.
Values for small $g$ are given in the table below.
\par
\begin{figure}[H]
$$ \begin{array}{|c|c|c|c|c|c|}
\hline
g = \tfrac{n}2 + 1 & 2 & 3 & 4 & 5 & 6 \\
\hline &&&&& \\ [-\halfbls]
\frac{\pi^2}{3} c_{\area}(\omoduli(1^{2g-2})) & \,\frac{5}{4}\, & \, \frac{39}{28} 
\, & \frac{2225}{1508}  & 
\frac{142333}{93428} & \frac{ 102396315}{65973212}  \\
[-\halfbls] &&&&&\\
\hline
\end{array}
$$
\end{figure}
\par
The leading order in the following theorem had also been conjectured by
Eskin and Zorich (\cite{ezvol}).
\par
\begin{Thm} \label{thm:asySVarea}
For $g \to \infty$
\bes
c_\area(\omoduli(1^{2g-2}))\,\sim\, \frac12  \,-\, \frac1{{8}g} \,-\, 
\frac{5}{{{32}}g^2}  \,-\, \frac{4\pi^2 + 75}{{384}g^3}\+ \cdots  \,,
\ees
where the coefficient of $1/g^\ell$ is a polynomial in~$\pi^2$ of degree $\ell-2$ for all~$\ell \geq 2$.
\end{Thm}
\par
It is remarkable that although the individual area Siegel-Veech constants
all have a factor of $1/\pi^2$, the dominating term of the asymptotics is
rational. 
\par
\begin{proof}
We will show at the end of this section that 
\be \label{eq:careakappa}
c_\area(\omoduli(1^{2g-2})) \= -\frac{1}{8\pi^2} \,\frac{\kappa_n}{v_n}\qquad (n=2g-2),
\ee
where $\kappa_n$ and $v_n$ are as in Section~\ref{sec:OneVariable}. The assertion 
then follows immediately from the asymptotic results in Section~\ref{sec:AsHur}
since the asymptotics of $v_n$ is given in Theorem~\ref{thm:vnasy}, while the 
generating series~$K$ for the~$\kappa_n$ was expressed in 
Theorem~\ref{thm:GFkappa} in terms of $\Br_{-1/2}$, and the asymptotics of
its coefficients is given in Theorem~\ref{thm:bmasy}.
\end{proof}
\par
To prove~\eqref{eq:careakappa}, we will use the approximation of the Siegel-Veech 
constants that we gave in Proposition~\ref{prop:HurStrLim}. By the
asymptotic formula for the coefficients of a modular form in 
Proposition~\ref{prop:evgrowth} it suffices to compute the leading terms
of the $X$-evaluations of the modular forms whose coefficients
are summed up in the numerator and denominator of~\eqref{eq:HurStrL}
respectively. The denominator has been taken care of by~\eqref{eq:cumuleadN} 
and we now treat the numerator. Recall that we defined
$c_{-1}^0(\Tr^{2k})$ in Section~\ref{sec:genser} as the generating function of
covers with $(-1)$-Siegel-Veech weight and that we showed in 
Theorem~\ref{thm:prstr_minus1} that this generating series is a quasimodular 
form. 
\par
\begin{Thm} \label{thm:asySVm1}
The $X$-evaluation of the quasimodular form $c_{-1}^0(\Tr^{n})$ has 
degree~$\tfrac{n}2+1$. Its leading term $c_{-1}^0(\Tr^{n})_L = [X^{\tfrac{n}{2}+1}]\,
\evX[c_{-1}^0(\Tr^{n})]$ 
is given by 
\ba \label{eq:blam1}
c_{-1}^0(\Tr^{n})_L  &\= n!\, (-B_2)\,  
\,\sum_{k = 2}^{n} 
\frac{ k\, \lda \overbrace{2,\ldots,2 }^{n-k},k-1  \rda_\QQ}
{2^{n-k+2}\, ({n-k})!} \= -\frac{1}{24}\,\frac{n!}{2^{n}}\,\kappa_n \ea
where $B_2 = \tfrac16$ is the second Bernoulli number.
\end{Thm}
\par
From the formula for $\kappa_n$ given in Section~\ref{sec:OneVariable} 
we find the following values.
\par
\begin{figure}[H]
$$ \begin{array}{|c|c|c|c|c|c|}
\hline
n & 2 & 4 & 6 & 8 & 10 \\
\hline &&&&& \\ [-\halfbls]
c_{-1}^0(\Tr^{n})_L  & \frac{1}{144} &  -\frac{13}{144}
& \frac{2225}{288}  & -\frac{996331}{432} & \frac{  170660525}{96}  \\
[-\halfbls] &&&&&\\
\hline
\end{array}
$$
\end{figure}
\par
We introduced $p$-Siegel-Veech weight and $c_p^0(\Tr^n)$ in Part~III
as a crucial tool for interpolation and to prove the quasimodularity of $c_{-1}^0(\Tr^{n})$.
For comparison we give the analogous statement to Theorem~\ref{thm:asySVm1} for $p \geq 1$.
\par
\begin{Prop} \label{thm:asySVp}
Let $p \geq 1$ be odd and $n \geq 2$ even. Then the $X$-evaluation of the quasimodular form $c_{p}^0(\Tr^{n})$
has degree $\tfrac{n+p+1}2$ and the leading term  
$c_{p}^0(\Tr^{n})_L = [X^{\tfrac{n+p+1}2}] \, \evX[c_{p}^0(\Tr^{n})]$
is given either  in terms of the mixed cumulants~\eqref{defvnk} as
\be \label{eq:cpLuseful}
\frac{1}{n!}c_p^0(\Tr^{n})_L  \= \frac{n!}{2^n} \,\langle T_p|\underbrace{p_2|\cdots|p_2}_n \rangle_\QQ
\ee
or explicitly in terms of Bernoulli numbers by formula~\eqref{eq:blap} below. 
\end{Prop}
\par
We emphasize that, although the statements of Theorem~\ref{thm:asySVp} and Theorem~\ref{thm:asySVm1} 
are quite parallel, we cannot deduce the latter from the former, because the leading terms correspond 
to different powers of~$X$. Moreover, we cannot deduce the asymptotics of $c_{-1}^0(\Tr^{n})_L$
by extrapolation the asymptotics of $c_{p}^0(\Tr^{n})_L$ to $p=-1$, as the following corollary shows.
\par
\begin{Cor} \label{cor:cpas}
For $p \geq 1$ odd
$$c_{p}^0(\Tr^{2h})_L \,\sim\, \frac{(-1)^{h}}{\sqrt{\pi}} \,\frac{(2h)!^2}{h^{3/2} \pi^{2h}} \,\,  \cdot h^{p+1} \cdot \Bigl(\frac{2}{\pi}\Bigr)^{p+1}\,\frac{(-1)^{(p+1)/2}}{p(p+1)}
$$
as $h \to \infty$, while
$$c_{-1}^0(\Tr^{2h})_L \; \sim \;  \, \frac{(-1)^{h}}{\sqrt{\pi}} \,\frac{(2h)!^2}{h^{3/2} \pi^{2h}}  \cdot \frac{-1}{24} 
\,. $$
\end{Cor} 
\par
\begin{proof} The second line follows by~\eqref{eq:blam1} from the 
asymptotics of $\kappa_n$ that we already discussed.
\par
For the first line we use~\eqref{eq:cpLuseful}. That is, we first
compute the asymptotics of the cumulants $\langle p_{k-1}|p_2|\cdots|p_2 \rangle$
encoded in the generating series $\psi_{k}(u)$ (see~\eqref{defpsik}) 
by linearly combining with the help of~\eqref{GFpsik}
the asymptotics given in Theorem~\ref{thm:bmasy}. (The result
is stated in the introduction.) Since $T_p$ is
a quadratic polynomial in the $Q_k$'s by Theorem~\ref{thm:TpisinLambda}, 
the generating series of the cumulants we are interested in is
by Proposition~\ref{prop:PhiHOMO} a linear combination of products
of the $\psi_k$'s. Consequently, we can apply the product rule from 
the appendix to conclude.
\end{proof}
\par
\medskip
To prove the main results of this section, we form the generating 
series of Siegel-Veech constants for the principal stratum (as power series 
with quasimodular form coefficients)
\be \label{eq:Cpser}
 C'_p(u) \,:=\, \sum_{n = 0}^\infty c'_p(\Tr^{n}) \frac{u^{n}}{2^n n!} \,,
\quad  \quad C^0_p(u) \,:=\, \sum_{n = 0}^\infty c^0_p(\Tr^{n})\frac{u^{n}}{2^n n!}
\ee
for coverings without unramified components and for connected covers, where
$c_p'$ and $c_p^0$ are the generating series defined in~\eqref{eq:def:cpseries}.
By definition these power series are even and have no constant term.
Note that our notation emphasizes that so far, in Parts~I and~III, we
have been working with Siegel-Veech constants for a fixed ramification
pattern (i.e.\ in fixed genus) and we studied the generating series as
{\em the number $d$ of sheets is growing}, denoted by small letters~$c$
with appropriate decorations. Only now, the {\em number of branch points
is growing} and the corresponding generating series are denoted by
decorated capital letters~$C$.
\par
Recall that $f_2 = \tfrac{p_2}2 = Q_3$ and that by~\eqref{eq:cpviaTp} the 
series for coverings without unramified components is given for any $p \geq -1$ by
\bes  C'_p(u)\=
\sum_{n=1}^\infty \bigl(\langle T_p f_2^{n} \rangle_q - \langle T_p \rangle_q
\langle f_2^{n} \rangle_q \bigr)\,\frac{u^{n}}{2^n n!}
\= \sum_{n=1}^\infty  \langle T_{p} | p_2^n \rangle \frac{u^{n}}{n!}
\= \sum_{n=1}^\infty  \langle \widetilde{T}_{p} | p_2^n \rangle \frac{u^{n}}{n!}\,.
\ees
(This explains why we included the factor $2^{-n}$ in~\eqref{eq:Cpser}). 
For $p=-1$, using the definition~\eqref{eq:modsbqs} of the bracket
$\sbqs{\ }$ and noting $\partial_2 f_2 =0$, we have instead
the identity
\bes  C'_{-1}(u)\= \sum_{n=1}^\infty \bqs{f_2^n} \,\frac{u^{n}}{n!}\,.
\ees
Since $N'(\Tr^{n}) = \sbq{f_2^n}$ (as recalled in Proposition~\ref{prop:volfromcumu}) 
the two generating series are related by
\be \label{CpCprime}
C_p^0(u) \= \frac{C_p'(u)}{\sum_{n \geq 0} N'(\Tr^{n}) \frac{u^{n}}{2^n n!}}
\= \frac{\sbq{\,\widetilde{T}_p\,|\, \exp(u p_2)\,}}{\sbq{\,\exp(u p_2)\,}} \,,
\ee
since~\eqref{eq:cmupcmu0} specializes to this identity in the case that all elements in 
the ramification profile are equal. 
\par
\begin{Prop} The generating series of Siegel-Veech constants for the
principal stratum is given for $p>0$ by
\be \label{eq:Cpcontributions}
C^0_p(u) = \sum_{i,\,k \geq 0, \atop i+k>0} G_{p+i+1}^{(i+k)}  \frac{u^{i+k}}{2^i(i+1)!}
\sum_{m=0}^\infty \langle Q_k | \underbrace{p_2|\ldots|p_2}_{m} \rangle_q \frac{u^m}{m!} 
\ee
and for $p=-1$ by
\ba \label{eq:Cm1contributions}
C^0_{-1}(u) 
&= \sum_{i \geq 2, k \geq 0} G_{i}^{(i+k)}  \frac{u^{i+k}}{2^i(i+1)!}
\sum_{m=0}^\infty \langle Q_k | \underbrace{p_2|\ldots|p_2}_{m} \rangle_q \frac{u^m}{m!} \\
&\phantom{=} + 
\sum_{k \geq 2} G_2^{(k-1)} u^{k} \sum_{m=0}^\infty \langle Q_k | 
\underbrace{p_2|\ldots|p_2}_{m} \rangle_q \frac{u^m}{m!}
\,.  \\
\ea
\end{Prop}
\par
\begin{proof}
For $p \geq 1$ we use Theorem~\ref{thm:rhoijDiffOp} with~\eqref{eq:veryspecialcases} and the effect of
the $\rho_{i,j}$-operator on powers of $Q_3$ given in ~\eqref{eq:rhoijforQ3} 
to deduce from the preceding formulas that
\bas
C^0_p(u) 
& =  \frac {\sum_{n=0}^\infty \sum_{i\geq 0, j \geq 1} \langle  
\rho_{i,j}(p_2^{n}) \rangle_q\, 
G_{p+i+1}^{(j)}\, \frac{u^{n}}{n!}} 
{ \sum_{n=0}^\infty \langle p_2^{n} \rangle_q \,\frac{u^{n}}{n!}} \\
&= \sum_{i \geq 0, j \geq 1} G_{p+i+1}^{(j)}  \frac{u^j}{2^i(i+1)!} \frac{\sum_{n=0}^\infty  
\langle  Q_{j-i} p_2^{n} \rangle_q  \frac{u^{n}}{n!} }
{ \sum_{n=0}^\infty \langle p_2^{n} \rangle_q \frac{u^{n}}{n!}}\,.
\eas
The equality to the statement in the lemma follows from 
the definition  of cumulants.
\par
For $p=-1$ recall that by Theorem~\ref{thm:QMofbqs}
\bas \langle T_{-1} f \rangle _q & - \langle T_{-1}\rangle \langle f \rangle_q \\
& = \sum_{j \geq 1} \bigl(D^{j-1}(G_2) + \frac1{24}\delta_{j,1} \bigr) \langle \rho_{0,j}(f) \rangle_q
+ \sum_{i \geq 2, j \geq 1} D^j(G_i) \langle \rho_{i,j}(f) \rangle_q \\
\eas
and since $\partial_2 Q_3^n =0$ the extra term given by $\delta_{j,1}$ 
disappears here.
\end{proof}
\par
The leading coefficient of the expression in the preceding lemma differs 
upon $p \geq 1$ or not as we now discuss.
\par
\begin{proof}[Proof of  Theorem~\ref{thm:asySVm1} and Proposition~\ref{thm:asySVp}]
By the definition of cumulants
\be \label{eq:Qf2cumu}
\langle Q_k | \underbrace{p_2|\ldots|p_2}_{m} \rangle_X  = 
\frac{\lda \overbrace{2,\ldots,2}^{m},k-1 \rda_\QQ}{(k-1)!}\,\,  
X^{\tfrac{k+m}{2}} \, + \,
{\rm O}(X^{\tfrac{k+m}{2}-1})\,.
\ee
On the other hand, the leading term of the derivative of an Eisenstein series is
determined by
\be \label{eq:evXDG}
\langle D^{i+\ell}(G_{p+i+1}) \rangle_X \= \frac{(2i+\ell + p)!}{(p+i)!} 
\frac{-B_{p+i+1}}{2(p+i+1)}\, X^{\tfrac{p+i+1}{2}}\, 
+{\rm O}(X^{\tfrac{p+i+1}{2}-1})\,.
\ee
Consequently, the degree of the $X$-evaluation of all the summands 
in~\eqref{eq:Cpcontributions} is $k + \tfrac{p+1}{2}$ and all of them  
contribute to the leading term. Adding the contributions gives
\begin{flalign} 
\frac{1}{n!}c_p^0(\Tr^n)_L  &\= \sum_{i = 1}^{n-2} \sum_{k = 2}^{n-i} 
\frac{(2i+k + p)!}{(p+i + 1)!} \,
\frac{- B_{p+i+1}}{2^{n-k-i+1}(i+1)! (k-1)!} \frac{\lda k-1, 
\overbrace{2,\ldots,2 }^{n-i-k}  \rda_\QQ}{({n-i-k})!} & \nonumber\\
& \phantom{=} +  \frac{(2n + p)!}{(p+n+ 1)!} \frac{-B_{p+n+1}}{2^{n+1}(n+1)!} \,.
\label{eq:blap}
\end{flalign}
\par
This is the alternative formula mentioned in the proposition. The formula stated
in~\eqref{eq:cpLuseful} follows directly from~\eqref{CpCprime} and the definition of cumulants.
\par
Now we address the case $p=-1$. For all the terms with $i>0$ the preceding 
formulas are also valid in this case and contribute to the $X^k$-term. 
However, the summands in the last line of~\eqref{eq:Cm1contributions}
contributes to the $X^{k+1}$-term of the $X$-evaluation.
Applying~\eqref{eq:evXDG} and~\eqref{eq:Qf2cumu} gives the formula
in the theorem.
\end{proof}
\par
\begin{proof}[Proof of~\eqref{eq:careakappa}] By 
Proposition~\ref{prop:HurStrLim} we need to take $3/\pi^2$ times the ratio 
of the asymptotics of the sum of the coefficients of $c_{-1}(\Tr^{2k})$ and the asymptotics 
of the sum of the coefficients of $N^0(\Tr^{2k})$. By Proposition~\ref{prop:evgrowth} we can
equivalently take the ratio of the leading coefficients of $\evh$
applied to the two modular forms. Since the numerator and denominator are
of the same degree in~$h$, we can work as well with the $\evX$-images. 
The claim now follows from~\eqref{eq:blam1} and~\eqref{eq:cumuleadN}, 
together with the definition of $v_n$ in~\eqref{defvn}.
\end{proof}
\par
\newpage

%% file: appendix.tex
\begin{appendix}

\addtocontents{toc}{\protect\setcounter{tocdepth}{0}}
\section*{{\bf \Large Appendix:  Asymptotics of very rapidly divergent series}}
\renewcommand{\thesection}{A}
\addcontentsline{toc}{part}{Appendix: Asymptotics of very rapidly divergent
series} \label{sec:asyrapid}

\numberwithin{equation}{section} \setcounter{equation}{0}
\setcounter{Defi}{0}

The aim of this appendix is to study the asymptotic behaviour of
powers, inverses, functional inverses, products, and compositions
of power series whose coefficients have very rapid growth.  More
specifically, we will verify that each of these operations preserves
the class of functions having coefficients that grow like $n!^\a$,
or that have an asymptotic expansion of the form
\be \label{asym}  a_n \;\sim \; 
  n!^\a\b^nn^\g\Bigl( A_0 + \frac{A_1}{n} + \frac{A_2}{n^2} + \cdots \Bigr) \ee
for some real constants $\a>1$, $\b>0$, and~$\g \in \mathbb R$, and where
 ``asymptotic expansion" has the usual meaning that the series
 in~\eqref{asym} may be divergent but that $a_n/n!^\a\b^nn^\g$ equals
 $A_0+\cdots+A_{r-1}n^{-r+1}+\text O(n^{-r})$ as~$n\to\infty$ for any fixed~$r>0$.
For multiplication and powers we need only  $\alpha >0$
 (``rapidly divergent''), but for composition and functional inverse 
the assumption $\alpha >1$ (``very rapidly divergent'') is crucial.
For each of these operations we will give explicit formulas for the asymptotics
of the corresponding coefficients in the case~$\a=2$, which is the case
 that is of interest for 
our applications to the asymptotics of Siegel-Veech constants.
\par
The results that we give in the case of products or fixed powers may be known in the
literature, though even here we could not find any convenient reference, but 
for the cases of composition and functional inverse we could not find any reference 
at all, and it seemed best to give a
self-contained account.  Our proofs depend on a simple estimate for the coefficients of powers of 
series with coefficient growth of type~$n!^\a$, given as Lemma~\ref{naive} below.  This estimate is good 
enough for our applications, but out of curiosity we did numerical computations to study the actual asymptotic 
behavior, and since the results are of some interest we report on them briefly at the end of this appendix.

For real numbers $\a>0$, $\b>0$, and $\g \in \mathbb R$ we denote by $\Gv(\a,\b,\g)$ the class of power series (say, with
complex coefficients) $\sum a_nx^n$ whose Taylor coefficients~$a_n$ satisfy the bound $a_n=\text O(n!^\a\b^nn^\g)$
and by $\Gva(\a,\b,\g)$ the subclass for which $a_n$ has a full asymptotic development as in~\eqref{asym}. We
also write $\Gv(\a,\b)$ for $\cup_\g\Gv(\a,\b,\g)$ and $\Gv(\a)$ for $\cup_\b\Gv(\a,\b)$. (The letter~$\Gv$
stands for Gevrey, who first studied series of these types.)  We also define the class $\Gva(\a,\b)$, but
here it is too restrictive to simply take the union of the $\Gva(\a,\b,\g)$ for all~$\g\in\RR$, since this
class would not be closed under multiplication or even under addition. Instead, we define it to be the
space of power series whose coefficients have an asymptotic expansion 
\be \label{weakasym}  a_n \;\sim \; 
  n!^\a\b^n\bigl(A_0\,n^{\g_0} + A_1\,n^{\g_1} + A_2\,n^{\g_2} + \cdots \bigr) \ee
with real exponents $\g_0>\g_1>\g_2>\cdots$, $\g_i\to\-\infty$. 
In our applications all 
of the exponents $\g_i$ are rational, with bounded denominators.  Note that any two classes $\Gv(\a)$, $\Gv(\a,\b)$,   
or $\Gv(\a,\b,\g)$ have the property that one (namely, the one with the larger exponents~$(\a,\b,\g)$ in lexicographical 
order) contains the other.  Note also that in both the expansions~\eqref{asym} and~\eqref{weakasym}, we do not 
require that $A_0$, or for that matter any of the coefficients~$A_i$, be non-zero, since otherwise these classes 
would not form vector spaces, let alone rings.  This means that any space $\Gv(\a',\b')$ with $\a'<\a$ or 
with $\a'=\a$ and $\b'<\b$ can be considered as a subspace of $\Gva(\a,\b)$ (or of any $\Gva(\a,\b,\g)$) having
an expansion~\eqref{asym} or~\eqref{weakasym} with all~$A_i$ equal to~0. This is convenient because it means that 
in statements about, say, the product of two functions of these types, we can assume without loss of generality 
that both belong to the same Gevrey class, thus avoiding fussy notational distinctions.
\par
\begin{Thm} \label{closed} Let $\a>1$, $\b>0$, and $\g$ be real numbers.  Then each of the classes 
$\Gv(\a)$, $\Gv(\a,\b)$, $\Gv(\a,\b,\g)$, $\Gva(\a,\b,\g)$, and $\Gva(\a,\b)$ is closed under the operations 
\begin{itemize}
\item[(i)] addition~$(\,f(x)+g(x)\,)$, 
\item[(ii)] multiplication~$(\,f(x)g(x)\,)$,
\item[(iii)] composition~$(\,g(f(x))$, where $f(x)=x+{\rm O}(x^2)\,)$,
\item[(iv)] complex powers $(\,f(x)^r$, where~$f(x)=1+{\rm O}(x)\,)$, and
\item[(v)] functional inverse~$(\,f^{-1}(x)$, where $f(x)=x+{\rm O}(x^2)\,)$, \end{itemize}
where in the cases of $\Gva(\a,\b,\g)$ 
and $\Gva(\a,\b)$ the asymptotic expansion to any fixed order of the result of the operation depends only
on the asymptotic expansions to the same order and on a bounded number of initial values of the Taylor
coefficients of the input function or functions.
\end{Thm}
We have formulated the theorem in a purely qualitative way, without writing out the full asymptotic expansions
of the result of each of the operations, in order to keep the statement reasonably short and to emphasize the
main point, but in the course of the proof we will write out explicitly the first few terms of the asymptotics
for each operation in the case~$\a=2$.

\smallskip
{\it Sums.}  This case is trivial, since one just adds the asymptotic expansions.

\smallskip
{\it Products.}  This is the next easiest case.  Let $f$ and~$g$ belong to the Gevrey class~$\Gv(\a,\b)$
(without restriction of generality with the same~$\a$ and~$\b$, for the reasons explained above). We want to show that $fg$~also
belongs to $\Gv(\a,\b)$ and that it has an asymptotic expansion of the form~\eqref{asym} if~$f$ and~$g$ do.  Set 
$$ f(x)\= \sum_{n=0}^\infty a_nx^n\,,\qquad g(x)\= \sum_{n=0}^\infty b_nx^n\,, \qquad f(x)\,g(x)\= \sum_{n=0}^\infty c_nx^n\,.$$
It is convenient, here and in the later proofs, to introduce the rescaled variables $\wta_n=a_n/n!^\a\b^n$, and 
similarly for $\wtb_n$ and~$\wtc_n$.  Then
$$ \wtc_n \= \sum_{m=0}^n \binom nm^{\!\!-\a}\,\wta_m\,\wtb_{n-m}\,.$$
To study the asymptotics of this for large~$n$, we fix an integer $L>0$ and break up the sum into three subsums 
  (which we call ``first", ``middle", and ``last")
according to $m<L$, $L\le m\le n-L$, and $m>n-L$, respectively. It is clear that the ``first'' and ``last'' 
sums are bounded
by $\,\text O(n^\g)$ if $\wta_n$ and $\wtb_n$ satisfy this bound, and also that they have asymptotic expansions in 
$n^\g\CC[[1/n]]$ if $\wta_n$ and $\wtb_n$ do.  For instance, if $a_n$ has the expansion~\eqref{asym} with~$\a=2$ 
then the ``last'' sum has the asymptotic expansion
\bas  \sum_{m=n-L+1}^n\binom nm^{\!\!-2}\, \wta_m\wtb_{n-m}
  & \=\wtb_0\,\Bigl(A_0n^\g+A_1n^{\g-1}+A_2n^{\g-2}+\cdots\Bigr) \\
   &\quad \+\frac{\wtb_1}{n^2}\,\Bigl(A_0(n-1)^\g+A_1(n-1)^{\g-1}+\cdots\Bigr) \\
   &\quad \+\frac{4\,\wtb_2}{n^2(n-1)^2}\,\Bigl(A_0(n-2)^\g+\cdots\Bigr) \+ \cdots \\
  \= n^\g\,\biggl(A_0\wtb_0 \+\frac{A_1\wtb_0}n &\+ \frac{A_2\wtb_0+A_0\wtb_1}{n^2}
   \+ \frac{A_3\wtb_0+A_1\wtb_1-\g A_0\wtb_1}{n^3}\+ \cdots\biggr)\eas
as $n\to\infty$, and if $b_n$ has an expansion like~\eqref{asym} with $A_i$ replaced by~$B_i$ then 
the ``first'' sum is given by a similar expression with $A_i$ and $\wtb_j$ replaced by~$B_i$ and~$\wta_j$.
For the ``middle'' sum, we note that because each row of Pascal's triangle is unimodal (rising to a maximum and then falling), 
we have $\binom nm\le\binom nL=\text O_L(n^L)$ for $L\le m\le n-L$ and hence 
$\sum_{m=L}^{n-L}\wta_m\wtb_{n-m} \= \text O_L(n^{2\g+1-\a L})$, which is smaller
than any fixed negative power of~$n$ if $L$ is large enough. It follows that the coefficients~$c_n$ have
an asymptotic expansion of the same form~\eqref{asym} with the same parameters $\a$, $\b$, and~$\g$ as for
$a_n$ and~$b_n$, the beginning of this expansion being
\bes  n!^2\b^nn^\g\Bigl(A_0b_0+a_0B_0 + \frac{A_1b_0+a_0B_1}n
+ \frac{A_2b_0+a_0B_2 + (A_0b_1+a_1B_0)/\b}{n^2} + \cdots \Bigr) \ees
in the case $\a=2$.
\par
\smallskip  
{\it Compositions.} Since we can only compose two series if the second one has zero constant term, we
will write our composed power series as $g(xf(x))=\sum c_nx^n $ for some power series $f=\sum a_nx^n$ and $g=\sum b_nx^n$.
We assume that $a_0\ne0$ and can further assume (by replacing the power series $g(x)$ by $g(a_0x)$) that~$a_0=1$. Then 
\be\label{compos}  c_n \= [x^n]\bigl(g(xf(x))\bigr) \= \sum_{k=1}^n b_k\,a_{n-k}^{(k)}\qquad(n\ge1)\,, \ee
where the coefficients $a_m^{(k)}$ ($m\ge0$) are defined by the generating series
\be \label{ankdef}   \sum_{m=0}^\infty a_m^{(k)}\,x^m \= f(x)^k \= 1\+ ka_1\,x \+ \Bigl(ka_2+\frac{k(k-1)a_1}2\Bigr)\,x^2\+\cdots \;. \ee
Now we want to apply the same decomposition ``first + middle + last" of the sum in~\eqref{compos} as we did for multiplication,
with the first and last terms of the sum dominating the whole sum for $n$~large.  But unlike the case of multiplication, 
where it would have sufficed to assume $\a>0$, here the assumption $\a>1$ is crucial.  For instance,
if $\a=1$ then the last two terms $b_n$ and $(n-1)a_1b_{n-1}$ of the sum have the same
order of magnitude, and if $\a<1$ then each successive term starting at the end is actually larger than its
predecessor, so that we do not get the desired asymptotic expansion.  If, on the other hand, $\a$ is larger than~1, 
then it is clear from the expressions for the first few $a_m^{(k)}$ as given in~\eqref{ankdef} that each of the first
and last terms of the sum~\eqref{compos}, counting from the ends, is of a smaller order than its predecessor, so that
the ``first" and ``last" subsums have well-defined asymptotic expansions by the same principle as we used for products.
But this is not enough for our purposes.  We are assuming that both $f$ and $g$ belong to the same growth class
$\Gv(\a,\b,\g)$, and since we have already proved that this class is closed under multiplication, it follows that the
coefficients $a_m^{(k)}$ have the same order of growth $\,\text O(m!^\a\b^mm^\g)\,$ as $m\to\infty$ for each {\it fixed}~$k$,
but since the summand~$k$ in~\eqref{compos} goes all the way up to~$n$ we need an estimate that is uniform in~$k$.  
Such an estimate is provided by the following lemma, which, as already mentioned, is not sharp but is sufficient 
for proving the required growth properties of the coefficients~$c_n$.  We will formulate this lemma in detail for the 
specific growth estimate $|a_n|\le n!^\a$, in order to keep its statement and proof short and clean, and then indicate 
briefly afterwards the modifications needed for the general case.
\begin{Lemma} \label{naive} Suppose that $f(x)=\sum_{n=0}^\infty a_nx^n$ with $|a_n|\le n!^\a$ for all~$n\ge0$ for 
some~$\a\ge1$. Then the coefficients $a_n^{(k)}$ defined by~\eqref{ankdef} satisfy the estimates
\be\label{prodest} \bigl|a_n^{(k)}\bigr|\;\le\;C(\a)^{k-1}\,n!^\a\,, \qquad 
  \bigl|a_n^{(k)}\bigr|\;\le\;\frac{(n+k-1)!\,n!^{\a-1}}{(k-1)!} \ee
for all $n\ge0$ and $k\ge1$, where $C(\a)$  $($e.g.~$C(1)=\frac83,\;C(2)=\frac94)$ denotes the maximum over all 
integers~$n\ge1$ of the quantity~$\sum_{m=0}^n{\binom nm}^{-\a}\,$.  
\end{Lemma}
\begin{proof} We rewrite the estimates~\eqref{prodest} as $\bigl|\wta_n^{(k)}\bigr|\le C(\a)^{k-1}$ 
and $\bigl|\wta_n^{(k)}\bigr|\le\binom{n+k-1}n$, where $\wta_n^{(k)}:=n!^{-\a}a_n^{(k)}$.
Both of them follow by induction on~$k$: the case~$k=1$ (i.e. $\bigl|\wta_n^{(1)}\bigr|\le1$) is true by assumption, 
and if~\eqref{prodest} holds for all~$n\ge0$ then from
$$ \bigl|\wta_n^{(k+1)}\bigr| \=  \biggl|\,\sum_{m=0}^n{\binom nm}^{-\a}\,\wta_m^{(k)}\,\wta_{n-m}^{(1)}\,\biggr|
  \;\le\; \sum_{m=0}^n{\binom nm}^{-\a}\,\bigl|\wta_m^{(k)}\bigr|$$
we get the upper bounds
$$ \bigl|\wta_n^{(k+1)}\bigr|  \;\le\;C(\a)^{k-1}\sum_{m=0}^n{\binom nm}^{-\a} \;\le\;C(\a)^k $$
and
\be\label{binom} \bigl|\wta_n^{(k+1)}\bigr| \;\le\;\sum_{m=0}^n \binom{m+k-1}m\=\binom{n+k}n  \ee
as required. \end{proof}
For the general case, we first note that if a series $f=\sum a_nx^n$ with $a_0=1$ belongs to $\Gv(\a,\b)$ for some real 
numbers $\a\ge1$, $\b>0$, then its coefficients can be estimated by both $|a_n|\le n!^\a\b^n(n+1)^c$ and
$|a_n|\le n!^\a\b^n\binom{n+c}n$ for some integer~$c\ge0$.  We then replace the two 
estimates~\eqref{prodest} by two different estimates involving these two different hypotheses, namely
 \be\label{prodest2} |a_n|\le n!^\a\b^n(n+1)^c \;\Rightarrow\; \bigl|a_n^{(k)}\bigr|\;\le\;C^{k-1}\,n!^\a\,\b^n\,(n+1)^c\,, \ee
 \be\label{prodest3} |a_n|\le n!^\a\b^n\binom{n+c}c \;\Rightarrow\; \bigl|a_n^{(k)}\bigr|\;\le\;k^C\,n!^\a\,\b^n\,\binom{n+k+c-1}n \ee
for some sufficiently large constant $C$ depending only on~$\a$ and~$c$.  The proof of~\eqref{prodest2} mimics the one in 
the lemma, with $\wta_n^{(k)}$ defined as $a_n/n!^\a\b^n(n+1)^c$ and $C$~defined as the maximum of 
$\sum_{m=0}^n\binom nm^{-\a}\,\bigl(\frac{(m+1)(n-m+1)}{n+1}\bigr)^c$ over all~$n\ge0$.  For~\eqref{prodest3} we 
first note that the case $n<c$ is trivial (even with $k^C$ in~\eqref{prodest2} replaced by a constant $2^C$ for all $k\ge2$)
since $a_k^{(n)}$ for $n$ fixed is a polynomial of degree~$n$ in~$k$.  For $n\ge c$ we define $\wta_n^{(k)}$
as $a_n/n!^\a\b^n$, so that $\wta_\ell^{(1)}\le\binom{\ell+c}\ell\le\binom n\ell^\a$ for~$n\ge\ell+c$, 
and then use the induction assumption and~\eqref{binom} with $k$ replaced by~$k+c$ to obtain the upper bound 
\bas \bigl|\wta_n^{(k+1)}\bigr| & \le\;\sum_{m=0}^n{\binom nm}^{-\a}\,\bigl|\wta_{n-m}^{(1)}\,\wta_m^{(k)}\bigr|
  \;\le\; k^C\,\sum_{m=c}^n \binom{m+k+c-1}m \+ \text O\bigl(n^ck^{c-1}\bigr) \\
 & \le\; k^C\,\binom{n+k+c}n \,\Bigl(1\+\text O\bigl(1/k\bigr)\Bigr) \;\le\; (k+1)^C\binom{n+k+c}n\eas
for sufficiently large $C$ and all $k\ge1$. 

\smallskip

Using these estimates, we find easily that the sum of the ``middle" terms in the sum in~\eqref{compos}
is of smaller order of magnitude than the first and last terms.  (More precisely, for any $H>0$ there
is a constant $K$ depending on~$H$ such that each term with $K<k<n-K$ in~\eqref{compos} is $\text O(n^{-H})$
times the dominant asymptotic $n!^\a\b^nn^{\g_0}$ for $n$ sufficiently large if $a_n$ and $b_n$ both satisfy estimates 
of the type~\eqref{weakasym}, and this implies the assertion since the number of terms in the sum is also 
bounded by~$n$.) As an explicit example, if $a_n$ has an asymptotic expansion of the form~\eqref{asym} 
with $\a=2$, $\b=1$, $\g=0$ and $b_n$ has an asymptotic expansion of the same form with $A_j$ replaced by~$B_j$,
then the coefficient $c_n$ of $g(xf(x))$ has the asymptotic expansion
$$  c_n \;\sim\; n!^2\,\Bigl(B_0 \+ \frac{B_1+a_1B_0}n
\+ \frac{B_2+a_1(B_1-B_0)+a_1^2B_2/2+b_1A_0}{n^2} \+ \cdots\Bigr) $$
as $n\to\infty$.
\par
{\it Arbitrary powers.} This is a special case of the preceding case, but important enough
to be stated separately.  If $f(x)=1+\cdots$ belongs to the class $\Gv(\a,\b)$ and $c$ and 
$\lambda$ are arbitrary complex numbers, then one can obtain $f^\lambda$ by writing $f(x) = 1+cxf_1(x)$
with $f_1(x) = 1+ \text{O}(x)$ and applying the previous result for $g(xf_1(x))$ to
the power series $g(x)=(1+cx)^\lambda=\sum_k\binom\lambda kc^kx^k$.
Here only the ``first'' coefficients (corresponding to small~$k$) contribute, 
because the power series~$g$ is of Gevrey order zero.
As an explicit example, if $f(x)=\sum a_nx^n$ with $a_0=1$ satisfies~\eqref{asym} 
with $\a=2$, then
the coefficients~$a_n^{(\lambda)}$ of $f(x)^{\lambda}$ have the asymptotic expansion
$$  a_n^{(\lambda)} \;\sim\; n!^2 \beta^n n^\g \,\Bigl(\lambda A_0 + \frac{\lambda A_1}n
+ \frac{\lambda A_2 + a_1 \lambda (\lambda - 1) A_0 /\b}{n^2} \+ \cdots\Bigr) $$
as $n\to\infty$.

{\it Inverse power series.} Let $h(x)$ be a power series beginning with $x$ belonging to the Gevrey class
$\Gv(\a,\b)$ with some $\a>1$.  We want to show that the inverse power series $h^{-1}(x)$ also
belongs to this class, and to give an explicit formula for the asymptotic expansion of its
coefficients. Write $h(x)=x-F(x)$ where $F(x)=\text O(x^2)$. Then the inverse power series is given by
$$ h^{-1}(x) \= x \+ \sum_{k=1}^\infty \frac1{k!}\,\frac{d^{k-1}}{dx^{k-1}}\,F(x)^k $$
by one of the forms of the Lagrange inversion formula (essentially the same one as we already used 
in the proof of Proposition~\ref{Prop10.4}), so if we write $F(x)=cx^{r+1}f(x)$ with $r\ge1$, $c\ne0$, 
and $f(x)$ a power series beginning with~1, and define coefficients $a_m^{(k)}$ by~\eqref{ankdef}, then 
$$ h^{-1}(x) \= x\+\sum_{n=2}^\infty c_nx^n\,, \qquad 
  c_n \= \sum_{0<k<n/r} \frac{c^k}k\,\binom{n+k-1}{k-1}\,a^{(k)}_{n-rk-1}\,. $$
We can now apply the same estimates as for the case of composition (Lemma~\ref{naive} and its extensions)
to show that the asymptotic expansion of $c_n$ is given to any given order by summing the first $\,\text O(1)$ terms of this sum.  Once again, 
only the ``first'' coefficients ($k$~small) contribute, and we give as a concrete example the expansion
of $c_n$ for $a_n$ satisfying~\eqref{asym} with $\a=2$, namely, 
$$ c_n \;\sim\; n!^2  \b^n n^{\g -4}\,\Bigl(cA_0 + \frac{cA_1 + 2c(c+2)A_0}n \+ \cdots\Bigr)\,. $$
This completes the proof of Theorem~\ref{closed}.
\par
\medskip
{\it True asymptotics.} As mentioned in the introductory paragraphs, we end this appendix 
by describing the complete asymptotic behavior of the coefficients $a_n^{(k)}$ defined 
by~\eqref{ankdef}  when $n$ and $k$ tend to infinity independently of one another, 
even though this is not used in the paper, 
because it is surprisingly subtle and because finding it even numerically is not easy.  We will concentrate 
on the special but typical case $a_n=n!^2$.  We will work with the renormalized values $\wta_n^{(k)}=a_n^{(k)}/n!^2$ 
as before, since these are bounded as functions of~$n$ for fixed~$k$, and will also describe the large~$k$ asymptotics 
of the numbers $M_k=\max\limits_n\wta_n^{(k)}$.  We have not given complete analytic proofs of all results.

We first consider small~$n$. The coefficient $a_n^{(k)}$ for $n$ fixed is a polynomial in~$k$ of degree~$n$ 
with leading term $k^n/n!\,$, the first values being given by
\bas \sum_{n=0}^\infty a_n^{(k)}\,x^n &\=  \bigl(1\+x\+4x^2\+36x^3\+\cdots\bigr)^k \\
  &\= 1 \+ k\,x \+ \frac{k^2+7k}2\,x^2 \+ \frac{k^3+21k^2+194k}6\,x^3 \+ \cdots \;. \eas
This gives the asymptotics of $a_n^{(k)}$ for $n$ fixed, e.g.
$$ \wta_3^{(k)} \= \frac{k^3}{6^3}\,\biggl(1\+\frac{21}{k}\+\frac{194}{k^2}\biggr) 
\= \frac{k^3}{6^3}\,\exp\biggl(\frac{21}{k}\,-\,\frac{53}{2k^2}\,-\,\frac{987}{k^3}\+\cdots\biggr)$$
and for general~$n$
\ba \label{1stapprox}
  \wta_n^{(k)} \;\sim\; A_1(n,k) \;:=\; \frac{k^n}{n!^3}\,&\exp\Bigl[\frac{n(n-1)}k\,
 \Bigl(\frac72\+\frac{94n-335}{12k} \\   & \qquad \+\frac{1711n^2-11215n+16272}{12k^2}\+\cdots\Bigr)\Bigr]\;, \ea
as one sees by writing $f(x)^k$ as $\exp(k\log f(x))$ and expanding the power series.  This approximation is valid not 
only for $n$ fixed and $k\to\infty$, but also for large~$n$ and experimentally gives the correct asymptotic behavior
of~$\wta_n^{(k)}$ as long as $n \ll k^{1/3}$.
\par
At the opposite extreme, when $n$ tends to~$\infty$ with $k$ fixed, we have a quite different 
asymptotic expansion.  The estimates for products given above show that when we write
$a_n^{(k)}$ as a sum of products of $k$ coefficients $a_{n_i}$, the dominant terms are those
where all but one of the $n_i$ are bounded, so 
$$ a_n^{(k)} \;\sim\;\sum_{r\ge0}\Biggl(\sum_{j=1}^k\sum_{n_1,\dots,n_k\ge0\atop 
 {n_1+\dots+\widehat{n}_j+\cdots+n_k=r\atop n_j=n-r}}a_{n_1}\cdots a_{n_k}\Biggr)
= k\,\sum_{r\ge0} a_{n-r}a_r^{(k-1)}$$
in the sense that for any $C>0$ the sum of the terms on the right with $0\le r\le R$
approximates $a_n^{(k)}$ to within a relative error of $\,\text O(n^{-C})$ if $R$ is
sufficiently large.  Thus in the case $a_n=n!^2$ we find
\bas \wta_n^{(k)} &\;\sim\; k\,\Bigl(1\+\frac{k-1}{n^2}\+\frac{(k-1)(k+5)}{2n^2(n-1)^2} 
    \+\frac{(k-1)(k^2+19k+174)}{6n^2(n-1)^2(n-2)^2}\+\cdots\Bigr) \\
  &\= k\,\exp\Bigl[\frac{k-1}{n^2}\Bigl(1+\frac7{2n^2}+\frac{k+6}{n^3}+\frac{9k+248}{6n^4}
 +\cdots\Bigr)\Bigr]\,.\eas
The series in $\Q[k][[1/n]]$ occurring in the exponent in the last expression on the right
is an asymptotic series (in the sense that 
there are only finitely many terms of order greater than $n^{-C}$ for any $C>0$)
not only for $k$ fixed but as long as $k\ll n^3$,  and in that range it continues (experimentally)
to give the correct asymptotic expansion of $\wta_n^{(k)}$ to all orders in $1/n$.  If $k$ has the 
same order of magnitude as~$n^3$, then the series contains infinitely many terms of any given order 
in $1/n$.  If we collect them together we get the expansion
\be \label{2ndapprox}
  \wta_n^{(k)} \;\sim\; A_2(n,k) \;:=\; k\,\exp\Biggl(\sum_{i=-1}^\infty G_i\Bigl(\frac k{n^3}\Bigr)\,n^{-i}\Biggr)
 \ee
where the $G_i(t)$ are power series with radius of convergence $\frac4{27}$, the first few being
\bas G_{-1}(t) &\= t+t^2+\frac73 t^3 + \frac{15}2t^4 + \frac{143}5t^5 + \frac{364}3t^6 + \frac{3876}7t^7 + \cdots\,,\\ 
G_0(t) &\= \frac32t^2 + 10t^3 + \frac{243}4t^4 + 366t^5 + 2218t^6 + 13554t^7 + \cdots\,,\\
G_1(t) &\= \frac72t + 16t^2 + 94t^3 + \frac{1271}2t^4 + \frac{9141}2t^5 + 33608t^6  + \cdots\,,\\
G_2(t) &\= -1 + 5t + \frac{131}2t^2 + 621t^3 + \frac{11209}2t^4 + 50042t^5  + \cdots \,.
 \eas
We can easily recognize the coefficients of $G_{-1}(t)$ and then use Lagrange inversion to write it in closed form:
$$ G_{-1}(t) \= \sum_{n=1}^\infty\frac{2\,(3n-2)!}{n!\,(2n)!}\,t^n \= 3a\+2\log(1-a)\,, $$
where $a=t + 2t^2 + 7t^3 + 30t^4 + 143t^5 +\cdots$ is related to~$t$ by
\be \label{ta}  t \= a\,(1-a)^2 \qquad \text{with $0\,<\,a\,<\,\frac13\,$.} \ee
Making the same substitution in the other $G_i$, we can recognize them too:
\bas G_0(t)& \,=\,\log\frac{(1-a)^{3/2}}{(1-3a)^{1/2}}\,, \;\quad
G_1(t) \,=\, \frac{7a-59a^2+191a^3-204a^4+9a^5}{2(1-a)^2(1-3a)^3}\,, \\
G_2(t) & \,=\, \frac{-2+50a-433a^2+1884a^3-4065a^4+4122a^5-1458a^6}{2(1-a)^2(1-3a)^6}\,,\;\cdots\,. \eas
(Rigorous proofs of each of these expansions are not hard to give.)
\par
\smallskip
We have now found two approximations $A_1(n,k)$ and $A_2(n,k)$ to $\wta_n^{(k)}$, the first of which
makes sense as an asymptotic series to all orders if $n\ll k^{1/2}$ and is (experimentally) correct
to all orders if $n\ll k^{1/3}$ and the second of which makes sense as an asymptotic series to all orders 
if $n > ck^{1/3}$ for any $c>2^{-2/3}3$ and is (experimentally) correct to all orders if $n\gg k^{1/3}$.
In the transition region where $k=tn^3$ for fixed $t\in(0,\frac4{27})$, we have
$$ A_1(n,k)\;\sim\;\frac{k^n}{n!^3} \= \frac{(tn^3)^n}{n!^3}\;\sim\; (2\pi)^{-3/2}\cdot n^{-3/2}\cdot (e^3t)^n \qquad (k=tn^3 \to \infty)$$
by Stirling's formula and
$$A_2(n,k)\;\sim\;C(t)\cdot n^3\cdot B(t)^n \qquad (k=tn^3 \to \infty)$$
by the formulas given above, where $B(t)$ and $C(t)$ are given by
$$  B(t)\= e^{G_{-1}(t)} \= (1-a)^2e^{3a},\qquad C(t)\=te^{G_0(t)}\=a\,\frac{(1-a)^{7/2}}{(1-3a)^{1/2}}$$
with $a$ and~$t$ related by~\eqref{ta}. Thus $A_1(n,tn^3)$ is exponentially larger than $A_2(n,tn^3)$
for $t>t_0$ fixed and $n\to\infty$, and $A_2(n,tn^3)$ is exponentially larger than $A_1(n,tn^3)$
for $t<t_0$ fixed and $n\to\infty$, where $t_0=0.0526457\cdots$ is the unique solution in $(0,\frac4{27})$
of the equation $B(t)=e^3t$, given by $t_0=a_0(1-a_0)^2$ where $a_0=0.0595202\cdots$ is the unique solution 
in $(0,\frac13)$ of the equation $e^{3a-3}=a$.  Near $k=t_0n^3$ both approximations have the same order
of magnitude and the true value of $\wta_n^{(k)}$ is given to high accuracy by their sum.  Thus our final
heuristic asymptotic formula is that
$$ \wta_n^{(k)} \;\sim\; \begin{cases} A_2(n,k) & \text{for $k/n^3 <t_0-\varepsilon$,} \\
A_1(n,k)+A_2(n,k) & \text{for $t_0-\varepsilon<k/n^3<t_0+\varepsilon$,} \\ 
A_1(n,k) & \text{for $k/n^3>t_0+\varepsilon$} \end{cases} $$
to all orders in $n$.  That this works well in practice is illustrated by the following table,
in which $k=50000$ is fixed and we let $n$ vary near $\sqrt{k/t_0}=98.29\cdots\,$:

\begin{table}[h]
\begin{tabular}{ |c|c|c|c|c|}
\hline
$n$ & $\wta_n^{(k)}$ &  $A_1(n,k)/\wta_n^{(k)}$ & $A_2(n,k)/\wta_n^{(k)}$ & sum \\
\hline
80 & $3.517 \,\times\,10^{19}$ & 1.00000000  & 0.00000000 & 1.00000000 \cr
85 & $1.909 \,\times\,10^{14}$ & 0.99999944  & 0.00000056 & 1.00000000 \cr
90 & $4.732 \,\times\,10^8$    & 0.91404303  & 0.08595697 & 1.00000000 \cr
91 & $6.347 \,\times\,10^7$    & 0.45791582  & 0.54208418 & 1.00000000 \cr
92 & $3.119 \,\times\,10^7$    & 0.06059598 &  0.93940402 & 1.00000000 \cr
93 & $2.524 \,\times\,10^7$    & 0.00471615  & 0.99528385 & 1.00000000 \cr
94 & $2.167 \,\times\,10^7$    & 0.00033500  & 0.99966500 & 1.00000000 \cr
95 & $1.879 \,\times\,10^7$    & 0.00002283 &  0.99997717 & 1.00000000 \cr
100 & $9.945 \,\times\,10^6$   & 0.00000000  & 1.00000000 & 1.00000000 \cr
\hline
\end{tabular}\medskip
\end{table}
\noindent Here, of course, the last columns of the table are not rigorously defined, since
both $A_1(n,k)$ and $A_2(n,k)$ are given only by divergent asymptotic series, but in both cases
the approximations obtained by breaking off the series after a few terms is insensitive (to high
order) to where we break it off: for the values in the table, the numbers $A_1(n,k)/\wta_n^{(k)}$
and $A_2(n,k)/\wta_n^{(k)}$ have the value given to the indicated number of digits if we 
take $m$ terms of the defining series for any $m$ between 7 and~57.
\par
\smallskip
Finally, if the above asymptotics are correct, then we can give the precise asymptotics
of the optimal constant $M_k=\max\limits_n\wta_n^{(k)}$ in the uniform estimate $a_n^{(k)}\le M_kn!^2$
for~$k$ fixed and all~$n$: this value is attained for $n=k^{1/3}+\text O(1)$ and is given by 
$$ M_k \= \frac{\exp(3k^{1/3})}{(2\pi)^{3/2}\,k^{1/2}}\, \bigl(1+\,\text O\bigl(k^{-1/3}\bigr)\bigr)\,,$$
whereas Lemma~\ref{naive} gave only the much cruder estimate $M_k < (9/4)^{k-1}$.
\end{appendix} 